\documentclass[english]{amsart}
\usepackage[final]{microtype}
\usepackage{caption}
\RequirePackage{amsmath}
\RequirePackage{amssymb}
\RequirePackage{amsxtra}
\RequirePackage{amsfonts}
\RequirePackage{latexsym}
\RequirePackage{euscript}
\RequirePackage{amscd}
\RequirePackage{amsthm}
\usepackage[all,cmtip]{xy}

\usepackage[T1]{fontenc}

\usepackage{stmaryrd}
\usepackage{mathtools}

\RequirePackage[dvipsnames,usenames]{xcolor}
\usepackage{ifpdf}
\ifpdf \usepackage[colorlinks,bookmarksopen,bookmarksdepth=2, linkcolor=Mahogany, citecolor=ForestGreen]{hyperref} \fi

\usepackage{mathrsfs}
\usepackage{tabu}
\usepackage{tikz-cd}
\usetikzlibrary{arrows.meta,calc,positioning,decorations.pathreplacing}
\usepackage[shortlabels]{enumitem}
\usepackage{enumitem}

\usepackage{relsize}
\usepackage[bbgreekl]{mathbbol}
\usepackage{dsfont}

  \usepackage[
    style=ext-alphabetic,
    backend=biber,
    doi=false,
    isbn=false,
    maxalphanames=99,
    maxnames=99,
    articlein=false
  ]{biblatex}

  \DeclareFieldFormat{language}{}
  \renewbibmacro*{language}{}

  \AtEveryBibitem{%
    \clearfield{language}%
    \ifboolexpr{
      test {\ifentrytype{article}}
      or test {\ifentrytype{book}}
      or test {\ifentrytype{incollection}}
    }
      {\clearfield{url}}
      {}%
    \iffieldundef{eprint}
      {}
      {\clearfield{url}}%
  }

  \DefineBibliographyStrings{english}{page={}, pages={}}
  \DefineBibliographyStrings{french}{page={}, pages={}}
  \DefineBibliographyStrings{german}{page={}, pages={}}

\newcommand{\explicit}{{\operatorname{explicit}}}
\newcommand{\Herzig}{{\operatorname{Her}}}

\newcommand{\WHer}{W^{\Herzig}}
\newcommand{\KisinS}{\mathfrak{S}}
\newcommand{\fib}{\mathrm{fib}}
\newcommand{\Vect}{\mathrm{Vect}}
\newcommand{\refl}{\mathrm{refl}}
\newcommand{\specializable}{\mathrm{sp}}

\newcommand{\Hk}{\mathcal{H}\mathrm{k}}
\newcommand{\pHk}{p\text{-}\mathcal{H}\mathrm{k}}
\newcommand{\rotorder}{N}
\newcommand{\rotHk}{\rotorder\text{-}\Hk}

\newcommand{\Ghat}{\widehat{G}}
\newcommand{\Hhat}{\widehat{H}}
\newcommand{\Bhat}{\widehat{B}}
\newcommand{\Phat}{\widehat{P}}
\newcommand{\Mhat}{\widehat{M}}

\newcommand{\canonicalmapnotation}{\wp}
\newcommand{\can}{\operatorname{can}}
\newcommand{\Mrhobarss}{M(\rhobar^{\semis})}

\DeclareSymbolFontAlphabet{\mathbb}{AMSb} %
\DeclareSymbolFontAlphabet{\mathbbl}{bbold}

\makeatletter%
\@mparswitchfalse%
\makeatother%
\normalmarginpar%

\usepackage{xcolor}

\renewcommand{\mathbb}{\mathbf}

\newcommand{\Bru}{\operatorname{Bru}}
\newcommand{\dom}{\operatorname{dom}}
\newcommand{\Flag}{\mathcal{F}\ell}
\newcommand{\Waff}{W_{\operatorname{Aff}}}
\newcommand{\Wt}{\widetilde{W}}

\newcommand{\gMbar}{\overline{\mathfrak{M}}}
\newcommand{\gNbar}{\overline{\mathfrak{N}}}

\newcommand{\Nyg}{\operatorname{Nyg}}
\newcommand{\Syn}{\operatorname{Syn}}
\newcommand{\ZpSyn}{\Z_{p}^{\Syn}}

\newcommand{\ZpNyg}{\Z_{p}^{\Nyg}}

\DeclareMathOperator{\colim}{\mathop{colim}}

\newcommand{\Prism}{{\mathlarger{\mathbbl{\Delta}}}}
\newcommand{\prism}{\Prism}
\newcommand{\Triv}{\text{\rm triv}}

\def\A{\mathbb A}

\def\C{\mathbb C}
\def\F{\mathbb F}

\def\N{\mathbb{N}}
\def\Q{\mathbb{Q}}
\def\R{\mathbb{R}}

\def\Z{\mathbb{Z}}

\def\Gr{\mathrm{Gr}}
\def\Mod{\mathrm{Mod}}

\def\id{\mathrm{id}}
\def\tf{\widetilde{f}}

\def\alg{\mathrm{alg}}

\def\red{\mathrm{red}}

\def\ss{\mathrm{ss}}

\def\GL{\operatorname{GL}}

\def\Gal{\mathrm{Gal}}

\def\Aut{\mathrm{Aut}}

\def\Ext{\mathrm{Ext}}

\def\End{\mathrm{End}}

\def\Hom{\mathop{\mathrm{Hom}}\nolimits}

\def\Spec{\mathop{\mathrm{Spec}}\nolimits}
\def\Spf{\mathop{\mathrm{Spf}}\nolimits}

\def\Ind{\mathop{\mathrm{Ind}}\nolimits}

\def\Fil{\mathop{\mathrm{Fil}}\nolimits}

\def\rhobar{\overline{\rho}}

\def\crys{\mathrm{crys}}

\def\dR{\mathrm{dR}}

\def\That{\widehat{T}}

\def\triv{\mathds{1}}
\def\ev{\mathrm{ev}}

\newcommand{\onto}{\twoheadrightarrow}

\newcommand{\into}{\hookrightarrow}

\newcommand{\To}{\longrightarrow}
\newcommand{\isoto}{\stackrel{\sim}{\To}}

\newlength{\ownl}

\newcommand{\diag}{{\operatorname{diag}}}

\newcommand{\gr}{{\operatorname{gr}\,}}

\newcommand{\Id}{{\operatorname{Id}}}

\newcommand{\Iw}{{\operatorname{Iw}}}

\newcommand{\Rep}{{\operatorname{Rep}}}

\newcommand{\Res}{{\operatorname{Res}}}
\newcommand{\rk}{{\operatorname{rk}\,}}

\newcommand{\Gm}{{\mathbb{G}_m}}

\newcommand{\LG}{{{}^{L}G}}
\newcommand{\LH}{{{}^{L}H}}
\newcommand{\LM}{{{}^{L}M}}
\newcommand{\LP}{{{}^{L}P}}

\newcommand{\cris}{{\operatorname{cris}}}

\newcommand{\der}{{\operatorname{der}}}

\newcommand{\semis}{{\operatorname{ss}}}

\newcommand{\univ}{{\operatorname{univ}}}

\newcommand{\D}{{\mathbb{D}}}

\newcommand{\G}{{\mathbb{G}}}

\newcommand{\CC}{{\mathcal{C}}}

\newcommand{\cC}{\mathcal{C}}
\newcommand{\cD}{\mathcal{D}}
\newcommand{\cE}{\mathcal{E}}

\newcommand{\cN}{\mathcal{N}}
\newcommand{\cO}{\mathcal{O}}
\renewcommand{\O}{\cO}

\newcommand{\cR}{\mathcal{R}}

\newcommand{\cX}{\mathcal{X}}

\newcommand{\gM}{{\mathfrak{M}}}
\newcommand{\gN}{{\mathfrak{N}}}

\newcommand{\gS}{{\mathfrak{S}}}

\newcommand{\gm}{{\mathfrak{m}}}

\newcommand{\tu}{\widetilde{{u}}}

\newcommand{\tw}{{\widetilde{{w}}}}

 \newcommand{\barrho   }{{\overline{\rho}}}

 \newcommand{\taubar     }{\overline{\tau}}

\newcommand{\s}{\mathcal{S}} %

\newcommand{\rbar}{{\bar{r}}}

\newcommand{\HT}{\operatorname{HT}}

 \newcommand{\Qp}{{\Q_p}}

\newcommand{\Zp}{{\Z_p}}
 
\newcommand{\Qpbar}{{\overline{\Q}_p}}

\newcommand{\Zpbar}{{\overline{\Z}_p}}

\newcommand{\Fpbar}{{\overline{\F}_p}}
\newcommand{\Fpbartimes}{{\overline{\F}_p^\times}}

\newcommand{\Fp}{{\F_p}}

\newcommand{\rank}{\operatorname{rank}}

\usepackage{chngcntr}
\counterwithin{equation}{subsection}

\makeatletter
\let\c@figure\c@equation

\ifdefined\theHfigure
  \let\theHfigure\theHequation
\fi

\let\oldsubsubsection\subsubsection

\renewcommand{\subsubsection}{\@ifstar{\subsubsectionstar}{\newsubsubsection}}

\newcommand{\subsubsectionstar}{\oldsubsubsection*}

\newcommand{\newsubsubsection}[1]{%
  \refstepcounter{equation}%
  \@startsection{subsubsection}{3}%
  {\z@}{.5\linespacing\@plus.7\linespacing}{-.5em}%
  {\normalfont\itshape}{#1}%
}
\makeatother

\newtheorem{theorem}[equation]{Theorem}
\newtheorem{thm}[equation]{Theorem}
\newtheorem{lemma}[equation]{Lemma}
\newtheorem{lem}[equation]{Lemma}

\newtheorem{cor}[equation]{Corollary}

\newtheorem{optimistic-conj}[equation]{Optimistic Conjecture}

\newtheorem{prop}[equation]{Proposition}

\theoremstyle{definition}

\newtheorem{defn}[equation]{Definition}

\theoremstyle{remark}
\newtheorem{remark}[equation]{Remark}
\newtheorem{rem}[equation]{Remark}

\newtheorem{ithm}[subsection]{\bf Theorem}
\newtheorem{iprop}[subsection]{\bf Proposition}
\newtheorem{iconj}[subsection]{\bf Conjecture}
\newtheorem{iquestion}[subsection]{\bf Question}

\newtheorem{example}[equation]{Example}

\newtheorem{hypothesis}[equation]{Hypothesis}

\theoremstyle{definition}
\newtheorem{para}[equation]{\bf}
\title
[Reduction mod~$p$ and $\mu_p$-equivariance]{Reduction modulo~$p$ of crystalline Galois representations via  $\mu_p$-equivariance}
\author[B. Bhatt]{Bhargav Bhatt} 
\address{School of Mathematics, Institute for Advanced Study \newline \& Department of Mathematics, Princeton University} \email{bhargav.bhatt@gmail.com} 
\author[T. Gee]{Toby Gee} \address{Department of
  Mathematics, Imperial College London} \email{toby.gee@imperial.ac.uk} 
\author[M. Kisin]{Mark Kisin} \address{Department of
  Mathematics, Harvard University}  \email{kisin@math.harvard.edu}
\thanks{B.B.\ was partially supported by grants from the Packard and Simons Foundations.
  T.G.\ was
  supported in part by an ERC Advanced grant and by the Simons Collaboration on Perfection in Algebra, Geometry, and Topology.
  This project has received funding from the European Research Council (ERC) under the European Union’s Horizon 2020 research and innovation programme (grant agreement No. 884596).  M.K. was partially supported by NSF grant DMS-2200449.
For the purpose of open access, the authors have applied a CC BY public copyright licence to any author accepted manuscript arising from this submission.}

\begin{document}
\begin{abstract} 
For a crystalline representation of $\Gal(\Qpbar/\mathbf{Q}_p),$ with given Hodge--Tate weights, we obtain new constraints on the inertial weights of its mod $p$ reduction. This allows us to formulate an explicit Serre weight conjecture, in the generality of $L$-parameters for unramified connected reductive groups over~$\mathbf{Q}_p $, and to prove the elimination direction of this conjecture. The proof uses prismatic techniques to show that the reductions modulo~$p$ of the Breuil--Kisin modules attached to crystalline representations of $\Gal(\Qpbar/\mathbf{Q}_p)$ acquire a natural $\mu_p$-equivariant structure. Combining this with results on the geometry of the $\mu_p$-fixed points of affine Grassmannians leads to our new constraint.
\end{abstract}
 \maketitle

\setcounter{tocdepth}{2}
\tableofcontents

\section{Introduction}\label{sec:intro}  
\subsubsection*{The main theorem}
Write~$\Gal_{\Q}$ for the absolute Galois group of~$\Q$, and let~$p$ be a prime.
The weight part of Serre's conjecture \cite{SerreDuke} predicts the weights of mod~$p$ modular
forms giving rise to a modular Galois representation
$$ \rhobar: \Gal_{\Q} \rightarrow \GL_2(\Fpbar). $$
This has been generalized to reductive groups~$G$ other than $\GL_2$ and general number fields $F$;
see in particular Herzig's thesis~\cite{MR2541127} and~\cite{zbMATH06991335} (the introduction to which gives a brief overview of the history of this problem with further references to the literature).
For simplicity of exposition we mostly stick to the case~$F=\Q$ and~$G=\GL_n$ in this introduction; the results and proofs are formulated uniformly for $L$-groups of unramified connected reductive groups~$G$ over $\Q_p,$ as explained below. This may then be applied to the weight part of
Serre's conjecture for any such $G,$ and $F/\Q$ unramified at all primes $v|p.$

The form of these conjectures is that the set of predicted weights depends only on
the restriction $\rhobar|_{\Gal_{\Qp}}$, or even on $\rhobar|_{I_{\Qp}}$, where~$I_{\Qp}$ is the inertia subgroup at~$p$. This is a form of local-global compatibility for the Galois representations associated to mod~$p$ automorphic forms. It is closely related to the reduction modulo~$p$ of the corresponding statements in characteristic zero, and thus to the following question.
\begin{iquestion}\label{iq} Let
$ \overline r: \Gal_{\Qp} \rightarrow \GL_n(\Fpbar) $ be
a continuous representation. What are the possible Hodge--Tate weights of
crystalline representations
$\rho: \Gal_{\Qp} \rightarrow \GL_n(\Zpbar)$ such that
the mod~$p$ reduction $\rhobar$ of $\rho$ satisfies $\rhobar \simeq \overline r$?
\end{iquestion}
Turning the question around, we may reformulate this as asking for all the representations
$\overline r: \Gal_{\Qp} \rightarrow \GL_n(\Fpbar)$ such that $\overline r$ arises as the reduction modulo~$p$ of a crystalline representation $\rho$ with given Hodge--Tate weights. The main result of this paper is the following necessary condition on the restriction to~$I_{\Qp}$ of the semisimplification of~$\overline r$.

\begin{ithm}[Theorem~\ref{v3-thm:existence-of-semisimple-etale-phi-module-with-invariants}]\label{ithm:mainthem}
Let $\rho: \Gal_{\Qp} \rightarrow \GL_n(\Zpbar)$ be a crystalline
representation with Hodge--Tate weights~$\mu$. Then there exist
$\lambda \uparrow \mu$  and $w\in S_n$ such that
$\barrho^{\ss}|_{I_{\Qp}} \simeq \tau(\lambda,w)$.
\end{ithm}
Let us explain the notation in the theorem.
Firstly, $\lambda=(\lambda_1 ,\dots,\lambda_n)$ and $\mu=(\mu_1 ,\dots,\mu_n)$ are tuples of integers with $\lambda_1 \ge\lambda_2 \ge\dots\ge\lambda_{n}$ and $\mu_1 \ge\mu_2 \ge\dots\ge\mu_n$. 
The partial order $\uparrow$ is defined in Section~\ref{subsec:uparrow}. This is the relation which appears in the theory of mod~$p$ representations of reductive groups.
It means (see Definition~\ref{defn:uparrow-coweights}, and Figure~\ref{fig:dominant-chain-in-defn-uparrow}, with $\rotorder=p$, for the example of~$\GL_3 $)
that there is a sequence of reflections $s_1, \dots, s_r$ in the
$p$-dilated affine Weyl group of $G=\GL_n$ such that
$$ \lambda \leq s_1(\lambda) \leq s_2s_1(\lambda) \leq \dots \leq s_r \dots s_1(\lambda)
= \mu,$$
where $\leq$ denotes the usual dominance order.

The representation $\tau(\lambda,w)$ is an explicit semisimple representation of $I_{\Qp}$. Such a representation is a sum of powers of fundamental characters and, roughly speaking, $ \lambda$ determines the ``$p$-adic digits'' which appear in the exponents; see \ref{v3-para:defntau} below.
For example, suppose that~$n=2$ and~$\lambda=(a,b)$.
Then if $w =1= \Id$, we have \[\tau(\lambda,1)=\omega^{-a}\oplus\omega^{-b},\] where~$\omega$ is the mod~$p$ cyclotomic character, while if~$w=w_0=(12)$, we have \[\tau(\lambda,w_0 )=\omega_2 ^{-(a+pb)}\oplus\omega_2 ^{-(b+pa)},\] where~$\omega_2 $ is a fundamental character of niveau~$2$.

Continuing to assume that~$n=2,$ suppose that
$\rho$ is crystalline with Hodge--Tate weights $\mu = (k,0)$ with $p \leq k \leq 2p$; our convention in this paper is that the Hodge--Tate weight of the cyclotomic character is $-1$.
Then $\lambda \uparrow \mu$ implies that $\lambda = (k,0)$ or $(p,k-p)$, and Theorem~\ref{ithm:mainthem} gives four possibilities for $(\rhobar|_{I_{\Q_p}})^{\ss}$, namely
\[\tau((k,0),1) = \omega^{-k}\oplus 1,\  \tau((p,k-p),1) = \omega^{-1} \oplus \omega^{-(k-p)}
\]

\[\tau((p,k-p),w_0 ) =  \omega_2^{-(p+p(k-p))} \oplus \omega_2^{-(k-p+p^2)},\ \tau((k,0),w_0 ) =  \omega_2^{-k} \oplus \omega_2^{-pk}
\]
(note that some of these possibilities coincide for particular values of~$k$).
A result of Berger--Breuil (see e.g.\ \cite[Thm.~3.2.1]{MR2642408})  asserts that all
$4$ cases occur as the reduction of a crystalline representation of weights $0,k$.
In fact, a calculation using the (proven) Breuil--M\'ezard conjecture for $\GL_2(\Qp),$ shows that 
the converse of Theorem \ref{ithm:mainthem} holds in general when~$n=2$: 
if $\lambda \uparrow \mu$ then for any $w,$ 
there is a crystalline $\rho$ with Hodge--Tate weights $\mu$ and reduction $\tau( \lambda, w).$

For general $n,$ the converse of Theorem \ref{ithm:mainthem} does not hold. 
However we conjecture that this is true when the Hodge--Tate weights are {\em $p$-restricted} in the sense that~$\mu_i-\mu_{i+1}\le p$ for $i=1,\dots, n-1$,
which is the case relevant for Serre's conjecture; see Conjecture~\ref{optimistic-conj:crystalline-Serre-weight} below.

\subsubsection*{Breuil--Kisin modules}
We now begin the discussion of the proof of Theorem~\ref{ithm:mainthem}. 
To simplify notation we assume that the crystalline representation $\rho$ takes values in $\GL_n(\Zp).$ 
Then associated to $\rho$ one has its corresponding Breuil--Kisin module, which
 is a finite free $\Z_p\llbracket u \rrbracket$-module $\gM,$ equipped
with an isomorphism
$$ \varphi_{\gM}:  \varphi^*(\gM)\left[\frac{1}{u-p}\right] \simeq \gM\left[\frac{1}{u-p}\right],$$ where
$\varphi$ is the endomorphism of $\Z_p\llbracket u \rrbracket$ given by $u \mapsto u^p.$
We set $\overline \gM = \gM\otimes \Fp$ with its induced Frobenius $\varphi_{\overline \gM}.$

Now fix a $\Z_p\llbracket u \rrbracket$-basis for $\gM.$ Then $\varphi_{\gM}$
may be viewed as an element of the group $\GL_n(\Q_p((u-p)))$,
and the Iwasawa
decomposition, applied over the local field $\Q_p((u-p))$, associates to $\varphi_{\gM}$ a tuple $\alpha_{\gM} =(\alpha_{1},\dots\alpha_{n})$  of integers $\alpha_1 \ge\dots\ge\alpha_{n}$.
Similarly, there is a tuple $\alpha_{\overline \gM}$ associated to
$\varphi_{\overline \gM}.$ It is known that~$\alpha_{\gM}=\mu$, the Hodge--Tate weights of~$\rho$, and that
$\alpha_{\overline \gM} \leq \alpha_{\gM} = \mu,$ cf.~\cite{br2016hardernarasimhan} (here~$\le$ is again the usual dominance order).
It turns out that this can be upgraded to
\begin{iprop}[Corollary~\ref{cor:specializableuparrow}]\label{iprop:uparrowrel} We have $\alpha_{\overline \gM} \uparrow \alpha_{\gM} = \mu.$
\end{iprop}
\subsubsection*{$\mu_{p}$-equivariance}
The proof of Proposition~\ref{iprop:uparrowrel} relies on one of the key observations of this paper, which is that Breuil--Kisin modules coming from crystalline representations enjoy a $\mu_p$-equivariance property. The scheme $\Spec \Z_p\llbracket u \rrbracket$ is equipped with an action of the scheme~$\mu_p$ of $p^{\text {\rm th}}$-roots of unity, given by ``loop rotation''; $\zeta\in \mu_p$ acts by $u \mapsto \zeta\cdot u.$ The Frobenius $\varphi:\Spec \Z_p \llbracket u \rrbracket \to \Spec \Z_p \llbracket u \rrbracket$ is $\mu_p$-equivariant for the trivial $\mu_p$-action on the target, so $\varphi^*(-)$ always produces $\mu_p$-equivariant objects. In terms of this structure, we show:

\begin{iprop}[Theorem~\ref{thm:mainNthm}]\label{iprop: mupequiv} 
There exists a $\mu_p$-equivariant finite free $\Z_p\llbracket u \rrbracket$-module $\gN$, a $\mu_p$-equivariant isomorphism $\psi_\gN:\varphi^* \gM[1/u] \simeq \gN[1/u]$, and an $\F_p\llbracket u\rrbracket$-linear isomorphism $
\gNbar \coloneq \gN/p \simeq \overline{\gM}=\gM/p$ such that the diagram 
\[\xymatrix{
\varphi^*\gM[1/u] \ar[r]^{\psi_{\gN}}_{\sim}\ar@{->>}[d] &  \gN[1/u] \ar@{->>}[d]
\\
\varphi^*\overline \gM[1/u] \ar[r]^{\varphi_{\overline \gM}}_{\sim} & \overline{\gM}[1/u]
}\]
commutes (where the right-hand vertical arrow is the composite of reduction modulo~$p$ and the isomorphism $\gNbar  \simeq \overline{\gM}$); in particular, the Breuil--Kisin module $(\overline \gM, \varphi_{\overline{\gM}})$ acquires a natural $\mu_p$-equivariant structure. Moreover, under the Iwasawa decomposition over $\Q_p (( u )),$ $\psi_{\gN}$ corresponds to $\alpha_{\gM}$.
\end{iprop}

Roughly, $\psi_\gN$ realizes $\gN$ as a modification of $\varphi^* \gM$ along $\{u=0\}$, and is constructed using the Hodge filtration on the associated filtered $\Q_p$-vector space; in fact, the subset of $\mathbf{Z}/p$ determined by the support of the $\mu_p$-equivariant structure on $\gN$ (and thus also on $\overline{\gN} \simeq \overline{\gM}$ by Proposition~\ref{iprop: mupequiv}) is exactly given by (the image in $\mathbf{Z}/p$ of) the Hodge--Tate weights of the underlying Galois representation (Remark~\ref{rmk:mupweightcanmod}).  Such a characteristic $0$ modification can in fact be constructed for any Breuil--Kisin module; the key consequence of crystallinity is the congruence of $\psi_\gN \equiv \varphi_\gM$ modulo $p$. We offer two constructions of $\gN$, one relying on the prismatization \cite{bhatt2022prismatization} that minimally proves Proposition~\ref{iprop: mupequiv}, the other using the syntomification \cite{drinfeldprismatization,BBFgaugenotes} that characterizes $\gN$ in classical terms along the lines of the opening sentence of this paragraph. We do not know a proof using more classical $p$-adic Hodge theory.

For the first construction recall that, by \cite{bhatt2021prismatic}, a Breuil--Kisin module $\gM$ 
associated to a crystalline representation arises from a prismatic $F$-crystal $\mathcal{E}.$ 
By \cite{bhatt2022prismatization}, $\mathcal{E}$ may in turn be regarded as a vector bundle with Frobenius structure on the prismatization $\mathbf{Z}_p^{\Prism}$; in this perspective, the module $\gM$ is simply the pullback of $\mathcal{E}$ along the map 
$\rho_{\mathrm{std}}: \mathrm{Spf}(\mathbf{Z}_p\llbracket u \rrbracket) \to 
\mathbf{Z}_p^{\Prism}$ classifying the Breuil--Kisin prism. We realize $\gN$ as the pullback of $\mathcal{E}$ along a new map $\widetilde{\rho_\dagger}:\Spf(\Z_p\llbracket u \rrbracket) \to \Z_p^\Prism$ (Proposition~\ref{prop:FrobDescMap}). The map $\widetilde{\rho_\dagger}$ is congruent to $\rho_{\text{std}}$ modulo $p$, but does {\em not} come from a prism structure on $\Z_p\llbracket u \rrbracket$ integrally. Moreover, the $\mu_p$-equivariance of $\gN$ is  a reflection of the $\mu_p$-equivariance of the map $\widetilde{\rho_\dagger}$ itself.

For the second construction, recall that the Nygaard filtration on $\varphi^*(\gM)$
induces the Hodge filtration on $D = \bigl(\varphi^*(\gM)/(u-p)\varphi^*(\gM)\bigr)[1/p].$
On the other hand, parallel transport along Frobenius allows us to identify
$\varphi^*(\gM)\otimes_{\Z_p\llbracket u \rrbracket}\Qp\llbracket u \rrbracket$ with
$D\otimes_\Qp\Qp\llbracket u \rrbracket.$ Thus the Rees module, with Rees parameter $u$, of the Hodge filtration on $D$ gives a modification of $\varphi^*(\gM)\otimes_{\Z_p\llbracket u \rrbracket}\Qp\llbracket u \rrbracket$ along $\{u=0\}.$ By Beauville--Laszlo glueing and purity for vector bundles on $\Spec \Z_p\llbracket u \rrbracket$, this extends uniquely to a modification $\gN$ of $\varphi^*(\gM)$ along $\{u=0\}$. It is not hard to see
that $\gN$ has all the properties in the proposition {\em except} for the isomorphism $\gN/p\gN \simeq \overline \gM.$ The proof of this last, crucial, property (Theorem~\ref{thm:crysBKspecial})
uses essentially that if
$\gM$ corresponds to a crystalline representation, then it arises from an $F$-gauge $\mathcal{E}^{\Syn}$, which is a reflexive coherent sheaf on the stack $\Z_p^{\Syn}$ \cite{drinfeldprismatization,BBFgaugenotes}, extending the crystal $\mathcal{E}$ from the previous paragraph from the open substack $\Z_p^\Prism \subset \Z_p^{\Syn}$. 

\subsubsection*{$\mu_{p}$-fixed points in the affine Grassmannian}
To explain how to deduce Proposition \ref{iprop:uparrowrel} from
Proposition \ref{iprop: mupequiv}, it is useful to reformulate the invariants
$\alpha_{\gM}$ and~$\alpha_{\overline \gM}$ in terms of orbits on the affine Grassmannian
$\Gr$ for $\GL_n$. Recall that this is defined as a quotient $L\GL_n/L^+\GL_n$, where for any ring
$R$, $L\GL_n(R) = \GL_n(R((u)))$ and $L^+\GL_n(R) = \GL_n(R\llbracket u \rrbracket)$.

Write~$T$ for the usual maximal torus of $\operatorname{GL}_n$ given by the diagonal matrices, write $X_*(T)^+$ for the set of tuples of integers $\lambda=(\lambda_1 , \dots, \lambda_{n})$ with $\lambda_1 \geq \lambda_2 \geq \dots \geq \lambda_n $, and write $u^{\lambda}=\diag(u^{\lambda_1},\dots, u^{\lambda_n})\in L \operatorname{GL}_n$. 
Then the orbits of $L^+\GL_n$ on $\Gr$ are precisely the $L^+\GL_n\cdot u^{\mu}$, and in particular are  indexed by $X_*(T)^+$; this is a geometric reformulation of the Iwasawa decomposition.
The closure relation on orbits is given by the dominance order.
Thus one may view $\alpha_{\gM}$ and~$\alpha_{\overline \gM}$ as points of $L^+\GL_n\backslash \Gr,$ and the closure relations imply that
$\alpha_{\overline \gM} \leq \alpha_{\gM}.$ More geometrically, the quotient $L^+\GL_n\backslash \Gr$ parameterizes modifications of vector bundles on $\Spec R[[u]]$, and $\alpha_{\gM}$ (respectively $\alpha_{\gMbar}$) is given by the modification $\varphi_{\gM}$ (respectively $\varphi_{\gMbar}$).

Now $\mu_p$ acts on $\Gr$ and $L^+\GL_n$ by loop rotation.
In particular we can consider the $\mu_p$-fixed points
$\Gr^{\mu_p}.$ Define the subgroup $L_p^+\GL_n \subset L^+\GL_n$ by $L_p^+\GL_n(R) = \GL_n(R \llbracket u^p \rrbracket).$
Using Proposition \ref{iprop: mupequiv}, one may promote $\alpha_{\gM}$ and $\alpha_{\overline \gM}$ to points of $L^+_p\GL_n\backslash\Gr^{\mu_p}.$
By a result of Riche--Williamson \cite{MR4517647}, the orbits of $L^+_p\GL_n$ on $\Gr^{\mu_p}$ are also indexed by $X_*(T)^+,$ but the closure relations between the orbits $L_p^+\GL_n\cdot u^{\lambda}$ in $\Gr^{\mu_p}$ are different from those between the corresponding orbits $L^+\GL_n\cdot u^{\lambda}$ in~$\Gr$.
We show

\begin{iprop}[Proposition~\ref{prop:K-orbit-closure-mu-fixed}]\label{iprop:uparroworbits} For $\lambda, \mu \in X_*(T)^+,$ 
the orbit $L_p^+\GL_n\cdot u^{\lambda} \subset \Gr^{\mu_p}$ is
contained in the closure of $L_p^+\GL_n\cdot u^{\mu}$ if and only if $\lambda \uparrow \mu.$
\end{iprop}
Now $\alpha_{\overline \gM}$ is a specialization of $\alpha_{\gM},$ so we obtain
$\alpha_{\overline \gM} \uparrow \alpha_{\gM} = \mu.$

\subsubsection*{Semisimple Galois representations and Breuil--Kisin modules}
\label{sec:semis-galo-repr}
Identify the Weyl group~$S_n$ of~$T$ with the group of permutation matrices, so that $N_{\operatorname{GL}_n }(T)=T\rtimes S_n$.
It is easy to show (using the explicit description of the tame inertia subgroup of~$I_{\Qp}$) that any semisimple representation $\rbar:\Gal_{\Qp}\to\GL_{n}(\Fpbar)$ is conjugate to a representation $\rbar:\Gal_{\Qp}\to N_{\GL_n}(T)(\overline{\mathbf{F}}_p )$ with $\rbar(I_{\Qp})\subset T(\Fpbar)$.

We now return to the proof of Theorem~\ref{ithm:mainthem}. 
First suppose that~$\rhobar^{\semis}$ is conjugate to a representation $\Gal_{\Qp}\to T(\Fpbar)$, which we continue to denote by $\rhobar^{\semis}$. Then it is straightforward to show that we may semisimplify $\gMbar$ to obtain a Breuil--Kisin module $\gMbar^{\semis}$ with~$\varphi_{\gMbar^{\semis}}$ represented by a diagonal matrix in some choice of basis.
The semisimplification procedure exhibits~$\gMbar^{\semis}$ as a specialization of~$\gMbar$ (in an appropriate moduli stack of Breuil--Kisin modules), and it follows easily that $\gMbar^{\semis}$ is $\mu_p$-equivariant, and $\alpha_{\gMbar^{\semis}}$ is a specialization in $L_p^+ \operatorname{GL} _n\Gr^{\mu_p}$ of~$\alpha_{\gMbar}$.
Setting $\lambda=\alpha_{\gMbar^{\semis}}$, we deduce from Propositions \ref{iprop:uparrowrel}, \ref{iprop: mupequiv} and \ref{iprop:uparroworbits} that $\lambda\uparrow \mu$, while it is easy to see that after a possible further change of basis, $\varphi_{\gMbar^{\semis}}$ is given by the diagonal matrix $\diag(t_1 u^{\lambda_1 },\dots,t_{n}u^{\lambda_n})$ for some $t_i\in\Fpbartimes$.
From this one can read off that $\rhobar^{\semis}|_{I_{\Qp}}\cong \tau(\lambda,1)= \omega^{-\lambda_1 }\oplus\dots\oplus \omega^{-\lambda_n}$, completing the proof in this case.
More generally, if~$\varphi_{\overline{\mathfrak{M}}^\semis}$ happens to be conjugate to a matrix in $N_{\GL_n}(T)$, then it is easy to see that after a change of basis, it is of the form
\[
  \diag(t_1 u^{\lambda_1 },\dots,t_{n}u^{\lambda_n})\cdot w
\] for some $t_i\in\Fpbartimes$ and~$w\in S_n$, where $\lambda=\alpha_{\gMbar}$.
One then checks that we have $\rhobar^{\semis}|_{I_{\Qp}}\cong \tau(\lambda,w)$ (indeed, this is essentially our definition of $\tau(\lambda,w)$).

When ~$\rhobar$ is irreducible, we do not expect that~$\varphi_{\gMbar}$ can always be conjugated into $N_{\GL_n}(T).$ 
Firstly, we show that after a change of basis, $\varphi_{\overline{\mathfrak{M}}}$ is of the form 
\[
  x\cdot \diag(t_1 u^{\lambda_1 },\dots,t_{n}u^{\lambda_n})\cdot w  
\]where~$x$ is a conjugate of a matrix in $\mathrm{GL}_n(\overline{\mathbf{F}}_p \llbracket u \rrbracket)$ (in fact, a matrix in the Iwahori subgroup) by a diagonal matrix $\operatorname{diag}(u^{\nu_1},\dots, u^{\nu_n}) $.
Since~$x$ need not be integral, it is not necessarily the case that $\alpha_{\overline{\mathfrak{M}}}=\lambda$.
However, using an idea that we learned from a paper of Chen--Nie~\cite{MR4402497}, we are able to show that we can specialize from~$\gMbar$ to a Breuil--Kisin module which admits a reduction of structure group to $N_{\GL_n}(T)$, and then conclude as above.

In the general case that~$\rhobar^{\semis}$ neither factors through a torus nor is irreducible, we combine the methods above: we first semisimplify, and then pass to a Levi subgroup to reduce to the case that~$\rhobar$ is irreducible.

\subsubsection*{$L$-groups}
The statements and proofs above are not special to~$\GL_n$. In the body of the paper we work with $L$-groups of arbitrary unramified connected reductive groups. Thus if~$G$ is such a group over~$\Zp$, with (unramified) splitting field~$L$, dual group~$\Ghat$, and dual torus~$\That$, we consider
\[
  \LG=\Ghat\rtimes \Gal(L/\Qp),
\]
and $L$-parameters $\Gal_{\Qp}\to \LG$ compatible with the projection to~$\Gal(L/\Qp)$; see Section~\ref{subsec:L-groups}.
The role of~$S_n$ is played by the Weyl group~$W$ of~$\Ghat$, and Hodge--Tate weights are dominant cocharacters of~$\That$.
In particular, for any connected reductive group~$\Ghat$ over~$\Zpbar$, we can choose~$G$ to be the split form of the dual group of~$\Ghat$, in which case $
  \LG=\Ghat
$.

In this generality we prove Theorem~\ref{v3-thm:existence-of-semisimple-etale-phi-module-with-invariants}: if $\rho:\Gal_{\Qp}\to\LG(\Zpbar)$ is crystalline with Hodge--Tate cocharacter~$\mu(\rho)$, then there are $\lambda$ and~$w\in W$ such that $\lambda_{\dom}\uparrow\mu(\rho)$ and
\[
  \rhobar^{\semis}|_{I_{\Qp}}\cong \tau(\lambda,w);
\] here  $\lambda_{\dom}$ denotes the unique dominant element of the orbit~$W\lambda$.
The formulation in terms of $L$-groups also gives, by restriction of scalars, the corresponding theorem for crystalline $L$-parameters of $\Gal_K$ for any finite unramified extension~$K/\Qp$; see Theorem~\ref{thm:existence-of-semisimple-etale-phi-module-with-invariants-L-parameter-version}.

\subsubsection*{The weight part of Serre's conjecture}
We now explain the connection with Serre weights more precisely. In Section~\ref{sec:weight-Serre-conjecture}, under the usual hypotheses on an unramified group~$G/\Zp$ recalled there, the Serre weights are the irreducible $\Fpbar$-representations of~$G(\Fp)$, parameterized as $F_{\lambda-\eta}$ for suitable $p$-restricted highest weights~$\lambda$ and a fixed twisting element~$\eta$.

For a semisimple $L$-parameter $\rhobar:\Gal_{\Qp}\to\LG(\Fpbar)$ we define two sets of Serre weights. The set $W^{\cris}(\rhobar)$ consists of the weights $F_{\lambda-\eta}$ for which~$\rhobar$ admits a crystalline lift with Hodge--Tate cocharacter~$\lambda$. The set $W^{\explicit}(\rhobar)$ consists of the weights $F_{\lambda-\eta}$ for which there are $\lambda'\uparrow\lambda$ and~$w\in W$ with
\[
  \rhobar|_{I_{\Qp}}\cong \tau(\lambda',w).
\]
Theorem~\ref{thm:Serre-weight-upper-bound} gives the inclusion
\[
  W^{\cris}(\rhobar)\subseteq W^{\explicit}(\rhobar),
\]
and therefore gives a weight elimination result whenever the expected inclusion from the true Serre weight set~$W(\rhobar)$ into~$W^{\cris}(\rhobar)$ is known. We make the following optimistic conjecture.
\begin{iconj}[Conjecture~\ref{optimistic-conj:crystalline-Serre-weight}]\label{conj:mainconj} Suppose that $\rhobar$ is semisimple. Then 
\[
  W(\rhobar)=W^{\cris}(\rhobar)=W^{\explicit}(\rhobar).
\]
\end{iconj}
 This gives a uniform candidate set of Serre weights, with no genericity assumptions, and a corresponding weight elimination result, for general unramified reductive groups. Combining this conjecture with the Breuil--M\'ezard conjecture, one obtains a complete conjectural answer to Question \ref{iq} for regular Hodge--Tate weights and semi-simple $\rbar.$ We leave the straightforward formulation to the reader.

We also show that our results are compatible with the existing literature on the weight part of Serre's conjecture.
In particular, for generic~$\rhobar$, we compare our set $W^{\explicit}(\rhobar)$  with the predicted set~$\WHer(\rhobar)$ of Gee--Herzig--Savitt, which generalizes Florian Herzig's conjectural set for~$\GL_n$. Proposition~\ref{prop:generic-agreement-GHS} proves that $W^{\explicit}(\rhobar)=\WHer(\rhobar)$ under an explicit genericity hypothesis. Combined with standard local-global compatibility, this improves existing weight elimination results for automorphic forms on unitary groups; see
Remark~\ref{rem:compare-to-LLLM-elimination}.
Note also  the key structural theorem in the proof of the Buzzard--Diamond--Jarvis conjecture~\cite{GLSII} is
 an immediate corollary of our results  (see Corollary~\ref{v3-cor:shape-p-small-weight}).

\subsubsection*{A brief history of this paper}\label{subsubsec:historical-remarks} 
The project began as joint work of T.G. and M.K., and the earliest version of the results of this paper was announced in T.G.'s talk \cite{TG-IAS-talk} in November 2023.
At that time we had a slightly weaker version of Proposition~\ref{iprop:uparrowrel} 
(see Corollary \ref{cor:Wp-orbit}), proved by a different argument with the stack~$\ZpSyn$, 
with the help of Jacob Lurie. 
This argument is presented in Section~\ref{subsec:Jacob-lemma}, which uses these earlier techniques to give a complete proof of Proposition~\ref{iprop:uparrowrel} for~$G=\GL_n$.
Inspired by T.G.'s talk~\cite{TG-IAS-talk}, Tong Liu proved a statement related to Corollary 
\ref{cor:Wp-orbit} in his paper~\cite{liu2026torsiongradedpiecesnyggard}.
Further proofs of Liu's result were given by
Gao--Liu~\cite{gao2024integralsentheoryintegral} and Pham~\cite{PhamFGaugeLiu}.

We attempted to deduce Theorem~\ref{ithm:mainthem} directly from the aforementioned version of Proposition~\ref{iprop:uparrowrel}, but did not succeed in doing so. At that point, B.B. joined the collaboration. By late summer 2024, we realized that the $\mu_p$-equivariance statement for~$\gMbar$ in Proposition~\ref{iprop: mupequiv} appeared to be true, and that this, together with the ideas of Chen--Nie~\cite{MR4402497}, would give a proof of Theorem~\ref{ithm:mainthem} for general reductive groups.
We initially imagined that the $\mu_p$-equivariance property would hold for any prism, and were confused by the lack of any structure in~$\ZpSyn$ that would enable us to prove this. Once we realized that we should only be trying to prove it for a Breuil--Kisin prism, we obtained our first complete proof of Theorem~\ref{ithm:mainthem} in the fall of 2024.
This first proof of $\mu_p$-equivariance, which uses $\Z_p^{\Prism}$, is described in Section~\ref{sss:canmodprismproof}; Proposition~\ref{iprop: mupequiv} was found later, in the course of writing this paper.
In late 2024 Robin Bartlett (personal communication) independently suggested that a $\mu_p$-equivariant structure on Breuil--Kisin modules could give another proof of Proposition~\ref{iprop:uparrowrel}; we thank him for bringing the reference~\cite{MR4517647} to our attention, which saved us a good deal of work in writing Section~\ref{sec:affine-Grassmannians}.

\subsection{A brief guide to the paper}
Section~\ref{sec:affine-Grassmannians} studies the $\mu_{\rotorder}$-fixed locus in affine Grassmannians for an arbitrary integer $\rotorder>1$. The main results are the comparison between the $\uparrow$ order and Bruhat order in Proposition~\ref{prop:Bruhat-equivalent-uparrow}, the orbit-closure description in Proposition~\ref{prop:K-orbit-closure-mu-fixed}, and the integral specialization result Corollary~\ref{cor:closure-from-zpbar-point}.
In the rest of the paper we apply these results with $\rotorder=p$, but in a forthcoming sequel we will also make use of $\rotorder=p^n$ for~$n\ge 1$ (see Remark~\ref{rem:not-best-possible-BM}).

Section~\ref{sec:Rees-stacks} recalls the Rees-stack formalism used to package relative positions of modifications and their specializations. Section~\ref{sec:Frob-descent-F-gauge} proves the main structural theorem for crystalline Breuil--Kisin modules, Theorem~\ref{thm:mainNthm}, including the $\mu_p$-equivariant specialization in Corollary~\ref{cor:mainNcor} and the numerical consequence Corollary~\ref{cor:specializableuparrow}. It also proves the corresponding statements for $G$-valued crystalline
representations. 

Section~\ref{v3-sec:inertial-weights-mod-p-Galois-rep} applies these results to Galois representations. It recalls $L$-groups and the inertial representations~$\tau(\lambda,w)$, proves the shape bound Theorem~\ref{v3-thm:shape-uparrow}, and proves our main result on inertial weights, Theorem~\ref{v3-thm:existence-of-semisimple-etale-phi-module-with-invariants}.

Section~\ref{sec:weight-Serre-conjecture} formulates Conjecture~\ref{optimistic-conj:crystalline-Serre-weight} on the weight part of Serre's conjecture, and proves  Proposition~\ref{prop:generic-agreement-GHS} which compares our explicit set of weights with the Herzig/Gee--Herzig--Savitt predicted set in the generic case.

Finally, Appendix~\ref{sec:appendix-Tannakian} recalls some material on the Tannakian formalism which we use in the body of the paper.
\subsection{Acknowledgements}We are all grateful to Jacob Lurie for many helpful conversations about this work. In particular, T.G.\ and M.K.\ wish to acknowledge his assistance in the early stages of this project, on the material recorded in Section~\ref{subsec:Jacob-lemma}. T.G.\ thanks George Boxer for many helpful conversations about Breuil--Kisin modules, the weight part of Serre's conjecture, and related topics over the last ten years.
We thank Simon Riche and Geordie Williamson for answering a question about~\cite{MR4517647}, and in addition we are grateful to Riche for his helpful comments on an earlier version of section~\ref{sec:affine-Grassmannians} below.
We would like to thank Robin
Bartlett, Florian Herzig, Hui Gao, Dan Le, Bao Le Hung, Tong Liu and Dat Pham for conversations about this work.
We would also like to thank Robin Bartlett, Matt Emerton, Brandon Levin and Dat Pham for their comments on an earlier draft of this paper.
The proof of Proposition~\ref{prop:effective-functions-imply-uparrow} was found with assistance from GPT‑5.4 Pro.
ChatGPT also produced Figure~\ref{fig:dominant-chain-in-defn-uparrow}.
In the first public version of this paper, the proof of Theorem~\ref{v3-thm:existence-of-semisimple-etale-phi-module-with-invariants} used a specialization argument in the affine Grassmannian.
In a subsequent conversation with GPT-6 Astra, we discovered how to lift this argument to the level of Breuil--Kisin modules, and we have accordingly rewritten the proof to use this lifting.
While the argument is essentially the same, we find this version conceptually clearer.
We also used ChatGPT/Claude/Codex/Gemini to proofread the paper.

\section{\texorpdfstring{$\mu_{\rotorder}$}{mu-\rotorder}-fixed points in affine Grassmannians} \label{sec:affine-Grassmannians}
Throughout this section, fix an integer $\rotorder>1$.
In the rest of the paper we will apply the results of this section with $\rotorder=p$, but in a forthcoming sequel we will also make use of $\rotorder=p^n$ for~$n\ge 1$  (see Remark~\ref{rem:not-best-possible-BM}).
Our main aim is to prove some results about the $\mu_{\rotorder}$-fixed locus for the loop rotation action on the affine Grassmannian.
We first establish Proposition~\ref{prop:Bruhat-equivalent-uparrow}, a combinatorial result which we then combine with results of Riche--Williamson~\cite{MR4517647} 
to study the closure relations in this $\mu_{\rotorder}$-fixed locus.
We have made an effort to give a self-contained treatment of this material, recalling some standard results along the way.

\subsection{Coxeter groups}
\label{sec:Coxeter}We briefly recall some more or less well-known results about Coxeter groups (see \cite{MR1890629, zbMATH00053657}).
Let~$(W_S,S)$ be a Coxeter system, with~$S$ finite.
Write~$\ell$ for the length function on~$W_S$, i.e.\ for the length of a minimal expression of an element in terms of elements of~$S$.
We write~$\le_{\Bru}$ for the Bruhat (partial) order on~$W_S$, so that $w\le_{\Bru} w'$ if and only if some (not necessarily consecutive) substring of some (equivalently every) reduced expression for~$w'$ is a reduced expression for~$w$.
We have the following useful lemma.
\begin{lem}
  \label{lem:Bruhat-order-cancellation}Suppose that $a,b,c\in W_{S}$ with $\ell(ac)=\ell(a)+\ell(c)$ and $\ell(bc)=\ell(b)+\ell(c)$.
Then $a\le_{\Bru}b$ if and only if $ac\le_{\Bru}bc$.
\end{lem}
\begin{proof}
  By induction on~$\ell(c)$, this follows easily from the lifting property of the Bruhat order; see~\cite[Lem.~2.2]{zbMATH04077214}.
\end{proof}
\begin{para}\label{para:minimal-representatives} For any subset~$I\subseteq S$, we write~$(W_I,I)$ for the corresponding Coxeter system; the restriction of~$\ell$ to~$W_I$ agrees with the length function on~$W_I$, and~$W_I$ is called a \emph{standard parabolic subgroup} of~$W_S$.
The set \[
W^I \coloneq  \{ w \in W_S \mid w < ws \text{ for all } s \in I \}
\] is the set of minimal length representatives of $W_S/W_I$, so that each element $w\in W_S$ has a unique expression as $w=w^Iw_I$ with $w^I\in W^I$ and $w_I\in W_I$.
Furthermore, we have $\ell(w)=\ell(w^I)+\ell(w_{I})$.
The set ${}^IW\coloneq (W^I)^{-1}$ is the set of minimal representatives for ~$W_I\backslash W_S$.
If~$W_I$ is finite, then it has a longest element~$w_{0,I}$.
If $I\subseteq J$, then $W^I\cap W_J$ has a unique longest element, namely $w_{0,J}w_{0,I}^{-1}$.
We will be interested in representatives for double cosets $W_J\backslash W_S/W_I$, for which a convenient reference is~\cite{zbMATH07740443}.
In particular we have the following result.
\end{para}

\begin{prop}
  \label{prop:double-coset-representatives-basic-properties} Suppose 
  that~$W_I,W_J$ are finite.
Then:
  \begin{enumerate}
  \item\label{item:19}Let $u\in W_J\backslash W_S/W_I$. Then the subset~$u\subset W_S$ is an interval in the Bruhat order; i.e.\ there are elements $\underline{u}$, $\overline{u} \in W_S$ such that \[u=\{x\in W_S \mid \underline{u}\le_{\Bru} x\le_{\Bru}\overline{u}\}.\] We say that $\underline{u}$ is the minimal representative of~$u$, and $\overline{u}$ is its maximal representative.
  \item\label{item:20} The set of minimal elements~$\{\underline{u}: u \in W_J\backslash W_S/W_I\}$ is equal to ${}^JW\cap W^I$.
  \item\label{item:21} For each~$u$, the map
     \begin{align*}
  & (W^{J \cap \underline{u} I \underline{u}^{-1}} \cap W_J) \times W_I
  \longrightarrow u \notag \\
  & \hspace{5em} (x, y) \longmapsto x \underline{u} y
\end{align*}
    is a bijection, and furthermore we have $\ell(x\underline{u}y)=\ell(x)+\ell(\underline{u})+\ell(y)$. In particular, we have
    \begin{equation}
    \label{eq:16}
    \overline{u}=w_{0, J} \, w_{0, J \cap \underline{u} I \underline{u}^{-1}}^{-1} \, \underline{u} \, w_{0, I}.
    \end{equation}
  \item The set $u\cap W^I$ contains a unique element ${}_Ju^I$ of maximal length; furthermore we have
  \begin{equation}
  \label{eq:11}
  {}_J u^I = w_{0, J} \, w_{0, J \cap \underline{u} I \underline{u}^{-1}}^{-1} \, \underline{u} .
  \end{equation}
    \end{enumerate}
\end{prop}
\begin{proof}The first three parts are part of~\cite[Lem.~2.12]{zbMATH07740443}.
For the fourth part, note that by part~\eqref{item:21}, the set $u\cap W^I$ is equal to the set of elements $x\underline{u}$ with
$x\in W^{J \cap \underline{u} I \underline{u}^{-1}}\cap W_J$, and we have $\ell(x\underline{u})=\ell(x)+\ell(\underline{u})$ for all
$x\in W^{J \cap \underline{u} I \underline{u}^{-1}}\cap W_J$.
Since $W^{J \cap \underline{u} I \underline{u}^{-1}}\cap W_J$ has a unique longest element, namely $w_{0,J}w_{0,J\cap \underline{u} I \underline{u}^{-1}}^{-1}$, we obtain~\eqref{eq:11}.
\end{proof}
The following result is well-known, but for lack of a reference we give a proof.
\begin{prop}\label{prop:equivalence-Bruhat-orderings-double-cosets}
  Assume that $W_I,W_J$ are finite, and let $u,v\in W_J\backslash W_S/W_I$.
Then the following conditions are equivalent:
\begin{enumerate}
\item\label{item:23} There exist $x\in u$ and $y\in v$ with $x\le_{\Bru} y$.
\item\label{item:22} $\underline{u}\le_{\Bru} \underline{v}$.
\item\label{item:29} $\overline{u}\le_{\Bru}\overline{v}$.
\item\label{item:24} ${}_Ju^I\le_{\Bru} {}_Jv^I$.

\end{enumerate}
\end{prop}
\begin{proof} Conditions \eqref{item:23}--\eqref{item:29} are equivalent by~\cite[Lem.~2.2]{MR1047315}.
By~\eqref{eq:16} and~\eqref{eq:11} we have $\overline{u}={}_Ju^Iw_{0,I}$, $\overline{v}={}_Jv^Iw_{0,I}$,  where $\ell(\overline{u})=\ell({}_Ju^I)+\ell(w_{0,I})$ and $\ell(\overline{v})=\ell({}_Jv^I)+\ell(w_{0,I})$.
It follows from Lemma~\ref{lem:Bruhat-order-cancellation} that \eqref{item:29} implies \eqref{item:24}.
Finally~\eqref{item:24} trivially implies~\eqref{item:23}.
\end{proof}

\subsection{The affine Weyl group}
\label{sec:affine-weyl-group}
We now recall some standard material on affine Weyl groups; see for example~\cite[Ch.~VI, \S 2]{MR1890629} for any facts for which we do not give an explicit reference.

\begin{para}\label{para:notation-for-root-systems} Let $(X^*(T),\Phi,X_*(T),\Phi^\vee)$ be the root system attached to
a split connected reductive group~$G$ and a choice of maximal torus $T \subseteq G$.
We write $\langle\cdot,\cdot\rangle:X^{*}(T)\times X_{*}(T)\to\Z$ for the natural pairing,
and $\Lambda\subseteq X^{*}(T), \Lambda^{\vee}\subseteq X_{*}(T)$ for the root and coroot lattices, respectively.
For any ring~$S$ we write $X_{*}(T)_{S}\coloneq X_{*}(T)\otimes_{\Z}S$.

Now fix a Borel subgroup $B \subseteq G$ containing $T$.
Let $\Phi^+ \subset \Phi$ and $\Phi^{\vee,+} \subset \Phi^{\vee}$ denote the positive
roots and coroots, respectively.
We let $X^{*}(T)^{+}$ and $X_{*}(T)^+$ denote the dominant characters and dominant cocharacters, respectively, so that in particular~$X_{*}(T)^+$ is the set of cocharacters~$\mu$ such that $\langle \alpha,\mu\rangle\ge 0$ for all $\alpha\in\Phi^{+}$.

We denote the Weyl group of~$G$ by~$W_{G}$, or by~$W$ if the choice of~$G$ is clear.
We let~$W$ act on the left on~$X^{*}(T)$ and on~$X_{*}(T)$, so that for~$w\in W$, $\chi\in X^{*}(T)$ and~$t\in T$ we have $(w(\chi))(t)=\chi(w^{-1}tw)$, while for~$\mu\in X_{*}(T)$ we have $(w(\mu))(x)=w\mu(x)w^{-1}$.
In particular, the pairing $\langle\cdot,\cdot\rangle$ is $W$-equivariant (for the trivial action
of~$W$ on~$\Z$).
For any~$\lambda\in X_{*}(T)$, we write $\lambda_{\dom}$ for the unique dominant element of the orbit~$W\lambda$.
\end{para}
\begin{para}\label{para:notn-affine-Weyl}
Write $\Wt$ for the extended affine Weyl group $X_{*}(T)\rtimes W$, and~$\Waff$ for its normal subgroup, the affine Weyl group $\Lambda^{\vee}\rtimes W$. These both act naturally on~$X_{*}(T)$ (with $X_{*}(T)$ acting by the usual addition of cocharacters), and we often view~$\Waff$ as a group of automorphisms of~$X_{*}(T)_{\R}$. We embed~$X_{*}(T)$ into~$G(R((u)))$ (where~$u$ is a formal variable) via \[\lambda\mapsto u^{\lambda}\coloneq \lambda(u).\] Accordingly, we will
frequently write elements of~$\Wt$ as~$\widetilde{w}=u^{\lambda}w$ with~$\lambda\in X_{*}(T)$ and~$w\in W$, so that
$u^{w(\lambda)}=wu^{\lambda}w^{-1}$.
Using a split Chevalley model, we fix integral representatives of the elements of~$W$ in the normalizer of~$T$, and use their images over each base ring~$R$. This gives an embedding (of sets) of $\Wt$ into~$G(R((u)))$.

For each $\alpha\in\Phi$ and~$n\in \Z$ we have a hyperplane in~$X_{*}(T)_{\R}$ given by the equation $\langle \alpha,v\rangle=n$.
We write~$s_{\alpha,n}$ for the affine reflection in this hyperplane, given by \begin{equation*}\label{eqn:formula-for-affine-reflection}s_{\alpha,n}(v)\coloneq v-(\langle\alpha,v\rangle-n)\alpha^{\vee}=s_{\alpha}(v)+n\alpha^{\vee}.\end{equation*} In particular $s_{\alpha}=s_{\alpha,0}$ has its usual meaning in the finite Weyl group~$W$, and in the notation above,
$s_{\alpha,n}$ is given by $u^{n\alpha^{\vee}}s_{\alpha}$.
The~$s_{\alpha,n}$ generate~$\Waff$ (note that $s_{\alpha}s_{\alpha,n}(\lambda)=\lambda-n\alpha^{\vee}$). 

We write $W_{\rotorder}\coloneq \rotorder\Lambda^{\vee}\rtimes W$ for the $\rotorder$-dilated affine Weyl group, which is the subgroup of~$\Waff$ generated by the reflections~$s_{\alpha,\rotorder n}$.
\end{para}

\begin{defn}\label{defn:facet}
  A \emph{facet} is an equivalence class of elements of~$X_{*}(T)_{\R}$ for the equivalence relation given by $x\sim y$ if and only if for each $\alpha\in\Phi$ and~$n\in \Z$, either $\langle \alpha,x\rangle=\langle \alpha,y\rangle= n\rotorder$, or $\langle \alpha,x\rangle,\langle \alpha,y\rangle< n\rotorder$ or $\langle \alpha,x\rangle,\langle \alpha,y\rangle> n\rotorder$.
  An \emph{alcove} is a facet of maximal dimension, and a \emph{wall} of an alcove is a codimension-one facet which lies in the closure of the alcove.
\end{defn}

\begin{para}
The group~$W_{\rotorder}$ acts simply transitively on the set of alcoves. We say that an alcove~$\Delta$ is dominant if $\langle \alpha ,x\rangle>0$ for all $\alpha\in \Phi^+,x\in \Delta$.

We write~$\Delta_{\rotorder}$ for the alcove \begin{equation*}\label{eqn:fundamental-p-alcove}\Delta_{\rotorder}=\{v\in X_*(T)_{\R}: 0<\langle\alpha,v\rangle<\rotorder\ \forall \alpha\in
  \Phi^+\}.\end{equation*} The closure~$\overline{\Delta}_{\rotorder}$ of~$\Delta_{\rotorder}$ is a fundamental domain for the action of~$W_{\rotorder}$ on~$X_{*}(T)_{\R}$.
Thus given any~$\lambda\in X_{*}(T)$, we can write $\lambda=(u^{\rotorder\nu }w)\lambda_0 =\rotorder\nu +w\lambda_0$ for some unique~$\lambda_0 \in \overline{\Delta}_{\rotorder}\cap X_{*}(T)$ (and some~$w\in W$ and~$\nu \in \Lambda^{\vee}$, which are not in general uniquely determined).

The alcove~$\Delta_{\rotorder}$ determines a Coxeter system~$(W_{\rotorder},S)$, where~$S$ is the set of reflections in the walls of~$\Delta_{\rotorder}$.
We write $\le_{\Bru}$ for the Bruhat order on~$W_{\rotorder}$, and we write~$\ell(v)$ for the length of~$v\in W_{\rotorder}$.

By \cite[(1.9)]{MR2998951} (a formula which goes back to Iwahori--Matsumoto \cite{MR185016}), for any~$\nu\in \Lambda^{\vee}$ and $w\in W$ we have
  \begin{equation}
    \label{eq:length-of-u-lambda-w}
    \ell(u^{\rotorder\nu}w)=\sum_{\alpha\in \Phi^+\cap w(\Phi^+)}|\langle \alpha,\nu\rangle| + \sum_{\alpha\in \Phi^+\cap w(\Phi^-)}|\langle \alpha,\nu\rangle-1|.
  \end{equation}

For each $\lambda_0 \in \overline{\Delta}_{\rotorder}\cap X_{*}(T)$, we write $W_{\rotorder}^{\lambda_0 }\le W_{\rotorder}$ for the stabilizer of~$\lambda_0 $.
This is a finite standard parabolic subgroup of~$W_{\rotorder}$, generated by the reflections in the walls of~$\Delta_{\rotorder}$ which contain~$\lambda_0 $.

\end{para}
\begin{prop}
  \label{prop:Bruhat-order-and-Wp}Suppose that $\lambda\in X_{*}(T)^+$, and write $\lambda=v(\lambda_0) $, where $\lambda_0 \in \overline{\Delta}_{\rotorder}\cap X_{*}(T)$, and~$v\in W_{\rotorder}$ is chosen so that~$v$ is the element of minimal length in $v W_{\rotorder}^{\lambda_0 }$.
  Write $v = u^{\rotorder\nu}w$ with $\nu\in \Lambda^{\vee}$ and $w\in W.$ Then:
\begin{enumerate}
\item\label{item:16} $\nu\in X_{*}(T)^{+}$, and if $\beta\in \Phi^{+}$ with $\langle\beta,\nu\rangle=0$, then $w^{-1}(\beta)\in\Phi^{+}$.
\item\label{item:17} $v{\Delta}_{\rotorder}$ is dominant.
\item\label{item:18} $v=\underline{v}$ is the unique minimal representative of the double coset~$WvW_{\rotorder}^{\lambda_0 }$.
\end{enumerate}
\end{prop}
\begin{proof} We begin by noting that for any $w'\in W$ we can write $w'u^{\rotorder\nu}w=u^{\rotorder w'(\nu)}w'w$, so that by~\eqref{eq:length-of-u-lambda-w} we have

  \begin{gather}
    \begin{aligned}
    \label{eq:length-of-left-translate-of-v}
      \ell(w'v) &=\sum_{\alpha\in \Phi^+\cap w'w(\Phi^+)}|\langle \alpha,w'(\nu)\rangle| + \sum_{\alpha\in \Phi^+\cap w'w(\Phi^-)}|\langle \alpha,w'(\nu)\rangle-1|\\
                &=\sum_{\alpha\in \Phi^+\cap w'w(\Phi^+)}|\langle (w')^{-1}\alpha,\nu\rangle| + \sum_{\alpha\in \Phi^+\cap w'w(\Phi^-)}|\langle (w')^{-1}\alpha,\nu\rangle-1|\\
      &=\sum_{\alpha\in (w')^{-1}(\Phi^+)\cap w(\Phi^+)}|\langle \alpha,\nu\rangle| + \sum_{\alpha\in (w')^{-1}(\Phi^+)\cap w(\Phi^-)}|\langle \alpha,\nu\rangle-1|.
  \end{aligned}
\end{gather}
Comparing \eqref{eq:length-of-u-lambda-w} and \eqref{eq:length-of-left-translate-of-v}, we obtain
\begin{multline}
    \label{eq:difference-of-lengths-counting-function}
      \ell(w'v)-\ell(v) =\sum_{\alpha\in \Phi^+\cap(w')^{-1}(\Phi^-)\cap w(\Phi^+)}\bigl(|\langle \alpha,\nu\rangle+1|-|\langle \alpha,\nu\rangle|\bigr) \\+ \sum_{\alpha\in \Phi^{+}\cap (w')^{-1}(\Phi^-)\cap w(\Phi^-)}\bigl(|\langle \alpha,\nu\rangle|-|\langle \alpha,\nu\rangle-1|\bigr).
    \end{multline}
First note that for any~$\alpha\in\Phi^{+}$ and any~$x\in\overline{\Delta}_{\rotorder}$, we have
\begin{equation}
  \label{eq:14}
  \langle\alpha,vx\rangle=\langle \alpha,\rotorder\nu+wx\rangle=\rotorder\langle \alpha,\nu\rangle+\langle w^{-1}(\alpha),x \rangle.
\end{equation}
To establish part~\eqref{item:16},  we take~$x=\lambda_0 $ in~\eqref{eq:14}.
Since $|\langle w^{-1}(\alpha),\lambda_0 \rangle|\le \rotorder$, and $\lambda=v(\lambda_0) $ is dominant, we see that $\langle \alpha,\nu\rangle\ge -1$ for all~$\alpha\in\Phi^+$. Suppose for the sake of contradiction that~$\nu$ is not dominant, so that there exists ~$\alpha\in \Phi^+$ with $\langle \alpha ,\nu\rangle= -1$.
Writing~$\alpha $ as a sum of simple roots, we see that there must in fact be a simple root~$\beta$ with $\langle \beta,\nu\rangle= -1$. Note that by~\eqref{eq:14} we must also have $\langle
w^{-1}(\beta),\lambda_0 \rangle=\rotorder$, so that (by the definition of~$\Delta_{\rotorder}$) we have $w^{-1}(\beta)\in\Phi^+$.
Since $\langle w^{-1}(\beta),\lambda_0 \rangle= \rotorder$ we have $s_{w^{-1}(\beta),\rotorder}(\lambda_0 )=\lambda_0 $, so that $s_{w^{-1}(\beta),\rotorder}\in W_{\rotorder}^{\lambda_0 }$.

We claim that $\ell(vs_{w^{-1}(\beta),\rotorder})=\ell(v)-1$, which contradicts the assumption that~$v$ is of minimal length in $v W_{\rotorder}^{\lambda_0 }$.
To see this, note that since $\langle \beta,\nu\rangle= -1$ we have $s_{\beta}(\nu)=\nu+\beta^{\vee}$, and an easy calculation shows that
\begin{equation*}
  \label{eq:13}
 vs_{w^{-1}(\beta),\rotorder} =u^{\rotorder s_{\beta}(\nu)}s_{\beta}w=s_{\beta}u^{\rotorder\nu}w=s_{\beta}v.
\end{equation*}
We can therefore compute using~\eqref{eq:difference-of-lengths-counting-function} with~$w'=s_{\beta}$.
Since~$\beta$ is assumed simple, we have $\Phi^+\cap s_{\beta}^{-1}(\Phi^-)=\{\beta\}$, and since we have already seen that $w^{-1}(\beta)\in\Phi^+$, we deduce from~\eqref{eq:difference-of-lengths-counting-function} that
\[\ell(vs_{w^{-1}(\beta),\rotorder})-\ell(v)=\ell(s_{\beta}v)-\ell(v)=|\langle \beta,\nu\rangle+1|-|\langle \beta,\nu\rangle|=-1,\]
as claimed.
Thus~$\nu$ is indeed dominant.

To complete the proof of~\eqref{item:16}, we assume for the sake of contradiction that there exists $\beta\in \Phi^{+}$ with $\langle\beta,\nu\rangle=0$ and $w^{-1}(\beta)\in\Phi^{-}$.
Writing~$\beta$ as a sum of simple roots, and using that~$\nu$ is dominant, we see that we may furthermore assume that~$\beta$ is simple.
By the definition of~$\Delta_{\rotorder}$, we have~$\langle w^{-1}(\beta),\lambda_0\rangle \le 0$, and comparing to~\eqref{eq:14} we see that in fact $\langle w^{-1}(\beta),\lambda_0\rangle =0$, so that~$s_{w^{-1}(\beta)}\in W_{\rotorder}^{\lambda_0 }$.
Again, we may assume that~$\beta$ is simple,
and to obtain a contradiction it suffices to show that $\ell(vs_{w^{-1}(\beta)})=\ell(v)-1$.
 Since $\langle\beta,\nu\rangle=0$ we have $s_{\beta}\nu=\nu$, so that
\[ vs_{w^{-1}(\beta)} =u^{\rotorder\nu}s_{\beta}w=s_{\beta}v,\]
 and we can again compute using~\eqref{eq:difference-of-lengths-counting-function} with~$w'=s_{\beta}$.
Since~$\beta$ is assumed simple, we again have $\Phi^+\cap s_{\beta}^{-1}(\Phi^-)=\{\beta\}$, but this time we are assuming that $w^{-1}(\beta)\in\Phi^-$. Since $\langle\beta,\nu\rangle=0$, it follows from~\eqref{eq:difference-of-lengths-counting-function} that
\[\ell(vs_{w^{-1}(\beta)})-\ell(v)=\ell(s_{\beta}v)-\ell(v)=|\langle \beta,\nu\rangle|-|\langle \beta,\nu\rangle-1|=-1,\]
as claimed.
This completes the proof of part~\eqref{item:16}.

We now turn to proving~\eqref{item:17}.
We need to show that  if~$x\in\Delta_{\rotorder}$ and $\alpha \in \Phi^+$, then $\langle \alpha,vx\rangle >0$.
To do this, we use~\eqref{eq:14}; if $\langle \alpha,\nu\rangle > 0$, then $\rotorder\langle \alpha,\nu\rangle+\langle w^{-1}(\alpha),x \rangle> \rotorder-\rotorder=0$, and otherwise $\langle \alpha,\nu\rangle = 0$, in which case $w^{-1}(\alpha)\in\Phi^{+}$, and $\rotorder\langle \alpha,\nu\rangle+\langle w^{-1}(\alpha),x \rangle= \langle w^{-1}(\alpha),x \rangle>0$.

It remains to prove part~\eqref{item:18}.
By Proposition~\ref{prop:double-coset-representatives-basic-properties}~\eqref{item:20} and our assumption that $v$ is minimal in $v W_{\rotorder}^{\lambda_0 },$ it suffices to show that~$v$ is of minimal length in~$Wv$.
Using part~\eqref{item:16}, it follows easily from~\eqref{eq:difference-of-lengths-counting-function} that for any $w'\in W$ we have 
  \[\ell(w'v)-\ell(v)=|\Phi^+\cap(w')^{-1}(\Phi^-)|=\ell(w'),\]
and the result follows immediately.
\end{proof}

\begin{cor}\label{cor:minimal-cosets-Wp-dominant}
Let $\lambda_0 \in \overline{\Delta}_{\rotorder}\cap X_{*}(T)$, and let $v \in W_{\rotorder}$ be an element which is minimal in $vW_{\rotorder}^{\lambda_0 }$.
Then the following are equivalent:
\begin{enumerate}
\item \label{item:101} $v$ is minimal in $Wv$.
\item \label{item:102} $v$ is minimal in $WvW_{\rotorder}^{\lambda_0}$.
\item \label{item:103} $v(\lambda_0)$ is dominant.
\end{enumerate}
\end{cor}
\begin{proof} The equivalence of~\eqref{item:101} and~\eqref{item:102} follows from Proposition~\ref{prop:double-coset-representatives-basic-properties}~\eqref{item:20}, while~\eqref{item:103} implies~\eqref{item:102} by Proposition~\ref{prop:Bruhat-order-and-Wp}~\eqref{item:18}.
To see that~\eqref{item:102} implies~\eqref{item:103}, let
$v'$ be any element of $WvW_{\rotorder}^{\lambda_0 }$. Replacing~$v'$ by~$wv'$ for some $w\in W$ if necessary,
we may assume that~$v'(\lambda_0) $ is dominant.
Replacing~$v'$ by the element of~$v'W_{\rotorder}^{\lambda_0 }$ of minimal length, it follows from Proposition~\ref{prop:Bruhat-order-and-Wp}~\eqref{item:18} that $v'=v$, so in particular $v'(\lambda_0) =v(\lambda_0)$ is dominant, as required.
\end{proof}

\subsection{The \texorpdfstring{$\uparrow$}{uparrow} partial order}\label{subsec:uparrow}
We define partial orderings~$\le,\uparrow$ on~$X_{*}(T)$ as follows; see Figure~\ref{fig:dominant-chain-in-defn-uparrow} for an illustration for~$G=\GL_3$.
The partial order~$\uparrow$ depends on the integer~$\rotorder$, but since~$\rotorder$ is fixed throughout this section, we omit it from the notation.
As we remarked earlier, in the remaining sections of the paper we will take $\rotorder=p$.
\begin{defn}\label{defn:uparrow-coweights}
  We say that $\mu\le \lambda$ if~$\lambda-\mu\in\Z_{\ge 0}\Phi^{\vee,+}$.
  We say that $\mu\uparrow\lambda$ if and only if~$\mu=\lambda$, or there exist reflections $s_i=s_{\alpha_i,\rotorder n_i}\in W_{\rotorder}$, $i=1,\dots,r$, such that
  \begin{equation}
    \label{eq:defn-of-uparrow}
    \mu\le s_1 (\mu) \le (s_2s_1)  (\mu) \le\dots\le (s_{r}\cdots s_1 )(\mu)=\lambda.
  \end{equation}
  Note in particular that if~$\mu\uparrow \lambda$, then $\mu\le \lambda$ and $\mu\in W_{\rotorder}\lambda$.
\end{defn}
\begin{rem}\label{rem:dominant-chain-in-defn-uparrow}
 By Corollary~\ref{cor:uparrow-dominant-chain} below, if~$\mu,\lambda$ are dominant, then one can always take the $(s_{i}\cdots s_1 )\mu$ in~\eqref{eq:defn-of-uparrow} to be dominant.
\end{rem}

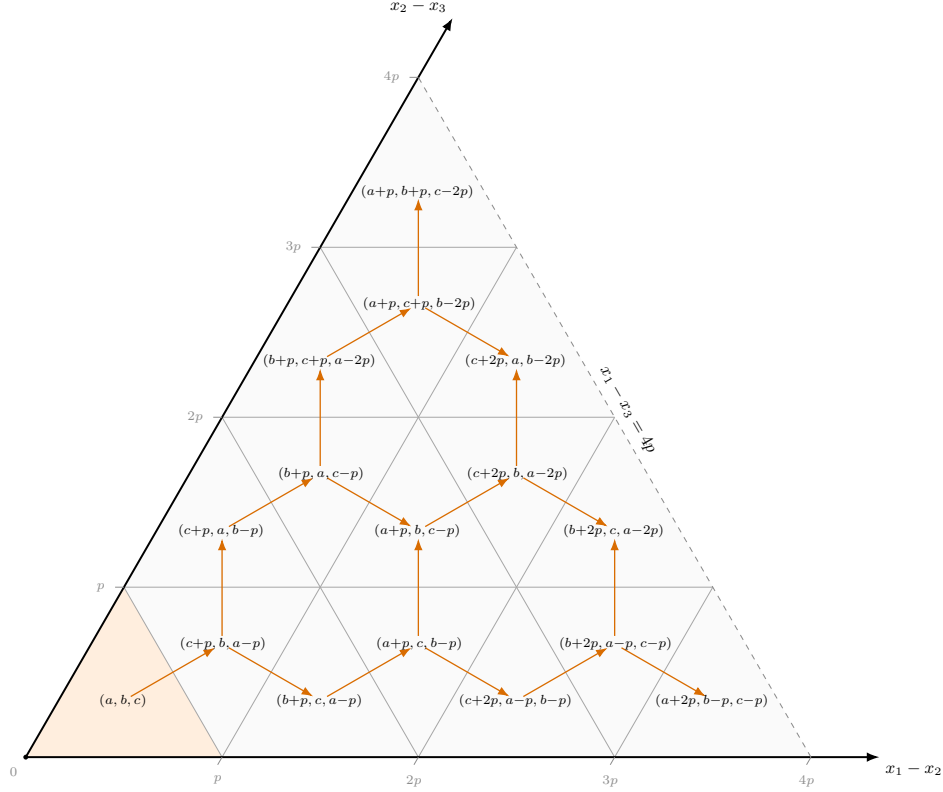
\begin{figure}[htbp]
\centering
\resizebox{\textwidth}{!}{
\begin{tikzpicture}[
  x={(3.666cm,0cm)},
  y={(1.833cm,3.17485cm)},
  >=Latex,
  font=\small,
  wall/.style={line width=0.35pt, draw=black!35},
  chamber/.style={line width=1.05pt, draw=black},
  trunc/.style={line width=0.55pt, draw=black!45, dashed},
  upmove/.style={-{Latex[length=2.2mm,width=1.6mm]}, line width=0.7pt, draw=orange!85!black, shorten >=4.2pt, shorten <=4.2pt},
  alcovefill/.style={fill=orange!13, draw=none},
  labelbox/.style={fill=white, fill opacity=.90, text opacity=1, inner sep=1.0pt, rounded corners=.6pt},
  tinylabel/.style={font=\scriptsize, inner sep=0.4pt, align=center}
]

\def\N{4} 

\fill[black!2] (0,0) -- (\N,0) -- (0,\N) -- cycle;

\fill[alcovefill] (0,0) -- (1,0) -- (0,1) -- cycle;

\foreach \k in {1,2,3} {
  \pgfmathtruncatemacro{\Nk}{\N-\k}
  \draw[wall] (\k,0) -- (\k,\Nk);          
  \draw[wall] (0,\k) -- (\Nk,\k);          
  \draw[wall] (\k,0) -- (0,\k);             
}

\draw[chamber,-{Latex[length=2.4mm]}] (0,0) -- (\N+.35,0)
  node[below right=-1pt] {$x_1-x_2$};
\draw[chamber,-{Latex[length=2.4mm]}] (0,0) -- (0,\N+.35)
  node[above left=-1pt] {$x_2-x_3$};
\draw[trunc] (\N,0) -- (0,\N)
  node[midway, above=3pt, sloped, labelbox] {$x_1-x_3=4\rotorder$};

\foreach \k/\lab in {1/\rotorder,2/2\rotorder,3/3\rotorder,4/4\rotorder} {
  \draw[black!45] (\k,0) -- ++(0,-.045) node[below=2pt, font=\scriptsize] {$\lab$};
  \draw[black!45] (0,\k) -- ++(-.045,0) node[left=2pt, font=\scriptsize] {$\lab$};
}

\foreach \i in {0,1,2} {
  \foreach \j in {0,1,2} {
    \pgfmathtruncatemacro{\s}{\i+\j}
    \ifnum\s<3
      \pgfmathsetmacro{\ax}{\i+1/3}
      \pgfmathsetmacro{\ay}{\j+1/3}
      \pgfmathsetmacro{\bx}{\i+2/3}
      \pgfmathsetmacro{\by}{\j+2/3}
      \draw[upmove] (\ax,\ay) -- (\bx,\by);
      \pgfmathsetmacro{\cx}{\i+4/3}
      \pgfmathsetmacro{\cy}{\j+1/3}
      \draw[upmove] (\bx,\by) -- (\cx,\cy);
      \pgfmathsetmacro{\dx}{\i+1/3}
      \pgfmathsetmacro{\dy}{\j+4/3}
      \draw[upmove] (\bx,\by) -- (\dx,\dy);
    \fi
  }
}

\node[tinylabel] at (.33,.33) {$(a,b,c)$};
\node[tinylabel] at (.67,.67) {$(c{+}\rotorder,b,a{-}\rotorder)$};
\node[tinylabel] at (1.33,.33) {$(b{+}\rotorder,c,a{-}\rotorder)$};
\node[tinylabel] at (.33,1.33) {$(c{+}\rotorder,a,b{-}\rotorder)$};
\node[tinylabel] at (1.67,.67) {$(a{+}\rotorder,c,b{-}\rotorder)$};
\node[tinylabel] at (.67,1.67) {$(b{+}\rotorder,a,c{-}\rotorder)$};
\node[tinylabel] at (2.33,.33) {$(c{+}2\rotorder,a{-}\rotorder,b{-}\rotorder)$};
\node[tinylabel] at (1.33,1.33) {$(a{+}\rotorder,b,c{-}\rotorder)$};
\node[tinylabel] at (.33,2.33) {$(b{+}\rotorder,c{+}\rotorder,a{-}2\rotorder)$};
\node[tinylabel] at (2.67,.67) {$(b{+}2\rotorder,a{-}\rotorder,c{-}\rotorder)$};
\node[tinylabel] at (1.67,1.67) {$(c{+}2\rotorder,b,a{-}2\rotorder)$};
\node[tinylabel] at (.67,2.67) {$(a{+}\rotorder,c{+}\rotorder,b{-}2\rotorder)$};
\node[tinylabel] at (3.33,.33) {$(a{+}2\rotorder,b{-}\rotorder,c{-}\rotorder)$};
\node[tinylabel] at (2.33,1.33) {$(b{+}2\rotorder,c,a{-}2\rotorder)$};
\node[tinylabel] at (1.33,2.33) {$(c{+}2\rotorder,a,b{-}2\rotorder)$};
\node[tinylabel] at (.33,3.33) {$(a{+}\rotorder,b{+}\rotorder,c{-}2\rotorder)$};

\fill[black] (0,0) circle[radius=1.3pt];
\node[below left=2pt and 1pt, font=\scriptsize, text=black!45] at (0,0) {$0$};

\end{tikzpicture}
}
\caption{The $\uparrow$ partial order on dominant coweights $(x_1,x_2,x_3)$ for $\GL_3$.
The highlighted alcove is~$\Delta_{\rotorder}$.
}
\label{fig:dominant-chain-in-defn-uparrow}
\end{figure}

We will sometimes make use of the following basic lemma.
\begin{lem}
  \label{lem:uparrow-dominant-Weyl-orbit}For each $\lambda\in X_*(T)$ and~$w\in W$ we have $w\lambda\uparrow \lambda_{\dom}$.
\end{lem}
\begin{proof}
  This may be proved by a straightforward induction on the length of~$w$; see~\cite[II.6.4~(5)]{MR2015057}.
\end{proof}
We have the following related definition for alcoves.
\begin{defn}\label{defn:uparrow-alcoves}
  Suppose that $\Delta$ is an alcove, that~$\alpha\in\Phi^+$ and that~$n\in\Z$.
  Then either $\langle \alpha,x\rangle<n\rotorder$ for all $x\in \Delta$, in which case we say that $\Delta\uparrow s_{\alpha,n\rotorder}\Delta$; or $\langle \alpha,x\rangle>n\rotorder$ for all~$x\in \Delta$, in which case we say that $s_{\alpha,n\rotorder}\Delta\uparrow \Delta$.
  We then say that $\Delta\uparrow\Delta'$ if and only if $\Delta=\Delta'$, or there are reflections $s_i=s_{\alpha_i,\rotorder n_i}\in W_{\rotorder}$, $i=1,\dots,r$, such that
  \begin{equation*}
    \label{eq:15}
    \Delta\uparrow s_1 \Delta\uparrow s_2 s_1 \Delta\uparrow\dots \uparrow (s_{r}\cdots s_1) \Delta=\Delta'.
  \end{equation*}
\end{defn}
The following theorem of Wang will be crucial below. 
\begin{thm}
  \label{thm:Wang}Suppose that $u,v\in W_{\rotorder}$ are such that $u\Delta_{\rotorder}$ and $v\Delta_{\rotorder}$ are both dominant.
Then the following conditions are equivalent:
\begin{enumerate}
\item\label{item:30} $u\Delta_{\rotorder}\uparrow v\Delta_{\rotorder}$.
\item\label{item:31} $u\le_{\Bru}v$.
\end{enumerate}
\end{thm}
\begin{proof} This is~\cite[Thm.~4.2]{MR4545005}, which was originally proved in~\cite{MR980127}.
See also~\cite[App.~A]{zbMATH06991335} for an exposition of some points of the proof.
\end{proof}
Definitions~\ref{defn:uparrow-coweights} and~\ref{defn:uparrow-alcoves} are related as follows.
\begin{lem}\label{lem:uparrow-from-alcove-to-cocharacter}
  Suppose that $\Delta\uparrow \Delta'$ and that~$\mu\in \overline{\Delta}\cap X_{*}(T)$. Let $\lambda$ be the unique element of $\overline{\Delta'}\cap W_{\rotorder}\mu$.
Then $\mu\uparrow \lambda$.
\end{lem}
\begin{proof} By induction, we can assume that $\Delta'=s_{\alpha,\rotorder n}\Delta$ for some $\alpha,n$, so that $\lambda=s_{\alpha,\rotorder n}(\mu)$. Since $\Delta\uparrow\Delta'$, we have $\langle \alpha,x\rangle<n\rotorder$ for all $x\in \Delta$, so that $\langle \alpha,\mu\rangle\le n\rotorder$.
Then $\lambda-\mu=(n\rotorder-\langle\alpha,\mu\rangle)\alpha^{\vee}\ge 0$, so $\mu\uparrow\lambda$, as required.
\end{proof}

\begin{rem}\label{rem:uparrow-rho-translation}
  When $\rotorder=p$, our definitions differ in two ways from the definition of~$\uparrow$ used in the representation theory of reductive groups in finite characteristic, as in~\cite[II.6.4]{MR2015057}. Firstly, our definitions are for coweights, rather than weights; this is harmless, as one can pass between the two cases by exchanging a root datum with its dual. Secondly, the definition in~\cite[II.6.4]{MR2015057} uses the ``dotted'' action of the affine Weyl group, which is defined as follows (after passing to the dual root datum): for $v\in W_p$ and $\mu\in X_{*}(T)$, we let \[
v \cdot \mu \coloneq v(\mu + \rho^{\vee}) - \rho^{\vee},
\] where~$\rho^{\vee}\coloneq \frac{1}{2}\sum_{\alpha\in \Phi^+}\alpha^{\vee}$.
Then the definition of $\uparrow$ in~\cite[II.6.4]{MR2015057} differs from ours in that in~\eqref{eq:defn-of-uparrow}, the $s_i$ act via the dotted action.

It follows that if  $\rotorder=p$, then $\mu\uparrow\lambda$ in our sense if and only if $(\mu - \rho^{\vee})\uparrow (\lambda-\rho^{\vee})$ in the sense of~\cite[II.6.4]{MR2015057} (where strictly speaking we have now extended the definition to~$X_{*}(T)_{\Z[1/2]}$).
Similarly, the definition of ``alcove'' in \cite[II.6.2]{MR2015057} differs from our definition by the same translation by~$\rho^{\vee}$.
The combinatorial statements about affine reflections and alcoves that we use from~\cite[II.6]{MR2015057} are unchanged by a positive dilation of the apartment, so they apply equally with $p$ replaced by~$\rotorder$.
\end{rem}
  The converse implication in Lemma~\ref{lem:uparrow-from-alcove-to-cocharacter} holds if $\mu\in \Delta\cap X_{*}(T)$, but it fails in general (consider the case $\mu=\lambda=0$).
However we still have the following weaker statement.

\begin{lem}
  \label{lem:equivalence-uparrow-alcoves-weights}Suppose that $\mu,\lambda\in X_{*}(T)^+$ are in the same $W_{\rotorder}$-orbit.
Then the following are equivalent:
\begin{enumerate}
\item\label{item:27} $\mu\uparrow \lambda$.
\item\label{item:28} There exist dominant alcoves $\Delta,\Delta'$ such that $\mu\in\overline{\Delta}$, $\lambda\in \overline{\Delta'}$, and $\Delta\uparrow\Delta'$.
\end{enumerate}
\end{lem}
\begin{proof}By Lemma~\ref{lem:uparrow-from-alcove-to-cocharacter}, the second statement implies the first.
  For the converse, we follow the proof of~\cite[Cor.~A.1.2]{zbMATH06991335}.
We may assume that~$\mu\ne \lambda$.
We let~$F$ denote the unique facet containing~$\mu$, and
we let~$\Delta(\mu)$ be the alcove denoted~$C^+(F)$ in \cite[II.6.11]{MR2015057}, which is defined as follows: for each~$\alpha\in\Phi^+$, there is a unique integer~$m_{\alpha}$ such that \[m_{\alpha}\rotorder\le \langle\alpha,\mu\rangle< (m_{\alpha}+1)\rotorder,\] and we set \[\Delta(\mu)=\{v\in X_*(T)_{\R}: m_{\alpha}\rotorder<\langle\alpha,v\rangle<(m_{\alpha}+1)\rotorder\ \forall \alpha\in
  \Phi^+\}.\]
Since~$\mu$ is dominant, so is~$\Delta(\mu)$. By~\cite[II.6.11(3)]{MR2015057}, $\Delta(\mu)$ is maximal with respect to~$\uparrow$; more precisely, if~$\Delta$ is an alcove whose closure contains~$\mu$, then $\Delta\uparrow\Delta(\mu)$.

We define~$\Delta(\lambda)$ in the same way.
 Since~$\mu\uparrow\lambda$, we can write \[\mu< s_1 (\mu) < (s_2s_1)  (\mu) <\dots< (s_{n}\cdots s_1 )(\mu)=\lambda\] for some reflections $s_1,\dots, s_n$, where we have replaced the inequalities in~\eqref{eq:defn-of-uparrow} by strict inequalities, by discarding any reflections giving equalities.
 Using the strictness of these inequalities, it follows from Definition~\ref{defn:uparrow-alcoves} that  \[\Delta(\mu)\uparrow s_1 \Delta(\mu) \uparrow s_2s_1  \Delta(\mu) \uparrow\dots\uparrow (s_{n}\cdots s_1 )\Delta(\mu).\]
By the maximality of~$\Delta(\lambda)$, we have $(s_n\cdots s_1 )\Delta(\mu)\uparrow\Delta(\lambda).$
Thus~$\Delta(\mu)\uparrow \Delta(\lambda)$, and since $\Delta(\mu)$ and~$\Delta(\lambda)$ are dominant, we are done.
\end{proof}

We note the following corollary, which is~\cite[Cor.~A.1.2(ii)]{zbMATH06991335} (with the same proof).
\begin{cor}
  \label{cor:uparrow-dominant-chain}If $\mu,\lambda\in X_{*}(T)^+$ are such that $\mu\uparrow \lambda$, then there exists a sequence of dominant coweights \[\mu=\lambda_0 \uparrow\cdots\uparrow \lambda_n=\lambda\] and reflections~$s_i\in W_{\rotorder}$ such that $\lambda_i=s_{i}(\lambda_{i-1})$ for all~$i$.
\end{cor}
\begin{proof}
  By Lemma~\ref{lem:equivalence-uparrow-alcoves-weights}, there are dominant alcoves $\Delta,\Delta'$ such that $\mu\in\overline{\Delta}$, $\lambda\in \overline{\Delta'}$, and $\Delta\uparrow\Delta'$.
By \cite[Thm.~2.7]{MR4545005} there exists a sequence of dominant alcoves \[\Delta=\Delta_0 \uparrow\cdots\uparrow \Delta_n=\Delta'\] and reflections~$s_i\in W_{\rotorder}$ such that $\Delta_i=s_{i}\Delta_{i-1}$ for all~$i$.
Taking $\lambda_i\coloneq (s_i\cdots s_1)\mu$, the result follows as in the proof of Lemma~\ref{lem:uparrow-from-alcove-to-cocharacter}.
\end{proof}

The following is the main combinatorial result of this section; under an additional regularity assumption (which is irrelevant for applications to the representation theory of reductive groups, but would be insufficient for our applications in $p$-adic Hodge theory) it was earlier proved by Riche~\cite[Prop.~2.35]{Riche-lectures}.
\begin{prop}
  \label{prop:Bruhat-equivalent-uparrow}Suppose that $\mu,\lambda\in X_{*}(T)^+$.
Then the following are equivalent:
\begin{enumerate}
\item\label{item:25} $\mu\uparrow\lambda$.
\item\label{item:26} There exist $u,v\in W_{\rotorder}$ with $u\le_{\Bru}v$, and $\lambda_0 \in \overline{\Delta}_{\rotorder}\cap X_{*}(T)$, such that $u(\Delta_{\rotorder})$ and~$v(\Delta_{\rotorder})$ are dominant, $\mu=u(\lambda_0) $, and $\lambda=v(\lambda_0)$.
\end{enumerate}
\end{prop}
\begin{proof} First assume that $\mu\uparrow\lambda$.
Then by Lemma~\ref{lem:equivalence-uparrow-alcoves-weights} we can find dominant alcoves $\Delta\uparrow\Delta'$ such that $\mu\in\overline{\Delta}$ and $\lambda\in \overline{\Delta'}$.
Write $\Delta=u\Delta_{\rotorder}, \Delta'=v\Delta_{\rotorder}$ for (unique) elements $u,v\in W_{\rotorder}$.
Since $\Delta,\Delta'$ are dominant, it follows from Theorem~\ref{thm:Wang} that $u\le_{\Bru}v$.
As we are assuming that $\mu\uparrow\lambda$, we in particular have $\mu\in W_{\rotorder}\cdot\lambda$, whence we have $u^{-1}(\mu)=v^{-1}(\lambda)=:\lambda_0 \in\overline{\Delta}_{\rotorder}$, as required.

Now suppose that~\eqref{item:26} holds. Since $u\Delta_{\rotorder}$ and
$v\Delta_{\rotorder}$ are dominant, Theorem~\ref{thm:Wang} gives
$u\Delta_{\rotorder}\uparrow v\Delta_{\rotorder}$,
whence Lemma~\ref{lem:uparrow-from-alcove-to-cocharacter} implies
that $\mu\uparrow \lambda$, as required.
\end{proof}

\subsection{Affine flag varieties}\label{subsec:affine-Grassmannian-recalling}  

Let $A$ be a ring and let $G$ be a smooth affine group scheme over $A$.
We recall the definitions of the loop groups and the affine Grassmannian associated to~$G$,
for which \cite{Cesnavicius} is a general reference; see also
\cite{epiga:12352} and \cite{MR4061978}.
\begin{para}\label{para:Affgrasssetup} We write~$L_nG$, $L^+_nG$ for the $n$th loop and arc groups, so that for any $A$-algebra~$R$ we have $L_nG(R)=G(R((u^n)))$ and~$L_n^+G(R)=G(R\llbracket u^n\rrbracket)$. When~$n=1$ we drop it from the notation.
We define the affine Grassmannian $\Gr_G$ as the \'etale sheafification of the presheaf $LG/L^+G$.
By \cite[Prop.~7]{Cesnavicius}, $\Gr_G(R)$ is in natural bijection with the isomorphism classes of pairs
$(\cE,\iota)$ consisting of a $G$-bundle $\cE$ over $R\llbracket u\rrbracket$ and a trivialization $\iota:\cE|_{R((u))}\isoto \cE^0|_{R((u))}$
of $\cE$ over $R((u))$ (where~$\cE^0$ denotes the trivial $G$-bundle). We sometimes write $\Gr_{G,A}$ for $\Gr_G$ when the ring $A$ may not be clear from the context.
When $G$ is a Chevalley group, i.e.\ a split reductive group, then $\Gr_G$
is representable by an Ind-projective Ind-scheme.

The multiplicative group $\G_m$ acts on $LG, L^+G$ and $\Gr_G$ by ``loop rotation''. That is, $z \in R^\times$ sends $f = f(u) \in G(R((u)))$ to $f(uz)$. In particular, we have the action of the subgroup $\mu_{\rotorder} \subset \G_m$ on $\Gr_G$. We will write $(\Gr_G)^{\mu_{\rotorder}} \subset \Gr_G$ for the fixed points of $\mu_{\rotorder}$.
\end{para}

\begin{para} Now suppose that $G$ is a Chevalley group, equipped with a Borel subgroup $B\subseteq G$ and a maximal torus $T \subseteq B$.
For~$\lambda\in X_{*}(T)^+$ we write $\Gr_{G}^{\le\lambda}$ for the image of the orbit map
$L^+G \rightarrow \Gr_G$ sending $f$ to $f\cdot u^{\lambda}$. This is a projective $A$-scheme, which contains
a unique open, fiberwise dense $L^+G$ orbit $\Gr_G^{\lambda}$.
The Ind-scheme $\Gr_G$ admits a stratification by the disjoint union of the $\Gr_G^{\lambda}$. 

\end{para}

\begin{para}\label{para:integral-Schubert} We continue to assume that $G$ is a Chevalley group.
To any facet~$\mathbf{f}$ (in the sense of Definition~\ref{defn:facet}), Bruhat--Tits theory associates a parahoric group scheme~$P_{\mathbf{f}}$ over $A\llbracket t \rrbracket$,
which is a smooth affine group scheme over~$A\llbracket t \rrbracket$ with connected geometric fibres such that~$P_{\mathbf{f}}\times_{A\llbracket t \rrbracket} A((t)) = G\times_AA((t))$.
Here the corresponding facet in the usual apartment where $t$ has valuation~$1$ is $\rotorder^{-1}\mathbf{f}$.
See \cite[Lem.~11.1]{epiga:12352} for the case $A=\Z$, which implies the general case.

As in \cite[\S 4]{MR4517647}, for $R$ an $A$-algebra, set
$L^+_{\rotorder} P_{\mathbf{f}}(R) = P_{\mathbf{f}}(R \llbracket u^{\rotorder} \rrbracket)$.
Here we view $R \llbracket u^{\rotorder} \rrbracket$ as an $A\llbracket t \rrbracket$
-algebra via $t \mapsto u^{\rotorder}$.
This defines a subgroup $L^+_{\rotorder} P_{\mathbf{f}} \subset L_{\rotorder}G$.
We denote by $\Flag^{\rotorder}_{\mathbf{f},A} = L_{\rotorder}G/L^+_{\rotorder}P_{\mathbf{f}}$ the associated affine flag variety. It is an Ind-projective Ind-scheme. We sometimes write
$\Flag^{\rotorder}_{\mathbf{f}}$ when the base is clear.

For each $v\in W_{\rotorder}$ we have the Schubert scheme $\Flag^{\rotorder,\le v}_{\mathbf{f},A}$, which is by definition (\cite[Defn.~4.3.4]{MR4061978} or   \cite[11.2.3]{epiga:12352}) the scheme-theoretic image of the map $L^+_{\rotorder}G\to \Flag^{\rotorder}_{\mathbf{f},A}$, $g\mapsto gv$, and its open subscheme $\Flag^{\rotorder,v}_{\mathbf{f},A}$, which by definition is the \'etale sheaf-theoretic image $L^+_{\rotorder}G\cdot v\subseteq \Flag^{\rotorder,\le v}_{\mathbf{f},A}$.

By~\cite[Lem.~4.3.7]{MR4061978} and~\cite[Lem.~11.8]{epiga:12352},
$\Flag^{\rotorder,\le v}_{\mathbf{f},A}$ is a projective faithfully flat $A$-scheme with geometrically integral fibres, while $\Flag^{\rotorder,v}_{\mathbf{f},A}$ is smooth, fiberwise geometrically connected, and fiberwise dense in $\Flag^{\rotorder,\le v}_{\mathbf{f},A}$.
\end{para}
\begin{para} We now take $A = \F$ an algebraically closed field, and assume that $G$ is a Chevalley group.
We recall some results from~\cite[\S 4]{MR4517647}, where the order of the roots of unity is an arbitrary positive integer (see~\cite[Rem.~4.1]{MR4517647}).\footnote{In~\cite[\S 4]{MR4517647} $\F$ is assumed to be of positive characteristic, but this assumption is never used.} Write~$\Iw$ for the Iwahori subgroup of~$G(\F\llbracket u\rrbracket )$ given by the
preimage of $B^{-}(\F)$ under the natural map $G(\F\llbracket u\rrbracket )\to G(\F)$,  where $B^{-}$ is opposite to~$B$.
Write~$\Iw_{\rotorder}=\Iw^{\mu_{\rotorder}}$ for the corresponding subgroup of $G(\F\llbracket u^{\rotorder}\rrbracket )$.
\end{para}
\begin{lem}
  \label{lem:mup-invariants-spherical}For each~$\lambda\in X_{*}(T)^+$ we have \[(\Gr_G^{\lambda})^{\mu_{\rotorder}}=L^+_{\rotorder}G\cdot u^{\lambda}.\]
\end{lem}
\begin{proof} This can be deduced from the results of~\cite[\S 4.4]{MR4517647} as follows.
  We have \begin{equation}\label{eqn:K-vs-I-orbits}\Gr_G^{\lambda}=\coprod_{\mu\in W\lambda}\Iw\cdot u^{\mu},\end{equation} and by~\cite[Lem.~4.4]{MR4517647} we have
\begin{equation}
  \label{eq:7}
  (\Iw\cdot u^{\mu})^{\mu_{\rotorder}}=\Iw_{\rotorder}\cdot u^{\mu}.
\end{equation}
We certainly have $L^+_{\rotorder}G\cdot u^{\lambda}\subseteq (\Gr_G^{\lambda})^{\mu_{\rotorder}}$, so we only need to prove the opposite inclusion.
By~\eqref{eqn:K-vs-I-orbits} and~\eqref{eq:7}, it suffices to show that for each $w\in W$ we have $u^{w(\lambda)}\in L^+_{\rotorder}G\cdot u^{\lambda}$; this follows from writing $u^{w(\lambda)}=
wu^{\lambda}w^{-1}
$ (where as usual 
we have chosen a lift of~$w$ to $N_G(T)(\F)$).  
\end{proof}

\begin{para}We continue to assume that $A = \F$ is an algebraically closed field. We now explain the relationship between affine flag varieties and the $\mu_{\rotorder}$-fixed points $(\Gr_G)^{\mu_{\rotorder}}$.

Assume from now on that the facet~$\mathbf{f}$ is contained in~$\overline{\Delta}_{\rotorder}$; then $P_{\Delta_{\rotorder}}\subseteq P_{\mathbf{f}}$, and $P_{\Delta_{\rotorder}}$ is the Iwahori subgroup
given by
the inverse image of the opposite Borel~$B^-$ under evaluation at~$t=0$. 
If we let $W_{\rotorder}^{\mathbf{f}}\le W_{\rotorder}$ be the pointwise stabilizer of~$\mathbf{f}$, then \[P_{\mathbf{f}}=\coprod_{v\in W_{\rotorder}^{\mathbf{f}}}P_{\Delta_{\rotorder}}vP_{\Delta_{\rotorder}},\] where we regard~$v\in W_{\rotorder}$ as an element of $G(\F((u^{\rotorder})))$ via the embedding of~$\Wt$ into~$G(\F((u)))$ defined in Section~\ref{sec:affine-weyl-group}.
For example, we see that for $\mathbf{o}_{\rotorder}\coloneq 0\in X_{*}(T)$ we have $P_{\mathbf{o}_{\rotorder}}=G_{\F\llbracket u^{\rotorder}\rrbracket }$.
(Note that some of our formulas differ from those found in~\cite[\S 4]{MR4517647}, due to our differing conventions on signs; see~\cite[\S 2]{arXiv:2405.17174} for a helpful discussion of the possible conventions.)

The connected components of the affine flag variety $\Flag_{\mathbf{f}}^{\rotorder}$ are indexed by~$\pi_1(G)=X_{*}(T)/\Lambda^{\vee}$ via the Kottwitz homomorphism; concretely, each connected component of~$(\Flag_{\mathbf{f}}^{\rotorder})_{\red}$ is stratified by the orbits $L^+_{\rotorder}G\cdot u^{\rotorder\lambda}$ for~$\lambda$ in the corresponding coset of~$\Lambda^{\vee}$.
We write $\Flag_{\mathbf{f}}^{\rotorder,\circ}$ for the neutral component (for which~$\lambda\in \Lambda^{\vee}$).

For each $\lambda_0\in \overline{\Delta}_{\rotorder}\cap X_{*}(T)$, we write $\mathbf{f}_{\lambda_0}\subseteq \overline{\Delta}_{\rotorder}$ for the unique facet containing~$\lambda_0$.
Then the morphism $L_{\rotorder}G\to\Gr_{G}$, $g\mapsto gu^{\lambda_0}$ induces a morphism
\begin{equation*}
  \label{eq:6}
  \Flag^{\rotorder,\circ}_{\mathbf{f}_{\lambda_0}}\to \Gr_G^{\mu_{\rotorder}}.
\end{equation*}
By~\cite[Prop.~4.7]{MR4517647}, these morphisms induce an isomorphism of ind-schemes
  \begin{equation}
    \label{eq:5}
    \coprod_{\lambda_0\in \overline{\Delta}_{\rotorder}\cap X_{*}(T)}\Flag^{\rotorder,\circ}_{\mathbf{f}_{\lambda_0}}\isoto \Gr_G^{\mu_{\rotorder}}.
  \end{equation}

\end{para}
  \begin{prop}
    \label{prop:K-orbit-closure-mu-fixed}Suppose that  $A = \F$ is an algebraically closed field. Then for each $\lambda\in X_{*}(T)^+$, the Zariski closure $\Gr_{G}^{\mu_{\rotorder},\uparrow\lambda}$
    of~$(\Gr_{G}^{\lambda})^{\mu_{\rotorder}}$ in $\Gr_{G}^{\mu_{\rotorder}}$ is given (on the level of underlying sets)
    by the disjoint union of the $(\Gr_{G}^{\lambda'})^{\mu_{\rotorder}}$ where $\lambda'\in X_{*}(T)^+$ and $\lambda'\uparrow\lambda$.
    \end{prop}

    \begin{proof}
      Let $\lambda_0\in \overline{\Delta}_{\rotorder}\cap X_{*}(T)$ be the unique element with~$\lambda\in W_{\rotorder}\lambda_0 $.
Write $\lambda=v(\lambda_0) $ for some~$v\in W_{\rotorder}$.
By Lemma~\ref{lem:mup-invariants-spherical} we have $(\Gr_{G}^{\lambda})^{\mu_{\rotorder}}=L^+_{\rotorder}G\cdot u^{\lambda}$, which is identified with the orbit $L^+_{\rotorder}G\cdot v$ in $\Flag^{\rotorder,\circ}_{\mathbf{f}_{\lambda_0}}$
       via the isomorphism~\eqref{eq:5}; so $\Gr_{G}^{\mu_{\rotorder},\uparrow\lambda}$ is identified with the affine Schubert variety $\Flag^{\rotorder,\le v}_{\mathbf{f}_{\lambda_0}}$ (i.e.\ the Zariski closure of $L^+_{\rotorder}G\cdot v$ in $\Flag^{\rotorder,\circ}_{\mathbf{f}_{\lambda_0}}$).

By~\cite[Prop.~2.8]{MR2998951} and Proposition~\ref{prop:equivalence-Bruhat-orderings-double-cosets} above (with~$W_S=W_{\rotorder}$, $W_I=W_{\rotorder}^{\lambda_0 }$, and~$W_J=W$), the closure of this orbit in $\Flag^{\rotorder,\circ}_{\mathbf{f}_{\lambda_0}}$ is the disjoint union of the orbits $L^+_{\rotorder}G\cdot \underline{v'}$, where~$\underline{v'}$ runs over those minimal length representatives in $W_{\rotorder}$ of the elements of $W\backslash W_{\rotorder}/W_{\rotorder}^{\lambda_0 }$ which satisfy $\underline{v'}\le_{\Bru}\underline{v}$.

Using~\eqref{eq:5} again, we see that the closure of~$(\Gr_{G}^{\lambda})^{\mu_{\rotorder}}$ is the disjoint union of the $L^+_{\rotorder}G\cdot u^{\underline{v'}(\lambda_0) }$. By Corollary~\ref{cor:minimal-cosets-Wp-dominant} and Proposition~\ref{prop:Bruhat-equivalent-uparrow}, the cocharacters $\underline{v'}(\lambda_0)$ with
$\underline{v'}\le_{\Bru}\underline{v}$ are exactly the dominant cocharacters~$\mu$ with $\mu\uparrow\lambda$, as required.
\end{proof}

\begin{rem}
\label{affgrassfixedorbit}
Proposition~\ref{prop:K-orbit-closure-mu-fixed} implies that the inclusion $\mathrm{Gr}_G^{\mu_{\rotorder}} \to \mathrm{Gr}_G$ induces a bijection $L^+_{\rotorder}G \backslash \mathrm{Gr}_G^{\mu_{\rotorder}} \to L^+G\backslash \mathrm{Gr}_G$ on orbit {\em sets}\footnote{In the forthcoming \S \ref{ss:fungrass}, we shall give a functor of points description of the corresponding quotient stacks.}, with both being identified with $X_*(T)^+$. However, the closure relations on the right are given by the standard dominance order on $X_*(T)^+$, while those on the left are given by the more restrictive $\uparrow$ relation.
\end{rem}
\subsubsection{The switched affine Grassmannian}
It will be convenient in much of the rest of the paper to use the switched version $\Gr^{sw}_G = L^+G \backslash LG$ of the affine Grassmannian, where we quotient by $L^+G$ on the left; this is isomorphic to $\Gr_G$ as an ind-scheme (via the inversion map on $LG$ and $L^+G$), so all statements we have made about~$\Gr_{G}$ can be translated to statements about $\Gr^{sw}_G$ at the expense of taking inverses. In particular, using the fact that we have $\mu\uparrow \lambda$ if and only if $-w_0(\mu)\uparrow -w_0(\lambda)$ (see e.g.\ \cite[II.6.4 (6)]{MR2015057}), we see that the statement of Proposition~\ref{prop:K-orbit-closure-mu-fixed} goes over unchanged to the switched affine Grassmannian.

\subsection{An integral version}\label{subsec:integral-mu-fixed}
 In this subsection, let $k$ be an algebraically closed field of characteristic~$p$,
$K$ an algebraic closure of $W(k)[1/p]$, and $\O$ the ring of integers of~$K$. We let $G$ be a Chevalley group 
over~$\O$. The integer~$\rotorder>1$ remains arbitrary; in particular, we allow $p\mid\rotorder$.

\begin{lem} For each $\lambda\in X_*(T)^+$, 
$\Gr_{G,\O}^{\mu_{\rotorder},\uparrow \lambda},$ defined as the closure of
$(\Gr^{\lambda}_{G,\O})^{\mu_{\rotorder}}$ in $(\Gr_{G,\O})^{\mu_{\rotorder}},$
coincides with the scheme-theoretic image of the map
\[ L^+_{\rotorder}G\to  \Gr_{G,\O}^{\mu_{\rotorder}}; \quad  g\mapsto g u^{\lambda}.
\]
\end{lem}
\begin{proof} The map in the lemma factors through a map $L^+_{\rotorder}G\to (\Gr^{\lambda}_{G,\O})^{\mu_{\rotorder}}$, and this map is surjective, as by Lemma~\ref{lem:mup-invariants-spherical} it is surjective over each point of $\Spec \O$.
The lemma follows.
\end{proof}

  \begin{prop}\label{prop:closure-mu-fixed-Zpbar-Fpbar}
    If~$\F=k$ or~$K$, then the natural closed immersion \[\Gr_{G,\F}^{\mu_{\rotorder},\uparrow\lambda}\into ( \Gr_{G,\O}^{\mu_{\rotorder},\uparrow \lambda}\times_{\O}\F)_{\red}\]is an isomorphism.
  \end{prop}
  \begin{proof}
As in the proof of Proposition~\ref{prop:K-orbit-closure-mu-fixed}, we write $\lambda=v(\lambda_0)$ with~$v\in W_{\rotorder}$,
$\lambda_0\in \overline{\Delta}_{\rotorder}\cap X_{*}(T)$.
The obvious commutative diagram
\[\begin{tikzcd}
	& L^+_{\rotorder}G \\
\Flag^{\rotorder}_{\mathbf{f}_{\lambda_0 },\O} && \Gr_{G,\O}^{\mu_{\rotorder}}
	\arrow["g\mapsto gv"', from=1-2, to=2-1]
	\arrow["g\mapsto gu^{\lambda}", from=1-2, to=2-3]
	\arrow["g\mapsto gu^{\lambda_0 }", from=2-1, to=2-3]
\end{tikzcd}\]
induces a natural morphism
\begin{equation*}
\label{eq:241}
\Flag^{\rotorder,v}_{\mathbf{f}_{\lambda_0}, \O} \to (\Gr_{G,\O}^{\lambda})^{\mu_{\rotorder}},
\end{equation*}
which is surjective by Lemma~\ref{lem:mup-invariants-spherical}. Hence the induced map
\begin{equation}
\label{eq:25}
\Flag^{\rotorder,\le v}_{\mathbf{f}_{\lambda_0},\O} \to \Gr_{G,\O}^{\mu_{\rotorder},\uparrow \lambda},
\end{equation}
is scheme-theoretically surjective by the definitions of each side as a scheme-theoretic closure. As (\ref{eq:25}) is a map of projective schemes over~$\O$,
this implies that it is surjective.

We therefore have a commutative diagram
\[\begin{tikzcd}
\Flag^{\rotorder,\le v}_{\mathbf{f}_{\lambda_0 },\F}	 & (\Flag^{\rotorder,\le v}_{\mathbf{f}_{\lambda_0},\O}\times_{\O}\F)_{\red} \\
\Gr_{G,\F}^{\mu_{\rotorder},\uparrow\lambda}	 & (\Gr_{G,\O}^{\mu_{\rotorder},\uparrow \lambda}\times_{\O}\F)_{\red}
	\arrow["\sim", from=1-1, to=1-2]
	\arrow["\sim"', from=1-1, to=2-1]
	\arrow["", from=1-2, to=2-2]
	\arrow["", hook, from=2-1, to=2-2]
\end{tikzcd}\]
where the left-hand vertical isomorphism was noted in the proof of Proposition~\ref{prop:K-orbit-closure-mu-fixed}, and the top horizontal map is an isomorphism by \cite[Lem.~4.3.6]{MR4061978} or
~\cite[Cor.~11.15]{epiga:12352}.
The right-hand vertical map is surjective (since~\eqref{eq:25} is),
so the closed immersion $\Gr_{G,\F}^{\mu_{\rotorder},\uparrow\lambda}\into ( \Gr_{G,\O}^{\mu_{\rotorder},\uparrow \lambda}\times_{\O}\F)_{\red}$ is also surjective, and thus an isomorphism, as claimed.
  \end{proof}

\subsection{A functor of points interpretation}
\label{ss:fungrass}
Let $G$ be a smooth affine group scheme over a ring~$A$, as in~\ref{para:Affgrasssetup}.
For later use, it will be convenient to record the geometric perspective on the quotient stack $(\Gr_{G}^{sw})^{\mu_{\rotorder}}/L^+_{\rotorder}G$ in terms of the standard one for the Hecke stack $\Gr^{sw}_{G}/L^+G$.

In what follows, all quotients are interpreted as sheaves of groupoids for the \'etale\footnote{It would be more natural to use the pro-\'etale topology when forming quotients by the pro-(smooth affine) group scheme $L^+G$. However, since $L^+G \to G$ is surjective with a kernel that has vanishing higher pro-\'etale cohomology on affines, \'etale and pro-\'etale $L^+G$-torsors on affines coincide, so it suffices to work with the \'etale topology.} topology on $A$-algebras.
Starting with the (switched) affine Grassmannian $\Gr^{sw}_{G}$ of $G$ and its natural $\mu_{\rotorder}$-action (considered in \S \ref{subsec:affine-Grassmannian-recalling}), we obtain three stacks over $A$:
\[ \Hk_G \coloneq  \Gr^{sw}_{G}/L^+G, \quad  (\Hk_G)^{\mu_{\rotorder}}, \quad \text{and} \quad \rotHk_G = (\Gr^{sw}_{G})^{\mu_{\rotorder}}/L^+_{\rotorder} G, \]
Here, for a stack $\mathcal{X}$ equipped with an action of a group scheme~$\mu$, the fixed point stack $\mathcal{X}^{\mu}$ is, by definition, the stack of sections of $\mathcal{X}/\mu \to B\mu$. There are natural maps
\[ \rotHk_G \to (\Hk_G)^{\mu_{\rotorder}} \to \Hk_G.\]
Our goal is to describe the functor of points of these stacks (especially $\rotHk_G$). For a test $A$-algebra $R$, write
 $\mathbf{D}_{R} = \mathrm{Spec}(R\llbracket u \rrbracket)$ and
 $\mathbf{D}_{R,\rotorder} = \mathrm{Spec}(R\llbracket u^{\rotorder} \rrbracket)$ for the $\rotorder$-speed disc, so there is a natural degree $\rotorder$ map
 $\phi_{/R}:\mathbf{D}_R \to \mathbf{D}_{R,\rotorder}$, which factors through
 the stack quotient $\mathbf{D}_{R}/\mu_{\rotorder}$. Write $\mathbf{D}^\circ_R = \mathrm{Spec}(R((u))) \subset \mathbf{D}_R$ for the punctured disc, and similarly $\mathbf{D}^\circ_{R,\rotorder} \subset \mathbf{D}_{R,\rotorder}$ for the $\rotorder$-speed version. The restriction of $\phi_{/R}$ to the punctured discs is a $\mu_{\rotorder}$-torsor for the fppf topology, and hence induces an identification $\mathbf{D}^{\circ}_{R}/\mu_{\rotorder} \simeq \mathbf{D}^\circ_{R,\rotorder}$. By \cite[Prop.~7]{Cesnavicius}, the set $\Gr^{sw}_G(R)$ is naturally identified with the (discrete) groupoid of pairs
$(M,\iota)$ consisting of a $G$-torsor $M$ over $\mathbf{D}_R$ and a trivialization $\iota:M^0|_{\mathbf{D}^\circ_R}\isoto M|_{\mathbf{D}^\circ_R}$ of $M$ over $\mathbf{D}^\circ_R$ (where~$M^0$ denotes the trivial $G$-torsor); this has a natural $LG(R)$-action on the right via automorphisms of $M^0$. From this description, we obtain the following description of $\Hk_G$, $(\Hk_G)^{\mu_{\rotorder}}$ and $\rotHk_G$:

\begin{prop} Fix a test ring $R$.
\label{prop:funcmupfixed}
\begin{enumerate}
    \item $\Hk_G(R)$ is naturally (in $R$) identified with the groupoid $\mathcal{C}_1(R)$ of triples $(M_1,M_2,\tau)$, where the $M_i$ are $G$-torsors on $\mathbf{D}_R$ and $\tau:M_1|_{\mathbf{D}^\circ_R} \simeq M_2|_{\mathbf{D}^{\circ}_R}$ is an isomorphism of torsors.

    \item $(\Hk_G)^{\mu_{\rotorder}}(R)$ is naturally (in $R$) identified with the groupoid $\mathcal{C}_2(R)$ of triples $(M_1,M_2,\tau)$, where the $M_i$ are $\mu_{\rotorder}$-equivariant $G$-torsors on $\mathbf{D}_R$ and $\tau:M_1|_{\mathbf{D}^\circ_R} \simeq M_2|_{\mathbf{D}^{\circ}_R}$ is a $\mu_{\rotorder}$-equivariant isomorphism of torsors.

    \item $\rotHk_G(R)$ is naturally (in $R$) identified with the groupoid $\mathcal{C}_3(R)$ of triples $(M_1,M_2,\tau),$
where $M_1$ is a $G$-torsor on $\mathbf{D}_{R,\rotorder}$, $M_2$ is a $\mu_{\rotorder}$-equivariant $G$-torsor on $\mathbf{D}_{R}$, and
 $\tau:\phi_{/R}^* M_1|_{\mathbf{D}^\circ_R} \isoto M_2|_{\mathbf{D}^\circ_R}$ is a $\mu_{\rotorder}$-equivariant isomorphism of torsors.
\end{enumerate}
\end{prop}

\begin{proof}

For (1), the functor of points description of $\Gr^{sw}_G$ recalled above
gives rise to a natural $L^+G(R)$-equivariant map $\Gr^{sw}_G(R) \rightarrow \mathcal{C}_1(R)$ simply by forgetting the canonical trivialization on the trivial $G$-bundle on $\mathbf{D}_R$. The resulting map $\Gr^{sw}_G(R)/L^+G(R) \to \mathcal{C}_1(R)$ is fully faithful by construction, with essential image being exactly those triples $(M_1,M_2,\tau)$ where $M_1$ admits a trivialization. It then suffices to show that for any triple $(M_1,M_2,\tau) \in \mathcal{C}_1(R)$, after replacing $R$ with an \'etale cover, we can find a trivialization of $M_1$. By smoothness of $G$, a $G$-torsor on $\mathbf{D}_R$ is trivial if and only if its fibre over the origin $\mathrm{Spec}(R) \subset \mathbf{D}_R$ is trivial; as the latter can be achieved \'etale locally on $R$ by smoothness of $G$, this gives the result.

The description in (2) follows formally from (1) by unwinding the definition of $\mu_{\rotorder}$-fixed points.

Finally, for (3), note that $(\Gr^{sw}_G)^{\mu_{\rotorder}}(R)$ can be viewed as the groupoid of pairs $(M, \iota)$,
where $M$ is a $\mu_{\rotorder}$-equivariant $G$-bundle on $\mathbf{D}_R$, and $\iota:\phi_{/R}^*(M_0)|_{\mathbf{D}^\circ_R} \to M|_{\mathbf{D}^\circ_R}$ is a $\mu_{\rotorder}$-equivariant isomorphism, where $M_0$ denotes the trivial $G$-bundle on $\mathbf{D}_{R,\rotorder}$. One then argues as in (1) above.
\end{proof}

\begin{remark}[Comparing $\rotHk_G$ with $(\Hk_G)^{\mu_{\rotorder}}$]
Fix a test ring $R$. As pullback along the map $\mathbf{D}_R/\mu_{\rotorder} \to \mathbf{D}_{R,\rotorder}$ is fully faithful on quasi-coherent sheaves, the natural map $\rotHk_G(R) \to (\Hk_G)^{\mu_{\rotorder}}(R)$ is fully faithful. However, it is not essentially surjective, even when $R=k$ is an algebraically closed field. Indeed, using linear reductivity of $\mu_{\rotorder}$ and deformation theory, one can verify that a point $(M_1,M_2,\tau) \in \mathcal{C}_2(k) = (\Hk_G)^{\mu_{\rotorder}}(k)$ lies in $(\rotHk_G)(k)$ exactly when the induced $\mu_{\rotorder}$-equivariant $G$-torsor $0^* M_1$ on the $0$-section $0:\mathrm{Spec}(k) \to \mathbf{D}_k$ is $\mu_{\rotorder}$-equivariantly trivial; noting that $\mu_{\rotorder}$-equivariant $G$-torsors on $\mathrm{Spec}(k)$ are given by (conjugacy classes of) representations $\mu_{\rotorder} \to G$, it is easy to find examples where $0^* M_1$ is not $\mu_{\rotorder}$-equivariantly trivial, showing strictness of the containment $\rotHk_G(k) \to (\Hk_G)^{\mu_{\rotorder}}(k)$.
\end{remark}

\begin{lemma}\label{lem:pointquotient} Let $R$ be a local Henselian $A$-algebra, and suppose that either the residue field $k$ of $R$ is separably closed or that $G$ is connected, and $k$ is finite. Then
$$ \Hk_G(R) = \Gr^{sw}_{G}(R)/L^+G(R) \quad \text{and \,\,\,} \rotHk_G(R) = (\Gr^{sw}_{G})^{\mu_{\rotorder}}(R)/L^+_{\rotorder}G(R). $$
In particular, if $R=k$ is separably closed and $G$ is connected reductive, then we have
$$ |\rotHk_G(k)| \simeq |\Hk_G(k)| \simeq X_*(T)^+ $$
where $T\subseteq G\otimes k$ is a maximal torus, and $|-|$ denotes the set of isomorphism classes of objects.
\end{lemma}

\begin{proof} As $L^+_{\rotorder} G \simeq L^+G$ abstractly, the first claim in the lemma follows from the vanishing of $H^1(\mathrm{Spec}(R),L^+G).$
To see that this group is trivial, write $L^+G$ as an inverse limit of extensions of $G$ by iterated extensions of vector groups. Then this reduces to the corresponding vanishing claim for the smooth group $G$ itself, which follows from Lang's theorem in the finite residue field case and is clear in the separably closed residue field case.

The second claim now follows from Remark~\ref{affgrassfixedorbit}.
\end{proof}

\begin{cor}\label{cor:closure-from-zpbar-point} Let $\O$ be as in~\ref{subsec:integral-mu-fixed}, and let $G$ be a connected reductive group over~$\O$.
Let $\beta \in \rotHk_G(\O),$ and write $\beta(K) \in \rotHk_G(K)$ and
$\beta(k) \in \rotHk_G(k)$
for the images of~$\beta$. Then
$$ \beta(k) \uparrow \beta(K) \in X_*(T)^+.$$
\end{cor}
\begin{proof} By Lemma~\ref{lem:pointquotient},
$\beta(k)$ and $\beta(K)$ can be regarded as elements of $X_*(T)^+$.
The result now follows from Proposition~\ref{prop:K-orbit-closure-mu-fixed} and Proposition~\ref{prop:closure-mu-fixed-Zpbar-Fpbar}.
\end{proof}

\section{Rees stacks and quasi-coherent sheaves}\label{sec:Rees-stacks}
\label{ss:ReesGeneral}

In this section, let $B$ be a commutative ring equipped with an invertible ideal $J \subset B$.  Given a triple $(M,N,\psi)$ consisting of finite projective $B$-modules $M$ and $N$ with an isomorphism $\psi:M[1/J] \simeq N[1/J]$ over $B[1/J]$, we obtain a $\mathbf{Z}$-graded $B/J$-module given in degree $i$ by
\[ (J^i N \cap M) / (J^{i+1} N \cap M + J(J^{i-1}N \cap M)).\]
The first goal of this section is to describe a geometric construction of this graded module via the Rees construction for $J$-adically filtered $B$-modules, in order to make certain base change properties evident; see \S \ref{sss:numrees}. The geometric formulation also makes sense when the input datum $\mathrm{Spec}(B/J) \subset \mathrm{Spec}(B)$ has been replaced by a suitable morphism of formal stacks (see \S \ref{sss:numstack}), and will be used in the sequel for the closed immersion $W(k)^{\HT} \subset W(k)^\Prism$. For our applications, we shall also need a certain ``$p$-speed'' variant of the preceding geometric constructions, discussed in \S \ref{ss:twistedrees}.

For any scheme~$S$, we write~$\Vect(S)$ for the category of vector bundles on~$S$. For a commutative ring $A$, we write~$\Vect(A)$ for~$\Vect(\Spec A)$, which we shall often identify with the category of finite projective $A$-modules.

A recent more comprehensive account of the Rees construction and its uses in related contexts can be found in~\cite{viehmann2025modulitruncatedshtukasdisplays}.

\subsection{Rees stacks for filtered rings}
There is a $\mathbf{Z}_{\geq 0}$-indexed filtration $J^\bullet B$ on $B$ given by powers of $J$, and its extended Rees algebra is defined as the graded $B[t]$-subalgebra
\[\mathrm{Rees}(J^\bullet B) \coloneq  \bigoplus_{i \in \mathbf{Z}} J^{\max(i,0)} t^{-i} \subset B[1/J,t,t^{-1}], \]
where $t$ has degree $1$. In relating this expression to objects more familiar in $p$-adic Hodge theory, it might be helpful rewrite the right hand side above as
\[ \mathrm{Fil}^0(J^{\max(\bullet,0)} B \otimes_B t^\bullet B[t^{\pm 1}]) \]
where $t^\bullet B[t^{\pm 1}]$ denotes the $\mathbf{Z}$-indexed filtration on $B[t^{\pm 1}]$ given by powers of $t$, the object $J^{\max(\bullet,0)} B$ is the evident displayed $\mathbf{Z}$-filtered ring, and the tensor product is endowed with the natural $\otimes$-product $\mathbf{Z}$-indexed filtration.

\begin{example}
\label{ex:Rees}
Say $J=(d)$ is principal. Then
\[ \mathrm{Rees}(J^\bullet B) = B[\tilde{d},t]/(\tilde{d}t-d),\] where $\tilde{d}$ has degree $-1$. Note that the elements $t$ and $\tilde{d}$ form a regular sequence of length $2$.
\end{example}

Viewing gradings as $\mathbf{G}_m$-actions, we can define:

\begin{defn}
The Rees stack of $J^\bullet B$ is the quotient stack
\[ X=\mathrm{Spec}(\mathrm{Rees}(J^\bullet B))/\mathbf{G}_m,\]
regarded as an Artin stack over $\mathrm{Spec}(B)$.
\end{defn}

The functor of points of $X$ is described as follows, generalizing Example~\ref{ex:Rees}. For a $B$-algebra $R$, a point of $X(R)$ is determined by an invertible $R$-module $L$
as well as a factorization
\[ J \otimes_B R \xrightarrow{\tilde{d}'} L \xrightarrow{t'} R\]
of the tautological map $J \otimes_B R \to R$.  In particular, the identity map $X \to X$ corresponds to a universal factorization
\[ J \otimes_B \mathcal{O}_X \xrightarrow{\tilde{d}} \mathcal{O}_X(-1) \xrightarrow{t} \mathcal{O}_X,\]
where the second map is the datum defining the map to $\mathbf{A}^1/\mathbf{G}_m = \mathrm{Spec}(B[t])/\mathbf{G}_m$ coming from the Rees parameter. From now on, we use the symbols $\tilde{d}$ and $t$ to refer to the maps appearing above, viewed as sections of suitable line bundles on $X$. For future reference, note that the elements $t$ and $\tilde{d}$, locally on $X$, form a regular sequence of length $2$ by Example~\ref{ex:Rees}.

The functor of points description then gives the following description of certain associated substacks of $X$:

\begin{lemma}
\label{lem:opensubstackdesc}\leavevmode
\begin{enumerate}
    \item The open substacks $X_{t\neq 0}$, $X_{\tilde{d} \neq 0}$ of $X$ both map isomorphically to $\mathrm{Spec}(B)$ via the structure map.

    \item The intersection $X_{t \neq 0} \cap X_{\tilde{d} \neq 0}$ is the open subscheme $\mathrm{Spec}(B[1/J])$ in either copy of $\mathrm{Spec}(B)$ in the isomorphisms from (1).

    \item The closed substack $X_{t=\tilde{d}=0} \subset X$ is identified with $\mathrm{Spec}(B/J) \times B\mathbf{G}_m$.
    \end{enumerate}
\end{lemma}

Write $X^\circ \coloneq  X_{t \neq 0} \cup X_{\tilde{d} \neq 0}$ for the open substack of $X$ where at least one of $t$ or $\tilde{d}$ is invertible, and let $j:X^\circ \to X$ be the corresponding open immersion.

\begin{remark}
\label{rmk:tonothercopy}
For future reference, let us note that the restriction to $\mathrm{Spec}(B) = X_{\tilde{d} \neq 0}$ of the tautological section $t:\mathcal{O}_X(-1) \to \mathcal{O}_X$ coincides with the inclusion $J \subset B$ under our identifications.
\end{remark}

\subsection{Sheaves on Rees stacks}

Let us describe quasi-coherent sheaves on $X^\circ$ and $X$ in linear algebraic terms.  For $X^\circ$, we need the category of isogenies:

\begin{defn}
    The category $\mathrm{Isog}(B,J)$ of $J$-isogenies is the category of triples $(M,N,\psi)$, where $M,N \in \mathrm{Vect}(B)$ and $\psi:M[1/J] \simeq N[1/J]$ is an isomorphism of $B[1/J]$-modules.
\end{defn}

The relevance of this category to current considerations is:
\begin{lemma}
\label{lem:qcohxcirc}
    The category $\mathrm{Vect}(X^\circ)$ of vector bundles on the open substack $j:X^\circ \subset X$ is naturally identified with $\mathrm{Isog}(B,J)$.
\end{lemma}
\begin{proof}
Lemma~\ref{lem:opensubstackdesc} shows that $X^\circ$ is the union of two open copies of $\mathrm{Spec}(B)$ glued along the common open $\mathrm{Spec}(B[1/J])$ in either copy. The lemma then follows immediately from gluing considerations.
\end{proof}

\begin{remark}
    There are two natural choices for the equivalence in Lemma~\ref{lem:qcohxcirc}, depending on which copy of $\mathrm{Spec}(B)$ inside $X$ carries $M$; in this paper, we always follow the convention that the vector bundle corresponding to $(M,N,\psi) \in \mathrm{Isog}(B,J)$ recovers $M$ over $X_{t \neq 0}$ and $N$ over $X_{\tilde{d} \neq 0}$.
\end{remark}

Next, passing to $X$, we have the following:

\begin{lemma}
\label{lem:qcohX}
The following categories are equivalent:
\begin{enumerate}
    \item The category of $t$-torsion-free quasi-coherent sheaves on $X$.
    \item The category of $\mathbf{Z}$-filtered modules $F^\bullet M$ over the filtered ring $J^\bullet B$.
\end{enumerate}
\end{lemma}

It might also be helpful to note that the $t$-torsionfree condition in (1) corresponds to having an honest filtration in (2); we leave it to the reader to formulate a version of the above equivalence dropping these conditions.

\begin{proof}
This result is standard, so we simply explain the functors in either direction. Given a $t$-torsion-free quasi-coherent sheaf $F$ on $X$, we obtain a $\mathbf{Z}$-filtered $J^\bullet B$-module $F^\bullet M$ via $F^i M = \Gamma(X, F \otimes \mathcal{O}_X(-i))$, with the transition maps coming from multiplication by $t$. Conversely, given $F^\bullet M$ as in (2) above, consider its Rees module 
\[ \mathrm{Rees}(F^\bullet M) := \mathrm{Fil}^0(F^\bullet M \otimes_B t^\bullet B[t,t^{-1}]) = \oplus_i F^i M t^{-i}.\] 
This is naturally a graded $t$-torsion-free $\mathrm{Rees}(J^\bullet B)$-module, and hence defines an object of (1).
\end{proof}

\begin{remark}
\label{rmk:reesdictionary}
Under the equivalence in Lemma~\ref{lem:qcohX}, forgetting the filtration corresponds to restriction to $\mathrm{Spec}(B)= X_{t \neq 0}$, while passage to the associated graded corresponds to restriction to $X_{t =0} = \mathbf{A}^1_{B/J}\{1\}/\mathbf{G}_m$, where $\mathbf{A}^1_{B/J}\{1\} = \mathrm{Spec}(\mathrm{Sym}^*_{B/J}(J/J^2))$ is the normal bundle of $\mathrm{Spec}(B/J) \subset \mathrm{Spec}(B)$ and the $\mathbf{G}_m$-action has weight $-1$ (as the element $\tilde{d}$ locally has degree $-1$).
\end{remark}

\begin{remark}
Fix a $t$-torsion-free quasi-coherent sheaf $F$ on $X$ corresponding to a filtered $J^\bullet B$-module $F^\bullet M$ via Lemma~\ref{lem:qcohX}. One then obtains two $B$-modules via restriction to the open substacks $X_{t \neq 0}$ and $X_{\tilde{d} \neq 0}$ (see Lemma~\ref{lem:opensubstackdesc}); the former is simply the underlying unfiltered module
\[ M =\varinjlim_{i \to -\infty} F^i M\ \]
with the obvious transition maps (as mentioned in Remark~\ref{rmk:reesdictionary}), while the latter is given by
\[ N = \varinjlim_{i \to \infty} F^i M \otimes J^{-i},\]
with (increasing) transition maps coming from the natural maps $J \otimes F^i \to F^{i+1}$. When $J=(d)$ and $M$ is $d$-torsion-free, this is simply $\sum_{i\geq 0} F^iM d^{-i} \subset M[1/d]$.
\end{remark}

Finally, we have:

\begin{lemma}
\label{lem:qcohclosedpoint}
    The category of quasi-coherent sheaves on $X_{t=\tilde{d}=0}$ identifies with the category of $\Z$-graded $B/J$-modules.
\end{lemma}
\begin{proof}
This follows from the identification (3) in Lemma~\ref{lem:opensubstackdesc} and the standard identification of $\mathbf{G}_m$-actions with gradings.
\end{proof}

\subsection{Extending sheaves from \texorpdfstring{$X^\circ$}{Xo} to \texorpdfstring{$X$}{X}}

In our applications, we will $*$-extend vector bundles along $j:X^\circ \to X$. Let us describe the concrete meaning of this operation in terms of the equivalences above.

\begin{lemma}\label{lemma:cuhrg42u7j}
\label{lem:starextendrees}
Fix $(M,N,\psi) \in \mathrm{Isog}(B,J)$ corresponding to $V \in \mathrm{Vect}(X^\circ)$ under the equivalence in Lemma~\ref{lem:qcohxcirc}.
\begin{enumerate}
    \item The filtered $J^\bullet B$-module corresponding to $j_* V$ under Lemma~\ref{lem:qcohX} has underlying module $M$, with filtration given by $F^i M = M \cap J^i N$ (where the intersection takes place inside $M[1/J] \stackrel{\psi}{\simeq} N[1/J]$).

    \item The sections $(\tilde{d},t)$ form a Koszul-regular sequence of length $2$ on $j_* V$, so the restriction $(j_* V)|_{X_{t=\tilde{d}=0}}$ coincides with the derived pullback of $j_* V$ along $\mathrm{Spec}(B/J)\times B\mathbf{G}_m \to X$.

    \item The graded $B/J$-module corresponding to the restriction $(j_* V)|_{X_{t=\tilde{d}=0}}$ under Lemma~\ref{lem:qcohclosedpoint} is given, in degree $i$, by the $B/J$-module
\[ \gr^i_F(M)/\tilde{d}\gr^{i-1}_F(M) = F^iM/(F^{i+1}M + J F^{i-1}M).\]
Here $\tilde{d}$ denotes the coordinate on $X_{t=0} = \mathbf{A}^1_{B/J}\{1\}/\mathbf{G}_m$ with the $\mathbf{G}_{m}$-action having weight $-1$ (see Remark~\ref{rmk:reesdictionary}).

\end{enumerate}
\end{lemma}

In the sequel, we shall refer to $(j_* V)|_{X_{t=\tilde{d}=0}} \in \mathrm{QCoh}(\mathrm{Spec}(B/J) \times B\mathbf{G}_m)$ as the {\em residual sheaf} of the isogeny $(M,N,\psi)$.

\begin{proof}[Proof of Lemma~\ref{lemma:cuhrg42u7j}]
For (1), we need to show that $\Gamma(X,j_* V(-i)) = \Gamma(X^\circ, V(-i))$ equals $M \cap J^i N$ inside $M[1/J]$. Computing global sections via the cover $X^\circ = X_{t \neq 0} \cup X_{\tilde{d} \neq 0}$ shows that
\[ \Gamma(X^\circ, V(-i)) = \Gamma(X_{t \neq 0},V(-i)) \cap \Gamma(X_{\tilde{d} \neq 0}, V(-i))\]
inside $\Gamma(X_{t\neq 0} \cap X_{\tilde{d} \neq 0},V(-i))$. Now $\Gamma(X_{t \neq 0},V(-i)) = M$ via Remark~\ref{rmk:reesdictionary} (whence $\Gamma(X_{t\neq 0} \cap X_{\tilde{d} \neq 0},V(-i)) = M[1/J]$), while $\Gamma(X_{\tilde{d} \neq 0}, V(-i)) = J^i N$ via Remark~\ref{rmk:tonothercopy}; we leave the identification of the maps to the reader.

The transversality assertion in (2) translates under the Rees equivalence of Lemma~\ref{lem:qcohX} to the statement that $F^\bullet M$ is a genuine filtration and that the natural maps $J \otimes \gr^i_F M \to \gr^{i+1}_F M$ are injective; the first is clear, while the second follows from the definition of $F^\bullet M$ as a saturated filtration.

The claim in (3) is clear from the discussion in the proof of (2) above: $(j_*V)/t$ identifies with $\gr^\bullet_F M$, and further quotienting by $\tilde{d}$ amounts to passage to the cokernel of $J \otimes \gr^{\bullet-1}_F M \to \gr^\bullet_F M$.
\end{proof}

Under rather mild conditions, one can turn the above into an equivalence of categories:

\begin{prop}
\label{prop:CohSheafRees}
Assume $B$ is excellent. Then the following categories are equivalent:
\begin{enumerate}
    \item The category $\mathrm{Isog}(B,J)$ of $J$-isogenies.
    \item The category $\mathrm{Vect}(X^\circ)$ of vector bundles on $X^\circ$.
    \item The subcategory of $\mathrm{Coh}(X)$ spanned by coherent sheaves that are vector bundles over $X^\circ$ and are $(\tilde{d},t)$-regular.
    \item The category $\mathrm{Vect}(X)$ provided $B/J$ is a field.
    \item The category $\mathrm{Coh}^{\refl}(X)$ of reflexive coherent sheaves on $X$ provided $B/J$ is a DVR.
\end{enumerate}
The functors relating (2) and (3) are $*$-extension and restriction, while those relating (3), (4), and (5) are the identity functor.
\end{prop}
\begin{proof}
The equivalence of (1) and (2) was already shown in Lemma~\ref{lem:qcohxcirc}. The equivalence of (2) and (3) follows from Lemma~\ref{lem:starextendrees} since the $*$-extension procedure produces coherent sheaves by \cite[\href{https://stacks.math.columbia.edu/tag/0AWA}{Tag 0AWA}]{stacks-project}. For the equivalence of (3) and (4), assume $B/J$ is a field, so $J$ is a maximal ideal. We must show that every coherent sheaf on $X$ satisfying the assumptions in (3) is in fact a vector bundle on $X$. The statement is local on $\Spec(B)$ and clear after inverting $J$, so we can assume $B$ is local with maximal ideal $J$, and thus a DVR. Then the stack $X$ admits a faithfully flat cover by a $2$-dimensional noetherian scheme $\Spec(\mathrm{Rees}(J^\bullet B))$ which is regular and graded-local 
with $\tilde{d}$ and $t$ as regular parameters; as vector bundles coincide with depth $2$ coherent sheaves on such schemes, the equivalence of (3) and (4) follows. A similar argument also gives the equivalence of (3) and (5).
\end{proof}

\subsection{A \texorpdfstring{$p$}{p}-twisted variant}
\label{ss:twistedrees}

In our applications, we shall also need a twisted version of the above discussion for the ``$J$ in degree $p$ filtration''. Thus, consider the graded $B[t]$-algebra
\[ \mathrm{Rees}(J^{\lceil \bullet/p \rceil}) = \bigoplus_{i \in \mathbf{Z}} J^{\max(\lceil i/p \rceil,0)} t^{-i} \subset B[1/J,t,t^{-1}] \]
where $t$ has degree $1$. Write $Y=\mathrm{Spec}(\mathrm{Rees}(J^{\lceil \bullet/p \rceil}))/\mathbf{G}_m$; it comes equipped with a projection map to $\Spec(B)$ and a Rees map to $\mathbf{A}^1/\mathbf{G}_m=\Spec(B[t])/\Gm$.

\begin{example}
    If $J=(d)$ is principal, then we can write
\[ \mathrm{Rees}(J^{\lceil \bullet/p \rceil}) = B[y,t]/(yt^p-d),\]
where $y = dt^{-p} \in Jt^{-p} \subset \mathrm{Rees}(J^{\lceil \bullet/p \rceil})$ and $\deg(y)=-p$.
\end{example}

Note that $Y$ can also be described as the fibre product of
\[ X \xrightarrow{\text{Rees}} \mathbf{A}^1/\mathbf{G}_m \xleftarrow{ (-)^p} \mathbf{A}^1/\mathbf{G}_m,\]
with the Rees map for $Y$ corresponding to the map from the fibre product to $\mathbf{A}^1/\mathbf{G}_m$ appearing on the right. Consequently, for a test ring $R$, the groupoid $Y(R)$ is the groupoid of generalized Cartier divisors $t':\mathcal{O}(-1) \to R$ together with a factorization $J \otimes_B R \xrightarrow{y'} \mathcal{O}(-p) \xrightarrow{(t')^p} R$ of the canonical map. In particular, there is a universal factorization
\[ J \otimes_B \mathcal{O}_Y \xrightarrow{y} \mathcal{O}(-p) \xrightarrow{t^p} \mathcal{O}_Y\]
over $Y$, where $y$ and $t$ are sections of suitable line bundles over $Y$.
To describe the geometry of $Y$, recall the notion of a $p$-th root stack: given any scheme (or stack) $W$ with a Cartier divisor $D \subset W$, the $p$-th root stack $W[\sqrt[p]{D}] \to W$ is the pullback of the $p$-power endomorphism $\mathbf{A}^1/\mathbf{G}_m \xrightarrow{(-)^p} \mathbf{A}^1/\mathbf{G}_m$ along the classifying map $W \to \mathbf{A}^1/\mathbf{G}_m$ of the divisor $D$. Explicitly, $W[\sqrt[p]{D}] \to W$ is the space of $p$-th roots of the effective Cartier divisor $\mathcal{O}(-D) \subset \mathcal{O}_W$. In particular, the structure map  $W[\sqrt[p]{D}] \to W$ is an isomorphism if $D$ is empty.

Using the above notion, we have:

\begin{lemma}
\label{lem:LociTwistedRees}\leavevmode
    \begin{enumerate}
        \item The open substack $Y_{t \neq 0}$ maps isomorphically to $\mathrm{Spec}(B)$ via the projection.

        \item The open substack $Y_{y \neq 0}$ identifies with the $p$-th root stack $\mathrm{Spec}(B)[\sqrt[p]{V(J)}]$ of the divisor $V(J) \subset \mathrm{Spec}(B)$ via the projection.

        \item The intersection $Y_{t \neq 0} \cap Y_{y \neq 0}$ identifies with the open substack $\mathrm{Spec}(B[1/J]) \subset Y_{t \neq 0}$ and with the open substack
$\mathrm{Spec}(B[1/J]) = \mathrm{Spec}(B[1/J])[\sqrt[p]{V(J)}] \subset Y_{y \neq 0}$.
\end{enumerate}
\end{lemma}
\begin{proof}
(1) is straightforward, while (2) and (3) follow from the functor of points description of $Y$ given above.
\end{proof}

Using the above, we can describe sheaves on $Y$ and its open substacks concretely, as in Proposition~\ref{prop:CohSheafRees}:

\begin{prop}
\label{prop:CohSheafptwistedRees}
Consider the open substack $Y^\circ = Y_{t \neq 0} \cup  Y_{y \neq 0} \subset Y$. Then the following categories are equivalent:
\begin{enumerate}
    \item The category $\mathrm{Vect}(Y^\circ)$ of vector bundles on $Y^\circ$.
    \item The category of triples $(M,N,\psi)$, where $M$ is a vector bundle on $\Spec(B)$, $N$ is a vector bundle on the root stack $\mathrm{Spec}(B)[\sqrt[p]{V(J)}]$, and
    \[ \psi:M|_{\mathrm{Spec}(B[1/J])} \simeq N|_{\mathrm{Spec}(B[1/J])[\sqrt[p]{V(J)}]} \]
    is an isomorphism over the natural identification
    \[ \mathrm{Spec}(B[1/J]) = \mathrm{Spec}(B[1/J])[\sqrt[p]{V(J)}]. \]
    \item The category $\mathrm{Vect}(Y)$ provided $B$ is excellent and $B/J$ is a field.
    \item The category $\mathrm{Coh}^{\refl}(Y)$ of reflexive coherent sheaves on $Y$ provided $B$ is excellent and $B/J$ is a DVR.
\end{enumerate}
\end{prop}
\begin{proof}
As in the untwisted case, the identification between (1) and (2) arises by describing sheaves on $Y^\circ$ in terms of the open cover $Y_{t \neq 0} \cup  Y_{y \neq 0}$ using Lemma~\ref{lem:LociTwistedRees}. The functor relating the category in (1) with either (3) or (4) is given by $*$-extension, with inverse given by restriction. To prove these are equivalences, as in the untwisted case, we may assume $B$ is local with $J$ inside the maximal ideal. Then the graded coordinate ring $\mathrm{Rees}(J^{\lceil \bullet/p \rceil})$ of $Y$ is regular and graded-local of dimension $2$ 
in case (3) (whence $\mathrm{Vect}(Y) \simeq \mathrm{Vect}(Y^\circ)$) and of dimension $3$ in case (4) (whence $\mathrm{Coh}^{\refl}(Y) \simeq \mathrm{Vect}(Y^\circ)$ via restriction and $*$-extension).
\end{proof}

\begin{remark}[Quasi-coherent sheaves on root stacks for principal divisors]
\label{rmk:qcohroot}
Assume $J=(d)$ is principal. Then the root stack $\mathrm{Spec}(B)[\sqrt[p]{V(J)}]$ admits a concrete quotient description as $\Spec(B[\sqrt[p]{d}])/\mu_p$, where $B[\sqrt[p]{d}] = B[t]/(t^p-d)$ and the $\mu_p$-action is the standard one, giving $t$ weight $1$. In particular, quasi-coherent sheaves on $\mathrm{Spec}(B)[\sqrt[p]{V(J)}]$ can be identified as $\mu_p$-equivariant $B[\sqrt[p]{d}]$-modules or equivalently $\mathbf{Z}/p$-graded modules $M = \oplus_{i \in \mathbf{Z}/p} M_i$ over the $\mathbf{Z}/p$-graded $B$-algebra $B[t]/(t^p-d)$. Note that when $d$ is invertible, then so is $t$, in which case $M$ is determined uniquely by the $B$-module $M_0$; this is the concrete manifestation, at the level of quasi-coherent sheaves, of the isomorphism $\mathrm{Spec}(B)[\sqrt[p]{V(J)}] \simeq \Spec(B)$ in the case of the empty divisor.
\end{remark}

Let us specialize the above discussion to a case of relevance in the sequel:

\begin{example}
\label{ex:VectReesKu}
Fix a ground ring $K$, and let $B=K\llbracket u \rrbracket$ with $J=(u)$. Then there is a $K$-algebra endomorphism $\varphi_{/K}:K\llbracket u \rrbracket \to K \llbracket u \rrbracket$ given by $u \mapsto u^p$, which realizes the target as the algebra of $p$-th roots of $u \in B$ on the source. In particular, there is a $\mu_p$-action on the target $K\llbracket u \rrbracket$ giving the variable $u$ degree $1$. Passing to schemes and quotients, we obtain a factorization
\[ \Spec(K \llbracket u \rrbracket) \xrightarrow{\can} \Spec(K \llbracket u \rrbracket)/\mu_p \xrightarrow{\pi} \Spec(K\llbracket u \rrbracket) \]
of $\varphi_{/K}$. By Remark~\ref{rmk:qcohroot}, the map $\pi$ is naturally identified with the root stack map $\mathrm{Spec}(K\llbracket u \rrbracket)[\sqrt[p]{V((u))}] \to \Spec(K\llbracket u \rrbracket)$. Consequently, by the description of the quasi-coherent sheaves on the root stack in Remark~\ref{rmk:qcohroot}, vector bundles on $Y^\circ$ are identified with triples $(M,N,\psi)$, where $M$ is a vector bundle on $\Spec(K\llbracket u \rrbracket)$, $N$ is a $\mu_p$-equivariant vector bundle on $\Spec(K \llbracket u \rrbracket)$, and $\psi$ is a $\mu_p$-equivariant isomorphism $\varphi_{/K}^*M[1/u] \simeq N[1/u]$ on $\Spec(K((u)))$ (or equivalently an isomorphism $\pi^*M[1/u] \simeq N[1/u]$ on $\Spec(K((u)))/\mu_p \stackrel{\pi}{\simeq} \Spec(K((u)))$).
\end{example}

\begin{remark}
\label{rmk:ptwistedmupweights}
Assume $B$ is excellent and $B/J$ is a field. Fix a triple $(M,N,\psi)$ as in Proposition~\ref{prop:CohSheafptwistedRees} (2). This corresponds to a vector bundle $\mathcal{E} \in \mathrm{Vect}(Y)$ by Proposition~\ref{prop:CohSheafptwistedRees}; one recovers $M$ and $N$ by restricting $\mathcal{E}$ to $Y_{t \neq 0}$ and $Y_{y \neq 0}$ respectively. In this situation, let us explain how to read off certain numerical invariants for $N$ in terms of $\mathcal{E}$. 

First, from $N$, we obtain a $\mu_p$-representation, or equivalently a $\mathbf{Z}/p$-graded vector bundle, over $\Spec(B/J)$ via pullback along the natural map 
\[ \Spec(B/J) \times B\mu_p = \left(\Spec(B)[\sqrt[p]{V(J)}])\right)_{J=0,\text{red}} \hookrightarrow \Spec(B)[\sqrt[p]{V(J)}].\]
Tracking ranks of the graded pieces determines a multisubset $S_1$ of $\mathbf{Z}/p$. On the other hand, from $\mathcal{E}$, one obtains a $\mathbf{Z}$-graded vector bundle over $B/J$ corresponding to $\mathcal{E}|_{Y_{t=0,y=0}} \in \mathrm{Vect}(\Spec(B/J)/\mathbf{G}_m)$; as before, determines a multisubset $S_2$ of $\mathbf{Z}$. These two multisets are related in the natural way: the projection of $S_2$ to a multisubset $\overline{S_2}$ of $\mathbf{Z}/p$ agrees with $S_1$. To prove this, observe that $Y_{t=0} \simeq \Spec(B/J[y])/\mathbf{G}_m$, where $y$ has weight $p$. Pulling back $\mathcal{E}$ along the composition 
\[ \Spec(B/J[y]) \times B\mu_p \to \Spec(B/J[y])/\mathbf{G}_m = Y_{t=0} \to Y\]
and evaluating at $y=1$ and $y=0$ gives the multisets $S_1$ and $\overline{S_2}$ respectively. But then the coincidence is clear: taking the fibral rank of a vector bundle is a locally constant function on $\Spec(B/J[y])$.
\end{remark}

\subsection{Numerical invariants}
\label{sss:numrees}
In this subsection, we extract numerical invariants from isogenies. It will be convenient to make the following definition:

\begin{defn}
An isogeny $(M,N,\psi) \in \mathrm{Isog}(B,J)$ corresponding to a vector bundle $V$ on $X^\circ$ is called {\em good} if the quasi-coherent sheaf $j_*V$ on the Rees stack $X$ is a perfect complex. Write $\mathrm{Isog}^{\text{good}}(B,J)$ for the full subcategory of good isogenies.
\end{defn}

While good isogenies need not be stable under arbitrary pullback, they are stable under flat pullback as the formation of $j_*$ commutes with flat pullback. This observation, coupled with the following example, provides an abundant supply of good isogenies (and, in particular, covers all cases we need).

\begin{example}
\label{ex:manygoodisog}
If $B$ is excellent and $B/J$ is regular, any isogeny is good. Indeed, by \cite[\href{https://stacks.math.columbia.edu/tag/0AWA}{Tag 0AWA}]{stacks-project}, the sheaf $j_* V$ is coherent. To check perfectness of $j_* V$, it suffices to do so after derived pullback to $X^\circ$ and $X_{t=\tilde{d}=0}$ separately; the former is clear (as $j_*V|_{X^\circ}$ is a vector bundle), while the latter follows from regularity of $B/J$ and Lemma~\ref{lem:starextendrees} (2).
\end{example}

A good isogeny has a naturally attached numerical invariant:

\begin{defn}
\label{def:relpos}
Fix a good isogeny $(M,N,\psi) \in \mathrm{Isog}(B,J)$ corresponding to a vector bundle $V$, so the residual sheaf $(j_* V)|_{X_{t=\tilde{d}=0}}$ is a perfect complex on $\mathrm{Spec}(B/J) \times B\mathbf{G}_m$. Let $f_\psi$ be the function $X^*(\mathbf{G}_m) = \mathbf{Z} \to \mathbf{Z}^{\pi_0(\mathrm{Spec}(B/J))}$ recording the fibrewise Euler characteristic of the graded pieces; thanks to the formula in Lemma~\ref{lem:starextendrees} (3), this is concretely given by
\[ f_{\psi}(i) = \chi_{\fib} \left( F^iM/(F^{i+1}M + J F^{i-1}M)\right).\]
We call $f_\psi$ the {\em relative position} of the good isogeny $\psi$.
\end{defn}

As $\chi_{\fib}(-)$ can be computed by taking the rank of the argument at generic points, the values of $f_\psi$ are always non-negative. Moreover, as the formation of $j_*$ commutes with flat pullback, the assignment
\begin{equation}
\label{eq:relpos}
(M,N,\psi) \in \mathrm{Isog}^{\text{good}}(B,J) \to f_\psi \in \mathrm{Map}(\mathbf{Z},\mathbf{Z}_{\geq 0}^{\pi_0(\mathrm{Spec}(B/J))})
\end{equation} 
is compatible with flat pullback.  In particular, if $B \to C$ is a flat map such that $\Spec(C/JC) \to \Spec(B/JB)$ is a bijection on connected components, then $f_\psi$ agrees with $f_\psi \otimes_B C$.

\begin{example}[Relative position of lattices]
\label{ex:RelPosDef}
Assume $B=V$ is an excellent discrete valuation ring and $J=\mathfrak{m}$ is the maximal ideal. Given two finite projective $V$-modules $M,N$ and an isomorphism $\psi:M[1/\mathfrak{m}] \simeq N[1/\mathfrak{m}]$, by \eqref{eq:relpos}, we obtain a function $f_\psi:\mathbf{Z} \to \mathbf{Z}$. This function is called {\em the relative position} of the modification $\psi$. For example, if $M=V$ and $N=J^a$ for some integer $a$ (and $\psi$ is the obvious identification), then one computes that $f_\psi$ has value $1$ at $-a$ and is $0$ otherwise.
\end{example}

For a modification of lattices over $V=k\llbracket t \rrbracket$, we now have two potential definitions of relative position: one via the affine Grassmannian, and the stack-theoretic one in Definition~\ref{def:goodisog}. Unsurprisingly, the two coincide:

\begin{remark}
\label{RelPosGrAgree}
Fix a field $k$. Consider the special case of Example~\ref{ex:RelPosDef} where $V=k\llbracket t \rrbracket$ and we have fixed a framing $M=V^{\oplus n}$ for $M$. Then an isomorphism $\psi:M[1/t] \simeq N[1/t]$ can be viewed as a $k$-point of the affine Grassmannian $\Gr_{\GL_n} = L\GL_n/L^+\GL_n$; the $L^+\GL_n$-orbit of this point therefore has parameter given by a dominant cocharacter $\lambda(\psi) \in X_*(T)^+ = (\mathbf{Z}^n)^+$ of $\GL_n$ (which is often called the relative position of the lattices $M$ and $N$ with respect to $\psi$ classically). Recall that there is a standard embedding $\tau:X_*(T)^+ \subset \mathrm{Map}(\mathbf{Z},\mathbf{Z})$ carrying an $n$-tuple $\lambda \coloneq (\lambda_1,\dots,\lambda_n) \in (\mathbf{Z}^n)^+$ to the function $\tau(\lambda)$ where $\tau(\lambda)(a)$ is the multiplicity of $a$ in the multiset $\lambda$. Using the calculation in Example~\ref{ex:RelPosDef}, one verifies that $\tau(\lambda(\psi)) = f_\psi$, so the notion of relative position in Definition~\ref{def:relpos} is indeed a generalization of the classical notion of relative position for lattices.
\end{remark}

\subsection{Generalization to stacks}
\label{sss:numstack}
In this subsection, we observe that the assignment \eqref{eq:relpos} generalizes to certain formal stacks; we have written what we need, and have not explored the natural generality of the constructions.

Fix a closed immersion $D \subset Y$ of fpqc stacks on $p$-nilpotent rings satisfying the following condition:

\begin{itemize}
    \item[$(\ast)$] There exists a representable affine flat cover $\pi:\mathrm{Spf}(B) \to Y$, where $B$ is a Noetherian ring with $I_B \coloneq  \pi^* I_D \subset B$ an invertible ideal, and such that $(p,I_B)$ is an ideal of definition for $B$.
\end{itemize}

The following observation will be used repeatedly to ensure that the completions coming from formal algebraic geometry do not cause significant problems.

\begin{lemma}
\label{lem:flatcov}
For a flat cover $\mathrm{Spf}(B) \to Y$ as in $(\ast)$, write $\mathrm{Spf}(B^\bullet)$ for its \v{C}ech nerve. Then each face map $B=B^0 \to B^i$ is faithfully flat.
\end{lemma}
\begin{proof}
By assumption, $B$ is $(p,I_B)$-adically complete and the face maps $B \to B^i$ are $(p,I_B)$-completely flat. It is then enough to note that over a Noetherian complete ring, complete flatness implies genuine flatness; see \cite[page 116, footnote 70]{BBFgaugenotes} for an elaboration of this argument that also explains that $Y$ admits a good notion of coherent sheaves compatible with pullback to any such $\mathrm{Spf}(B)$.
\end{proof}

Using Lemma~\ref{lem:flatcov}, one checks that $Y$ is complete along $D$ with $I_D \subset \mathcal{O}_Y$ being an invertible ideal sheaf.

\begin{defn}
The category $\mathrm{Isog}(Y,D)$ consists of triples $(M,N,\psi)$, where $M,N \in \mathrm{Vect}(Y)$ and $\psi$ is an $I_D$-isogeny\footnote{As $Y$ is $I_D$-formal, the locus $Y-D$ is empty, so we cannot simply define isogenies of vector bundles as isomorphisms of their restrictions to $Y-D$.}, i.e., an isomorphism $\{I_D^{\otimes -n} \otimes M\}_{n \geq 0} \simeq \{I_D^{\otimes -n} \otimes N\}_{n \geq 0}$ of ind-objects in $\mathrm{Vect}(Y)$.
\end{defn}

It is straightforward to check that when $Y=\mathrm{Spf}(A)$ is affine, the category $\mathrm{Isog}(Y,D)$ defined above agrees with $\mathrm{Isog}(A,I_D)$.

\begin{defn}
\label{def:goodisog}
The full subcategory $\mathrm{Isog}^{\text{good}}(Y,D) \subset \mathrm{Isog}(Y,D)$ of good isogenies consists of those isogenies which become good after pullback to some (or, equivalently, any\footnote{\label{foot:comparecech}As $*$-extension from a quasi-compact open commutes with flat pullback, and because perfectness can be detected after faithfully flat pullback, this equivalence follows from the following observation: if $\mathrm{Spf}(B) \to Y$ and $\mathrm{Spf}(B') \to Y$ are two flat covers as in $(\ast)$, then their fibre product is $\mathrm{Spf}(C) \to Y$ for a (possibly non-Noetherian) ring $C$ that is faithfully flat over both $B$ and $B'$.}) flat cover $\mathrm{Spf}(B) \to Y$ as in $(\ast)$.
\end{defn}

If there exists a flat cover $\mathrm{Spf}(B) \to Y$ as in $(\ast)$ with $B$ excellent and $B/I_B$ regular, then every isogeny is good by Example~\ref{ex:manygoodisog}; this will be the case in our main application. The main result in this subsection is an extension of the construction \eqref{eq:relpos} to $(Y,D)$:

\begin{lemma}
    There is a unique assignment
    \[ \mathrm{Isog}^{\mathrm{good}}(Y,D) \to \mathrm{Map}(\mathbf{Z}, \mathbf{Z}_{\geq 0}^{\pi_0(D)})\]
    compatible with pullback to any $\mathrm{Spf}(B) \to Y$ as in $(\ast)$.
\end{lemma}
\begin{proof}
Granting existence, the uniqueness is clear as $\pi_0 (\mathrm{Spec}(B/I_B)) \to \pi_0(D)$ is surjective for any $\mathrm{Spf}(B) \to Y$ as in $(\ast)$.

For existence, fix a good isogeny $(M,N,\psi)$ over $Y$. We will construct a perfect complex $K$ over $D \times B\mathbf{G}_m$ whose pullback to any $\mathrm{Spf}(B) \to Y$ as in $(\ast)$ yields the residual coherent sheaf (as in Lemma~\ref{lem:starextendrees}) for the pullback of the good isogeny $(M,N,\psi)$ to $B$; once such a $K$ is constructed, taking fibrewise Euler characteristic yields the desired assignment.

To construct $K$, fix $\mathrm{Spf}(B) \to Y$ as in $(\ast)$ with \v{C}ech nerve $\mathrm{Spf}(B^\bullet)$. Then pullback of $(M,N,\psi)$ yields a compatible system of good isogenies on $(B^\bullet,I_{B^\bullet})$. By Lemma~\ref{lem:flatcov}, forming the residual sheaf termwise then yields an object $K$ of $\lim_\Delta \mathrm{Perf}(\mathrm{Spec}(B^\bullet/I_{B^\bullet}) \times B\mathbf{G}_m)$ thanks to the flat pullback compatibility of the formation of the residual sheaf of a good isogeny in the case of rings. But we have
\[
\begin{aligned}
\lim_\Delta \mathrm{Perf}(\mathrm{Spec}(B^\bullet/I_{B^\bullet}) \times B\mathbf{G}_m)
&\simeq \lim_\Delta \mathrm{Perf}(\mathrm{Spf}(B^\bullet/I_{B^\bullet}) \times B\mathbf{G}_m) \\
&\simeq \mathrm{Perf}(D \times B\mathbf{G}_m),
\end{aligned}
\]
so $K$ can be viewed as a perfect complex on $D \times B\mathbf{G}_m$. The function associated to $K$ clearly descends the function associated to the pullback of $(M,N,\psi)$ to $\mathrm{Spf}(B) \to Y$. To finish, it is thus enough to check that $K$ is independent of the cover $\mathrm{Spf}(B) \to Y$ used to construct it; this follows by comparing the \v{C}ech nerves associated to two different covers as in Footnote~\ref{foot:comparecech}.
\end{proof}

\section{Structures on crystalline Breuil--Kisin modules}\label{sec:Frob-descent-F-gauge}
Fix a perfect field $k$ of characteristic~$p$. In this section, our main goal is to record why the Breuil--Kisin modules attached to prismatic $F$-crystals over $W(k)$ enjoy some additional structures related to the natural $\mu_p$-action on a Breuil--Kisin prism, see  Theorem~\ref{thm:mainNthm}, Proposition~\ref{cor:FrobdescentfiltBKmod}, Theorem~\ref{thm:crysBKspecial}, and Corollary~\ref{cor:crysGdesc}. 

\subsection{Notation: the Breuil--Kisin prism and its \texorpdfstring{$\mu_p$}{mu p}-action}
\label{not:mupaction}

We will write $K = W(k)[1/p];$ note that this is a change of notation from Section~\ref{sec:affine-Grassmannians}, where $K$ denoted an algebraically closed field. Fix a uniformizer $\pi \in K$, so $\pi = pz$ for some unit $z \in W(k)^*$. This choice determines the {\em Breuil--Kisin} prism $(\KisinS,(E))$, where $\KisinS=W(k)\llbracket u \rrbracket$, given the unique $\delta$-structure where $\delta(u)=0$ (or equivalently $\varphi(u)=u^p$), and $E=E(u)=u-\pi = u-pz$ is the minimal polynomial of $\pi$. Unless otherwise specified, $\KisinS$ is endowed with the $(p,u)$-adic topology for which it is complete. Write $w:\KisinS \to W(\KisinS)$ for the map associated to the $\delta$-structure, so $w(u)=[u]$ while $w|_{W(k)}$ endows $W(\KisinS)$ with a $W(k)$-algebra structure.  For any $\delta$-ring/space $X$, we shall write $\varphi_X$ for the associated Frobenius lift (or simply $\varphi$ if there is no risk of confusion).

There is a natural $W(k)$-linear $\mu_p$-action on $\mathrm{Spec}(\KisinS)$ giving $u$ weight $1$. Equipping $\mu_p$ with its usual $\delta$-structure (where the Frobenius lift kills the group scheme), this action is compatible with the $\delta$-structure: indeed, this $\mu_p$-action on $\mathrm{Spec}(\KisinS)$ is simply
induced from the scaling $\mathbf{G}_m$-action on $\mathbf{A}^1 = \mathrm{Spec}(W[u])$ restricting the group and by completing the ring. In particular,  there is an induced $\delta$-structure on $\Spec \gS/\mu_p$ such that the canonical map $\canonicalmapnotation:\Spec \gS \to \Spec \gS/\mu_p$ is a $\delta$-map. Moreover, as $\varphi:\KisinS \to \KisinS$ has image inside the invariant subring $\KisinS^{\mu_p} \subset \KisinS$, the map $\varphi:\Spec \KisinS \to \Spec \KisinS$ factors as 
\begin{equation} 
\label{eq:Kisin-Phi-factorize}\Spec \KisinS \xrightarrow{\canonicalmapnotation} \Spec \KisinS/\mu_p \xrightarrow{q} \Spec \KisinS
\end{equation}
where $q$ is a $\delta$-map that is linear over $\varphi$ on $W(k)$. Completing all structures for the $(p,u)$-adic topology on $\gS$ yields similar statements with $\Spec(-)$ replaced by $\mathrm{Spf}(-)$. The following simple consequence of this discussion will be used repeatedly:

\begin{lemma}
\label{lem:mupactionKisinSmod}
For any $\gM \in \mathrm{Coh}(\KisinS)$, the pullback $\varphi^* \gM$ is naturally $\mu_p$-equivariant. 
\end{lemma}
\begin{proof}
This follows immediately from the factorization in \eqref{eq:Kisin-Phi-factorize}.
\end{proof}

In the sequel, we shall implicitly use Lemma~\ref{lem:mupactionKisinSmod} to $\varphi$-pullbacks of objects on $\Spec \KisinS$ with a $\mu_p$-equivariant structure.

\subsection{Canonical modifications of crystalline Breuil--Kisin modules}
\label{ss:canmodthms}

Recall that a Breuil--Kisin module is a pair $(\gM,\varphi_{\gM})$ where $\gM$ is a  finite projective $\KisinS$-module and $\varphi_\gM:\varphi^* \gM[1/E] \simeq \gM[1/E]$ is an isomorphism of $\KisinS[1/E]$-modules. We write $\mathrm{Vect}^{\varphi}(\KisinS)$ for the category of Breuil--Kisin modules. This category receives a natural fully faithful functor from the category of $\Gal_K$-stable $\mathbf{Z}_p$-lattices in crystalline $\Gal_K$-representations by the main result of \cite{KisinCrys}. We denote the essential image of this functor by $\Vect^{\varphi}(\gS)_{\cris}$; its
objects are dubbed {\em  crystalline} Breuil--Kisin modules.

The goal of this subsection is to formulate  Theorem~\ref{thm:mainNthm}, which gives a new structure enjoyed by crystalline
Breuil--Kisin modules.

\begin{para} Denote by $K\llbracket E \rrbracket$  the completion of $\gS[1/p]$ at the ideal $E(u).$
Recall that the double cosets $\GL_n(K\llbracket E \rrbracket)\backslash \GL_n(K\llbracket E \rrbracket[1/E])/ \GL_n(K\llbracket E \rrbracket)$ are indexed by unordered $n$-tuples of integers, or equivalently
by dominant cocharacters of $\GL_n.$

Let $\gM$ in  $\Vect^{\varphi}(\gS)_{\cris}$ be of rank $n.$
 As $\varphi^*(\gM)[1/E] \simeq \gM[1/E],$ the relative position of
$\varphi^*(\gM)\otimes_{\gS} K\llbracket E \rrbracket$ and $\gM \otimes_{\gS} K\llbracket E \rrbracket$
in $\varphi^*\gM \otimes_{\KisinS} K((E)) \simeq \gM \otimes_{\gS} K((E))$ corresponds to a dominant cocharacter
which we denote by $\mu_{\gM}.$

If $\gM$ arises from a lattice $L$ in a crystalline representation $V,$ then $\mu_{\gM}$ agrees with the Hodge--Tate
cocharacter of $V.$
\end{para}

\begin{para}\label{para:category-D} Let $\cD$ denote the category of tuples $\underline{\gN} = (\gN_1, \gN_2,\psi),$ where $\gN_1$ is a finite free $\gS$-module, $\gN_2$ is a finite free $\mu_p$-equivariant $\gS$-module, and $\psi: \varphi^* \gN_1[1/u] \simeq \gN_2[1/u]$ is a $\mu_p$-equivariant isomorphism (see Lemma~\ref{lem:mupactionKisinSmod} for the $\mu_p$-action on the left hand side). In the notation of \S \ref{ss:fungrass} (or the forthcoming \S \ref{sss:xprimezprimeBK}), the underlying groupoid of $\mathcal{D}$ coincides with $\pHk(W(k)) \coloneq  \bigsqcup_n \pHk_{\mathrm{GL}_n}(W(k))$. For any $\underline{\gN} = (\gN_1, \gN_2,\psi) \in \mathcal{D}$, we denote by $\mu_{\underline{\gN}}$ the dominant cocharacter of $\GL_n$ determined by the attached point of $\pHk(K)$, i.e., by the relative position of the $K\llbracket u \rrbracket$-lattices $\varphi^*\gN_1 \otimes_{\KisinS} K\llbracket u \rrbracket$ and $\gN_2 \otimes_{\KisinS} K\llbracket u \rrbracket$ inside $\varphi^* \gN_1 \otimes_{\KisinS} K((u)) \simeq \gN_2 \otimes_{\KisinS} K((u))$, where $n$ is the rank of $\gN_1.$

Our main result is that any crystalline Breuil--Kisin module has an attached object in $\mathcal{D}$ whose generic and special fibres have predictable relative positions:

\end{para}

\begin{thm}\label{thm:mainNthm} There is an exact tensor functor
$$\Vect^{\varphi}(\gS)_{\cris} \rightarrow \cD: \quad \gM \mapsto \underline{\gN}(\gM) = (\gM, \gN, \psi) $$
such that
\begin{enumerate}
    \item\label{item:specializable} {\em Specializability:} There is a  $\KisinS/p$-linear identification $\tau:\gN/p\gN \simeq \gM/p\gM$ fitting into a commutative diagram
\[\xymatrix{ \bigl(\varphi^*(\gM)/p\varphi^*(\gM)\bigr)[1/u] \ar[r]^-{\psi}\ar@{=}[d] & \bigl(\gN/p\gN\bigr)[1/u] \ar[d]^{\tau,\sim} \\
\bigl(\varphi^*(\gM)/p\varphi^*(\gM)\bigr)[1/u] \ar[r]^-{\varphi_\gM} & \bigl(\gM/p\gM\bigr)[1/u]
}\]
of $\KisinS/p[1/u]$-modules. (Note that such a $\tau$ is unique if it exists.)

    \item\label{item:generic-position} {\em Generic relative position:} We have $\mu_{\gM} = \mu_{\underline{\gN}(\gM)}.$
\end{enumerate}
\end{thm}

We shall give two proofs of Theorem~\ref{thm:mainNthm}: a shorter one via the prismatization in \S \ref{sss:canmodprismproof}, and a slightly more conceptual longer one via the syntomification in \S \ref{ss:PfCanModThm}.

\begin{remark}
There is an asymmetry in parts~\eqref{item:specializable} and~\eqref{item:generic-position} of Theorem~\ref{thm:mainNthm}: the former identifies the isogeny $\psi \otimes_{\KisinS} k\llbracket u \rrbracket$ uniquely, while the latter only pins down the relative position of $\psi \otimes_{\KisinS} K\llbracket u \rrbracket$. The proof via the syntomification mentioned above upgrades~\eqref{item:generic-position} to a more precise assertion: in fact, the isogeny $\psi$ is itself intrinsically described in terms of $(\gM,\varphi_\gM)$ (see Proposition~\ref{prop:canmodBK} and Theorem~\ref{thm:crysBKspecial}).
\end{remark}

In the rest of this subsection, we record some consequences. First, one has the following $\mu_p$-equivariance of mod $p$ reductions of crystalline Breuil--Kisin modules:

\begin{cor}\label{cor:mainNcor}
Let $(\gM,\varphi_\gM) \in \Vect^{\varphi}(\gS)_{\cris}.$ Then $\gM/p\gM$ has a unique $\mu_p$-equivariant structure compatible with the map $\varphi^*(\gM/p\gM)[1/u] \simeq \gM/p\gM[1/u]$ induced by $\varphi_\gM$.
\end{cor}
\begin{proof}
The existence follows from Theorem \ref{thm:mainNthm}~\eqref{item:specializable}. For uniqueness: a $\mu_p$-equivariant structure on $\gM/p\gM$ corresponds to a $\mathbb{Z}/p$-grading (with some properties), so it is determined by the induced structure on $\gM/p\gM[1/u]$, and the latter is unique (if it exists) thanks to its compatibility with $\varphi_{\gM}$.
\end{proof}

Secondly, one has the following numerical consequence:

\begin{cor}
\label{cor:specializableuparrow}
Let $(\gM,\varphi_\gM) \in \Vect^{\varphi}(\gS)_{\cris}.$  Let $\mu_{\gM}$ and $\mu_{\overline{\gM}}$ be the relative positions of the isogenies over $K\llbracket E \rrbracket$ and $k\llbracket E \rrbracket$ defined by $\varphi_\gM$. Then
\[ \mu_{\overline{\gM}} \uparrow \mu_\gM \]
in $X_*(T)^+$, where $n=\rank(\gM)$ and $T\subseteq \mathrm{GL}_n$ is the maximal torus.
\end{cor}

\begin{proof}
Consider the $u$-isogeny $\psi:\varphi^*\gM[1/u] \simeq \gN^{\can}[1/u]$ provided by Theorem~\ref{thm:mainNthm}. As this isogeny is $\mu_p$-equivariant, via Proposition~\ref{prop:funcmupfixed}, it defines a point $\beta_{W(k)} \in \pHk_{\mathrm{GL}_n}(W(k))$. The induced points $\beta_{\overline{K}} $ and $\beta_{\overline{k}}$ over the geometric generic and special points of $\mathrm{Spec}(W(k))$ can be viewed as elements of $X_*(T)^+$ by Lemma~\ref{lem:pointquotient}. These elements satisfy
\begin{equation*}\label{eq:uparrowbetaGL} \beta_{\overline{k}} \uparrow \beta_{\overline{K}} \end{equation*}
by Corollary~\ref{cor:closure-from-zpbar-point}. By Theorem~\ref{thm:mainNthm} (1) and (2) respectively, we also know that $\beta_{\overline{k}}=\mu_{\overline{\gM}}$ and $\beta_{\overline{K}} = \mu_{\gM}$, so the claim follows.
\end{proof}

Our methods also naturally give the following Frobenius-descent property enjoyed by crystalline Breuil--Kisin modules (which is not used in the sequel).

\begin{prop}
\label{cor:FrobdescentfiltBKmod}
Let $(\gM,\varphi_\gM)$ be the Breuil--Kisin module attached to a crystalline Galois representation of $\Gal_{K}$. Consider the two filtrations on $(\varphi^* \gM)/p$, given by $(\varphi^* \gM \cap E^\bullet \gM)/p \subset \varphi^* \gM/p$ and $\varphi^*(\gM/p) \cap E^\bullet \gM/p \subset \varphi^* \gM/p$. Both these filtered $(u)^\bullet \KisinS/p$-modules descend naturally along the Frobenius $(u)^{\lceil \bullet/p \rceil} \KisinS/p \to (u)^\bullet \KisinS/p$.
\end{prop}

The proof is given at the end of \S \ref{ss:synstack}. The descent along the Frobenius for the second filtration $\varphi^*(\gM/p) \cap E^\bullet \gM/p$ considered above is essentially equivalent to Corollary~\ref{cor:mainNcor} and was also independently established by Gao--Liu \cite{gao2024integralsentheoryintegral} via slightly more classical techniques in integral $p$-adic Hodge theory. We are not aware of a classical argument for the first filtration.

The reader interested only in the {\em statements} about Breuil--Kisin modules should feel free to skip ahead to \S \ref{v3-sec:inertial-weights-mod-p-Galois-rep} at this point.

\begin{remark}
\label{rmk:mupweightcanmod}
Let $(\gM,\varphi_\gM)$ be a crystalline Breuil--Kisin module. By Theorem~\ref{thm:crysBKspecial}, we obtain a $\mu_p$-equivariant $\KisinS$-module $\gN$ and an isomorphism $\psi:\varphi^*\gM[1/u] \simeq \gN[1/u]$ whose relative position is given by $\mu_\gM$, the Hodge--Tate cocharacter of $\gM$. On the other hand, by $\mu_p$-equivariance, the mod $u$ reduction of $\gN$ determines a $\mathbf{Z}/p$-graded finite free $W(k)$-module, and thus a multisubset $S \subset \mathbf{Z}/p$ by tracking ranks of graded pieces. It follows from Remark~\ref{rmk:ptwistedmupweights} that $S$ is determined by the Hodge--Tate weights of $\gM$: regard the latter as a multisubset of $\mathbf{Z}$ and project to $\mathbf{Z}/p$. Informally, we may simply say that the $\mu_p$-weights of $\gN/u\gN$ (and thus also $\overline{\gM}/u\overline{\gM}$) coincide with the mod $p$ reduction of the Hodge--Tate weights of $\gM$. 
\end{remark}

\subsection{Recollections on the prismatization of \texorpdfstring{$\mathrm{Spf}(W(k))$}{Spf(W(k))}}
\label{ss:reviewprismatization}

In this subsection, we recall (without proof) some definitions and results surrounding the prismatization of $W(k)$ that are relevant to this paper; see \cite{bhatt2022absolute,bhatt2022prismatization,drinfeldprismatization} for more.

All our geometric objects are regarded as $p$-adic formal objects over $W(k)$, i.e., presheaves on $p$-nilpotent $W(k)$-algebras; in fact, we shall only work with fpqc sheaves in practice.

\begin{para}
{\em The stack $W(k)^\Prism$.
} Write $W(-)$ (or just $W$) for the ring scheme of $p$-typical Witt vectors; it is pro-smooth over $W(k)$.
Write $F:W \to W$ for the Witt vector Frobenius.
Let $W_{\text{dist}} \subset W$ be the subfunctor parameterizing {\em distinguished elements}, i.e., $d = (d_0,d_1,d_2,\dots)
\in W(R)$ with $d_0$ nilpotent and $d_1$ invertible.
By construction, $W_{\text{dist}}$ is the formal completion of a quasi-compact open subscheme over $W$ along a Cartier divisor, and it is easily seen to be stable under multiplication by the group scheme $W^*$ of units in $W$.

\begin{defn}
\label{defn:wkprism}
The {\em prismatization of $W(k)$} is defined to be the quotient
\[ W(k)^\Prism \coloneq  W_{\text{dist}}/W^*,\]
regarded as an fpqc stack on $p$-nilpotent $W(k)$-algebras.
\end{defn}

\begin{remark}[$W(k)^\Prism$ via Cartier--Witt divisors]
There is an implicit fpqc sheafification involved in Definition~\ref{defn:wkprism}. While this can be problematic in general, there is no issue above as one can identify the functor of points $W(k)^\Prism$: for a $p$-nilpotent $W(k)$-algebra $R$, the groupoid $W(k)^\Prism(R)$ is the groupoid of maps $I \xrightarrow{\alpha} W(R)$, where $I$ is an invertible $W(R)$-module, and $\alpha$ is a map of $W(R)$-modules such that the image of $I$ is, Zariski locally on $\mathrm{Spec}(R)$, generated by a
distinguished element, see \cite[Proposition 3.2.3]{bhatt2022absolute}. We prefer to emphasize the above quotient presentation as it is directly relevant for our main new geometric result (Proposition~\ref{prop:FrobDescMap}) about the prismatization.
\end{remark}

The stack $W(k)^\Prism$ carries a Frobenius lift $\varphi$ determined by the Witt vector Frobenius $F$ on $W_{\text{dist}}$ (and its compatibility with the action of $W^*$) and $\varphi$ on $W(k)$ itself. In fact, as $W(k)^\Prism$ admits a flat cover by the $\mathrm{Spf}(\mathbf{Z}_p)$-flat scheme $W$, it follows that this Frobenius lift $\varphi$ underlies a (unique) $\delta$-structure on $W(k)^\Prism$.
\end{para}

\begin{remark}[Comparison with relative prismatization]\label{rem:comparison-relative-prismatization}
The papers \cite{bhatt2022absolute,drinfeldprismatization} concern the prismatization of $\mathrm{Spf}(\mathbf{Z}_p)$ itself (which agrees with the above for $k=\mathbf{F}_p$), while  \cite{bhatt2022prismatization} provides a general definition of the prismatization $X^\Prism$ of any bounded $p$-adic formal scheme $X$. It is easy to see, from deformation theory, that $X^\Prism \simeq X \times_{\mathrm{Spf}(\mathbf{Z}_p)} \mathrm{Spf}(\mathbf{Z}_p)^\Prism$ when $X/\mathbf{Z}_p$ is formally \'etale (i.e., $L_{X/\mathbf{Z}_p} = 0$); here the projection map to $X$ arises from the natural identifications $X(R) \xleftarrow{{\simeq}} X(W(R)) \xrightarrow{\simeq} X(W(R)/I)$ for a test $p$-nilpotent ring $R$ and Cartier--Witt divisor $(I \xrightarrow{\alpha} W(R)) \in \Z_p^\Prism(R)$. As the formal \'etaleness holds true for $X=\mathrm{Spf}(W(k))$, we have simply defined $W(k)^\Prism$ as $\mathrm{Spf}(\mathbf{Z}_p)^\prism \times_{\mathrm{Spf}(\mathbf{Z}_p)} \mathrm{Spf}(W(k))$ above. Note that this would not be the case if we replaced $W(k)$ with a ramified extension.
\end{remark}

\begin{para} 
{\em The Hodge--Tate locus in $W(k)^\Prism$.}
There is a divisor $W(k)^{\HT} \subset  W(k)^\Prism$ determined by the $W^*$-stable divisor $W_{\text{dist},d_0=0} \subset W_{\text{dist}}$; write $I_{\HT} \subset \mathcal{O}_{W(k)^\Prism}$ for the corresponding ideal sheaf. This divisor is called the {\em Hodge--Tate divisor}. To describe it explicitly, note that there is a point $\rho_{\HT}:\mathrm{Spf}(W(k)) \to W(k)^{\HT} \subset W(k)^\Prism$ determined by the distinguished element $V(1) \in W(W(k))$. This point is called the Hodge--Tate point. Calculating its stabilizer in $W^*$ and using the surjectivity of $F$ on $W$, one shows  that
\[ \mathrm{Spf}(W(k))/W^*[F] \simeq W(k)^{\HT},\]
where $W^*[F] \subset W^*$ is the subgroup scheme determined by the kernel of $F$ on $W^*$ (see \cite[Theorem 3.4.13]{bhatt2022absolute}). Note that $\mu_p \subset W^*[F]$ via the Teichm\"uller map. In fact, one can show that projection to the $0$-th component induces an identification $W^*[F] \simeq \mathbf{G}_m^\sharp$, where $\mathbf{G}_m^\sharp$ is the PD-hull of the identity in $\mathbf{G}_m$ (see \cite[Proposition 3.4.13]{bhatt2022absolute}).
\end{para}

\begin{para}\label{para:ptsofprism}
{\em Some points of $W(k)^\Prism$ from bounded prisms.}
Fix a $\delta$-ring $B$ over $W(k)$ and a nonzerodivisor $d \in B$. Assume that $B$ is $(p,d)$-complete, that $\delta(d) \in B^*$, and $B/(d)$ has bounded $p$-power torsion (i.e., $(B,(d))$ is a bounded prism in the sense of \cite[Definition 1.4]{bhatt2019prisms}.) Then the image of $d$ under the natural $\delta$-map $B \to W(B)$ classifying the $\delta$-structure on $B$ gives a point of $W_{\text{dist}}(\mathrm{Spf}(B))$, where $B$ is given the $(p,d)$-adic topology. This yields a map
\begin{equation}\label{eqn:rho-B-d} \rho_{(B,(d))}:\mathrm{Spf}(B) \to W(k)^\Prism,\end{equation}
     which is a $\delta$-map and in particular intertwines the maps labelled $\varphi$ on both the source and the target. It has the following crucial property:

     \begin{lemma}[Flat covers of $W(k)^\Prism$ from prisms]
     \label{lem:flatcoverWkprism}
     With $(B,(d))$ as above, assume that $B/(d)$ is $p$-torsion-free (i.e., $(B,(d))$ is a transversal prism) and nonzero. Then $\rho_{(B,(d))}$ is faithfully flat, representable and affine.
     \end{lemma}
     \begin{proof}
     This is \cite[Corollaries 3.2.9 and 3.2.10]{bhatt2022absolute}.
     \end{proof}

We shall use the above map frequently for the Breuil--Kisin prism $(\KisinS,(E))$, so let us give it the name
\begin{equation}
\label{eq:rhoforBK}
\rho_{\mathrm{std}} = \rho_{(\KisinS,(E))   }:\mathrm{Spf}(\KisinS) \to W(k)^\Prism.
\end{equation}
It is faithfully flat by Lemma~\ref{lem:flatcoverWkprism}, and   $\rho_{\mathrm{std}}^*(I_{\HT}) = (E)$ as ideals in $\KisinS$.
\end{para}

\subsection{The canonical modification via the prismatization}
\label{sss:canmodprismproof}
In this subsection, we give our first proof of Theorem~\ref{thm:mainNthm}, via the prismatization. The main geometric innovation is the following natural map $\Spf(\KisinS) \to W(k)^\Prism$ that does not come from a prism structure via the construction in \S \ref{para:ptsofprism}. More precisely:

\begin{prop}
\label{prop:FrobDescMap}
There is a natural $W(k)$-linear map
\[ \rho_{\dagger}:\mathrm{Spf}(\KisinS)/\mu_p \to W(k)^\Prism \]
with the following features:
\begin{enumerate}
    \item The composition $\mathrm{Spf}(\KisinS) \xrightarrow{\canonicalmapnotation} \mathrm{Spf}(\KisinS)/\mu_p \xrightarrow{\rho_{\dagger}} W(k)^\Prism$ is determined by the distinguished element $[u] - zV(1) \in W(\KisinS)$. In particular, the pullback of the Hodge--Tate divisor on $W(k)^\Prism$ along $\rho_\dagger$ is the $\mu_p$-equivariant divisor $V((u)) \subset \mathrm{Spf}(\KisinS)$.

    \item We have an identification $\varphi \circ \rho_{\dagger} \simeq \rho_{\mathrm{std}} \circ q$ of maps $\mathrm{Spf}(\KisinS)/\mu_p \to W(k)^\Prism$ (both of which are linear over $\varphi$ on $W(k)$).

    \item The mod $p$ reductions of the maps $\rho_\dagger\circ \canonicalmapnotation$ and $\rho_{\mathrm{std}}$ are naturally identified (as maps $\mathrm{Spf}(\KisinS/p) \to W(k)^\Prism \times_{\mathrm{Spf}(\mathbf{Z}_p)} \mathrm{Spec}(\mathbf{F}_p)$).

    \item The map $\rho_{\dagger}$ is faithfully flat.
\end{enumerate}
    \end{prop}

We warn the reader that the map $\rho_\dagger$ does {\em not} come from a prism structure on $\KisinS$: if it did, it would commute with $\varphi$, but $\rho_\dagger$ does not (see Remark~\ref{rmk:rhodaggernotprism}).

\begin{proof}

In this proof, we interpret geometry on $\mathrm{Spf}(\KisinS)/\mu_p$ as $\mu_p$-equivariant geometry on $\mathrm{Spf}(\KisinS)$, e.g., to check if two $\mathrm{Spf}(\KisinS)/\mu_p$-valued points of a stack $F$ are identified, it is necessary and sufficient to identify the corresponding $\mathrm{Spf}(\KisinS)$-valued points $\mu_p$-equivariantly.

For the existence of $\rho_\dagger$ satisfying (1), we shall use the quotient presentation in Definition~\ref{defn:wkprism}. Note that $[u] - zV(1) \in W_{\text{dist}}(\mathrm{Spf}(\KisinS))$. Now, for any $\KisinS$-algebra $R$ equipped with $\zeta \in \mu_p(R)$, we have the equality
\[ [\zeta] \cdot ([u]-zV(1)) = [\zeta u] - [\zeta]zV(1) = [\zeta u] - zV(F[\zeta]) = [\zeta u] - zV(1)\]
in $W(R)$. Consequently, the map $\mathrm{Spf}(\KisinS) \to W_{\text{dist}}$ classifying $[u]-zV(1)$ intertwines the $\mu_p$-action on the source given by rescaling $u \in \KisinS$ with the action of the subgroup $\mu_p \stackrel{[\cdot]}{\hookrightarrow} W^*$ on the target. Passing to quotients gives the map $\rho_\dagger$ in (1). To obtain the description of the Hodge--Tate locus, we simply note that the $0$th Witt component of $[u]-zV(1) = [u] - V(Fz)$ is $u \in \KisinS$.

Part (2) follows from noting that $\varphi$ on $W_{\text{dist}}/W^*$ is obtained from $F$ on the Witt vectors (and $\varphi$ on $W$), coupled with the calculation
\[ F([u]-zV(1)) = [u]^p - \varphi(z)p = w(\varphi(u-pz)).\]

Part (3) is immediate from the construction of $\rho_\dagger$ since $p=V(1)$ in $W(\mathbb{F}_p)$.

For part (4), the faithful flatness modulo $p$ follows from part (3) by Lemma~\ref{lem:flatcoverWkprism}. As both $\mathrm{Spf}(\KisinS)$ and $W(k)^\Prism$ are $\mathrm{Spf}(\mathbf{Z}_p)$-flat, part~(4) then follows from local flatness criterion: if a map $A \to B$ of $p$-complete and $p$-torsion-free rings is (faithfully) flat modulo $p$, then $A/p^n \to B/p^n$ is (faithfully) flat for all~ $n$.
\end{proof}

    \begin{rem}
    \label{rmk:rhodaggernotprism}
    Even though the source and target of $\rho_\dagger$ admit natural $\delta$-structures, the map $\rho_\dagger$ is {\em not} a $\delta$-map. Indeed, even ignoring the $\mu_p$-equivariance, the map in Proposition~\ref{prop:FrobDescMap} (1) is not a $\delta$-map: if it were, then $[u] - zV(1) = w(u) - V(\varphi(z))$ 
    would have to lie in $w(\KisinS) \subset W(\KisinS)$, which is clearly not the case, since $w(B) \cap VW(B) = 0$ for any $\delta$-ring $B$. Concretely, the $\KisinS$-point underlying $\rho_\dagger \circ \varphi$ is given by $[u^p] - \varphi(z)V(1)$, while the $\KisinS$-point underlying $\varphi \circ \rho_\dagger$ is given by $w(\varphi(u-pz)) = [u^p] - \varphi(z)p$.
    \end{rem}

Let us also record a consequence of the discussion in \S \ref{sss:numstack} in relation to the map $\rho_\dagger$.

\begin{example}[Relative position over $W(k)^\Prism$]
\label{ex:relposWkprism}
Consider the closed immersion $W(k)^{\HT} \subset W(k)^\Prism$, so $\pi_0(W(k)^{\HT}) = \{\ast\}$. The analysis of Section~\ref{sss:numstack} applies to this closed immersion. The Breuil--Kisin prism provides a flat cover $\rho_{\text{std}}:\mathrm{Spf}(\KisinS) \to W(k)^\Prism$ with $\KisinS$ excellent and 
$\KisinS/\rho_{\text{std}}^*(I_{\HT}) \simeq \KisinS/(E(u)) = \mathcal{O}_K$ being regular, so every $I_{\HT}$-isogeny over $W(k)^\Prism$ is good in the sense of Definition~\ref{def:goodisog}. In particular, given an $I_{\HT}$-isogeny $\psi$ between $E,F \in \mathrm{Vect}(W(k)^\Prism)$, we obtain a function $f_\psi:\mathbf{Z} \to \mathbf{Z}_{\geq 0}$. This function can be computed by first pulling back $\psi$ to any flat cover $\pi:\mathrm{Spf}(B) \to W(k)^\Prism$ where $B$ is excellent and $B/\pi^*(I_{\HT})$ is regular and connected, and then applying \eqref{eq:relpos}. In particular, taking $\pi$ to be either $\rho_{\text{std}}$ or $\rho_\dagger$ shows that the relative positions of $\rho_{\text{std}}^*(\psi)$ and $\rho_\dagger^*(\psi)$ agree.
\end{example}

Let us use Proposition~\ref{prop:FrobDescMap} to prove Theorem~\ref{thm:mainNthm}.

\begin{proof}[Proof of Theorem~\ref{thm:mainNthm}]

Since $(\gM,\varphi_\gM)$ is crystalline, it is (uniquely) the pullback of a prismatic $F$-crystal $(\mathcal{E},\varphi_{\mathcal{E}})$ along the map $\rho_{\text{std}}$; more general statements are recalled in the forthcoming Theorem~\ref{thm:CrysGalPrismCrys}. Let $\gN \coloneq  {\canonicalmapnotation}^* \rho_\dagger^* \mathcal{E}$, regarded as a $\mu_p$-equivariant vector bundle on~ $\KisinS$. Note that there is a $\mu_p$-equivariant identification
\[ {\canonicalmapnotation}^* \rho_\dagger^* \varphi^* \mathcal{E} \simeq \varphi^* \gM \]
by Proposition~\ref{prop:FrobDescMap} (2) and \eqref{eq:Kisin-Phi-factorize}. Moreover, the pullback of the Hodge--Tate divisor $W(k)^{\HT} \subset W(k)^\Prism$ along $\rho_\dagger$ is simply $\{u=0\} \subset \Spf(\KisinS)/\mu_p$ by Proposition~\ref{prop:FrobDescMap} (1). Consequently, the isogeny $\varphi_{\mathcal{E}}$ pulls back along $\rho_\dagger$ to a $\mu_p$-equivariant identification $\psi:\varphi^* \gM[1/u] \simeq \gN[1/u]$. The assignment carrying  $(\gM,\varphi_\gM)$ to the triple $(\gM,\gN,\psi)$ then defines a functor
\[ \underline{\gN}(-):\mathrm{Vect}^{\varphi}(\KisinS)_\cris \to \mathcal{D}.\]
As the construction of $\underline{\gN}(-)$ only involves pullbacks (and the $\otimes$-equivalence between prismatic $F$-crystals and crystalline Breuil--Kisin modules), it is clearly a $\otimes$-functor. To finish proving Theorem~\ref{thm:mainNthm}, it suffices to verify (1) and (2) in the theorem, and show exactness of $\underline{\gN}(-)$.

Part (1) is immediate from Proposition~\ref{prop:FrobDescMap} (3) and the construction of $\gN$.

For (2), we need to show that the relative positions $\mu_{\gM}$ and $\mu_{\underline{\gN}(\gM)}$ of the isogenies defined by $\varphi_\gM$  over $K\llbracket E \rrbracket$ and  $\psi$ over $K\llbracket u \rrbracket$ agree; this follows from Example~\ref{ex:relposWkprism}.

Finally, it remains to prove exactness. As exactness in $\mathcal{D}$ can be detected after reduction modulo $p$, the claim follows formally from the functorial identification $\gM/p\gM \simeq \gN/p\gN$ constructed in the proof of (1) above.
\end{proof}

\begin{remark}[$\mu_p$-equivariance for mod $p$ $F$-crystals]
\label{rmk:modpFcrysmupequiv}
There is an evident notion of a prismatic $F$-crystal in vector bundles on $(W(k)^\Prism)_{p=0}$. 
Given such a crystal $(\overline{\mathcal{E}}, \varphi_{\overline{\mathcal{E}}})$ with Breuil--Kisin realization $(\overline{\gM},\varphi_{\overline{\gM}})$ over $\mathrm{Spf}(\KisinS/p)$,  the proof of Theorem \ref{thm:mainNthm} (and ultimately the isomorphism of
Proposition~\ref{prop:FrobDescMap} (2)) shows that $(\overline{\gM},\varphi_{\overline{\gM}})$ is naturally $\mu_p$-equivariant.
\end{remark}

At this point, we have given our first proof of Theorem~\ref{thm:mainNthm}. The reader interested only in applications of this statement to Galois representations should feel free to skip ahead to \S \ref{v3-sec:inertial-weights-mod-p-Galois-rep} at this point (possibly after looking at the statements for $G$-valued representations in \S \ref{subsec:Gstr}).

\subsection{Specializable Breuil--Kisin modules}

The main goal of this subsection is to formulate Theorem~\ref{thm:crysBKspecial}, which isolates an algebraic property --- specializability --- enjoyed by crystalline Breuil--Kisin modules (Definition~\ref{def:specializable}); this formally implies Theorem~\ref{thm:mainNthm}, as we explain after Theorem~\ref{thm:crysBKspecial}. The specializability property concerns the reduction modulo $p$ of a certain canonical $\mu_p$-equivariant modification of any Breuil--Kisin module, which is constructed in Proposition~\ref{prop:canmodBK}. For the latter construction, it is convenient to use the formalism of the $p$-Hecke stack from \S \ref{ss:fungrass}, so we recall some surrounding notation in \S \ref{sss:xprimezprimeBK}.

\subsubsection{Recollections on the Hecke stacks}
\label{sss:xprimezprimeBK}
It will be convenient to use the stacks from \S \ref{ss:fungrass}. Since we change notation nominally\footnote{More precisely, for compatibility with standard notation for Breuil--Kisin modules, the relative Frobenius is now viewed as an $R$-algebra endomorphism $R\llbracket u \rrbracket \to R\llbracket u \rrbracket$ sending $u$ to $u^p$, rather than the inclusion $R\llbracket u^p \rrbracket \subset R\llbracket u \rrbracket$ as in \S \ref{ss:fungrass}. Moreover, we stick to the case of $\GL_n$, in which case we can reformulate statements in terms of vector bundles, leaving the rank unspecified.} let us simply give the definition.

Let $\Hk$ be the stack on $W(k)$-algebras sending a $W(k)$-algebra $R$ to the groupoid of triples $(\mathfrak{M}_1,\mathfrak{M}_2,\tau)$, where $\mathfrak{M}_i \in \mathrm{Vect}(R\llbracket u \rrbracket)$ and $\tau:\gM_1[1/u] \simeq \gM_2 [1/u]$ is an isomorphism over $R((u))$.

Similarly, $\pHk$ is the stack sending $R$ to the groupoid of triples $(\gM_1,\gM_2,\psi)$, where $\gM_1$ is a  vector bundle on $R\llbracket u \rrbracket$, $\gM_2$ is a $\mu_p$-equivariant vector bundle on $R\llbracket u \rrbracket$, and 
$\psi: \varphi_{/R}^* \gM_1[1/u] \simeq \gM_2[1/u]$ is a $\mu_p$-equivariant isomorphism over $R((u))$, where $\varphi_{/R}:R\llbracket u \rrbracket \to R\llbracket u \rrbracket$ is the $R$-algebra map sending $u$ to $u^p$.

There is a natural map $\pHk \to \Hk$ given by sending $(\gM_1,\gM_2,\psi) \in \pHk(R)$ to $(\varphi_{/R}^* \gM_1, \gM_2,\psi) \in \Hk(R)$. This map fits into a commutative diagram
\[ \xymatrix{ \pHk(R) \ar[r] \ar[d] & \Hk(R) \ar[d] \\ \mathrm{Vect}(R\llbracket u \rrbracket) \ar[r]^-{\varphi_{/R}^*} & \mathrm{Vect}(R\llbracket u \rrbracket) }\]
where the vertical maps remember the first component bundle from the functor of points description.

Given a $W(k)$-algebra $R$ and a vector bundle $\gM \in \mathrm{Vect}(R\llbracket u \rrbracket)$, write $\pHk(R)_\gM$ (resp. $\Hk(R)_\gM$) for the (groupoid-theoretic) fibre of the left vertical map (resp. right vertical map) over $\gM$ on the bottom left (resp. bottom right) in the diagram above. In this situation, given an $R$-algebra $S$, we shall often abuse notation and write $\Hk(S)_\gM$ for $\Hk(S)_{\gM \otimes_R S}$ (and similarly for the $\pHk$ variant).

Given a $W(k)$-algebra $R$ and a vector bundle $\gM \in \mathrm{Vect}(R\llbracket u \rrbracket)$, the displayed diagram above induces a natural map $\pHk(R)_\gM \to \Hk(R)_{\varphi^*_{/R}\gM}$ by passing to fibres. After fixing a basis of $\gM$ (when it exists, always achievable locally) and using the induced basis of $\varphi_{/R}^* \gM$, this map agrees with the obvious map $\Gr_{\GL_n}^{\mu_p}(R) \to \Gr_{\GL_n}(R)$.

\subsubsection{Canonical $\mu_p$-equivariant modifications}
\label{sss:canmupmod}

Fix a Breuil--Kisin module $(\gM,\varphi_{\gM})$. We shall construct a canonical $\mu_p$-equivariant modification of $\gM$ supported at $\{u=0\} \subset \mathrm{Spec}(\KisinS)$ using the Hodge filtration of the associated $F$-crystal (Proposition~\ref{prop:canmodBK}).

Let us first recall some basic constructions from \cite{KisinCrys}. Write
\[ (D,\varphi_D) \coloneq  \varphi^* (\gM,\varphi_{\gM}) \otimes_\KisinS \KisinS/(u)[1/p] \]
for the associated $F$-isocrystal over $W(k)$. The open unit disc $\mathbf{D}{[0,1)}$ can be viewed as the generic fibre of the (non $p$-adic) formal scheme $\mathrm{Spf}(\KisinS)$, so finitely generated $\KisinS$-modules give coherent sheaves on $\D[0,1)$ via $M \mapsto M \otimes_\KisinS \mathcal{O}_{\D[0,1)}$.  Write $\mathbf{D}[0,r) \subset \mathbf{D}[0,1)$ for the smaller open unit disc of radius some fixed $r \in (|p|,|p|^{1/p})$. By \cite[Lemma 1.2.6]{KisinCrys}, there is a unique $\varphi$-equivariant isomorphism
\begin{equation}
\label{eq:ConnBKRational}
    D \otimes_K \mathcal{O}_{\D[0,r)} \simeq \varphi^* \gM \otimes_{\gS} \mathcal{O}_{\D[0,r)}
\end{equation}
whose reduction modulo $u$ is the identity. In particular, as $r > |p|$, we can also identify $D$ with $(\varphi^*\gM/(E)\varphi^*\gM)[1/p]$. On the other hand, we have  the relative ``Nygaard filtration''
\begin{equation}
\label{eq:NygaardBK} \Fil^i \varphi^* \gM = (E)^i \gM \cap \varphi^* \gM
\end{equation}
on $\varphi^*\gM$. Inverting $p$ and taking its image under the map
\[
  \varphi^* \gM[1/p] \to (\varphi^*\gM/(E)\varphi^* \gM)[1/p] = D
\]
then defines a filtration $F^\bullet D$ that we call the Hodge filtration on $D$.

\begin{rem} If $(\gM,\varphi_\gM)$ comes from a crystalline Galois representation $T$, then $D = D_{\crys}(T)$ and $F^\bullet D$ agrees with the classical Hodge filtration on $D_{\crys}(T) \simeq D_{\dR}(T)$ (by the construction of the equivalence in \cite[Theorem 1.2.15]{KisinCrys}). 
\end{rem}

Observe that $\varphi^* \gM$ is naturally a $\mu_p$-equivariant vector bundle on $\Spec(\KisinS)$ with respect to the scaling $\mu_p$-action on $\KisinS$. Thanks to the factorization
\[ \varphi = \varphi_{/W(k)} \circ \varphi_{W(k)} = \varphi_{W(k)} \circ \varphi_{/W(k)}\]
of the Frobenius on $\KisinS$ in terms of the Frobenius $\varphi_{W(k)}$ on $W(k)$ and the relative Frobenius $\varphi_{/W(k)}$, for any $W(k)$-algebra $R$, the groupoid $\pHk(R)_{\varphi_{W(k)}^* \gM}$ from \S \ref{sss:xprimezprimeBK} is simply the groupoid of pairs 
$(\gN,\psi)$ consisting of a $\mu_p$-equivariant $R\llbracket u \rrbracket$-module $\gN$ and a $\mu_p$-equivariant isomorphism $\varphi^* \gM \otimes_\KisinS R((u)) \simeq \gN[1/u]$. The promised canonical modification of $\gM$ is a natural point of this groupoid:

\begin{prop}
\label{prop:canmodBK}
There is a unique pair $(\gN^{\can},\psi^{\can}) \in \pHk(W(k))_{\varphi_{W(k)}^* \gM}$ whose image in $\pHk(K)_{\varphi_{W(k)}^* \gM}$ is the $\mu_p$-equivariant modification of
\[ D \otimes_K K\llbracket u \rrbracket \simeq (\varphi^* \gM) \otimes_\KisinS K\llbracket u \rrbracket \simeq \varphi_{/K}^* \left(\varphi_{W(k)}^* \gM \otimes_\KisinS K\llbracket u \rrbracket\right)\]
given by the Rees module of the Hodge filtration $F^\bullet D$, i.e., by the $u$-adic completion of the graded $K[u]$-module $\oplus_i F^i D u^{-i}$ (see Lemma~\ref{lem:qcohX}) or equivalently as  $\mathrm{Fil}^0(F^\bullet D \otimes_K u^\bullet K((u)))$.
\end{prop}

\begin{proof}
Note that the ``first component'' map $\pHk \to \mathrm{Vect}$ is an ind-proper morphism of stacks: indeed, locally on the base, it is a $\Gr_{\mathrm{GL}_n}^{\mu_p}$-bundle. In particular, by the valuative criterion, the  natural map
\[ \pHk(W(k))_{\varphi_{W(k)}^* \gM} \to \pHk(K)_{\varphi_{W(k)}^*\gM}\] is an isomorphism. The claim then follows as the Rees construction indeed gives a $\mu_p$-equivariant modification\footnote{In fact, the natural $\mu_p$-equivariant structure on $\varphi^* \gM$ upgrades to a $\mathbf{G}_m$-equivariant structure on $\varphi^* \gM \otimes_{\KisinS} K\llbracket u \rrbracket$ thanks to the isomorphism \eqref{eq:ConnBKRational}. The
  Rees modification, by construction, is even $\mathbf{G}_m$-equivariant. However, this $\mathbf{G}_m$-equivariant structure does not spread out to $\Spec(\KisinS)$, so we can only talk about $\mu_p$-equivariance integrally.} of $\varphi^* \gM \otimes_{\gS} K\llbracket u\rrbracket \simeq D \otimes_K K\llbracket u\rrbracket$.
\end{proof}
Examining the construction shows that the assignment
\begin{equation}
    \label{eq:functorcanmodBK}
(\gM,\varphi_\gM) \mapsto (\gM, \gN^{\can},\psi^{\can})
\end{equation}
gives a $\otimes$-functor from Breuil--Kisin modules to the category~$\cD$ defined in~\ref{para:category-D}. 
We warn the reader that, since the construction of $\gN^{\can}$ involves $*$-extension of vector bundles across punctured spectra (via the proof of the valuative criterion invoked in Proposition~\ref{prop:canmodBK}), the functor $(\gM,\varphi_\gM) \mapsto \gN^{\can}$ is not exact in general.

\begin{remark}[Geometric realization of the canonical modification]
\label{rmk:GeomIntCanMod}
Consider the graded ring $R_K = K\llbracket u \rrbracket[y,t]/(yt^p-u)$. Fix a finite projective $K\llbracket u \rrbracket$-module $M$, so its base change $\varphi_{/K}^* M$ under the $K$-algebra map $\varphi_{/K}:K\llbracket u \rrbracket \xrightarrow{u \mapsto u^p} K\llbracket u \rrbracket$ is naturally $\mu_p$-equivariant. By Example~\ref{ex:VectReesKu} (see also Proposition~\ref{prop:CohSheafptwistedRees}), the category of $\mu_p$-equivariant modifications of $\varphi_{/K}^* M$ supported at $\{u=0\}$ is naturally identified with the category of vector bundles $E$ on $\mathrm{Spec}(R_K)/\mathbf{G}_m$ equipped with an identification of $E_{t \neq 0} \in \mathrm{Vect}_{\gr}(R_{K}[1/t]) \simeq \mathrm{Vect}(K\llbracket u \rrbracket)$ with $M$. Taking $M = D \otimes_K K\llbracket u \rrbracket$ (so $M \simeq \varphi_{/K}^* M$), the output of Proposition~\ref{prop:canmodBK} then yields such a bundle $E$. This bundle can be described geometrically as follows: the Rees module of $F^\bullet D$ defines a vector bundle $E' \in \mathrm{Vect}_{\gr}(K[t])$, and the bundle $E$ is simply the pullback of $E'$ along the evident map $r_K:\Spec(R_K)/\mathbf{G}_m \to \mathrm{Spec}(K[t])/\mathbf{G}_m$.
\end{remark}

\subsubsection{The main theorem on specializability}

Let $(\gM,\varphi_{\gM})$ be a Breuil--Kisin module. Let us write $\overline{\gM} \coloneq \gM/p\gM$, and similarly for mod $p$ reductions of other objects. In Proposition~\ref{prop:canmodBK}, we constructed a $\mu_p$-equivariant modification $\psi^{\can}$ of $\varphi^* \gM$ supported at $\{u=0\}$. On the other hand, tautologically, $\varphi_{\gM}$ is a modification of $\varphi^* \gM$ supported at $\{E=0\}$. Since $E \equiv u \mod p\KisinS$, both  $\overline{\psi^{\can}}$ and $\overline{\varphi_\gM}$ give modifications of $\overline{\varphi^* \gM}$ supported at $\{u=0\}$ on $\Spec(\KisinS/p)$.

\begin{defn}
\label{def:specializable}
The Breuil--Kisin module  $(\gM,\varphi_{\gM})$ is called {\em specializable} if $\overline{\gN^{\can}}$ and $\overline{\gM}$ are equal as submodules of $\overline{\varphi^* \gM}[1/u]$ via $\overline{\psi^{\can}}$ and $\overline{\varphi_\gM}$ respectively. Write $\mathrm{Vect}^{\varphi}(\KisinS)_{\specializable} \subset \mathrm{Vect}^{\varphi}(\KisinS)$ for the full subcategory spanned by specializable Breuil--Kisin modules.
\end{defn}

Using the functoriality of the canonical modification, one checks that the inclusion $\mathrm{Vect}^{\varphi}(\KisinS)_{\specializable} \subset \mathrm{Vect}^{\varphi}(\KisinS)$ is a $\otimes$-subcategory. Moreover, on this subcategory, the functor \eqref{eq:functorcanmodBK} is actually exact: indeed, the right exactness of $(\gM,\varphi_\gM) \to \gN^{\can}$ can be checked modulo $p$, and then follows from the agreement of $\overline{\psi^{\can}}$ with $\overline{\varphi_{\gM}}$. 

\begin{remark}
The canonical modification is defined using the isomorphism \eqref{eq:ConnBKRational}, which is closely related to the canonical (meromorphic) connection $N_\nabla$ on $\gM \otimes_\KisinS \mathcal{O}_{\D[0,1)}$ (see \cite[Corollary 1.3.15]{KisinCrys}); from this optic, the ``specializability'' condition can be regarded as an integrality constraint on $N_\nabla$.
\end{remark}

The main existence theorem about such specializable objects is the following theorem, whose proof uses prismatic $F$-gauges, and will be given at the end of~\S \ref{ss:PfCanModThm}.

\begin{theorem}
\label{thm:crysBKspecial}
Crystalline Breuil--Kisin modules are specializable.
\end{theorem}

The proof of Theorem~\ref{thm:crysBKspecial} 
uses some new results on the syntomification. Given the relatively classical nature of the statement, it would be quite interesting to have a more direct proof. For now, let us simply prove Theorem~\ref{thm:mainNthm} assuming Theorem~\ref{thm:crysBKspecial}:

\begin{proof}[Proof of Theorem~\ref{thm:mainNthm} via specializability]
We shall show that the conclusion of Theorem~\ref{thm:mainNthm} holds true with the category $\mathrm{Vect}^\varphi(\KisinS)_{\text{sp}}$ replacing $\mathrm{Vect}^\varphi(\KisinS)_{\cris}$; this gives the desired result by Theorem~\ref{thm:crysBKspecial}. Moreover, we already observed in~\eqref{eq:functorcanmodBK} 
that the assignment
$(\gM,\varphi_\gM) \mapsto (\gM,\gN^{\can},\psi^{\can})$ yields an exact $\otimes$-functor $\mathrm{Vect}^\varphi(\KisinS)_{\text{sp}} \to \mathcal{D}$, so it remains to verify that (1) and (2) from Theorem~\ref{thm:mainNthm} hold true for this functor. But (2) holds true by construction of the canonical modification, while  (1) is exactly the specializability condition, so we win.\end{proof}

\subsection{The Nygaard stack}
\label{ss:Nygaard}

In this subsection, we study the compactification of $W(k)^\Prism$ provided by the Nygaard stack $W(k)^\cN$ (see \cite{drinfeldprismatization,BBFgaugenotes}). Quasi-coherent sheaves on $W(k)^\cN$, also called {\em gauges} over $\Spf(W(k))$, encode the Nygaard filtration (whose Breuil--Kisin realization already appeared earlier in \eqref{eq:NygaardBK}) on a prismatic crystal.

\subsubsection{A ``divisorial'' description of $\Z_p^\cN$: }
As we saw in \S \ref{ss:reviewprismatization}, the stack $\Z_p^\Prism$ parametrizes generalized Cartier divisors in the ring scheme $W$ of Witt vectors. Drinfeld constructed an enlargement $\Z_p^\cN$ of $\Z_p^\Prism$ (or, more precisely, of $\Z_p^\Prism \sqcup \Z_p^\Prism$) as a moduli space of certain (non-invertible) quasi-ideal schemes in $W$, phrased in terms of ``admissible $W$-module schemes'', see \cite{drinfeldprismatization}. The notes \cite{BBFgaugenotes} gave an alternate perspective on $\Z_p^\cN$ via quasi-syntomic descent. While we do not recall either perspective here, we shall offer an alternative, and perhaps more conceptual, description of $W(k)^\cN$ entirely in terms of generalized Cartier divisors. It relies on the Teichm\"uller construction $[\cdot]$, regarded as a $\otimes$-functor from the category of invertible modules over a ring $R$ to the category of invertible $W(R)$-modules; this functor is right inverse to the obvious base change functor in the other direction, and, due to the identity $[x^p] = F([x]) \in W(R)$ for $x \in R$, carries the $p$-power operation on invertible $R$-modules to $F^*$ on invertible $W(R)$-modules.

\begin{prop}
\label{prop:ZpNygDesc}
For a $p$-nilpotent ring $R$, the groupoid $\mathbb{Z}_p^{\cN}(R)$ identifies with the groupoid of the following triples:
\begin{enumerate}
\item $(d: I \longrightarrow W(R))$ is a Cartier--Witt divisor over $\Spec(R)$, i.e., a point of $\mathbb{Z}_p^\Prism(R)$;
\item $(t: \mathcal{O}(-1) \longrightarrow R)$ is a generalized Cartier divisor over $\Spec(R)$, i.e., a point of $\mathbb{A}^1/\mathbb{G}_m(R)$;
\item $I/p \xrightarrow{y} [\mathcal{O}(-p)]/p \xrightarrow{[t]^p} W(R)/p$ is a factorization, in invertible $W(R)/p$-modules, of the mod $p$ reduction of $d$, where $[t]: [\mathcal{O}(-1)] \longrightarrow W(R)$ is the Teichm\"uller lift of $t$.
\end{enumerate}
Moreover, for any perfect $\F_p$-algebra $k$, the map\footnote{\label{fn:wklinearity} This isomorphism is normalized as follows: the $\Z_p^\cN$ component of this map is the obvious one, while the $\mathrm{Spf}(W(k))$ component is induced by the projection $W(k)^\cN \to W(k)^\Prism$ and the natural $W(k)$-structure on the target. In particular, this is {\em not} the $W(k)$-structure on $W(k)^\cN$ induced by the universal ring stack $\mathbf{G}_a^\cN$, but rather the composition of the latter with the Frobenius on $W(k)$.} $W(k)^\cN \to \Z_p^\cN \times_{\mathrm{Spf}(\Z_p)} \mathrm{Spf}(W(k))$ is an isomorphism, so we obtain a similar description for $W(k)^\cN$.
\end{prop}

As the preceding description of $\Z_p^\cN$ is in terms of invertible modules (over $W$, $\mathbf{G}_a$ and $W/p$), we shall refer to it informally as the {\em divisorial description}.

In the proof below, we shall write $W$ for the ring scheme of $p$-typical Witt vectors, and $F:W \to W$ for the Witt vector Frobenius. To emphasize $W$-linearity, we shall write $F_* W$ to indicate the group scheme $W$, endowed with the $W$-module structure via $F$; thus, the Frobenius itself is a $W$-module map $W \to F_* W$, while the Verschiebung is a $W$-module map $V:F_*W \to W$, etc.

\begin{proof}[Proof of Proposition~\ref{prop:ZpNygDesc}]
The last part is standard from deformation theory, so we focus on the first part. Fix a test $p$-nilpotent ring $R$. We have a natural map $\alpha:\mathbf{Z}_p^\cN(R) \to \Z_p^\Prism(R) \times \mathbf{A}^1/\mathbf{G}_m(R)$ given by remembering the underlying Cartier--Witt divisor on the first component and the Rees map on the second component; these maps are labelled $\pi$ and $t$ in \cite[Construction 5.3.3]{BBFgaugenotes}. Fix a pair $(d,t)$ as in (1) and (2) above, regarded as a map $\mathrm{Spec}(R) \to \Z_p^\Prism(R) \times \mathbf{A}^1/\mathbf{G}_m(R)$. It suffices to identify the fibre $\mathcal{F}$ of $\alpha$ over the point $(d,t)$ with the groupoid of factorizations as in (3) above. By definition (\cite[Definition 5.3.1]{BBFgaugenotes}), $\mathcal{F}$ is the groupoid of pairs $(M,\widetilde{d})$, where $M$ is an extension of $F_* I$ by $\mathbf{V}(\mathcal{O}(-1))^\sharp$ (in a necessarily unique fashion, see \cite[Remark 5.2.5]{BBFgaugenotes}), and $\tilde{d}:M \to W$ is a morphism of $W$-module schemes such that the diagram
\[ \xymatrix{ 0 \ar[r] & \mathbf{V}(\mathcal{O}(-1))^\sharp \ar[r] \ar[d]^-{t} & M \ar[r] \ar[d]^-{\widetilde{d}} & F_* I \ar[d]^-{F_*d} \ar[r] & 0 \\
0 \ar[r] & \mathbf{G}_a^\sharp \ar[r] & W \ar[r]^-F & F_*W \ar[r] & 0 }\]
commutes (and thus gives a map of short exact sequences). As with any morphism of short exact sequences, this diagram can be enlarged into a diagram
\begin{equation}\label{eqn:pullback-pushout-diagram} \xymatrix{ 0 \ar[r] & \mathbf{V}(\mathcal{O}(-1))^\sharp \ar[r] \ar[d]^-{t} & M \ar[r] \ar[d] & F_* I \ar@{=}[d] \ar[r] & 0 \\
    0 \ar[r] & \mathbf{G}_a^\sharp \ar[r] \ar@{=}[d] & M' \ar[r] \ar[d] & F_* I \ar[d]^-{F_*d} \ar[r] & 0 \\
    0 \ar[r] & \mathbf{G}_a^\sharp \ar[r] & W \ar[r]^-F & F_*W \ar[r] & 0, }\end{equation}
where the second row is both the pushout of the first row along the map $t$, and the pullback of the last row along the map $F_{*}d$. Conversely, specifying such an isomorphism of the respective pullback and pushouts uniquely determines  $\widetilde{d}$.  It follows that $\mathcal{F}$ is also the fibre product of
\[ \xymatrix{
  & \tau \ar[d] \\
  & \Ext_W(F_* W, \mathbf{G}_a^\sharp) \ar[d]_-{F_*d} \\
  \Ext_W(F_* I,\mathbf{V}(\mathcal{O}(-1))^\sharp) \ar[r]^-{[t]} & \Ext_W(F_* I,\mathbf{G}_a^\sharp)
}\]
where $\tau$ denotes the tautological extension appearing in the last row of~\eqref{eqn:pullback-pushout-diagram}. But now recall that $\Ext_W(F_*W,\mathbf{G}_a^\sharp) = F_*W/p$ with $1$ on the right corresponding to $\tau$ on the left (see \cite[Proposition 5.2.1 and Corollary 2.6.8]{BBFgaugenotes}). Twisting, this implies that for any pair of invertible $W$-modules $M$ and $N$, we have a natural identification
\[ \mathrm{Ext}_W(F_* M, N \otimes_W \mathbf{G}_a^\sharp) \simeq N \otimes_W F_*(M^\vee/p) \simeq F_*\left(M^\vee/p \otimes_{W/p} F^* N/p\right).\]
Applying this with $M=I$ and $N=[\mathcal{O}(-1)]$ or $N=W$, we can rewrite the above fibre product as
\[ \xymatrix{
  && 1 \ar[d] \\
  && F_* W/p \ar[d]_-{F_*d} \\
  F_* \left(I^\vee/p \otimes_{W/p} [\mathcal{O}(-p)]/p\right) \ar[rr]^-{F_*[t^p] = [t] } && F_* I^\vee/p
} \]
As all the linear maps appearing are actually $F_*W/p$-linear, this simplifies to the fibre product of
\[ \xymatrix{
  & 1 \ar[d]_-{d} \\
  \Hom_{W/p}(I/p,[\mathcal{O}(-p)]/p) \ar[r]^-{[t]^p} & \Hom_{W/p}(I/p,W/p)
} \]
which is exactly what the groupoid of data in (3) records.
\end{proof}

\begin{remark}
In \cite[Remark 7.2.3]{drinfeldprismatization}, Drinfeld noted that the explicit quotient presentation of $\Z_p^\cN$ in \cite[\S 7.2.1-7.2.2]{drinfeldprismatization} showed a slightly mysterious descent property: the Rees map $\Z_p^\cN \to \mathbf{A}^1/\mathbf{G}_m$ descends\footnote{However, he also observed that this descent property is not shared by the universal ring stack $\mathbf{G}_a^\cN$ over $\Z_p^\cN$; this can also be seen from our description of $\mathbf{G}_a^\cN$ as a fibre product of a diagram \eqref{eq:fibproductGaNyg} whose last term requires the knowledge of $t$ (and not just $t^p$).} along the $p$-power map $\mathbf{A}^1/\mathbf{G}_m \to \mathbf{A}^1/\mathbf{G}_m$. The divisorial description in Proposition~\ref{prop:ZpNygDesc} makes this descent more transparent: the Rees parameter $t$ only appears via its $p$-th power in part (3) of the description.
\end{remark}

\begin{remark}
We were led to the divisorial description of $\Z_p^\cN$ quite directly by the applications of this paper. Specifically, the most natural way to geometrically realize the canonical modification from \S  \ref{sss:canmupmod} from the data of an $F$-gauge is via pullback along a map $f:Y \to \Z_p^\cN$ (as we do in \S \ref{ss:PfCanModThm}). In constructing this map, the form of the co-ordinate ring of $Y$ (see \eqref{eq:coordringY}) almost immediately led us to  the divisorial description of $\Z_p^\cN$.  But an essentially equivalent form of Proposition~\ref{prop:ZpNygDesc} in fact appeared earlier in \cite[\S 6.4]{GardnerMadapusiAlg}; as they stress, the divisorial description adapts straightforwardly to the derived setting.
\end{remark}

\subsubsection{Reinterpreting classical structures via the divisorial description}
\label{sss:ReinterpretClassicalDivisorial}

It is a rather pleasant exercise to understand structures over $\Z_p^\cN$ in terms of its divisorial description. To carry this out, let us name some universal objects over $\Z_p^\cN$ coming from Proposition~\ref{prop:ZpNygDesc}: we have a universal Cartier--Witt divisor $d^{\univ}:\mathcal{I}_W^{\univ} \to W$, a universal Cartier divisor $t^{\univ}:\mathcal{O}(-1) \to \mathcal{O}_{\Z_p^\cN}$, and a universal factorization
\[ \mathcal{I}_W^{\univ}/p \xrightarrow{y^{\univ}} [\mathcal{O}(-p)]/p \xrightarrow{[t^{\univ}]^p} W/p \]
of $d^{\univ} \mod p$. Using these, we explain how to recover some geometric statements about $\Z_p^\cN$:

\begin{enumerate}[wide,itemsep=4pt]
    \item {\em The universal ring stack:} We have a diagram of animated $W$-algebra stacks

    \begin{equation}
\label{eq:fibproductGaNyg}
\xymatrix{
  & F_* W/\mathcal{I}_W^\univ \ar[d] \\
  & F_* W/(p,\mathcal{I}_W^\univ) \ar[d] \\
  \mathbf{G}_a/\mathcal{O}(-1) \ar[r] & F_*W/(p,[\mathcal{O}(-p)])
}
\end{equation}
where all quotients are interpreted in the Koszul sense, the first vertical map is the obvious one, the second vertical map is induced by $y^\univ$, and the horizontal map is induced by Frobenius $\mathbf{G}_a = W/VW \xrightarrow{F}  F_*W/p$. The fibre product of this diagram is the universal $W$-algebra stack $\mathbf{G}_a^\cN$.

    \item  {\em The universal ideal:}
Recall that \cite{drinfeldprismatization} defines $\Z_p^\cN$ as a moduli stack of pairs $(M,d)$, where $M$ is an ``admissible $W$-module scheme'' and $d:M \to W$ is a ``primitive'' map. In particular, there is a universal such map $d^\univ:M^\univ \to W$ over $\Z_p^\cN$ itself. In the ring stack picture, $d^\univ$ is simply the fibre of the structure map $W \to \mathbf{G}_a^\cN$. Thus, part (1) already contains a description of $d^\univ$ implicitly. But the proof of Proposition~\ref{prop:ZpNygDesc} actually gives a much more direct description of $M^{\univ}$: the section $y^\univ$ classifies $M^{\univ}$ in the sense that $M^{\univ}$ is the pullback of the tautological extension
    \[ \mathcal{O}(-1) \otimes_{\mathbf{G}_a} \left( 0 \to \mathbf{G}_a^\sharp \to \mathbf{G}_a \to \mathbf{G}_a^{\dR} \simeq F_* W/p \to 0\right)\]
    along the map $F_* \mathcal{I}^\univ \to F_* [\mathcal{O}(-p)]/p \simeq \mathcal{O}(-1) \otimes_{\mathbf{G}_a} F_*W/p$ induced by $y^\univ$.

    \item {\em The two open copies of $\Z_p^\Prism$ inside $\Z_p^\cN$:}
    \label{reinterpret-two-open}The de Rham open copy $j_{\dR}:\Z_p^\Prism \to \Z_p^\cN$ from \cite[Construction 5.3.5]{BBFgaugenotes} is simply the open locus $(\Z_p^\cN)_{t^\univ \neq 0},$ as the Rees map resulting from Proposition~\ref{prop:ZpNygDesc} is the usual Rees map for $\Z_p^\cN$. On the other hand, the Hodge--Tate open copy $j_{\HT}:\Z_p^\Prism \to \Z_p^\cN$ from \cite[Construction 5.3.2]{BBFgaugenotes}, which is the locus where the universal
      admissible module $M^{\univ}$ is invertible, is the open substack $(\Z_p^\cN)_{y^{\univ} \neq 0}$: indeed, it follows from (2) above that $y^\univ$ exactly classifies the failure of the universal admissible module to be invertible. For future use, let us give an explicit construction of this isomorphism in both directions. Fix a test ring $R$. 
      
      For the forward direction, given a test ring $R$, fix a point $(d:I \to W(R)) \in \Z_p^\Prism(R)$. Denoting base change along the projection $W(R) \to R$ with a subscript of $0$, the point $j_{\HT}(d) \in \Z_p^\cN(R)$ is given by the triple 
      \[ \Big(F^*d:F^*I \to W(R), t=d_0:I_0 \to R, F^*I/p \stackrel{y}{\simeq} [(I_0)^{\otimes p}]/p \xrightarrow{[t]^p} W(R)/p\Big) \]
      under the divisorial description, where the isomorphism $y$ in the third entry is the effect on invertible modules of the identity $F(x) \equiv x_0^p \mod pW(R)$ for $x \in W(R)$ (and can also be seen more formally from the square just below).
      
    Conversely, since $FV=p$ on $W$, we have a fibre square 
    \[ \xymatrix{W \ar[r]^-F \ar[d] & F_*W \ar[d] \\
     W/VW  \simeq \mathbf{G}_a \ar[r]^-{\overline{[\cdot]^p}} & F_*W/p }\]
    of animated $W$-algebra stacks: indeed, the induced map on vertical fibres is the tautological isomorphism $VW \simeq F_*W$. Applying $\mathrm{Vect}(-)$ to the value of this square on our test ring $R$ yields a fibre square of categories by Milnor glueing \cite[Theorem 16.2.0.2]{LurieSAG}. In particular, the groupoid $\Z_p^\Prism(R)$ of Cartier--Witt divisors $I \to W(R)$ is identified, via pullback along the maps appearing above, with the groupoid of triples $(d:J \to W(R),\tau:L\to R,\xi)$, where $d:J \to W(R)$ is a Cartier--Witt divisor, $\tau:L \to R$ is a generalized Cartier divisor, and
    \[  \xi: \Big(J/p \to W(R)/p\Big)  \simeq \Big(\overline{[L]^p} \to W(R)/p\Big)\]
    is an identification of the generalized Cartier divisors over $W(R)/p$; this latter groupoid is exactly $(\Z_p^\cN)_{y^{\univ}\neq0}(R)$, so $\Z_p^\Prism \simeq (\Z_p^\cN)_{y^{\univ}\neq0}$ via this construction.

\begin{remark}
      The argument in the last paragraph above is particularly transparent when all line bundles appearing are trivial. Indeed, it is tantamount to the following elementary fact: if $R$ is a test ring and $d \in W(R)$ is an element of the form $d=[t]^p + pz$ for some $t \in R $ and $z \in W(R)$, then $d$ has a preferred Frobenius root given by
      $[t] + Vz$.
\end{remark}

    \begin{rem}
    Our choice of the $W(k)$-structure on $W(k)^\cN$ (see Footnote~\ref{fn:wklinearity}) ensures that the composition $W(k)^\Prism \xrightarrow{j_{\dR}} W(k)^\cN \to W(k)^\Prism$ is the identity (and hence $W(k)$-linear), while the composition $W(k)^\Prism \xrightarrow{j_{\HT}} W(k)^\cN \to W(k)^\Prism$ is the Frobenius (and hence $\varphi_{W(k)}$-linear).
    \end{rem}

    \item {\em The Hodge--Tate divisor:} The divisor $(\Z_p^\cN)_{t^{\univ}=0} \subset \Z_p^\cN$ is identified with the total space of $\mathcal{O}(-1) \otimes_{\mathbf{G}_a} \mathbf{G}_a/\mathbf{G}_a^\sharp = \mathbf{V}(\mathcal{O}(-1))/\mathbf{V}(\mathcal{O}(-1))^\sharp$ over $B\mathbf{G}_m$. This was proven directly in \cite[Proposition 5.3.7]{BBFgaugenotes}. Let us deduce from the divisorial description. In the latter, given a test $p$-nilpotent ring $R$,  a point of $(\Z_p^\cN)_{t^{\univ}=0}(R)$  is given by a Cartier--Witt divisor $d:I \to W(R)$, a line bundle $\mathcal{O}(-1) \in \mathrm{Pic}(R)$,  a $W/p$-module map $y:I/p \to [\mathcal{O}(-p)]/p$, and a trivialization\footnote{A ``trivialization'' here means an identification of the mod $p$ reduction of $d$ with $0$ in groupoid of $W(R)/p$-module maps $I/p \to W(R)/p$.} $h$ of the map $I/p \to W(R)/p$ induced by $d$. Note that $y$ does not interact with $d$ or $h$ in this description. The trivialization $h$ yields  a map $\left(I \xrightarrow{d} W(R)\right) \to \left(W(R) \xrightarrow{p} W(R)\right)$ of generalized Cartier divisors in $W(R)$, which must be an isomorphism by the irreducibility lemma for distinguished elements (see \cite[Lemma 3.5]{bhatt2019prisms} or \cite[Lem.~5.1.5]{BBFgaugenotes}). Conversely, any such isomorphism certainly yields a trivialization $h$ of $d$ modulo $p$. Thus, we learn that $(\Z_p^\cN)_{t^{\univ}=0}(R)$ is simply the groupoid of line bundles $\mathcal{O}(-1) \in \mathrm{Pic}(R)$ equipped with a section of $[\mathcal{O}(-p)]/p$. The groupoid of such pairs is exactly $\mathcal{O}(-1) \otimes_{\mathbf{G}_a} \mathbf{G}_a/\mathbf{G}_a^\sharp$ over $B\mathbf{G}_m$, as wanted.

\item {\em The filtered de Rham point:}  Consider the space $(\Z_p^\cN)_{y^\univ=0}$ of trivializations of $y^\univ$, regarded as a section of an invertible $W/p$-module over $\Z_p^\cN$. As $W/p$ is a sheaf of animated rings, the structure map $(\Z_p^\cN)_{y^\univ=0} \to \Z_p^\cN$ is not a closed immersion\footnote{In fact, locally on $\Z_p^\cN$, the map $(\Z_p^\cN)_{y^\univ=0} \to \Z_p^\cN$ is  the base change of the $0$-section $\mathrm{Spf}(\mathbb{Z}_p) \xrightarrow{0} W/p$. As $W/p \simeq \G_a \times B\G_a^\sharp$ in characteristic $p$, the fibres of $(\Z_p^\cN)_{y^\univ=0} \to \Z_p^\cN$ over geometric points are either empty or given by $\G_a^\sharp$.}. This morphism turns out to be an important one for cohomological purposes:

\begin{cor}[The filtered de Rham point]
\label{cor:filtdRlocus}
The filtered de Rham map $i_{\dR}:\mathbf{A}^1/\mathbf{G}_m \to \Z_p^\cN$ from \cite[Construction 5.3.4]{BBFgaugenotes} is  identified with the structure map $(\Z_p^\cN)_{y^\univ=0} \to \Z_p^\cN$.
\end{cor}

\begin{proof}
Note that the pullback of the universal admissible module along $i_{\dR}$ is the split extension $F_* W \oplus \mathbf{V}(\mathcal{O}(-1))^\sharp$ by construction. As $y^\univ$ classifies the extension determined by the universal admissible module, it follows that the map $i_{\dR}$ factors naturally as $\mathbf{A}^1/\mathbf{G}_m \to  (\Z_p^\cN)_{y^\univ=0} \to \Z_p^\cN$. We must show the first map is an isomorphism. As $i_{\dR}$ has a left inverse given by Rees map, it suffices to show that the
composition $(\Z_p^\cN)_{y^\univ=0} \to \Z_p^\cN \to \mathbf{A}^1/\mathbf{G}_m$ with the Rees map is an isomorphism. Given a test $p$-nilpotent ring $R$, the fibre $\mathcal{F}$ of $(\Z_p^\cN)_{y^\univ=0}(R) \to \mathbf{A}^1/\mathbf{G}_m(R)$ over a point $\left(t:\mathcal{O}(-1) \to R\right) \in \mathbf{A}^1/\mathbf{G}_m(R)$ is the groupoid of Cartier--Witt divisors $d:I \to W(R)$ equipped with a trivialization of $I/p \to W(R)/p$. As in (4) above, by the irreducibility lemma for distinguished elements, the trivialization promotes uniquely to an identification $\left(d:I \to W(R)\right) \simeq \left(p:W(R) \to W(R)\right)$ of Cartier--Witt divisors, so $\mathcal{F}$ is contractible, as wanted.
\end{proof}

\end{enumerate}

\subsubsection{A new chart of $W(k)^\cN$}
\label{ss:chartsNygaard}
Let us use the divisorial description to construct some charts for $W(k)^\cN$. First, we reconstruct a standard chart attached to any transversal prism, see \cite[Remark 5.5.19]{BBFgaugenotes}.

\begin{example}[Standard charts of $\Z_p^\cN$ from prisms]
\label{ex:ReesChartZpNyg} 
Let $(A, (d))$ be an oriented prism  with $A/(d)$ a $p$-torsion-free $W(k)$-algebra. Consider the $(p,d)$-completed Rees stack $X_{(A,(d))}$ of the $(d)$-adic filtration on $A$, so $X_{(A,(d))}=\mathrm{Spf}(R_{(A,(d))})/\mathbf{G}_m$, where 
\[ R_{(A,(d))} \coloneq  \mathrm{Rees}(d^\bullet A)^{\wedge} = A[x,t]^{\wedge}/(xt-d), \]
where the topology and completion are $(p,d)$-adic, the element $t$ has degree $1$ and $x$ has degree $-1$. Note that $X_{(A,(d))}$ has a natural map to $\Spf(A)$ and hence a $W(k)$-structure. We will construct a faithfully flat $\varphi_{W(k)}$-linear map
\[ \rho^\cN_{(A,(d))}:X_{(A,(d))} \to W(k)^\cN\]
using the divisorial description. Since our prism is oriented, all line bundles that appear will be naturally trivial, up to grading shifts.

Via the isomorphism $W(k)^\cN \simeq \mathrm{Spf}(W(k)) \times \Z_p^{\cN}$, it suffices to specify the maps to each factor separately. The first component map  $X_{(A,(d))} \to \Spf(W(k))$ is uniquely determined by demanding it be $\varphi_{W(k)}$: it must agree with the postcomposition of the canonical map $X_{(A,(d))} \to \Spf(A) \to \Spf(W(k))$ with $\varphi_{W(k)}$. The divisorial components of the second component $X_{(A,(d))} \to \Z_p^\cN$ are given by:
\begin{itemize}[wide]

    \item {\em The underlying Cartier--Witt divisor:} This is given by the image of the distinguished element $w(\varphi(d)) \in W(A)$ under the natural map $W(A) \to W(R_{(A,(d))})$.

    \item {\em The underlying Cartier divisor:} This is simply given by the element $t \in R_{(A,(d))}$.

    \item {\em The factorization:}   Note that $w(d) = [d] + Vd_1$ for some $d_1 \in W(A)$, so $w(\varphi(d)) = F(w(d)) = [d]^p + pd_1$. Consequently, $w(\varphi(d)) \equiv [d]^p \mod pW(A)$ and hence the same holds for its image in $W(R_{(A,(d))})$; taking $z$ to be the image of $[x]^p \in W(R_{(A,(d))})$ then provides the needed factorization since $[x]^p [t]^p = [d]^p$ in $W(R_{(A,(d))})$.
\end{itemize}
By construction, $\rho_{(A,(d))}^\cN$ restricts to $\varphi \circ \rho_{(A,(d))}$  on the de Rham open point (recall that~$\rho_{(A,(d))}$ was defined in~\eqref{eqn:rho-B-d}); in fact, this desideratum dictated the choice of the map to $\Spf(W(k))$ made above. By inspection, one also checks $\rho^\cN_{A,(d)}$ recovers $\rho_{(A,(d))}$ at the Hodge--Tate open point. Using this, as well as the $p$-torsion-freeness of source and target of $\rho_{(A,(d))}^\cN$, one checks that  $\rho_{(A,(d))}^\cN$  is faithfully flat by checking it on the stratification of $(\Z_p^\cN)_{p=0}$ with strata given by the de Rham open point, the Hodge--Tate locus of the Hodge--Tate open point, and the Hodge point; we omit the details here. 
\end{example}

As the special case of the Breuil--Kisin prism is important for the sequel, let us give it a short name: write
\begin{equation*}
\label{eq:BKRees}
X = X_{(\KisinS,(E))} \quad \text{and} \quad \left(g:X \to W(k)^\cN\right)  = \rho^\cN_{(\KisinS,(E))}
\end{equation*}
for the resulting $\varphi_{W(k)}$-linear map. It turns out that  one can obtain a ``more efficient''\footnote{More precisely, the chart $g$ realizes $\varphi \circ \rho_{\text{std}}$ over the de Rham open point $W(k)^\Prism \subset W(k)^\cN$, while the chart $f:Y \to W(k)^\cN$ in Example~\ref{ex:ptwistedReesChartBK} realizes $\rho_{\text{std}}$ itself.} chart for $W(k)^\cN$ than the one provided by $g$, and this is critical for our applications.

\begin{example}[A ``more efficient'' approximation of $\Z_p^\cN$]
\label{ex:ptwistedReesChartBK}
Let $Y$ be the $(p,u)$-completed Rees stack of the ``$u$ in degree $p$'' filtration on $\KisinS$ (which is also the $\varphi$-preimage of the $u$-adic filtration on $\KisinS$); see \S \ref{ss:twistedrees} for more on such stacks. Explicitly, $Y=\mathrm{Spf}(R)/\mathbf{G}_m$, where
\begin{equation}
    \label{eq:coordringY}
R = \mathrm{Rees}((u)^{\lceil \bullet/p \rceil} \KisinS)^{\wedge} = \KisinS[y,t]^{\wedge}/(yt^p-u),
\end{equation}
where the completion/topology is $(p,u)$-adic, the Rees parameter $t$ has degree $1$ and the element $y$ has degree $-p$ (and corresponds to $u$ living in filtration degree $p$). There are natural structure maps $Y \to \Spf(\KisinS) \to \Spf(W(k))$. We now construct a $W(k)$-linear map
\[ f:Y \to W(k)^\cN\]
using the divisorial description of the target, and establish some of its properties. Again, by the $W(k)$-linearity requirement, it suffices to specify the divisorial components of the induced map to $\Z_p^\cN$, and these are given by:
\begin{itemize}[wide]
    \item {\em The underlying Cartier--Witt divisor:} This is given by the image of the distinguished element $w(E) = [u] - pw(z) \in W(A)$ under the natural map $W(A) \to W(R)$.

    \item {\em The underlying Cartier divisor:} This is simply given by the element $t \in R$.

    \item {\em The factorization:} Note that $w(E) \equiv u \mod pW(A)$ by design. Consequently, the factorization $yt^p = u$ in $R$ yields the desired factorization $[y][t]^p$ of $w(E)$ in $W(R)/p$.
\end{itemize}
By construction, the restriction of $f$ to the open substack $j_{\dR}:W(k)^\Prism \hookrightarrow W(k)^\cN$ agrees with the map $\rho_{\text{std}}$.
\end{example}

We now study the geometry of the map $f$ constructed above. Our first goal is to show it is faithfully flat; this will be deduced from its compatibility with $\rho_{\text{std}}^\cN$ in a suitable way. Since $E\equiv u \mod p\KisinS$, the Frobenius on $\KisinS/p$ carries the ``$u$ in degree $p$ filtration'' to the $u$-adic filtration on $\KisinS/p$, and hence induces a  natural map $X_{p=0} \to Y_{p=0}$ over $\mathbf{A}^1/\mathbf{G}_m$. Explicitly, on co-ordinate rings, this is induced by the graded  $\mathbf{F}_p[t]$-algebra map
\[ \KisinS/p[y,t]/(yt^p-u) \to \KisinS/p[x,t]/(xt-u)\]
determined by Frobenius on $\KisinS/p$ and $y \mapsto x^p$. This map is faithfully flat by inspection. Moreover, we have:

\begin{prop}
\label{prop:relatefandg}
The composition $X_{p=0} \to Y_{p=0} \xrightarrow{f_{p=0}} W(k)^\cN_{p=0}$ is simply the map $g_{p=0}$. In particular, $f$ is faithfully flat.
\end{prop}
\begin{proof}
The first part is clear from the construction in terms of divisorial components, while the second part follows from the first part by a suitable $2$-out-of-$3$ property for faithfully flat maps, and the $\Z_p$-flatness of the source and target of $f$.
\end{proof}

Secondly, we explain the relationship of $f$ to the map $\rho_\dagger$ from Proposition~\ref{prop:FrobDescMap}:

\begin{prop}
\label{prop:rhodaggerviaNygaard}
The restriction of $f$ to the open substack $j_{\HT}:W(k)^\Prism \hookrightarrow W(k)^\cN$ agrees with the map $\rho_\dagger$ from Proposition~\ref{prop:FrobDescMap}
\end{prop}
\begin{proof}
We shall construct a stack $Y_1$ with a $\mu_p$-action and a map $Y_1/\mu_p \to Y$ inducing an isomorphism over $j_{\HT}:W(k)^\Prism \hookrightarrow W(k)^\cN$; this reduces the question from $Y$ to $Y_1/\mu_p$, where it is more tractable.

Let
\[ R_1 = \mathrm{Rees}(u^\bullet \KisinS)^{\wedge} \coloneq  \KisinS[y_1,t]^{\wedge}/(y_1t-u),\]
where $\deg(y_1)=-1$ and the completion is $(p,u)$-adic. The Frobenius on $\KisinS$ then induces a morphism 
\[ R \coloneq  \KisinS[y,t]^{\wedge}/(yt^p-u) \to R_1 \coloneq  \KisinS[y_1,t]^{\wedge}/(y_1t-u) \]
of graded $\Z_p[t]$-algebras linear over $\varphi:\KisinS \to \KisinS$ and sending $y$ to $y_1^p$. Passing to the corresponding quotient stacks yields a $\varphi_{\KisinS}$-linear faithfully flat map
\[ h:Y_1 \to Y\]
of stacks over $\mathbf{A}^1/\mathbf{G}_m$. Moreover, the natural $\mu_p$-action on $\varphi:\Spf(\KisinS) \to \Spf(\KisinS)$ extends to a $\mu_p$-action on $Y_1$ over $Y$, where $u,y_1 \in R_1$ have weight $1$. The map $h_{t \neq 0}$ identifies with $\varphi:\Spf(\KisinS) \to \Spf(\KisinS)$. On the other hand, since $(Y_1)_{y_1 \neq 0} \simeq \Spf(\KisinS)$ via the projection map $Y_1 \to \Spf(\KisinS)$, the map $h_{y \neq 0}$ induces a map $\overline{h}:\Spf(\KisinS)/\mu_p \to Y_{y \neq 0}$ that one checks to be an isomorphism. These constructions are summarized in the following commutative diagram:
\[
\xymatrix{
\Spf(\KisinS)/\mu_p \simeq (Y_1)_{y_1  \neq 0} / \mu_p\ar[r]^-{\overline{h},\simeq} \ar@{^{(}->}[d] \ar@/_3pc/@{=}[dd]  & Y_{y \neq 0} \ar[r]^-{f_{y \neq 0}} \ar@{^{(}->}[d] & (W(k)^\cN)_{y \neq 0} \simeq W(k)^\Prism  \ar@{^{(}->}^-{j_{\HT}} [d]  \ar@/^3pc/[dd]^{\varphi}  \\
Y_1/\mu_p\ar[r] \ar[d] & Y \ar[r]^f \ar[d] & W(k)^\cN \ar[d]  \\
\Spf(\KisinS)/\mu_p \ar[r]^{q} & \Spf(\KisinS) \ar[r]^{\rho_{\text{std}}} & W(k)^\Prism
}
\]
where the squares relating the first two rows are Cartesian and have open immersions as  vertical maps, while the vertical maps relating the second and third rows are the structure maps. Our task is then to identify the map
\[ \Spf(\KisinS)/\mu_p \xrightarrow{\overline{h}} Y_{y \neq 0} \xrightarrow{f_{y \neq 0}} W(k)^\Prism\]
with $\rho_\dagger$.  We first observe that the above commutative diagram gives a natural identification $\varphi \circ f_{y \neq 0} \circ \overline{h} \simeq \rho_{\text{std}} \circ q$ (as in Proposition~\ref{prop:FrobDescMap} (2)): both sides correspond to paths from the top left to the bottom right in the diagram above. This compatibility also shows that $f_{y \neq 0} \circ \overline{h}$ is $W(k)$-linear. It is now a pleasant exercise in unwinding the isomorphism $W(k)^\Prism \simeq (W(k)^\cN)_{y^\univ \neq 0}$ (see item (\ref{reinterpret-two-open}) at the start of \S \ref{sss:ReinterpretClassicalDivisorial}) to show that $f_{y \neq 0} \circ \overline{h}$ agrees with $\rho_{\dagger}$.
\end{proof}

\begin{remark}
\label{rmk:GeomDescCanModKisinS}
The proof above shows that the Frobenius on $\Spf(\KisinS)$ induces an isomorphism
\[ \Spf(\KisinS)/\mu_p \simeq Y_{y \neq 0}.\]
We can therefore rewrite the conclusion of Proposition~\ref{prop:CohSheafptwistedRees} (in the current special case) as follows:  the category of reflexive coherent sheaves on $Y$ identifies with the category of triples $(\gM,\gN,\psi)$, where $\gM$ is a vector bundle on $\Spf(\KisinS)$, $\gN$ is a $\mu_p$-equivariant vector bundle on $\Spf(\KisinS)$, and $\psi:\varphi^* \gM[1/u] \simeq \gN[1/u]$ is a $\mu_p$-equivariant isomorphism. (In fact, we are simply imitating what we said in Remark~\ref{rmk:GeomIntCanMod},
replacing $K$ with $W$ and the relative Frobenius $\varphi_{/K}$ with the absolute Frobenius $\varphi$ on $\KisinS$.)
\end{remark}

\subsection{The syntomic stack and Galois representations}
\label{ss:synstack}
The syntomic stack $W(k)^{\Syn}$ from \cite{drinfeldprismatization,BBFgaugenotes} is obtained by gluing the two disjoint opens $j_{\HT},j_{\dR}: W(k)^\Prism \to W(k)^\cN$ in $W(k)^\cN$ to each other via the identity. Coherent sheaves on $W(k)^{\Syn}$ are closely related to coherent sheaves on $W(k)^\Prism$ equipped with a Frobenius structure, and the latter are closely related to $\Gal_{K}$-representations. More precisely, putting together some recent results, we obtain the following geometric perspective on crystalline Galois representations that will be critical in this paper:

\begin{theorem}[Crystalline Galois representations via prismatic $F$-crystals]
\label{thm:CrysGalPrismCrys}
    The following categories are naturally equivalent:
    \begin{enumerate}
        \item The category of reflexive coherent sheaves on $W(k)^{\Syn}$.

        \item The category $\mathrm{Vect}^{\varphi}(W(k)^\Prism)$ of pairs $(\mathcal{E},\varphi_{\mathcal{E}})$, where $\mathcal{E}$ is a vector bundle on $W(k)^\Prism$, and $\varphi_{\mathcal{E}}$ is an  $I_{\HT}$-isogeny between $\varphi^* \mathcal{E}$ and $\mathcal{E}$, i.e., an identification $\{I_{\HT}^{\otimes -n} \varphi^* \mathcal{E}\}_{n \geq 0} \simeq \{I_{\HT}^{\otimes -n} \otimes \mathcal{E}\}_{n \geq 0}$ of ind-objects.

        \item The category $\mathrm{Vect}^{\varphi}(W(k)_\Prism)$ of prismatic $F$-crystals in the sense of \cite{bhatt2021prismatic}.

        \item The category of crystalline representations of $\Gal_{K}$ on finite free $\mathbf{Z}_p$-modules.
    \end{enumerate}
\end{theorem}

\begin{proof}
    The equivalence of (3) and (4) is the main result of \cite{bhatt2021prismatic}, while the equivalence of (2) and (3) follows from the identification of crystals of vector bundles on prismatic site $W(k)_\Prism$ with vector bundles on the stack $W(k)^\Prism$, see \cite[Proposition 8.15]{bhatt2022prismatization}. The equivalence of (1) and (2) is explained in \cite[Theorem 6.6.13]{BBFgaugenotes}.
\end{proof}

\begin{remark} We remark that the equivalence between (2) and (3) in Theorem~\ref{thm:CrysGalPrismCrys} is realized by the
construction of \ref{para:ptsofprism}.
In the sequel, we shall often abuse notation and refer to objects of the category appearing in either (2) or (3) as prismatic $F$-crystals.
\end{remark}

\begin{remark}
The functors in the directions $(1) \Rightarrow (2) \Leftrightarrow (3) \Rightarrow (4)$ are essentially all obtained by localization and are exact for the evident exact structures on each of the categories. On the other hand, the functors in the direction $(4) \Rightarrow (3)$ and $(2) \Rightarrow (1)$ involve $*$-extensions of vector bundles, and are not exact.
\end{remark}

\begin{example}[Breuil--Kisin modules from crystalline Galois representations]\label{ex:BK-from-crystalline}
Fix a crystalline Galois representation $\rho:\Gal_{K} \to \mathrm{GL}_n(\mathbf{Z}_p)$. By Theorem~\ref{thm:CrysGalPrismCrys}, this corresponds to a rank $n$ prismatic $F$-crystal $(\mathcal{E},\varphi_{\mathcal{E}})$.  Consider the map $\rho_{\KisinS,(E)}:\mathrm{Spf}(\KisinS) \to W(k)^\Prism$ coming from the Breuil--Kisin prism. As this map is compatible with Frobenius, pullback yields a Breuil--Kisin module $(\gM,\varphi_\gM) \coloneq  \rho_{\KisinS,(E)}^*(\mathcal{E},\varphi_{\mathcal{E}})$ of rank $n$. In fact, this is simply the classical Breuil--Kisin module attached to $\rho$ in \cite{KisinCrys}.
\end{example}

Let us end this subsection by giving a proof of the Frobenius-descent property promised in \S \ref{ss:canmodthms}.

\begin{proof}[Proof of Proposition~\ref{cor:FrobdescentfiltBKmod}]
We shall use the Rees stack charts $g:X \to W(k)^\cN$ and $f:Y \to W(k)^\cN$ constructed in \S \ref{ss:chartsNygaard}. Let $(\mathcal{E},\varphi_{\mathcal{E}})$ denote the prismatic $F$-crystal in vector bundles lifting $(\gM,\varphi_\gM)$ via Theorem~\ref{thm:CrysGalPrismCrys}. This $F$-crystal determines two coherent sheaves on $(W(k)^\cN)_{p=0}$: either first take a reflexive hull (see \cite[\S 6.6]{BBFgaugenotes}) and then reduce modulo $p$, or do the operations in the reverse order. Pulling back the preceding two coherent sheaves along the map  $g_{p=0}:X_{p=0} \to (W(k)^\cN)_{p=0}$, and translating back from quasi-coherent sheaves on Rees stacks to filtered modules, then yields the two filtered objects considered in the Proposition. But $g_{p=0}$ factors as
\[ g_{p=0}:X_{p=0} \to Y_{p=0} \xrightarrow{f_{p=0}} (W(k)^\cN)_{p=0} \]
by Proposition~\ref{prop:relatefandg}. As $X_{p=0} \to Y_{p=0}$ is exactly the map on Rees stacks attached to the Frobenius $(u)^{\lceil \bullet/p \rceil} \KisinS/p \to (u)^\bullet \KisinS/p$ on filtered rings, one obtains the desired descent.
\end{proof}

\subsection{Crystalline Breuil--Kisin modules are specializable}\label{ss:PfCanModThm}

The goal of this subsection is to prove our main result (Theorem~\ref{thm:crysBKspecial}) on specializability of crystalline Breuil--Kisin modules $(\gM,\varphi_{\gM})$. We shall use the ingredients discussed in the previous subsections. Let us sketch the key ideas informally right away. First, using the map $f$ from Example~\ref{ex:ptwistedReesChartBK}, we construct an alternate candidate $(\gN,\psi)$ for the canonical modification $(\gN^{\can},\psi^{\can})$ from \S \ref{sss:canmupmod}; the pair $(\gN,\psi)$ visibly satisfies the specializability constraint that $\overline{\gN}=\overline{\gM}$ as submodules of $\varphi^* \overline{\gM}[1/u]$. To finish, we need to identify the modifications $(\gN^{\can},\psi^{\can})$ and $(\gN,\psi)$ of $\varphi^* \gM$. By the characterization in Proposition~\ref{prop:canmodBK}, this amounts to understanding the behaviour of $(\gN,\psi)$ after base change along $W(k) \to K$, i.e., after allowing some denominators. The latter is accomplished by a geometric argument that, roughly, requires understanding the ``crystallization'' of the map $f$ used to construct $(\gN,\psi)$.  To carry out the last step, we shall need the following abstract lemma:

\begin{lemma}
\label{lem:twistedPDenvelope}
Let $R$ be a $p$-nilpotent animated commutative ring and $y \in R$ be an element. Then the pullback of
\[
\Spec(R) \xrightarrow{[y]} W/p \xleftarrow{0} \operatorname{Spf}(\mathbb{Z}_p)
\]
is $\Spec(S)$, where $S$ is obtained by freely adjoining $y/p$ and its divided powers to $R$.
\end{lemma}

\begin{proof}
Write $A = \mathbb{Z}_p[y]^{\wedge}$ for the free $p$-complete $\delta$-ring on a rank $1$ element $y$. The map $\Spec(R) \xrightarrow{[y]} W/p$ classifying $[y] \in W(R)/p$ factors as
\[
\Spec(R) \longrightarrow \Spec(A) \longrightarrow W \longrightarrow W/p,
\]
where the first map is the one determined by sending $y \in A$ to $y \in R$, the second map is a $\delta$-map corresponding to the element $y \in A$ (where we recall that the coordinate ring of $W$ is the free $p$-complete $\delta$-ring on one generator), and the last map is the obvious one.  Pullback along $\operatorname{Spf}(\mathbb{Z}_p) \xrightarrow{0} W/p$ gives a commutative diagram
\[
\xymatrix{
\Spec(S) \ar[r] \ar[d] & \Spf(B) \ar[r] \ar[d] & W \ar[r] \ar[d]^{p} & \operatorname{Spf}(\mathbb{Z}_p) \ar[d]^{0} \\
\Spec(R) \ar[r] & \Spf(A) \ar[r] & W \ar[r] & W/p
}
\]
where the third vertical map is multiplication by $p$, and all squares are pullbacks (so $B$ and $S$ are defined this way).  It suffices to show that $B$ is obtained from $A$ by freely adjoining $y/p$ and its divided powers, i.e., $B$ is the $p$-complete PD-polynomial ring on $y/p$. But the middle vertical square is a fibre square in $p$-complete $\delta$-schemes, 
so (from the $\delta$-freeness of the coordinate ring of $W$) we learn that $B = A\{y/p\}$, i.e., $B$ is obtained by freely adjoining $y/p$ in $p$-complete $\delta$-rings to $A$. It remains to show that this is the $p$-completion of the subring of $A[1/p]$ generated by $y/p$ and its divided powers.

Consider the $\mu_p$-cover $A \to A' = \mathbb{Z}_p[y^{1/p}]^{\wedge}$, where $A'$ is endowed with the unique $\delta$-structure where $y^{1/p}$ has rank $1$. While this map is not a $\mu_p$-torsor, we still have $A = (A')_{\deg=0}$, where degrees are taken with respect to the $\mathbf{Z}/p$-grading induced by the $\mu_p$-action. Then \cite[Corollary 2.39]{bhatt2019prisms} shows that $A'\{y/p\}$ is exactly the $p$-complete PD-polynomial ring on $y^{1/p}$, so $A\{y/p\}$ is the $p$-complete graded subring spanned by the terms of degree divisible by $p$, which identifies with the $p$-complete PD-polynomial ring on $y/p$ by the Legendre formula for valuation of factorials, as wanted.
\end{proof}

We can now give the proof of Theorem~\ref{thm:crysBKspecial}.
\begin{proof}[Proof of Theorem~\ref{thm:crysBKspecial}]
Fix a crystalline Breuil--Kisin module $(\gM,\varphi_\gM)$; this necessarily arises from a unique reflexive coherent sheaf $\mathcal{E}$ on $W(k)^{\Syn}$ (Theorem~\ref{thm:CrysGalPrismCrys}). We shall give an alternate description of $(\gN^{\can},\psi^{\can})$ in terms of $\mathcal{E}$; in this alternate description, the specializability will become obvious.

Write $\mathcal{E}^\cN$ and $\mathcal{E}^\Prism$ for the restriction of $\mathcal{E}$ to $W(k)^\cN$ and $W(k)^\Prism$ respectively.

First, let us recall how to reconstruct $(\gM,\varphi_\gM)$ from $\mathcal{E}$. We have the faithfully flat map $g:X \to W(k)^\cN$ from Example~\ref{ex:ReesChartZpNyg}. Via pullback, $g^* \mathcal{E}^\cN$ gives a reflexive coherent sheaf on $X$. Such sheaves identify with triples $(\gM_1,\gM_2,\tau)$, where $\gM_i \in \mathrm{Vect}(\KisinS)$ and $\tau:\gM_1[1/E] \simeq \gM_2[1/E]$ is an isomorphism (Proposition~\ref{prop:CohSheafRees}). The descent of the sheaf $\mathcal{E}^\cN$ on $W(k)^\cN$ to the sheaf $\mathcal{E}$ on $W(k)^{\Syn}$ then additionally supplies an identification $\varphi^* \gM_2 \simeq \gM_1$. Taking $\gM=\gM_2$ thus yields the Breuil--Kisin module attached to $\mathcal{E}$. 

Next, consider the faithfully flat map $f:Y \to W(k)^\cN$ from Example~\ref{ex:ptwistedReesChartBK}. Again, pullback gives a reflexive coherent sheaf $f^* \mathcal{E}^\cN$ on $Y$. Such sheaves can be identified with triples $(\gM,\gN,\psi)$, where $\gM$ is a finite projective $\KisinS$-module, $\gN$ is a finite projective $\mu_p$-equivariant $\KisinS$-module, and $\psi:\varphi^* \gM[1/u] \simeq \gN[1/u]$ is a $\mu_p$-equivariant isomorphism (see Proposition~\ref{prop:CohSheafptwistedRees} and Remark~\ref{rmk:GeomDescCanModKisinS}). In particular, we obtain a candidate pair $(\gN,\psi)$. We then claim:

\begin{itemize}
    \item[$(\ast)$] There is a natural identification $(\gN,\psi) \simeq (\gN^{\can},\psi^{\can})$.
\end{itemize}

Note that $(\ast)$ implies the theorem: to show that the modifications $\overline{\psi}$ and $\overline{\varphi_\gM}$ of $\overline{\varphi^* \gM}$ agree, one combines the compatibility in the first sentence of Proposition~\ref{prop:relatefandg} with the observation that the map $X_{p=0} \to Y_{p=0}$ induces the tautological map $\Spf(\KisinS/p) \to \Spf(\KisinS/p)/\mu_p$ over the open $Y_{p=0,y\neq0} \subset Y_{p=0}$.

It remains to prove $(\ast)$. By the characterization given in Proposition~\ref{prop:canmodBK}, it suffices to construct such an identification after base change to $K$, i.e., to match the modifications of $\varphi^* \gM \otimes_\KisinS K\llbracket u \rrbracket$ determined by $\psi$ and $\psi^{\can}$. Now Remark~\ref{rmk:GeomIntCanMod} describes the target in geometric terms, while the former was already described in geometric terms. We will finish the proof by matching these geometric descriptions.

\begin{enumerate}[wide]
    \item {\em Description of the vector bundle encoding $\psi^{\can}$ via $\mathcal{E}^\cN$:} We reformulate Remark~\ref{rmk:GeomIntCanMod}, giving the description of the modification of $\varphi^* \gM \otimes_{\KisinS} K\llbracket u \rrbracket$ arising from $\psi^{\can}$, in terms of the gauge $\mathcal{E}^\cN$.  Recall from Remark~\ref{rmk:GeomIntCanMod} that $R_K$ is the graded $K\llbracket u\rrbracket$-algebra $K\llbracket u \rrbracket[y,t]/(yt^p-u)$. Consider the diagram of natural maps
\[ \xymatrix{ \Spec(R_K)/\mathbf{G}_m \ar[r]^-{r_K} & \mathbf{A}^1_{\mathrm{Spec}(W(k))}/\mathbf{G}_m \coloneq  \mathrm{Spec}(W(k)[t])/\mathbf{G}_m & \\
 W(k)^\cN & \mathbf{A}^1_{\mathrm{Spf}(W(k))}/\mathbf{G}_m \coloneq  \mathrm{Spf}(W(k)[t]^{\wedge})/\mathbf{G}_m \ar[u] \ar[l]^-{i_{\dR}} } \]
 where the completion on the bottom right is $p$-adic, the right vertical map is the tautological map, and $i_{\dR}$ is the filtered de Rham map. Pullback along the right vertical map  induces an equivalence on $\text{Perf}(-)$; write $(-)^{\alg}$ for its inverse.  On the other hand, $(i_{\dR})^* \mathcal{E}^\cN$ is exactly the Rees module of the Hodge filtration on (the integral form of) $D$. Consequently, the vector bundle on  $\Spec(R_K)/\mathbf{G}_m$ attached to the modification $\psi^{\can}$ is simply $r_K^* \left( (i_{\dR})^* \mathcal{E}^\cN\right)^{\alg}$.

\item {\em Description of the vector bundle encoding $\psi$ via $\mathcal{E}^\cN$:} Recall that $Y=\mathrm{Spf}(R)/\mathbf{G}_m$ where $R=\KisinS[y,t]^{\wedge}/(yt^p-u)$ with the completion and the topology being $(p,u)$-adic. Write $R^{\alg} = \KisinS[y,t]/(yt^p-u)$ for the obvious algebraization, so there is a natural map $R^{\alg} \to R_K$ of graded rings. This gives a diagram of stacks
\[ Y \coloneq  \mathrm{Spf}(R)/\mathbf{G}_m \to \mathrm{Spec}(R^{\alg})/\mathbf{G}_m \xleftarrow{j} \mathrm{Spec}(R_K)/\mathbf{G}_m \]
where the first map is the obvious $(p,u)$-completion map; pullback along this map induces an equivalence on $\text{Perf}(-)$, and we again write $(-)^{\alg}$ for its inverse. The modification arising from $\psi$, realized as a vector bundle on $\mathrm{Spec}(R_K)/\mathbf{G}_m$, is then simply $j^* (f^* \mathcal{E}^\cN)^{\alg}$.

\item {\em The key compatibility:} To relate the above descriptions, we need to understand the relationship of $f$ and $i_{\dR}$, so let us compute their fibre product.
Consider the ring
\[ R' = R[ \{\gamma_n(y/p)\}_{n \geq 1} ]^{\wedge} = W(k)\langle y/p \rangle[t]^{\wedge},\]
obtained by adjoining $\frac{y}{p}$ and its divided powers to $R$ in the $p$-complete world; note that $R'$ is $(p,u)$-complete since $u=yt^p = p(y/p)t^p$.  As $y$ is homogeneous, $R'$ inherits a grading from $R$ in the $p$-complete sense. Set $Y'=\mathrm{Spf}(R')/\mathbf{G}_m$, so there is a natural map $Y' \to Y$.  We claim this fits into a Cartesian diagram
\begin{equation}\label{eq:cuixal7nch} \xymatrix{ Y' \ar[r] \ar[d] & Y \ar[d]^-f \\
    \mathbf{A}^1/\mathbf{G}_m \ar[r]^{i_{\dR}} & W(k)^\cN }\end{equation}
where the left vertical map is the Rees map, and the bottom horizontal map is the filtered de Rham point. Indeed, since the bottom horizontal map is the space of trivializations of the section $y^{\univ}$ of the invertible $W/p$-module $(\mathcal{I}_W^{\univ})^{\otimes -1} \otimes_{W} [\mathcal{O}(-p)]/p$  over $W(k)^\cN$ by Corollary~\ref{cor:filtdRlocus}, the claim follows from Lemma~\ref{lem:twistedPDenvelope}.

\item {\em End of proof:} To simplify life and avoid divided powers we pass to the rational localization where $y$ becomes divisible by $p^2$, i.e.\ we consider the ring $S = W(k)[y_2,t]^{\wedge}$, where  the completion is $(p,y_2 t)$-adic. Regard $S$ as an $R'$-algebra (and thus also as an $R$-algebra) via $y/p = p y_2 \in pS$ (so $y_2 = y/p^2$); let $S^{\alg} = W(k)\llbracket u_2 \rrbracket[y_2,t]/(y_2t^p - u_2)$, so the $(p,u_2)$-adic completion of $S^{\alg}$ gives $S$. There is a natural grading on $S^{\alg}$ where $\deg(t)=1$ and $\deg(y_2)=-p$; this induces a grading on $S$ in the $(p,u_2)$-complete sense. Observe that the map $j:\mathrm{Spec}(R_K)/\mathbf{G}_m \to \mathrm{Spec}(R^{\alg})/\mathbf{G}_m$ factors naturally over the map $\mathrm{Spec}(S^{\alg})/\mathbf{G}_m \to \mathrm{Spec}(R^{\alg})/\mathbf{G}_m$: indeed, we have an obvious factorization
\begingroup
\setlength{\arraycolsep}{2pt}
\[
\begin{array}{c}
  R^{\alg} \coloneq  \\
  W(k)\llbracket u \rrbracket[y,t]/(yt^p-u)
\end{array}
\longrightarrow
\begin{array}{c}
  S^{\alg} \coloneq  \\
  W(k)\llbracket u_2 \rrbracket[y_2,t]/(y_2t^p - u_2)
\end{array}
\longrightarrow
\begin{array}{c}
  R_K \coloneq  \\
  K\llbracket u \rrbracket[y,t]/(yt^p-u)
\end{array}
\]
\endgroup
of maps of graded $W(k)$-algebras where the maps are determined by $u = u_2 p^2$ and $y = y_2 p^2$. This factorization, combined with our previous constructions, leads to the following commutative diagram
\[ \xymatrix{ & & \mathrm{Spec}(R_K)/\mathbf{G}_m \ar[dll]_-{r_K} \ar[d] \ar[drr]^-{j} & & \\ \mathbf{A}^1_{\mathrm{Spec}(W(k))}/\mathbf{G}_m & & \mathrm{Spec}(S^{\alg})/\mathbf{G}_m \ar[ll]_-{\text{Rees}} \ar[rr] & & \mathrm{Spec}(R^{\alg})/\mathbf{G}_m \\ \mathbf{A}^1_{\mathrm{Spf}(W(k))}/\mathbf{G}_m \ar[u] \ar[drr]_-{i_{\dR}} & & \mathrm{Spf}(S)/\mathbf{G}_m \ar[u] \ar[ll]_-{\text{Rees}} \ar[rr] & & Y = \mathrm{Spf}(R)/\mathbf{G}_m \ar[u] \ar[dll]^-{f} \\ & & W(k)^\mathcal{N} & & }\]
where the commutative (but not fibre!) square determined by the third and fourth rows comes from the commutative square~\eqref{eq:cuixal7nch}  by mapping $R'$ to $S$, the vertical maps  relating the second and third rows are all the evident completion maps whose pullbacks are equivalences on $\text{Perf}(-)$ (with inverses labelled $(-)^{\alg}$ above and below),  the right triangle in relating the first and second rows was the factorization just explained, while the left triangle relating those rows comes from compatibility of Rees maps.  Applying $\text{Perf}(-)$, inverting the vertical arrows relating the second and third columns of the resulting diagram and labelling them $(-)^{\alg}$, and tracing the effect of moving $\mathcal{E}^\cN \in \text{Perf}(W(k)^\cN)$ to $\text{Perf}(\Spec(R_K)/\mathbf{G}_m)$ through the two extreme paths in the diagram then shows that
\[ r_K^* \left( (i_{\dR})^* \mathcal{E}^\cN\right)^{\alg} \simeq j^* (f^* \mathcal{E}^\cN)^{\alg}, \]
as wanted.\qedhere
\end{enumerate}\end{proof}

\subsection{\texorpdfstring{$F$}{F}-crystals with \texorpdfstring{$G$}{G}-structure from \texorpdfstring{$G$}{G}-crystalline representations}\label{subsec:Gstr}

It will be convenient to use the following convention in relation to the loop rotation $\mu_p$-action on $\Spec(\KisinS)$ (see \S \ref{not:mupaction}), extending what we have already been using for vector bundles: for any geometric object $M$ over
 $\Spec W(k)\llbracket u \rrbracket$, a $\mu_p$-equivariant structure
on $M$ is a descent datum for the morphism
$$ \Spec W(k)\llbracket u \rrbracket \rightarrow  \Spec W(k)\llbracket u \rrbracket/\mu_p.$$
We will sometimes just say $M$ is $\mu_p$-equivariant, when we mean that $M$ is canonically equipped with such a structure.

\begin{para}\label{para:Gobjectssetup}  We want to extend  Theorem~\ref{thm:crysBKspecial}
to objects with $G$-structure. To do this we recall the general yoga of objects with $G$-structure explained in the Appendix~\ref{sec:appendix-Tannakian}.

Fix a finite extension $F/\Q_p,$ and let $G$ be a flat affine group scheme $G/\O_F.$ Following the convention of Appendix~\ref{para:scalarsnotn}, and taking the extension  $V'/V$ there to be $\cO_F/\Zp$, in this section an ``object with $G$-structure'' of a $\Zp$-linear category~$\cC$ is a $G$-object of $\cC\otimes_{\Zp}\cO_{F}$, the category of pairs $(x,\iota)$ consisting of an object~$x$ of~$\cC$ and a morphism of $\Zp$-algebras $\iota:\cO_F\to\End(x)$.
\end{para}

\begin{cor}[mod~$p$ Crystals with $G$-structure]
\label{cor:modpcrysGdesc}\leavevmode
Let $(\bar{{\mathcal{E}}},\varphi_{\bar{{\mathcal{E}}}})$ be a mod $p$ prismatic $F$-crystal over $W(k)$ with $G$-structure, and with associated mod $p$ $G$-Breuil--Kisin module $(\bar{\gM},\varphi_{\bar{\gM}})$. Then $(\bar{\gM},\varphi_{\bar{\gM}})$ is naturally $\mu_p$-equivariant.
\end{cor}
\begin{proof}It suffices to note that by Remark~\ref{rmk:modpFcrysmupequiv}, the exact tensor functor taking mod~$p$ prismatic $F$-crystals to mod~$p$ Breuil--Kisin modules takes values in the subcategory of $\mu_{p}$-equivariant Breuil--Kisin modules.
\end{proof}

\begin{cor}[Crystals with $G$-structure]
\label{cor:crysGdesc}\leavevmode
Let $({\mathcal{E}},\varphi_{\mathcal{E}})$ be a prismatic $F$-crystal over $W(k)$ with $G$-structure.
Then there is a canonical $G$-object $(\gM,\gN,\psi_{\gN})$ in the category $\mathcal D$ of \ref{para:category-D} associated to
$({\mathcal{E}},\varphi_{\mathcal{E}}).$ Explicitly we have
    \begin{enumerate}
        \item $\gN$ is a $G$-torsor on $\mathrm{Spec}(\KisinS\otimes_{\Z_p}\O_F)/\mu_p$ and $\psi_\gN$ is a $\mu_p$-equivariant isomorphism between $\varphi^*\gM$ and $\gN$ away from $V(u)$. 
        \item There is a natural identification $\gN/p \simeq \gM/p$, and the resulting $\mu_p$-equivariant structure on  $(\gM/p,\varphi_{\gM/p})$ provided by $\psi_\gN$ agrees with the one from Corollary \ref{cor:modpcrysGdesc}.
        \item For objects as in (1), if $G$ has reductive generic fiber, the relative positions of the modifications of $G$-torsors over $\mathrm{Spec}(K\llbracket E \rrbracket\otimes_{\Q_p}F)$  obtained from $\varphi_\gM$ and $\psi_\gN$ agree.
    \end{enumerate}
\end{cor}
\begin{proof}
This is immediate from Theorem \ref{thm:mainNthm} by Tannakian considerations (i.e.\ by  the definition of a $G$-object in terms of exact tensor functors). 
\end{proof}

\begin{para}   Now suppose that $G/\O_F$ is connected reductive, and fix
a maximal torus and Borel, $T \subseteq B \subseteq G.$
We now reformulate some parts of Corollary~\ref{cor:crysGdesc} in  terms of the stacks $\Hk_G$ and $\pHk_G$ considered in \S \ref{ss:fungrass}.

Let  $(\gM,\varphi_\gM)$ be the $G$-Breuil--Kisin module attached to a prismatic  $F$-crystal with $G$-structure $(\mathcal{E},\varphi_{\mathcal{E}}).$ We now change variables by writing the Breuil--Kisin ring $\KisinS$ as $W(k)\llbracket x \rrbracket$ for $x=E(u)=u-pz,$ and we regard
$x$ as the loop variable in the definition of the affine Grassmannian.
Then by Proposition \ref{prop:funcmupfixed}, the Breuil--Kisin module~$\gM$
defines a point  $\alpha_{\mathcal{E}} \in \Hk_G(W(k)\otimes_{\Z_p}\O_F)$,
encoding the relative position of the isogenies defined by $\varphi_\gM$ over the points of $\Spec W(k)\otimes_{\Z_p}\O_F.$ More precisely, we define  ~$\alpha_{\mathcal{E}}$ to be the point of $\Hk_G(W(k)\otimes_{\Z_p}\O_F)$ corresponding to the triple  \begin{equation*}\label{eqn:triple-defining-alpha}(\varphi^*\gM, \gM, \varphi_{\gM}:\varphi^{*}\gM[1/x]\isoto \gM[1/x]).\end{equation*}
For any $W(k)\otimes_{\Z_p}\O_F$-algebra $A,$ we denote by $\alpha_{\mathcal{E}}(A)$
the image of $\alpha_{\mathcal{E}}$ in $\Hk_G(A).$ In the following, we shall aim to lift $\alpha$, at least over geometric points, along the map $\pHk_G \to \Hk_G$ defined in \S \ref{ss:fungrass}.
\end{para}

\begin{para}\label{para:connected-reductive-uparrow} 
Let $\overline K$ be an algebraic closure  of $K,$ and let
$\overline k$ be the residue field of the ring of integers of $\overline{K}.$
For $A = \bar k, \bar K,$ we may identify  $\Hk_G(A)$ with $X_*(T)^+$ via Lemma~\ref{lem:pointquotient}.
Now fix an embedding $F \rightarrow \bar K,$ and consider the induced map $\xi:W(k)\otimes_{\Z_p}\O_F \rightarrow \bar K.$ This induces a map
$\bar\xi: W(k)\otimes_{\Z_p}\O_F \rightarrow \bar k.$
The specializations of $\alpha_{\mathcal E}$ along these maps satisfy
\begin{equation*}\label{eqn:le-version-of-alpha} \alpha_{\mathcal{E},\xi}(\bar k) \leq \alpha_{\mathcal{E},\xi}(\bar K) \in X_*(T)^+\end{equation*}
as $\xi$ is a generalization of  $\bar\xi$. We remark that these invariants depend only on $\ker \xi.$ The following theorem refines  this relation.
\end{para}

\begin{theorem}\label{thm:alpha-E-uparrow}
Suppose $G/\O_F$ is connected reductive, and $(\mathcal{E},\varphi_{\mathcal{E}})$ is a prismatic  $F$-crystal with $G$-structure. Then
\begin{enumerate}
\item \label{item:beta-exists}
There exists a $\beta_{\mathcal{E}} \in \pHk_G(W(k)\otimes_{\Z_p}\O_F),$ which maps to
 $\alpha_{\mathcal{E}}(K\otimes_{\Z_p}\O_F)$ and $\alpha_{\mathcal{E}}(k\otimes_{\Z_p}\O_F)$
 in  $\Hk_G(K\otimes_{\Z_p}\O_F)$ and $\Hk_G(k\otimes_{\Z_p}\O_F)$ respectively.
\item For any $\xi:W(k)\otimes_{\Z_p}\O_F \rightarrow \bar K$ as above,  we have
\[ \alpha_{\mathcal{E},\xi}(\overline k) \uparrow \alpha_{\mathcal{E},\xi}(\overline K) \in X_*(T)^+\]
\end{enumerate}
\end{theorem}
\begin{proof}   Using Proposition~\ref{prop:funcmupfixed}, the point $\beta_{\mathcal{E}} \in \pHk_G(W(k)\otimes_{\Z_p}\O_F)$ in (1) is obtained from the triple $(\varphi^*_W\gM,\gN,\psi_\gN)$ coming from Corollary~\ref{cor:crysGdesc},
where $\varphi_W$ denotes the Frobenius on $W(k).$
(Here we use $u$ as the loop variable.) 
The compatibility with $\alpha_{\mathcal{E}}$ over  $K$ and $k$ comes from parts (3) and (2) of Corollary~\ref{cor:crysGdesc} respectively.
Now (2) follows from Corollary \ref{cor:closure-from-zpbar-point}.
\end{proof}

We now turn to the relationship between prismatic $F$-crystals with $G$-structure and crystalline $G$-representations.
We write $\Gal_K$ for the absolute Galois group of $K,$ and 
$I_K \subset \Gal_K$ for its inertia subgroup.

We assume for the rest of this subsection that $G/\O_F$ is an extension of a finite \'etale group scheme by
a parahoric group scheme $G^\circ/\O_F.$

\begin{para}\label{para:repinner} Let $\mathrm{Rep}_{\mathbf{Z}_p}^{\crys}(\Gal_{K})$ denote the category of
crystalline representations of $\Gal_K$ on finite free $\Z_p$-modules,
i.e.\ the category of $\Gal_K$-stable $\Z_p$-lattices in crystalline $\Gal_K$-representations.
Following the conventions introduced in Appendix~\ref{para:scalarsnotn} and recalled above, we have the category
$G$-$\mathrm{Rep}_{\mathbf{Z}_p}^{\crys}(\Gal_{K})$ of exact $\otimes$-functors from $\Rep_{\O_F} G$ 
into $\mathrm{Rep}_{\mathbf{Z}_p}^{\crys}(\Gal_{K})\otimes_{\Z_p}\O_F$;
note that this latter category is nothing but the category of crystalline representations
of $\Gal_K$ on finite free $\O_F$-modules.

Let $P$ be a $G$-torsor corresponding to a class $c \in H^1(V,G).$ Set $G' = \underline{\Aut}_G P,$
so that $G'$ is the pure inner form of $G$ corresponding to $c.$
We make the set of crystalline representations
$\rho: \Gal_K \rightarrow G'(\O_F)$ into a category, by defining $\Hom(\rho,\rho')$ to be the set of $g \in G'(\O_F)$ such that $\rho' = g\rho g^{-1}.$ (Here as usual we say that a representation $\rho: \Gal_K \rightarrow G'(\O_F)$ is crystalline if its composite with any (equivalently, with all) faithful representation of~$G'$ is crystalline.)
\end{para}

\begin{lem}\label{lem:crystallineGreps}
There is a natural functor $\rho \mapsto \omega(\rho)$ from the category of crystalline representations
$\rho: \Gal_K \rightarrow G'(\O_F)$ to the category $G$-$\mathrm{Rep}_{\mathbf{Z}_p}^{\crys}(\Gal_{K}).$

Conversely, for any $\omega \in G$-$\mathrm{Rep}_{\mathbf{Z}_p}^{\crys}(\Gal_{K}),$ there is a pure inner
form $G'$ of $G,$ and a crystalline representation
$\rho: \Gal_K \rightarrow G'(\O_F)$ with $\omega \simeq \omega(\rho).$
\end{lem}
\begin{proof} 
For the first statement, note that the functor $W\mapsto (W\times P)/G$ is an equivalence of exact tensor categories
$ \Rep_{\cO_F} G \simeq \Rep_{\cO_F} G'$. It therefore suffices to consider the case that $G' = G$, in which case the functor is given by defining $\omega(\rho)(W)$ to be $W$ with the $\Gal_K$-action induced by $\rho,$
for each $W$ in $\Rep_{\O_F} G.$
If $\rho,\rho'$ are two such crystalline representations, and $g \in \Hom(\rho,\rho'),$ then we send
$g$ to the morphism $\omega(\rho) \rightarrow \omega(\rho')$ given by multiplication by $g$ on
the underlying $\O_F$-module $W$ in $\Rep_{\O_F} G.$

Conversely, suppose that $\omega$ is an object of $G$-$\mathrm{Rep}_{\mathbf{Z}_p}^{\crys}(\Gal_{K}).$
Then, as recalled in Appendix~\ref{subsec:torsors}, $P = \Spec (\omega(\O_G))$ is a $G$-torsor over $\O_F,$ 
equipped with an action of
 $\Gal_{K},$ and $\omega(W) \simeq (W\times P)/G$ for $W$ in $\Rep_{\O_F} G,$ compatible with
 the action of $\Gal_K.$ Let $G' = \underline{\Aut}_G P.$ Then the $\Gal_K$-action on $P$ factors through $G',$
 and hence the $\Gal_K$-action on $\omega(W)$ factors through $G'$.
 Unraveling the definitions, we see that $\omega \simeq \omega(\rho)$, as required.
 \end{proof}

If~$G$ is reductive, then the following theorem is a special case of~\cite[Thm.~2.28]{imai2024tannakian} (together with \cite[Prop.~3.8]{GuoReineckeCrysLoc}).
\begin{theorem}\label{thm:prismaticKeylemma}
Let $F/\Qp$ be a finite extension, and let~$G/\O_F$ be an extension of a finite \'etale group scheme by a parahoric group scheme $G^\circ.$
Then the \'etale realization functor $T$ induces an equivalence between the following categories:
\begin{enumerate}
    \item The category $G$-$\mathrm{Vect}^{\varphi}(\mathrm{Spf}(\mathcal{O}_K)^\Prism)$ of prismatic $F$-crystals in $G$-bundles over $\mathrm{Spf}(\mathcal{O}_K)$. 
    \item The category $G$-$\mathrm{Rep}_{\mathbf{Z}_p}^{\crys}(\Gal_{K})$ of crystalline $G$-representations of $\Gal_{K}$.
\end{enumerate}
\end{theorem}

\begin{proof} We begin by noting that each of the categories in (1) and (2) do not change if we replace $G$ by $\Res_{\O_F/\Z_p} G$ and $F$ by $\Q_p.$ For (1) this follows from the equivalence explained in Corollary \ref{cor:Gobjectequiv}, applied to $(\O_K)^{\prism}$.
For (2) it follows from Lemma \ref{lem:crystallineGreps}. 
Thus we assume $F=\Q_p$ for the rest of the proof. 
\footnote{While this reduction is not essential to the proof it is convenient when
citing the references at the end of the argument.}

By Theorem~\ref{thm:CrysGalPrismCrys} (i.e.\ the main result of  \cite{bhatt2021prismatic}), the functor $T$ gives a $\otimes$-functor
\begin{equation*}
    \label{prismcrystaleqn}
T:\mathrm{Vect}^{\varphi}(\mathrm{Spf}(\mathcal{O}_K)^\Prism) \to \mathrm{Rep}_{\mathbf{Z}_p}^{\crys}(\Gal_{K})
\end{equation*}
which is an equivalence  (but whose  inverse is not exact). Contemplating exact $\otimes$-functors from $\Rep_{\Z_p}(G)$ into either side then gives a fully faithful functor
\[ T_G:G\text{-}\mathrm{Vect}^{\varphi}(\mathrm{Spf}(\mathcal{O}_K)^\Prism) \to G\text{-}\mathrm{Rep}_{\mathbf{Z}_p}^{\crys}(\Gal_{K}).\]
To show that $T_G$ is an equivalence, it thus suffices to show essential surjectivity.
Let $\omega \in G\text{-}\mathrm{Rep}_{\mathbf{Z}_p}^{\crys}(\Gal_{K}).$ By Lemma \ref{lem:crystallineGreps}, $\omega = \omega(\rho)$
for some crystalline representation $\rho: \Gal_K \rightarrow G'(\O_F),$ with $G'$ a pure inner form of $G.$
As in the proof of Lemma~\ref{lem:crystallineGreps}, the equivalence $ \Rep_{\cO_F} G \simeq \Rep_{\cO_F} G'$ 
allows us to identify the functors $T_G$ and $T_{G'}.$ Thus to show that $\omega$ is in
the essential image of $T_G$ we may replace $G$ by $G',$ and assume $G' = G.$

For each~ $V$ in $\Rep_{\Z_p} G$ we set $T^{-1}_G(\omega)(V) = T^{-1}(\omega(V)).$ To show $T^{-1}_G(\omega)$ is a well defined
object of $G\text{-}\mathrm{Vect}^{\varphi}(\mathrm{Spf}(\mathcal{O}_K)^\Prism),$ we need to show that $V \mapsto T_G^{-1}(\omega)(V)$ is exact. To do this note that, if $K'/K$ is finite unramified extension, then $(\O_{K'})^{\Prism} \rightarrow (\O_K)^{\Prism}$ 
is a finite \'etale cover by \cite[Rmk.~3.9]{bhatt2022prismatization}, so it suffices to check exactness after replacing $K$ by $K'.$ 
(This may also be seen by using the flat cover $\iota$ below). 
Thus we may assume that $\rho$ factors through $G^\circ(\O_F).$
Then $\omega$ may be promoted to an object in $G^\circ\text{-}\mathrm{Rep}_{\mathbf{Z}_p}^{\crys}(\Gal_{K}),$ 
so we may assume that $G = G^\circ$ is a parahoric, which we do from now on. 

Now to show exactness of $T_G^{-1}(\omega)$ we use the point of view of \ref{subsec:torsors}. 
Let $\O_P \in \mathrm{QCoh}(\Spf(\cO_K)^\Prism)$ be the colimit of the ind-object $ T_G^{-1}(\omega)(\O_G)$ in vector bundles; this is a sheaf of commutative algebras, and let $P$ be the corresponding stack with $G$-action over $\Spf(\cO_K)^\Prism$.
By Corollary \ref{cor:equivfunctors}, and using that we know that $T_G$ is fully faithful, it suffices to show that $P$ is a $G$-torsor. 
Using the faithfully flat map $\iota:\mathrm{Spf}(\KisinS) \to (\mathcal{O}_K)^\Prism$ from \eqref{eq:rhoforBK}, it suffices to show that $\iota^*P$ is a $G$-torsor.

Observe that pulling back along $\mathrm{Spf}(\KisinS) \to \mathrm{Spec}(\KisinS)$ (i.e., completion) induces an exact $\otimes$-equivalence on the category of vector bundles, and hence also on the Ind-categories. Hence, the commutative algebra object $\O_{\iota^*P}$ arises as the completion of a flat commutative algebra object $\O_Q$ in $\mathrm{QCoh}(\Spec(\KisinS))$, corresponding to a flat affine $G$-scheme $Q \to \Spec(\KisinS)$ algebraizing the $G$-scheme $\iota^* P \to \Spf(\KisinS)$. It suffices to show that $Q$ is a $G$-torsor over $\KisinS$.

When $G$ is reductive this was proved in \cite[1.3.4]{KisinShimura}.
In fact, the argument there shows that if $G$ is any affine group scheme of finite type over $\Z_p,$ then
$Q$ is a $G$-torsor when restricted to the complement of the closed point
of $\Spec \KisinS$ (Steps 1-4 of {\it loc.~cit}). Here we are using the isomorphism 
$\varphi^*(Q)[1/E] \simeq Q[1/E],$ cf.~\ref{lem:Frobtorsor}. 
This could also be deduced from the analogous ``exactness outside the closed point'' property of the only non-exact functor, $T^{-1}$, used in the definition of $Q$.
When $G$ is smooth and connected, and every
$G$-torsor on the complement of the closed point in $\Spec \KisinS$ extends to $\Spec \KisinS$, then  Step 5
of {\em loc.~cit} shows that~$Q$ 
is a $G$-torsor.
It therefore suffices to note that for $G$ parahoric this final statement is
\cite[Cor.~1.2]{AnschutzTorsor}.
\end{proof}

\begin{para}\label{para:HTweights} As in~\ref{para:connected-reductive-uparrow}, fix an embedding $F \rightarrow \bar K,$ and denote by
$\xi:W(k)\otimes_{\Z_p}\O_F \rightarrow \bar K$ the induced map.
Let $\rho:\Gal_{K} \rightarrow G(\O_F)$ be a crystalline representation,
and let $\omega$ be the corresponding functor in $G$-$\mathrm{Rep}_{\mathbf{Z}_p}^{\crys}(\Gal_{K})$
given by Lemma \ref{lem:crystallineGreps}.
For $V$ in $\Rep_{\O_F} G,$ consider the graded $\bar K$-vector space
given by
$$ \gr_{\xi}D_{\dR}(V) \coloneq \gr^{\bullet} (D_{\dR}(\omega(V)[1/p])\otimes_{W(k)\otimes_{\Z_p}\O_F,\xi}\bar K). $$

 Replacing $K$ by a finite unramified
extension does not change $ \gr_{\xi}D_{\dR}(V)$, so we may promote $\omega$ to an object of $G^\circ$-$\mathrm{Rep}_{\mathbf{Z}_p}^{\crys}(\Gal_{K}).$
By Tannaka duality, this gives a $G^\circ$-conjugacy class of cocharacters
$\mu_{\xi}(\rho): \G_m \rightarrow G^\circ$ defined over $\bar K.$
Note that the conjugacy class $[\mu_{\xi}(\rho)]$ is functorial in $G^\circ$
and depends only on $\ker \xi.$ 
Fixing a maximal torus and Borel $T \subseteq B \subseteq G^\circ$ over $\overline{K},$
we may view $\mu_{\xi}(\rho) \in X_*(T)^+.$

Similarly, by replacing $K$ by a finite unramified extension, we may define an invariant $\alpha_{\mathcal E,\xi}(\bar K),$ 
as in \ref{para:connected-reductive-uparrow}. 

\end{para}

\begin{prop}\label{prop:HTweights} Let $\rho:\Gal_{K} \rightarrow G(\O_F)$
be a crystalline representation, and let $(\mathcal E, \varphi_{\mathcal E})$
be the corresponding prismatic $F$-crystal with $G$-structure.
Then, for each $\xi$ as above, 
we have
$$\mu_{\xi}(\rho) = \alpha_{\mathcal E,\xi}(\bar K) \in X_*(T)^+.$$
\end{prop}
\begin{proof} As above, after replacing $K$ by a finite unramified extension, we may assume that $G = G^\circ$.
  Consider a representation $r: G \rightarrow \GL(L)$ on
a finite free $\O_F$-module $L.$ By \cite[Lem.~2.2]{RapoportRichartz}, it suffices to show that, for any $L,$
$r\circ(\mu_{\xi}(\rho))$ and $r\circ \alpha_{\mathcal E,\xi}(\bar K)$ are conjugate in
$\GL(L).$ Keeping in mind the functoriality of the cocharacters
$\mu_{\xi}(\rho)$ and $\alpha_{\mathcal E,\xi}(\bar K)$ with respect to
$\rho,$ it then suffices to consider the case $G = \GL_n.$ In this case the result follows from
\cite[Thm.~1.2.1(1)]{KisinShimura}.
\end{proof}

\subsection{\texorpdfstring{$F$}{F}-gauges and the shape of Breuil--Kisin modules}\label{subsec:Jacob-lemma}
In this section, which is not used in the rest of the paper, we explain our first approach to proving results like Theorem~\ref{thm:alpha-E-uparrow}, 
as first explained in~\cite{TG-IAS-talk}. 

Write $X=(\ZpNyg)_{p=t=0}$, so $X=\mathbf{A}^1/(\mathbf{G}_a^\sharp \rtimes \mathbf{G}_m)$, with $x$ denoting the co-ordinate on $\mathbf{A}^1$; we normalize the action to ensure\footnote{This choice ensures compatibility with standard sign conventions for the larger stack $\Z_p^\cN$: there is a natural Rees map $t:\mathbf{Z}_p^\cN \to \mathbf{A}^1/\mathbf{G}_m$ whose restriction to the closed point of the target gives a natural map $X \to B\mathbf{G}_m$, and our  conventions ensure that $\mathcal{O}(-1)$ on $X$ is indeed the pullback of $\mathcal{O}(-1)$ from $B\mathbf{G}_m$.} $x \in H^0(X,\mathcal{O}(-1))$.
Note that we have a  composition of maps
\[ f:B\mu_p \to B\mathbf{G}_m \to \mathbf{A}^1/\mathbf{G}_m \to X,\]
where all maps are the natural ones. Write $\tilde{f}:B\mathbf{G}_m \to X$ for the map appearing above.

\begin{defn}
For $M \in \mathrm{Perf}(X)$, we write
\[ w_M = [\tilde{f}^* M] \in K_0(B\mathbf{G}_m)\]
and
\[ \overline{w}_M = [f^* M] \in K_0(B\mu_p).\]
Concretely, we can identify $K_0(B\mathbf{G}_m)$ with the group $\mathrm{Map}_{\mathrm{fin}}(\mathbf{Z},\mathbf{Z})$ of finitely supported $\mathbf{Z}$-valued functions on the character group $\mathbf{Z}$ of $\mathbf{G}_m$; a coherent sheaf $M$ on $X$ corresponds to a graded module $\oplus_n M_n$ over the Weyl algebra $\mathbf{F}_p\{x,D\}$, and we have
\[ w_M(n) = \dim(M_n) - \dim(M_{n+1}).\]
Similarly, identifying $K_0(B\mu_p) = \mathrm{Map}(\mathbf{Z}/p,\mathbf{Z})$, for $M=\oplus_n M_n$ as above, we have
\[ \overline{w}_M(n) = \sum_{j \in \mathbf{Z}} w_M(n+pj).\]
\end{defn}

\begin{defn}For~$i\in\Z$ we as usual write~$\delta_i:\Z\to\Z$ for the function
\[
  \delta_i(j)=\begin{cases}
1 & \text{if } i=j, \\
   0 & \text{otherwise,}
\end{cases}
\] and we set \[
 e_i\coloneq \delta_i+\delta_{i+p+1}-\delta_{i+1}-\delta_{i+p}.
\]
\end{defn}

With the above notation, we have the following lemma, which we learned from Jacob Lurie.

\begin{lem}\label{lem:Jacob-lemma}
\begin{enumerate}
\item\label{item:32} Let $M \in \mathrm{Perf}(X)$ be supported at the closed point, i.e.,  $M[x^{-1}]=0$. Then $\overline{w}_M=0$.
\item\label{item:33} Let $N \in \mathrm{Coh}( (\Z_p^\cN)_{p=0})$ be supported at the closed point, and let $M = N|_X \in \mathrm{Perf}(X)$ be the restriction of $N$ along $X \hookrightarrow (\Z_p^\cN)_{p=0}$. Then $w_M$ can be written as a finite sum $h=\sum_im_ie_i$ with each~$m_i\le 0$.
\end{enumerate}
\end{lem}

Note that $\overline{w}_M$ only depends on the pullback of $M$ to $\mathbf{A}^1/\mathbf{G}_m$, so the statement above also makes sense if $X$ is replaced by $\mathbf{A}^1/\mathbf{G}_m$. However, the resulting statement is false: the structure sheaf $M$ of the locus $x^2=0$ yields a counterexample for $p > 2$.

\begin{proof}[Proof of Lemma~\ref{lem:Jacob-lemma}]
We begin with the first part. By filtering, we may assume that $M$ is a coherent sheaf with $M[x^{-1}]=0$, and we must show $[f^*M] =0$. As $\alpha_p \subset \mathbf{G}_a$ is stable under translation by $\mathbf{G}_a^\sharp$, setting $x^p=0$ gives a closed substack $i:Z \subset X$. Moreover,
\[ Z = \alpha_p/(\mathbf{G}_a^\sharp \rtimes \mathbf{G}_m) \simeq B(\mathbf{G}_a^\sharp \rtimes^p \mathbf{G}_m),\] where the superscript of $p$ indicates a rescaling by $p$ of the standard action of $\mathbf{G}_m$ on $\mathbf{G}_a^\sharp$, and the second identification comes from the exact sequence
\[ 0 \to F_* \mathbf{G}_a^\sharp \xrightarrow{V} \mathbf{G}_a^\sharp \to \alpha_p \to 0\]
of $W$-module schemes, coming from taking the kernel of Frobenius in the standard exact sequence
\[ 0 \to F_* W \xrightarrow{V} W \to \mathbf{G}_a \to 0\]
of $W$-module schemes. 
In particular, as representations of $\mathbf{G}_a^\sharp$ are all unipotent, any coherent sheaf on $Z$ admits a finite filtration whose associated graded pieces are copies of $\mathcal{O}_Z(j)$ for suitable $j$. Further, as $I_Z = \mathcal{O}_X(p)$, any coherent sheaf $M$ on $X$ set-theoretically supported on $Z$ then also admits a finite filtration with graded pieces of the form $i_* \mathcal{O}_Z(j)$ for suitable $j$. Thus, it is enough to show $[f^* i_* \mathcal{O}_Z(j)]=0$ for any $j$. But
$f^* i_* \mathcal{O}_Z(j) = g^* i^* i_* \mathcal{O}_Z(j)$ for the (unique) map $g:B\mu_p \to Z$ factoring $f$, and $ i^*i_{*} \mathcal{O}_Z(j) = \mathcal{O}_Z(j) \oplus \mathcal{O}_Z(j+p)[1]$, so the claim follows by noting that $g^*$ is a $\otimes$-functor and $g^* \mathcal{O}(p) = \mathcal{O}_{B\mu_p}$.

For the second part, by filtering $N$, we may assume $N=k_*\mathcal{O}_Z(j)$ for suitable $j \in \Z$, where $k:Z \hookrightarrow (\Z_p^\cN)_{p=0}$ is the inclusion of the closed point. In this case, since $k$ factors as $Z \xrightarrow{i} X \hookrightarrow (\Z_p^\cN)_{p=0}$ where the second map is a Cartier divisor with ideal sheaf $\mathcal{O}_{ (\mathbf{Z}_p^\cN)_{p=0}}(-1)$, we have
\[ M = N|_X = i_* \mathcal{O}_Z(j) \oplus i_*\mathcal{O}_Z(j-1)[1].\]
Write $h:B\Gm\to Z$ for the evident map factoring $\tf$. We can then compute
\begin{align*}
    \tf^* M &= h^* i^* i_* \mathcal{O}_Z(j) \oplus h^* i^* i_* \mathcal{O}_Z(j-1)[1] \\
    &= \big(h^*  \mathcal{O}_Z(j) \oplus h^*  \mathcal{O}_Z(j+p)[1]\big) \oplus \big(h^*  \mathcal{O}_Z(j-1) \oplus h^*  \mathcal{O}_Z(j-1+p)[1]\big)[1] \\
    &= \mathcal{O}_{B\mathbf{G}_m}(j) \oplus \mathcal{O}_{B\mathbf{G}_m}(j+p)[1] \oplus \mathcal{O}_{B\mathbf{G}_m}(j-1)[1] \oplus \mathcal{O}_{B\mathbf{G}_m}(j-1+p)[2]
\end{align*}
where the second line follows from the first by noting that $i$ is a Cartier divisor with ideal sheaf $\mathcal{O}_X(p)$. One then computes that $w_M = -e_{j-1}$, which gives the desired claim.
\end{proof}

Let us deduce some numerical consequences from Lemma~\ref{lem:Jacob-lemma} for cocharacters attached to a crystalline Galois representation. It will be convenient in fact to work in a slightly broader context:

\begin{defn}
  We say that $f\in \mathrm{Map}_{\mathrm{fin}}(\mathbf{Z},\mathbf{Z})$  is \emph{effective} if $f(n)\ge 0$ for all~$n$.
If~$f$ is effective, then we write~$\nu(f)$ for the dominant cocharacter $(x_1 ,\dots,x_{N})$ of~$\GL_{N}$, where $N=\sum_nf(n)$, and the multiset~$\{x_1,\dots,x_N\}$ contains~$n$ with multiplicity~$f(n)$.
\end{defn}

An essentially immediate corollary is the following:

\begin{cor}\label{cor:Wp-orbit} Suppose that $\rho:\Gal_{\Qp}\to\GL_n(\Zp)$ is a crystalline representation with Hodge--Tate weights~$\mu(\rho)=(\mu_1 ,\dots,\mu_n)$. 
  Let~$\gM$ be the corresponding Breuil--Kisin module, and let $s(\gMbar)=(s_1 ,\dots,s_{n})$ be the shape of $\gMbar=\gM\otimes_{\Zp}\Fp$ (see Definition~\ref{v3-defn:shape}).

Then 
there is a permutation $\sigma\in S_n$ such that $s_i\equiv \mu_{\sigma(i)}\pmod{p}$ for all~$i$.
\end{cor}
\begin{proof}
Let $\mathcal{E}$ be the reflexive gauge on $\Z_p^{\mathcal{N}}$ corresponding to $\rho$ via Theorem~\ref{thm:CrysGalPrismCrys}; we shall prove the corollary by interpreting it as a relation between two invariants attached to this gauge.

Write $M=\mathcal{E}/p \in \mathrm{Coh}( (\Z_p^\cN)_{p=0})$, so $M$ is a $t$-torsion-free coherent sheaf on $(\Z_p^\cN)_{p=0}$. By (for example) \cite[Thm.~1.2.1(1)]{KisinShimura}, we have $\mu(\rho)=\nu(w_M)$. Write $M'$ for the reflexive hull of $M$, so there is a short exact sequence
\begin{equation*}\label{eqn:ses-for-pullback-to-ZpN}0\to M\to M'\to N\to 0\end{equation*} where $N$ is supported at the closed point of $(\Z_p^\cN)_{p=0}$. We then also have $s(\gMbar)=\nu(w_{M'})$, see Remark~\ref{RelPosGrAgree}. Now Lemma~\ref{lem:Jacob-lemma} (1) gives that $\overline{w}_{M'} -\overline{w}_{M}=\overline{w}_N=0 $, which immediately gives the claim.
\end{proof}
The rest of this section is devoted to proving Corollary~\ref{cor:uparrow-order-from-F-gauge}, which refines
 Corollary~\ref{cor:Wp-orbit}; the main geometric input comes from the (as yet unused) non-positivity of the coefficients appearing in Lemma~\ref{lem:Jacob-lemma} (2).
 We begin with some combinatorics.

  For any function $m:\Z\to\Z$, we will often write $m_i\coloneq m(i)$. If~$m$ is effective and has finite support, then we write \[
 H_m\coloneq \sum_i m_i e_i;
\]
equivalently,
\begin{equation}\label{eq:Hm-formula}
 H_m(x)=m_x-m_{x-1}-m_{x-p}+m_{x-p-1}.
\end{equation}

For any integer $a$ and integers $1\leq d\leq p$, $1\le k$, define the function
\[
 \rho_{a,d,k}(i)=
 \begin{cases}
  1,&\text{if }i=a+s+jp\text{ for some }0\leq s<d,\ 0\leq j<k,\\
  0,&\text{otherwise,}
 \end{cases}
\]
and set $E_{a,d,k}\coloneq   H_{\rho_{a,d,k}}$.
\begin{lemma}\label{lem:telescoping}
  We have 
  $E_{a,d,k}=\delta_a+\delta_{a+d+kp}-\delta_{a+d}-\delta_{a+kp}$.
\end{lemma}

\begin{proof}
Since
\[
 e_i=(\delta_i-\delta_{i+1})-(\delta_{i+p}-\delta_{i+p+1}),
\]
we have

\begin{align*}
 E_{a,d,k}=\sum_{s=0}^{d-1}\sum_{j=0}^{k-1}e_{a+s+jp}
  &= \sum_{s=0}^{d-1} \bigl( (\delta_{a+s} - \delta_{a+s+1})  - (\delta_{a+s+kp}-\delta_{a+s+kp+1}) \bigr) \\
  &= \delta_a - \delta_{a+d} - \delta_{a+kp} + \delta_{a+d+kp},
\end{align*}

as claimed.
\end{proof}

\begin{lemma}\label{lem:rectangle-uparrow}
Let $G:\Z\to\Z_{\geq 0}$ be effective, and suppose that
\[
 G^+\coloneq G+E_{a,d,k}
\]
is also effective.  Then
\[
 \nu(G)\uparrow \nu(G^+).
\]
\end{lemma}

\begin{proof}
  By Lemma~\ref{lem:telescoping}, passing from $G$ to $G^+$ removes one occurrence of $a+kp$ and one occurrence of $a+d$, and inserts one occurrence of $a$ and one occurrence of $a+d+kp$.
   Since $1\leq d\leq p$ and $k\geq 1$, we have $a+kp\ge a+d$, with equality  if and only if~$d=p$ and~$k=1$.

As usual we write the positive coroots of  $\GL_N$ as $\varepsilon_r-\varepsilon_s$ for $1\leq r<s\leq N$.  If $x=(x_1,\dots,x_N)$ and $M\in p\Z$, the affine reflection $s_{\varepsilon_r-\varepsilon_s,M}$ acts by
\[
 (x_r,x_s)\mapsto (x_s+M,\,x_r-M)
\]
and fixes all other coordinates.

Let $x=\nu(G)=(x_1,\dots,x_N)$.  Choose indices $r<s$ such that
\[
 x_r=a+kp,
 \qquad
 x_s=a+d.
\]
Let $\alpha=\varepsilon_r-\varepsilon_s$ and let
\[
 y=s_{\alpha,kp}(x).
\]
Then the two chosen coordinates change as
\[
 (a+kp,a+d)\mapsto (a+d+kp,a).
\]
Moreover
\[
 y-x=d(\varepsilon_r-\varepsilon_s),
\]
so $\nu(G)=x\uparrow y$.
The multiset of coordinates of $y$ is exactly the multiset attached to $G^+$, so we may write $y=w\cdot \nu(G^+)$ for some~$w\in S_{N}$.
By Lemma~\ref{lem:uparrow-dominant-Weyl-orbit}, we have $w \cdot \nu(G^+)\uparrow \nu(G^{+})$, and we are done.
\end{proof}

\begin{lemma}\label{lem:anchored-corner}
Let $c:\Z_{\geq 0}^2\to\Z_{\geq 0}$ be finitely supported and suppose that $c(0,0)>0$.  
For $D,K\geq 1$ define
\[
 (\Delta c)(D,K)\coloneq c(D,K)-c(D-1,K)-c(D,K-1)+c(D-1,K-1).
\]
Then there exist $D,K\geq 1$ such that
\[
 c(s,j)>0\quad\text{for all }0\leq s<D,\ 0\leq j<K,
\]
and
\[
 (\Delta c)(D,K)>0.
\]
\end{lemma}

\begin{proof}
Let $P$ be the finite set of pairs $(D,K)\in\Z_{\geq 1}^2$ such that
\[
 c(s,j)>0\quad\text{for all }0\leq s<D,\ 0\leq j<K.
\]
Thus $P$ is a finite downwards-closed subset of $\Z_{\geq 1}^2$, and $(1,1)\in P$.
We need to prove that for some $(D,K)\in P$ we have $(\Delta c)(D,K)>0$.
Setting
\[
 S\coloneq \sum_{(D,K)\in P}(\Delta c)(D,K),
\]it suffices to prove that~$S>0$.

Writing $\mathbf 1_P(D,K)=0$ if $D\leq 0$ or $K\leq 0$, we may write
\[
 S=\sum_{s,j\geq 0} A(s,j)c(s,j),
\]
where
\[
 A(s,j)=\mathbf 1_P(s,j)-\mathbf 1_P(s+1,j)-\mathbf 1_P(s,j+1)+\mathbf 1_P(s+1,j+1).
\]
Noting that $A(0,0)=\mathbf 1_P(1,1)=1$, so that $A(0,0)c(0,0)=c(0,0)>0$, it suffices to show that $A(s,j)c(s,j)\geq 0$ for all $(s,j)$.

If~$c(s,j)=0$ then there is nothing to prove, so we may assume that $c(s,j)>0$, and we need to show that $A(s,j)\geq 0$.  If $s,j\geq 1$, then, because $P$ is downwards closed, the only way to have $A(s,j)<0$ is to have
\[
 (s+1,j)\in P,
 \qquad
 (s,j+1)\in P,
 \qquad
 (s+1,j+1)\notin P.
\]
But the first two inclusions say that every entry in the two rectangles
\[
 0\leq u<s+1,
 \quad
 0\leq v<j,
\]
and
\[
 0\leq u<s,
 \quad
 0\leq v<j+1
\]
is positive.  Together with $c(s,j)>0$, this implies that every entry in the rectangle
\[
 0\leq u<s+1,
 \quad
 0\leq v<j+1
\]
is positive, i.e. $(s+1,j+1)\in P$, a contradiction.

If $s=0$ and $j\geq 1$, then
\[
 A(0,j)=-\mathbf 1_P(1,j)+\mathbf 1_P(1,j+1).
\]
This can be negative only if $(1,j)\in P$ but $(1,j+1)\notin P$.  Since $c(0,j)>0$, the inclusion $(1,j)\in P$ forces $(1,j+1)\in P$, again a contradiction.  The case $j=0$ is identical.
\end{proof}
\begin{lemma}\label{lem:backward-peeling}
Let $m:\Z\to\Z_{\geq 0}$ be finitely supported and non-zero.  Let $g:\Z\to\Z_{\geq 0}$ be effective, and put
\[
 f\coloneq g+H_m.
\]
Assume that $f$ is effective.
Then there exist an integer $a$ and integers $1\leq d\leq p$,  $k\geq 1$ such that the functions
\[
  m'\coloneq m-\rho_{a,d,k},
 \qquad
 f'\coloneq f-E_{a,d,k},
\]
 are effective.
\end{lemma}

\begin{proof}
Let~$a$ be minimal such that~$m(a)>0$.
By~\eqref{eq:Hm-formula},
\[
 H_m(a)=m_a>0,
\]
because all $m_t$ with $t<a$ vanish.  Since $g$ is effective, it follows that
\begin{equation}\label{eq:culjggpm46}
  f(a)>0.
\end{equation}
Define a finitely supported function $c:\Z_{\geq 0}^2\to\Z_{\geq 0}$ by
\[
 c(s,j)\coloneq m_{a+s+jp}.
\]
Then $c(0,0)=m_a>0$, so by Lemma~\ref{lem:anchored-corner}, there are $D,K\geq 1$ such that
\[
 c(s,j)>0\quad\text{for }0\leq s<D,\ 0\leq j<K,
\]
and
\[
 (\Delta c)(D,K)>0.
\]
Write
\[
 D=d+qp
\]
with $1\leq d\leq p$ and $q\geq 0$, set
\[
 k\coloneq K+q,
\]
and set $m'\coloneq  m-\rho_{a,d,k}$.
We claim that $m'$ is effective.
To see this, we must show that if $0\leq s<d$ and $0\leq j<k$ then $m_{a+s+jp}>0$.  If $j<K$, then
\[
 m_{a+s+jp}=c(s,j)>0
\]
because $s<d\leq D$ and $j<K$.  If on the other hand $j=K+t$ with $0\leq t<q$, then
\[
 m_{a+s+jp}=m_{a+s+(K+t)p}=c(s+(t+1)p,K-1),
\]
and
\[
 s+(t+1)p\leq (d-1)+qp=D-1,
\]
so this coefficient is also positive, as required.

It remains to show that~$f'=f-E_{a,d,k}$ is effective.
To this end, note that~\eqref{eq:Hm-formula} gives
\begin{align*}
 H_m(a+D+Kp)
 &=m_{a+D+Kp}-m_{a+D+Kp-1}-m_{a+D+(K-1)p}+m_{a+D+(K-1)p-1}\\
 &=c(D,K)-c(D-1,K)-c(D,K-1)+c(D-1,K-1)\\
 &=(\Delta c)(D,K)>0.
\end{align*}
But
\[
 a+D+Kp=a+d+kp,
\]
so, again using the effectivity of $g$, we see that
\begin{equation}\label{eq:culjgg0xaz}
  f(a+d+kp)>0.
\end{equation}

By Lemma~\ref{lem:telescoping},
\[
 E_{a,d,k}=\delta_a+\delta_{a+d+kp}-\delta_{a+d}-\delta_{a+kp},
\]so the effectivity of~$f'$ follows from~\eqref{eq:culjggpm46} and~\eqref{eq:culjgg0xaz}.
\end{proof}

\begin{prop}\label{prop:effective-functions-imply-uparrow}
Let $f,g:\Z\to\Z_{\geq 0}$ be finitely supported effective functions.  Assume that
\[
 f-g=\sum_i m_i e_i
\]
for a finitely supported function $m:\Z\to\Z_{\geq 0}$.  
Then
$
 \nu(g)\uparrow \nu(f)
$.
\end{prop}
\begin{proof}
Let
$
 M\coloneq \sum_i m_i.
$
We argue by induction on $M$, the case $M=0$ being trivial.
If~$M>0$, then by Lemma~\ref{lem:backward-peeling}, we can find $a,d,k$ such that $m'\coloneq m-\rho_{a,d,k}$,
 $f'\coloneq f-E_{a,d,k}=f-   H_{\rho_{a,d,k}}$ are effective.
  Moreover we have
\[
 \sum_i m'_i=\sum_i m_i-\sum_i \rho_{a,d,k}(i)<M,
\]
so by induction we have $ \nu(g)\uparrow \nu(f')$.
By Lemma~\ref{lem:rectangle-uparrow} we have $ \nu(f')\uparrow \nu(f), $ so $ \nu(g)\uparrow \nu(f) $, as required.
\end{proof}

\begin{cor}
  \label{cor:uparrow-order-from-F-gauge}Suppose that $\rho:\Gal_{\Qp}\to\GL_n(\Zp)$ is a crystalline representation with Hodge--Tate weights~$\mu(\rho)$.
  Let~$\gM$ be the corresponding Breuil--Kisin module, and let $s(\gMbar)=(s_1 ,\dots,s_{n})$ be the shape of $\gMbar$. Then $s(\gMbar)\uparrow \mu(\rho)$.
\end{cor}
\begin{proof}
  As in the proof of Corollary~\ref{cor:Wp-orbit}, this follows from Proposition~\ref{prop:effective-functions-imply-uparrow} and Lemma~\ref{lem:Jacob-lemma}.
  \end{proof}

\section{Inertial weights of mod~\texorpdfstring{$p$}{p} Galois representations}\label{v3-sec:inertial-weights-mod-p-Galois-rep}
\subsection{\'{E}tale \texorpdfstring{$\varphi$}{phi}-modules and Breuil--Kisin modules}\label{v3-subsec:etale-varphi}
We fix throughout this section an algebraic closure~$\Qpbar$ of~$\Qp,$
and denote by $\Zpbar$ and $\Fpbar$ its ring of integers and residue field, respectively.
We fix a uniformizer~$\pi$ of~$\mathbf{Q}_p$, and as usual write $E=u-\pi$.

\begin{para}
Let $R$ be a $\Zp$-algebra, and let $\varphi:R\llbracket u\rrbracket\to R\llbracket u\rrbracket$ be the $R$-linear Frobenius with $\varphi(u)=u^{p}$.
We denote by $\Mod^{\varphi}_{/R\llbracket u \rrbracket}$ the category of finite free $R\llbracket u \rrbracket$-modules $M$ equipped with an isomorphism $\varphi^*M[1/E] \isoto M[1/E].$
Similarly, if $R$ is an $\Fp$-algebra we denote by $\Mod^{\varphi}_{/R((u))}$ the category of finite free $R((u)) = R\llbracket u \rrbracket [1/u]$-modules $M$ equipped with an isomorphism $\varphi^*M \isoto M.$
We will usually apply these definitions with $R = \Fpbar$
or $R = \Zpbar.$

It is a straightforward consequence of the definitions that for each finite extension~$\F/\Fp$, extension of scalars gives natural equivalences of categories (with notation as in Appendix~\ref{para:extension-scalars-categories})
\[\Mod^{\varphi}_{/\Fp((u))}\otimes_{\Fp}\F=\Mod^{\varphi}_{/\Fp((u))\otimes_{\Fp}\F}=\Mod^{\varphi}_{/\F((u))} \]and \[\Mod^{\varphi}_{/\Fp\llbracket u \rrbracket}\otimes_{\Fp}\F=\Mod^{\varphi}_{/\Fp\llbracket u \rrbracket\otimes_{\Fp}\F}=\Mod^{\varphi}_{/\F\llbracket u \rrbracket}. \]
\end{para}

The following standard lemma extends this to the case of ~$\Fpbar$-coefficients.
\begin{lemma}\label{v3-lem:etalemodulecomp} The natural functors
\[\Mod^{\varphi}_{/\Fp((u))\otimes_{\Fp}\Fpbar}\to\Mod^{\varphi}_{/\Fpbar((u))},\] \[\Mod^{\varphi}_{/\Fp\llbracket u \rrbracket\otimes_{\Fp}\Fpbar}\to\Mod^{\varphi}_{/\Fpbar\llbracket u \rrbracket}\]
are exact equivalences of categories, with exact quasi-inverses.
\end{lemma}
\begin{proof}It evidently suffices to show that 
  any object (resp.~morphism) in the category $\Mod^{\varphi}_{/\Fpbar((u))}$
descends to  $\Mod^{\varphi}_{/\F_q((u))}$ for some $q=p^r,$ and similarly for $\Mod^{\varphi}_{/\Fpbar\llbracket u \rrbracket}$.
This can be checked directly, but it is more convenient to observe that it is an immediate consequence of 
\cite[Prop.~5.4.9, Thm.~5.4.11]{EGstacktheoreticimages}. More precisely,
using the notation of {\em loc.~cit}, for each $d \geq 1,$
 \cite[Thm.~5.4.11]{EGstacktheoreticimages} shows that the stack~$\cR_d$ is limit-preserving, so that there is an equivalence of groupoids\[\colim_{\F/\Fp\textrm{ finite}}\cR_d(\Spec\F)\to\cR_d(\Spec\Fpbar).\]
This shows that the first functor is essentially surjective, and induces a bijection on isomorphisms. However, this implies that the functor is an equivalence, as in any additive category, if $f: X \rightarrow Y$ is a morphism,
then $(\Gamma_f,0)+(0,\id): X\times Y \rightarrow X \times Y$ is an isomorphism.
Here $\Gamma_f$ denotes the graph of $f.$

The argument for the second equivalence is identical, using the stacks $\mathcal C_{d,F}$
of loc.~cit.
Finally, the exactness of the functors follows from the faithful flatness of all the extensions of scalars involved.
\end{proof}

\begin{para}\label{v3-para:Fontainecorresp}We let $(\pi^{1/p^{n}})_{n\ge 0} \subset \Qpbar$ 
be a choice of a compatible system of~$p$-power roots of~$\pi,$
and set~$\Q_{p,\infty}\coloneq \cup_{n}\Qp(\pi^{1/p^{n}}).$ 
  By Lemma \ref{v3-lem:etalemodulecomp}, and a result of Fontaine~\cite[A.1.2.6, A.3]{MR1106901} (see also~\cite[Lem.~1.2.7]{KisinModularity} for this precise statement), the functor $M\mapsto T(M)\coloneq (M\otimes_{\Fp((u))}\C^\flat)^{\varphi=1}$ gives an exact tensor equivalence
\begin{equation}\label{v3-eqn:Fontainecorresp} \Mod_{/ \Fpbar((u))}^{\varphi} \simeq \Rep_{\Fpbar} \Gal_{\Q_{p,\infty}},
\end{equation}
where the right hand side denotes the category of continuous representations of $\Gal_{\Q_{p,\infty}}$ on finite dimensional $\Fpbar$-vector spaces.
This functor has an exact quasi-inverse $\rhobar\mapsto M(\rhobar).$

The above correspondence then extends to $H$-objects:
\end{para}

\begin{prop}Let $H/\Fpbar$ be a smooth affine group scheme.
  \label{v3-prop:standard-properties-etale-phi-modules}\leavevmode
  \begin{enumerate}
  \item\label{v3-item:11} The functor~$T$ induces an equivalence of categories between $H$-$\Mod_{/ \Fpbar((u))}^{\varphi},$ 
  and the category of continuous representations 
  $$\rhobar: \Gal_{\Q_{p,\infty}}\to H(\Fpbar),$$
  with morphisms given by conjugation, as in \ref{para:repinner}. 
    We write $\rhobar\mapsto M(\rhobar)$ for the quasi-inverse functor.
  \item\label{v3-item:3} If~$\rhobar:\Gal_{\Q_{p,\infty}}\to H(\Fpbar)$ factors through~$H'(\Fpbar)$ for some algebraic subgroup~$H'$ of~$H$, then $M(\rhobar)$ admits a reduction of structure group to~$H'$.
  \end{enumerate}
\end{prop}
\begin{proof} Taking exact $\otimes$-functors from $\Rep_{\Fpbar} H$ to the
two sides of the equivalence \eqref{v3-eqn:Fontainecorresp}, we see that 
$H$-$\Mod_{/ \Fpbar((u))}^{\varphi}$ is equivalent to the category of
$H$-objects in $\Rep_{\Fpbar} \Gal_{\Q_{p,\infty}}.$ Thus we have to check that the
latter category is equivalent to the category of continuous representations
$\rhobar:\Gal_{\Q_{p,\infty}}\to H(\Fpbar).$ The argument for this is completely 
analogous to the proof of Lemma \ref{lem:crystallineGreps}.

Part \eqref{v3-item:3} follows formally from~\eqref{v3-item:11}.
\end{proof}

If $H$ is a smooth affine group scheme over $\Fpbar$,
and $\rhobar: \Gal_{\Qp} \rightarrow H(\Fpbar)$ is  a continuous representation,
we will write $M(\rhobar) \coloneq  M(\rhobar|_{\Gal_{\Q_{p,\infty}}}).$

\begin{para} Consider a continuous representation of $\Gal_{\Q_p}$ on a finite free $\Zpbar$-module $V.$
By the Baire category theorem, there is a finite extension~$F/\Qp$ such that this descends to a $\Gal_{\Q_p}$-representation on  a finite free $\O_F$-module $V^F.$
We say that $V$ is crystalline if $V^F$ is crystalline, a condition which does not depend on the choice of $F$ or $V^F.$
We denote by $\Rep^{\crys}_{\Zpbar}(\Gal_{\Q_p})$ the category of such crystalline representations $V.$
\end{para}

Now suppose that $H/\Zpbar$ is an extension of a finite \'etale group scheme
by a parahoric. 
As in \ref{para:repinner}, we define the groupoid of continuous representations $\rho: \operatorname{Gal}_{\mathbf{Q}_p}\rightarrow H(\overline{\mathbf{Z}}_p )$
 by declaring $\Hom(\rho,\rho')$ to be the set of $h \in H(\Zpbar)$ conjugating $\rho$ into $\rho'.$

Fix such a representation~$\rho$.
By the Baire category theorem, there is a finite extension $F/\Q_p,$ contained in $\Qpbar,$ such that  $H$ is defined over $\O_F,$ and $\rho$ factors through $H(\O_F).$ We denote by $\rho^F$ the corresponding $H(\O_F)$-valued representation.
We say that $\rho:\Gal_{\Qp} \rightarrow H(\Zpbar)$  is crystalline, if $\rho^F$ is crystalline.
This condition does not depend on the choice of $F.$

The $H(\O_F)$-valued representation $\rho^F$ may be thought of as an object of
$H$-$\Rep_{\Z_p}^{\crys} (\Gal_{\Qp}),$ by Lemma \ref{lem:crystallineGreps}, where we now think of $H$ as defined over $\O_F.$
By Theorem \ref{thm:prismaticKeylemma} and evaluation on the Breuil--Kisin prism we obtain an object  of $H$-$\Mod^{\varphi}_{/\Zp\llbracket u \rrbracket},$
and thus an object $\gM(\rho)$  of  $H$-$\Mod^{\varphi}_{/\Zpbar\llbracket u \rrbracket}$ by extension of scalars to 
$\Zpbar\llbracket u \rrbracket.$ 

\begin{prop}\label{v3-prop:crysreptoBMmodules} Suppose that $H/\Zpbar$ is an extension of a finite \'etale group scheme
by a parahoric.
Then there is a natural functor $\rho \mapsto \mathfrak{M}(\rho)$ from the category of crystalline $H$-representations
$ \rho:\Gal_{\Qp} \rightarrow H(\Zpbar)$
to the category $H$-$\Mod^{\varphi}_{/\Zpbar\llbracket u \rrbracket}$. Writing
$$\overline{\gM}(\rho) = \gM(\rho)\otimes_{\Zpbar}\Fpbar,$$
there is a natural isomorphism
$$ \overline{\gM}(\rho)[1/u] \simeq M(\rhobar),$$ where we write $\rhobar:\Gal_{\Qp} \rightarrow H(\Fpbar)$ for the composite of $\rho$ and the projection $H(\Zpbar) \rightarrow H(\Fpbar).$  
\end{prop}
\begin{proof} The functor $\rho \mapsto \mathfrak{M}(\rho)$ was described immediately above; the second part follows from
\cite[Rem.~7.11]{bhatt2021prismatic} and \cite[Prop.~2.1.5]{KisinCrys}.
\end{proof}

\subsection{\texorpdfstring{$L$}{L}-groups}\label{subsec:L-groups}
We will work throughout the rest of this section with Galois representations valued in a class of not-necessarily-connected reductive groups, namely the $L$-groups of unramified reductive groups.
We recall the definition of this class of groups in more detail in Section~\ref{subsec:unramified-groups} below, but in this section, it suffices to note that an unramified $L$-group~$\LG$ is in particular a group scheme over~$\Zp$ of the form  \[\LG = \Ghat \rtimes \Gal (L/\Qp),\] where
\begin{itemize}[wide]
\item $\Ghat$ is a split connected reductive group (i.e.\ a Chevalley group) over $\Zp$, with Borel~ $\Bhat$ and maximal torus $\That\subset\Bhat$,
\item $L/\Qp$ is a finite unramified extension, and
\item the semidirect product is defined by an action of~$\Gal(L/\Qp)$ on~$\Ghat$ which preserves~$\Bhat$ and~$\That$, and acts on the set of simple roots~$\Delta^{\vee}$.
\end{itemize}
Let~$W$ denote the Weyl group of~$\That$.
We write $\gamma\in\Gal(L/\Qp)$ for the geometric Frobenius.
If~$x\in\Ghat$ then we write $x^{\gamma}\coloneq  \gamma x \gamma^{-1}$.
We have $N_{\LG}(\That)=N_{\Ghat}(\That)\rtimes \Gal(L/\Qp)$, and for $g\in N_{\LG}(\That)$ (e.g.\ $g=\gamma$) we write $t^g\coloneq  gtg^{-1}$.
This induces an action of $N_{\LG}(\That)$ on~$X_{*}(\That)$; for~$\lambda\in X_{*}(\That)$, $g\in N_{\LG}(\That)$, we have $(g\lambda)(t)=\lambda(t)^{g}$.

\begin{defn}\label{defn:L-parameter} Let $A$ be a topological $\Zp$-algebra. 
 An \emph{$L$-parameter} is a continuous homomorphism $\Gal_\Qp \to \LG (A),$ 
 which is compatible with the projections to $\Gal (L/\Qp)$.
We make the collection of $L$-parameters into a groupoid by setting $\Hom(\rho_1,\rho_2)$ equal to the set of $g \in \Ghat(A)$
which conjugate $\rho_1$ into $\rho_2.$ We say that $\rho_1$, $\rho_2$ are \emph{equivalent} if such a $g$ exists.
\end{defn}

\begin{para} Let $\Gamma_0 = \Gal(L/\Q_p)$ thought of as a constant group scheme over $\Zpbar$,
and let $\omega_{\triv,L}$ in $\Gamma_0$-$\Rep^{\crys}_{\Zpbar}$
be the functor which takes a $\Gamma_0$-representation $W$ to $W$ thought of as a $\Gal_{\Q_p}$-representation via the projection to $\operatorname{Gal}(L/\mathbf{Q}_p )$.
The $\Gamma_0$-torsor corresponding to $\omega_{\triv,L}$ is trivial. 

Applying the equivalence of Theorem \ref{thm:prismaticKeylemma} to $\omega_{\triv,L}$ produces an object
$\gM(\omega_{\triv,L})$ of $\Gamma_0$-$\Mod^{\varphi}_{/\Zpbar\llbracket u \rrbracket}.$ 
The corresponding $\Gamma_0$-torsor over $\Zpbar\llbracket u \rrbracket$ is pulled back  from a $\Gamma_0$-torsor over 
$\Zpbar,$ which 
is the spectrum of
$$(\O_{\Gamma_0}\otimes_{\Z_p} \O_L)^{\Gamma_0} \simeq \Zpbar\otimes_{\Zp}\O_L$$ with the isomorphism being induced by 
the counit $\O_{\Gamma_0} \rightarrow \Zpbar.$ 
Here the action of $\Gamma_0$ on $\O_{\Gamma_0}$ is induced by $\Gamma_0$ acting on itself by right translation, while the 
$\Gamma_0$-torsor structure is induced by $\Gamma_0$ acting on itself by left translation. (Of course since $\Gamma_0$ is commutative the left and right actions are equal). 

Let $\varphi$ denote the Frobenius on $\O_L.$ On elements of $(\O_{\Gamma_0}\otimes_{\Z_p} \O_L)^{\Gamma_0},$ 
$1\otimes \varphi$ acts as $\gamma\otimes 1,$ so the action of $\varphi$ on the $\Gamma_0$-torsor $\gM(\omega_{\triv,L})$ is 
via the element $\gamma \in \Gamma_0(\Zpbar).$

For $\gM$ in $\LG$-$\Mod^{\varphi}_{\Zpbar\llbracket u \rrbracket}$ (respectively  $\LG$-$\Mod^{\varphi}_{\Fpbar\llbracket u \rrbracket},$) we denote by $\gM_0$  the induced object of  $\Gamma_0$-$\Mod^{\varphi}_{/\Zpbar\llbracket u \rrbracket}$ (respectively  $\Gamma_0$-$\Mod^{\varphi}_{/\Fpbar\llbracket u \rrbracket}$).
\begin{defn}
  We define an $\LG$-Breuil--Kisin module to be an object $\gM$ in $\LG$-$\Mod^{\varphi}_{\Zpbar\llbracket u \rrbracket},$ together with an isomorphism $\gM_0 \simeq \gM(\omega_{\triv,L})$ in $\Gamma_0$-$\Mod^{\varphi}_{/\Zpbar\llbracket u \rrbracket}.$ Similarly, we define an $\LG$-Breuil--Kisin module over $\Fpbar$ to be an object $\gM$ in $\LG$-$\Mod^{\varphi}_{\Fpbar\llbracket u \rrbracket},$ together with an isomorphism $\gM_0 \simeq \gM(\omega_{\triv,L})\otimes_{\Zpbar}\Fpbar$ in
  $\Gamma_0$-$\Mod^{\varphi}_{\Fpbar\llbracket u \rrbracket}$.  Similarly an $\LG$-\'etale $\varphi$-module is defined to be an object $M$ of $\LG$-$\Mod^{\varphi}_{\Fpbar((u))},$ together with an isomorphism $M_0 \simeq \gM(\omega_{\triv,L})\otimes_{\Zpbar \llbracket u \rrbracket}\Fpbar((u)),$ where $M_0$ denotes the object of $\Gamma_0$-$\Mod^{\varphi}_{\Fpbar((u))}$ induced by $M.$
\end{defn}
\end{para}

\begin{prop}\label{v3-prop:crysreptoBMmodules-again}
The functor in Proposition \ref{v3-prop:crysreptoBMmodules} induces functors 
$$ \rho \mapsto \gM(\rho), \quad \rho \mapsto M(\barrho)$$ 
from the category of $L$-parameters $\rho:\Gal_{\Qp} \rightarrow \LG(\Zpbar)$ 
to  the category of $\LG$-Breuil--Kisin modules and the category of $\LG$-\'etale $\varphi$-modules respectively. 
Moreover, writing
$$\overline{\gM}(\rho) = \gM(\rho)\otimes_{\Zpbar}\Fpbar,$$
there is a natural isomorphism of $\LG$-\'etale $\varphi$-modules
$$ \overline{\gM}(\rho)[1/u] \simeq M(\rhobar).$$
\end{prop}
\begin{proof} This is a formal consequence of Proposition \ref{v3-prop:crysreptoBMmodules} 
applied with $H = \LG$ and $H = \Gamma_0.$
\end{proof}

\begin{para} Let $\widetilde\gM(\omega_{\triv,L})$ be the $\LG$-torsor over $\Zpbar$, induced from $\gM(\omega_{\triv,L})$ by the canonical inclusion $\Gamma_0 \rightarrow \LG.$
 Let $\gM$ be an $\LG$-Breuil--Kisin module; bearing in mind Lemma~\ref{lem:Frobtorsor}, we will think of $\gM$ as a $\LG$-torsor over $\overline{\mathbf{Z}}_p \llbracket u \rrbracket$, equipped with an isomorphism $\Phi_{\gM}: \varphi^*\gM[1/E] \simeq \gM[1/E]$. 
Since $\Zpbar$ is strictly Henselian, $\gM$ is a trivial $\LG$-torsor, and there exists an isomorphism of $\LG$-torsors over $\Zpbar\llbracket u \rrbracket$, 
$\gM \simeq \widetilde\gM(\omega_{\triv,L})$ which lifts the isomorphism $\gM_0 \simeq \gM(\omega_{\triv,L}).$
As $\widetilde\gM(\omega_{\triv,L})$ is defined over $\Zpbar,$ pulling back by $\varphi$ induces an isomorphism 
$\varphi^*\gM \simeq \widetilde\gM(\omega_{\triv,L})$, so that the isomorphism $\Phi_{\gM}: \varphi^*\gM[1/E] \simeq \gM[1/E]$ is given by 
an element $x\cdot \gamma$ for some $x \in \Ghat(\Zpbar \llbracket u \rrbracket [1/E]),$ as 
$\varphi$ acts on $\gM(\omega_{\triv,L})$ via the geometric Frobenius $\gamma.$ 

Changing the isomorphism $\gM \simeq \widetilde\gM(\omega_{\triv,L})$ has the effect of changing~$x$ by 
$\gamma$-twisted Frobenius-conjugation. That is, $x$ is replaced by 
\begin{equation}\label{v3-phi-conjugation-G-tuple}
 g^{-1}\cdot x\cdot \varphi^\gamma(g)
  \end{equation}for some  $g \in \Ghat(\Zpbar \llbracket u \rrbracket)$, where we write
  \[
    \varphi^{\gamma}(g)\coloneq  (\varphi(g))^{\gamma}=\varphi(g^{\gamma}).
  \] 
  Similarly, up to $\gamma$-twisted Frobenius-conjugacy, $x$ depends only on the $\widehat G$-conjugacy class 
  of the splitting $\Gamma_0 \rightarrow \LG.$ 
  We will write $\gM \sim x.$

Similarly, if $\gM$ is an $\LG$-Breuil--Kisin module over $\Fpbar$ 
(resp.\  $\LG$-\'etale $\varphi$-module), we can assign to $\gM$ an element
$x \in \Ghat(\Fpbar ((u))),$ which is well defined up to $\gamma$-twisted Frobenius
conjugation by elements $g\in \Ghat(\Fpbar \llbracket u \rrbracket),$
(resp.~$g\in \Ghat(\Fpbar (( u )))$). Here in the case of
$\LG$-\'etale $\varphi$-modules we are using Steinberg's theorem that
a $\Ghat$-bundle  over $\Fpbar ((u))$ is trivial. We will again write
$\gM \sim x.$ 
\end{para}
  
\begin{defn}
  \label{v3-defn:shape}Let ~$\gM$ be an $\LG$-Breuil--Kisin module over $\Fpbar$, with $\gM \sim x.$
 Note that  if we view $x\in \Ghat(\Fpbar((u)))$ as an element of $\Gr^{sw}_{\Ghat}(\Fpbar),$ then by~\eqref{v3-phi-conjugation-G-tuple} the $L_p^{+}\Ghat$-orbit of $x$ depends only
  on $\gM$ and not on the choice of $x$, and thus the same holds for the $L^+\Ghat$-orbit. The \emph{shape} $s(\gM)$ is the element
  $\lambda\in X_*(\That)^{+}$ such that
\begin{equation*}
\label{v3-eq:20}
x\in \Gr_{\Ghat}^{sw,\lambda}(\Fpbar).
\end{equation*}
(The terminology \emph{shape} follows that of Booher--Levin, \cite[Defn.~3.1.9]{MR4616433}.)
\end{defn}

\begin{para}
The following theorem is a reinterpretation of some of the results of Section~\ref{sec:Frob-descent-F-gauge} in the terminology introduced in this section; it is the only result from Section~\ref{sec:Frob-descent-F-gauge} that we will use in the rest of this section.
\end{para}
\begin{thm}
   \label{v3-thm:shape-uparrow}Let ~$\rho:\Gal_{\Qp}\to \LG(\Zpbar)$ be  a crystalline $L$-parameter. Then:
   \begin{enumerate}
\item\label{v3-item:35}  $\overline \gM(\rho)$ is $\mu_p$-equivariant.
   \item\label{v3-item:34}  $s(\overline \gM(\rho))\uparrow \mu(\rho)$.
\end{enumerate}
 \end{thm}
 \begin{proof} Part~\eqref{v3-item:35} follows from Corollary~\ref{cor:modpcrysGdesc} 
   and Theorem~\ref{thm:prismaticKeylemma}.
Part~\eqref{v3-item:34} follows from Theorem \ref{thm:alpha-E-uparrow} and Proposition \ref{prop:HTweights}, 
applied to the restriction of $\rho$ to a finite unramified extension of $\Q_p$; indeed, these show that \[ s(\gMbar(\rho)) = \alpha_{\mathcal E}(\Fpbar) \uparrow \alpha_{\mathcal E}(\Qpbar) = \mu(\rho),\]as required.  \end{proof}

  \begin{para}\label{v3-para:small-weight-shape}
When the Hodge--Tate weights are sufficiently small, we have the following corollary, which together with Corollary~\ref{cor:modpcrysGdesc} 
in particular recovers~\cite[Thm.~4.22]{GLSII},  which was the key  input from $p$-adic Hodge theory used in the proof of the Buzzard--Diamond--Jarvis conjecture.
\end{para}
\begin{cor}\label{v3-cor:shape-p-small-weight} Let~$\rho:\Gal_{\Qp}\to \LG(\Zpbar)$ be a crystalline $L$-parameter such that $\mu(\rho)\in \overline{\Delta}_p$ (equivalently, for each positive root $\alpha$ we have $0\le \langle \alpha,\mu(\rho)\rangle\le p$). Then $s(\overline \gM(\rho))=\mu(\rho).$ Moreover, we have $\gM(\rho)\sim \mu(\rho)(E)X$ for some~$X\in \Ghat(\Zpbar\llbracket u\rrbracket)$.
\end{cor}
\begin{proof}By Theorem~\ref{v3-thm:shape-uparrow} we have $s(\overline \gM(\rho))\uparrow \mu(\rho)$.
Since  $\mu(\rho)\in \overline{\Delta}_p$, and $ \overline{\Delta}_p$ is the minimal dominant alcove for the~$\uparrow$ order, the first part is immediate from Lemma~\ref{lem:equivalence-uparrow-alcoves-weights}.
The second part follows from the first.
To see this, write~$\Gr^{sw}_{\Ghat,\Zpbar}$ for the switched affine Grassmannian in the loop variable~$E$. As  recalled in \ref{para:integral-Schubert}, there is an open subscheme~$\Gr^{sw,\mu(\rho)}_{\Ghat,\Zpbar}$ of~$\Gr^{sw}_{\Ghat,\Zpbar}$ given by the $L^+\Ghat$ orbit of~$\mu(\rho)(E)$.
The Breuil--Kisin module $\gM(\rho)$ determines a morphism $\Spf\Zpbar\to \Gr^{sw}_{\Ghat,\Zpbar}$, whose special fibre factors through~$\Gr^{sw,\mu(\rho)}_{\Ghat,\Fpbar}$ by the first part.
Thus the morphism $\Spf\Zpbar\to \Gr^{sw}_{\Ghat,\Zpbar}$ also factors through ~$\Gr^{sw,\mu(\rho)}_{\Ghat,\Zpbar}$, as required.
\end{proof}

\subsection{Semisimple mod \texorpdfstring{$p$}{p} representations and \'etale \texorpdfstring{$\varphi$}{phi}-modules.}
We will make use of the results on $L$-parameters $\Gal_{\Qp}\to \LG(\Fpbar)$ proved in~\cite[\S 3]{lin2023delignelusztigtypecorrespondencetame}, which works in the more general setting of $L$-groups of tamely ramified reductive groups.
Following this reference, we make the following definitions.
\begin{defn}
  \label{defn:standard-Levi}Let~$S\subseteq \Delta^{\vee}$ be a $\Gal(L/\Qp)$-stable subset, let~$\Mhat_S$ be the corresponding Levi subgroup of~$\Ghat$, and let~$\Phat_S$ be the parabolic subgroup of~$\Ghat$ which contains~$\Bhat$ and has Levi component~$\Mhat_S$.
  We call $\LM_S\coloneq \Mhat_S\rtimes \Gal(L/\Qp)$ (resp.\ $\LP_S\coloneq \Phat_S\rtimes\Gal(L/\Qp)$) a \emph{standard Levi subgroup} of~$\LG$ (resp.\ \emph{standard parabolic subgroup} of $\LG$).
  Then a \emph{Levi subgroup} (resp.\ \emph{parabolic subgroup}) of~$\LG$ is a subgroup~$\LM$ which is $\Ghat$-conjugate to a standard Levi subgroup (resp.\ a subgroup~$\LP$ which is $\Ghat$-conjugate to a standard parabolic subgroup).
\end{defn}

\begin{defn}\label{v3-defn:irred-ss}
Let $\Gamma$ be a profinite group.
    We say that a representation $\rhobar:\Gamma \to \LG(\Fpbar)$ is \emph{$\LG$-irreducible}, or simply \emph{irreducible}, if it does not factor through~$\LP(\Fpbar)$ for any proper parabolic subgroup~$\LP$ of~$\LG$. 
  We say that ~$\rhobar$ is
  \emph{semisimple} if whenever it factors through some parabolic~$\LP(\Fpbar)$, it factors through~$\LM(\Fpbar)$ for some Levi subgroup~$\LM$ of~$\LP$.
  \end{defn}

The following  result of Lin~\cite{lin2023delignelusztigtypecorrespondencetame} will be important in the sequel.
\begin{prop}Suppose that~$\rhobar:\Gal_{\Qp}\to \LG(\Fpbar)$ is semisimple. Then:
  \label{v3-prop:rhobar-irreducible-classification}
  \begin{enumerate}
  \item\label{v3-item:1}$\rhobar$ is tamely ramified.
  \item\label{v3-item:4} After replacing $\rhobar$ by an equivalent representation, we have $\rhobar(I_{\Qp})\subseteq \That(\Fpbar)$ and $\rhobar(\Gal_{\Qp})\subseteq N_{\LG}(\That)(\Fpbar)$.
  \item\label{v3-item:5} If~$\rhobar$ as in~\eqref{v3-item:4} is irreducible, then $(\Ghat^{\rhobar(I_{\Qp})})^{\circ}=\That$.
  \item $\rhobar|_{\Gal_{\Q_{p,\infty}}}$ is semisimple.
\end{enumerate}
\end{prop}
\begin{proof}Recalling that all maximal tori of~$\Ghat$ are conjugate over $\Fpbar,$ the first two parts are immediate from ~\cite[Lem.~3.2.6, Thm.~3.2.7]{lin2023delignelusztigtypecorrespondencetame}, and the third part follows from ~\cite[Thm.~2.3.5]{lin2023delignelusztigtypecorrespondencetame}.
For the final claim, it suffices to show that  $\rhobar(\Gal_{\Qp}) = \rhobar(G_{\Q_{p,\infty}}).$ Since $\mathbf{Q}_{p,\infty}/\mathbf{Q}_p $ is totally ramified, it suffices in turn to show that   $\rhobar(I_{\Qp}) = \rhobar(I_{\Q_{p,\infty}}).$  Since  $[\rhobar(I_{\Qp}):\rhobar(I_{\Q_{p,\infty}})]$ is a power of~$p,$ and~$\rhobar(I_\Qp)$ has order prime to~$p$ (because $\rhobar$ is tame), this is clear. 
\end{proof}

\begin{para}
As in \ref{para:notn-affine-Weyl}, we now fix a lift of each~$w\in W$ to $N_{\Ghat}(\That)(\Fpbar)$, which we continue to denote by~$w$.
\end{para}

\begin{lem}
  \label{v3-lem:semisimple-etale-phi-structure}If~$\rhobar:\Gal_{\Qp} \to \LG(\Fpbar)$ is semisimple, then  $M(\rhobar) \sim tu^{\lambda}w$ for some $t\in \That(\Fpbar)$, $\lambda\in X_*(\That)$, and $w\in W$.
\end{lem}
\begin{proof}
  By Propositions~\ref{v3-prop:rhobar-irreducible-classification} and~\ref{v3-prop:standard-properties-etale-phi-modules}, after replacing~$\rhobar$ by an equivalent $L$-parameter, we may assume that  $M(\rhobar)$  admits a reduction of structure group to an $N_{\LG}(\That)$-\'etale $\varphi$-module, so that $M(\rhobar)\sim x w$ for some $x\in \That(\Fpbar((u)))$ and $w\in W$.
   By~\eqref{v3-phi-conjugation-G-tuple}, we have  $M(\rhobar)\sim g^{-1}\cdot \varphi(g^{w\gamma})\cdot x w$ for any~$g\in \That(\Fpbar((u)))$, so it suffices to show that  the map \[	\That(1+u\Fpbar\llbracket u\rrbracket)\to 	\That(1+u\Fpbar\llbracket u\rrbracket),\] \[g\mapsto g^{-1}\cdot\varphi(g^{w\gamma})\]is a bijection.
   This is clear; indeed, an inverse is given by the  map \[g\mapsto \prod_{n\ge 0}\varphi^{n}(g^{(w\gamma)^{n}})^{-1}.
  \qedhere\]
\end{proof}

\begin{rem}\label{v3-rem:non-uniqueness-of-semisimple-description}The element $tu^{\lambda}w$ in Lemma~\ref{v3-lem:semisimple-etale-phi-structure} is of course far from  uniquely determined; indeed, we can modify it by $\gamma$-twisted $\varphi$-conjugation by any element $g\in N_{\Ghat}(\That)(\Fpbar((u)))$.
  In particular, if we take~$g=u^\nu \sigma$ for some ~$\nu\in X_{*}(\That)$, $\sigma\in W$, then we see that
  \begin{equation}
    \label{v3-eq:21}tu^{\lambda}w\sim
    t'u^{\sigma^{-1}(\lambda-\nu +p(w\gamma)\nu)}\cdot(\sigma^{-1}w\sigma^{\gamma})
   \end{equation}
  for some~$t'$.
 \end{rem}
\begin{para}  
For each~$n\ge 1$ we have a fundamental character  $\omega_n:I_{\Qp}\to\Fpbartimes$ given by $g\mapsto g(\sqrt[p^n-1]{p})/\sqrt[p^n-1]{p}$; 
in particular~$\omega_1$ is the mod~$p$ cyclotomic character.
For any $\lambda\in X_*(\That)$  we write \[\omega_n^{\lambda}\coloneq \lambda\circ\omega_n:I_{\Qp}\to \That(\Fpbar). \]
\end{para}

\begin{para}\label{v3-para:defntau} 
Let~$d\ge 1$ be any integer such
  that~$(w\gamma)^d=1$; note that in particular this implies that~$\Q_{p^{d}}$ contains~$L$.
 Write
   \begin{equation}
    \label{v3-eq:formula-for-inertial-character}
    \tau(\lambda ,w )\coloneq  \omega_d^{-\sum_{n=0}^{d-1}p^n(w\gamma)^{n}\lambda}: I_{\Qp}\to \That(\Fpbar)\subset \LG(\Fpbar).
  \end{equation} It is easy to check that this 
  does not depend on the choice of 
  ~$d$.

  Given~$\lambda\in X_{*}(\That)$, there is a unique element $e\in X_*(\That)_{\Q}$  such that \begin{equation}\label{v3-eqn:cyclic-equation-ej}e=\lambda+pw\gamma e;\end{equation} indeed, we have  \begin{equation}\label{v3-eqn:defn-of-ei}e\coloneq \frac{1}{1-p^{d}}\sum_{n=0}^{d-1}p^{n}(w\gamma)^{n}\lambda,  \end{equation}
  so that
 \begin{equation}  \label{v3-eq:18}   \tau(\lambda ,w )=\omega_d^{(p^{d}-1)e}.\end{equation}
\end{para}

\begin{lem}
  \label{v3-lem:inertia-representation-from-etale-phi} Suppose that $\rhobar:\Gal_{\Qp} \to \LG(\Fpbar)$ is semisimple, and write $M(\rhobar) \sim tu^{\lambda}w$ as in Lemma~\ref{v3-lem:semisimple-etale-phi-structure}.  
  Then ~$\rhobar$ is equivalent to an~$L$-parameter which factors through 
  $N_{\LG}(\That)(\Fpbar)$.
  Furthermore:
  \begin{enumerate}
  \item     $
      \rhobar|_{I_{\Qp}}\cong \tau(\lambda ,w ).
    $
  \item  The composite representation
  \begin{equation}\label{v3-eq:cuhntkl4s5}
    \Gal_{\Qp}\xrightarrow{\rhobar}N_{\LG}(\That)(\Fpbar)\to (N_{\LG}(\That)/\That)(\Fpbar)=W\rtimes\Gal(L/\Qp)
  \end{equation}is unramified, and takes  geometric Frobenius  to $w\cdot \gamma$.
  \item   If~$\rhobar$ is irreducible, then  for  each root $\alpha$ we have
  \begin{equation}
    \label{v3-eq:irreducibility-implies-fixed-point-not-integral}
    \langle
    e,\alpha\rangle\not\in\Z.
  \end{equation}

  \end{enumerate}
\end{lem}
\begin{proof}
  Since $M(\rhobar)\sim tu^{\lambda}w$, we see that~$M(\rhobar)$ admits a reduction of structure group to  $N_{\LG}(\That)$, so ~$\rhobar$ is equivalent to an $L$-parameter factoring through  $N_{\LG}(\That)(\Fpbar)$. Furthermore, the unramified representation~\eqref{v3-eq:cuhntkl4s5} corresponds to the unramified $N_{\LG}(\That)$-\'etale $\varphi$-module given by~$w$, which proves part (2).

  To compute~$\rhobar|_{I_{\Qp}}$, we choose as above some $d\ge 1$ such that $(w\gamma)^{d}=1$, and consider the \'etale $\varphi$-module~$M(\rhobar|_{\Gal_{\Q_{p^{d}}}})$; by e.g.\ \cite[\S 6.2]{MR4055172}, we have $M(\rhobar|_{\Gal_{\Q_{p^{d}}}})=\F_{p^{d}}\otimes_{\Fp}M(\rhobar)$, so that in a slight abuse of notation we again write
  $M(\rhobar|_{\operatorname{Gal}_{\Q_{p^d}}})\sim t u^{\lambda} w.$

  Write $e_0\in \F_{p^{d}}\otimes_{\Fp}\Fpbar$ for the unique idempotent such that $e_0(x\otimes 1)=(1\otimes \kappa(x))e_0$, where $\kappa:\F_{p^d}\to \Fpbar$ is the natural inclusion, 
  and set 
  \[
    g\coloneq \sum_{i=0}^{d-1}(w\gamma)^i\varphi^{i}(e_0 )\in \LG(\F_{p^{d}}\otimes_{\Fp}\Fpbar).
  \]Then $g=w\gamma\varphi(g)$, so that
  \[
    g^{-1}tu^{\lambda}w\gamma\varphi(g)=g^{-1}tu^{\lambda}g=\sum_{i=0}^{d-1}t^{(w\gamma)^{i}}u^{(w\gamma)^i\lambda}\varphi^{-i}(e_0 )\in \That(\F_{p^{d}}\otimes_{\Fp}\Fpbar((u))).
  \]

  Since~$\That$ is a (split) torus, we can reduce to the case ~$\That=\Gm$, which is standard, see e.g.\ \cite[Lem.~3.1.2]{MR3324938} (bearing in mind that in that reference the Galois representations corresponding to \'etale $\varphi$-modules are computed via contravariant functors).

  Suppose finally that~$\rhobar$ is irreducible, and assume for the sake of contradiction that $\langle e,\alpha\rangle\in\Z$ for some~$\alpha$. Since~$\omega_{d}$ has order $p^{d}-1$,
  it follows from ~\eqref{v3-eq:18} that $\rhobar(I_{\Qp})\subseteq Z_{\alpha}(\Fpbar)$, where~$Z_{\alpha}$ denotes the centre of the Levi subgroup of~$\Ghat$ determined by~$\alpha$. This contradicts part~\eqref{v3-item:5} of Proposition~\ref{v3-prop:rhobar-irreducible-classification}.
\end{proof}

\subsection{Semisimplification}\label{v3-subsec:semisimple}

\begin{defn} Let $\Gamma$ be a profinite group.
For any representation~$\rhobar:\Gamma \to \LG(\Fpbar)$, its \emph{semisimplification} $\rhobar^{\semis}$ is obtained as follows: choose ~$\LP$ to be minimal with the property that~$\rhobar$ factors through~$\LP(\Fpbar)$, let~$\LM$ be a Levi subgroup of~$\LP$, and let
  $$\rhobar^{\semis}:\Gamma \to \LP(\Fpbar)\to \LM(\Fpbar)\to \LG(\Fpbar)$$
  be the composite induced by the quotient by the unipotent radical of~$\LP$. Then~$\rhobar^{\semis}$ is semisimple
(and is $\LM$-irreducible as a representation $\Gamma\to \LM(\Fpbar)$), and is well-defined up to equivalence (see~\cite[Prop.~2.15]{quast2023deformationsgvaluedpseudocharacters}).
\end{defn}

\begin{para} For the remainder of this subsection, we fix an $\LG$-Breuil--Kisin module $\gM$ over~$\Fpbar$, with corresponding \'etale $\varphi$-module~$\gM[1/u]$ and Galois representation~$\rhobar:\Gal_{\Q_{p,\infty}}\to \LG(\Fpbar)$.

\end{para}

\begin{lem} \label{v3-lem:reduction-of-structure-BK-parabolic} Let~$\LP$ be a standard parabolic subgroup of~$\LG$. Then the following are equivalent:
  \begin{enumerate}
  \item\label{v3-item:6} $\gM$ admits a reduction of structure group to~$\LP$.
  \item\label{v3-item:7} $\gM[1/u]$ admits a reduction of structure group to~$\LP$.
      \item\label{v3-item:8} $\rhobar$ admits a reduction of structure group to~$\LP$. 
  \end{enumerate}
\end{lem}
\begin{proof}
  The first condition trivially implies the second, which is equivalent to the third by Proposition~\ref{v3-prop:standard-properties-etale-phi-modules}.
  It remains to show that the second condition implies the first.
Assuming the second condition, we have $\gM[1/u]\sim x$ with $x\in \Phat(\Fpbar((u)))$.
Making an arbitrary choice of trivialization, we can write $\gM \sim g^{-1}x\varphi^\gamma(g)$ for some ~$g\in \Ghat(\Fpbar((u)))$.
By the Iwasawa decomposition, 
  we may write $g=bk$ with $b\in \Bhat(\Fpbar((u)))\subseteq \Phat(\Fpbar((u)))$ and~$k\in \Ghat(\Fpbar\llbracket u\rrbracket )$.
After possibly changing the trivialization of~$\gM$, 
  we may suppose that $g=b$, so that $g^{-1}x\varphi^\gamma(g)\in \Phat(\Fpbar((u)))$, as required.
\end{proof}

\begin{lem}  \label{v3-lem:BK-module-parabolic}Suppose that $\gM$ is $\mu_p$-equivariant, and
   write
   $$\rhobar^{\semis}:\Gal_{\Q_{p,\infty}}\to \LM(\Fpbar)\subseteq \LG(\Fpbar)$$ for the semisimplification of~$\rhobar$, where~$\LM$ is a standard Levi subgroup of~$\LG$.
      Then there is a $\mu_p$-equivariant $\LG$-Breuil--Kisin module $\gM^{\semis}$ over~$\Fpbar$ such that:
   \begin{enumerate}
   \item $\gM^{\semis}[1/u]\cong \Mrhobarss$,
   \item $\gM^{\semis} $ admits a reduction of structure group to~$\LM$, and
   \item\label{v3-item:37}  $s(\gM^{\semis} )\uparrow s(\gM)$.
   \end{enumerate}
 \end{lem}
 \begin{rem}
   \label{v3-rem:L-vs-G-be-careful}We caution the reader that the shape~$s(\gM^{\semis} )$ and the relation~$\uparrow$ in part~\eqref{v3-item:37} of Lemma~\ref{v3-lem:BK-module-parabolic} are defined with respect to~$\Ghat$, rather than with respect to~$\Mhat$.
 \end{rem}
   \begin{proof}[Proof of Lemma~\ref{v3-lem:BK-module-parabolic}] Let~$\LP$ be the standard parabolic with Levi $\LM$.
By Lemma~\ref{v3-lem:reduction-of-structure-BK-parabolic}, $\gM$ admits a reduction of structure group to~$\LP$, so we can write $\gM \sim x$ for some ~$x\in \Phat(\Fpbar((u)))$.
 By~\cite[Cor.~2.1.5]{lin2023delignelusztigtypecorrespondencetame}, we can write $\LP=P_{\LG}(\xi)$ for some cocharacter $\xi\in X_*(\That)$, where $\LP_{\LG}(\xi)\subseteq \LG$ is the subgroup of elements~$g$ such that $ \lim_{t\to 0}\xi(t)g \xi(t)^{-1}$ is defined, and this limit then lies in the Levi
 $\LM$. 
 Thus the map $\G_m \rightarrow L\Ghat$  given by $t \mapsto \xi(t)x (\xi(t)^{\gamma})^{-1}$ extends to  $\A^1 \rightarrow L\Ghat,$ and over $0 \in \A^1,$ this map
 factors through $L\LM.$

Then we define a family of Breuil--Kisin modules $\gM_t$ over~$\Gm$ by
 \[
   \gM_t\sim x(t)\coloneq \xi(t)x(\xi(t)^{\gamma})^{-1}.
 \]
 This extends to a family over~$\A^{1}$, and by construction~$\gM^{\semis} \coloneq \gM_0,$ the fiber over
 $0 \in \A^1,$ admits a reduction of structure group to~$\LM$, and satisfies $\gM^{\semis}[1/u]\cong \Mrhobarss$.
Since~$\gM$ is $\mu_p$-equivariant, and the isomorphism class of the
family $\gM_t$ is constant for $t\neq 0,$~$\gM^{\semis}$ is again $\mu_p$-equivariant.
The family $\gM_t$  induces a map $\A^1\to \Gr_{\Ghat}^{sw,\mu_p}$ (sending~$\gM_t$ to ~$x(t)$), so it follows from Proposition~\ref{prop:K-orbit-closure-mu-fixed} that  $s(\gM^{\semis} )\uparrow s(\gM)$, as required.
\end{proof}

\subsection{Breuil--Kisin modules and inertial weights}
\label{v3-subsec:Inertial-weights-of-crystalline}
We now prove our main theorem connecting the Hodge--Tate weights of a crystalline $L$-parameter to the inertial weights of its reduction modulo~$p$.
Recall that a semi-standard Levi subgroup~$\widehat{M}$ of~$\widehat{G}$ is one which contains $\That$.
We say that 
$\LM=\Mhat\rtimes\Gal(L/\Qp)$ is a semi-standard Levi subgroup of~$\LG$ if $\Mhat$ is semi-standard, and we write
$W_{\Mhat}$ for the Weyl group of~$\Mhat$.
\begin{thm}\label{v3-thm:existence-of-semisimple-etale-phi-module-with-invariants}
  Suppose that~$\rho:\Gal_{\Qp}\to \LG(\Zpbar)$ is a crystalline $L$-parameter with Hodge--Tate weights~$\mu(\rho)$. 
  Then there exists a  cocharacter~$\lambda\in X_{*}(\That)$ with $\lambda_{\dom}\uparrow \mu(\rho)$ such that \begin{equation}\label{eqn:t-u-lambda-main-thm}\Mrhobarss \sim tu^{\lambda} w\end{equation} for some
  $t\in \That(\Fpbar)$, $w\in W$, so that in particular
  \begin{equation}\label{eqn:rhobar-on-inertia-in-main-thm}
    \rhobar^{\semis}|_{I_{\Qp}}\cong \tau(\lambda,w).
  \end{equation}
  Furthermore, the pair~$(\lambda,w)$ can be chosen to satisfy either of the following conditions, as desired:
  \begin{enumerate}
  \item\label{item:lambda-is-dominant} $\lambda$ is dominant, or
  \item\label{item:lambda-is-M-dominant-and-w-in-WM} there is a standard Levi subgroup~$\LM$ of~$\LG$ such that  $w\in W_{\Mhat}$, $\lambda$ is $\Mhat$-dominant, and there is an irreducible $L$-parameter
    \[
      \rhobar^{\semis}_{\Mhat}:\Gal_{\Qp}\to \LM(\Fpbar)
    \] such that
    $\rhobar^{\semis}$ is equivalent to the composite
    \begin{equation}\label{eqn:factor-L-parameter-through-irreducible}
      \Gal_{\Qp}\xrightarrow{\rhobar^{\semis}_{\Mhat}} \LM(\Fpbar)\into\LG(\Fpbar).
    \end{equation} 
        \end{enumerate}
  If~$\LG=\Ghat$ then  conditions~\eqref{item:lambda-is-dominant} and~\eqref{item:lambda-is-M-dominant-and-w-in-WM} can be arranged simultaneously, provided that \eqref{item:lambda-is-M-dominant-and-w-in-WM} is weakened to allow semi-standard Levi subgroups.
  \end{thm}
\begin{proof}Applying Proposition~\ref{v3-prop:rhobar-irreducible-classification}~(4) to~$\overline{\rho} ^{\semis}$, we see that $(\overline{\rho}|_{\operatorname{Gal}_{\mathbf{Q}_{p,\infty} }} )^{\semis}=\overline{\rho} ^{\semis}|_{\operatorname{Gal}_{\mathbf{Q}_{p,\infty} }}$.
   By definition, there is a standard Levi subgroup~$\LM$ and an irreducible $L$-parameter  $\rhobar^{\semis}_{\Mhat}:\Gal_{\Qp}\to \LM(\Fpbar)$ such that~$\rhobar^{\semis}$ is equivalent to the composite~\eqref{eqn:factor-L-parameter-through-irreducible}.

Let~$\gM=\gMbar(\rho)$ be as in Proposition~\ref{v3-prop:crysreptoBMmodules}, and let $\gM^{\semis}$ be the $\mu_p$-equivariant Breuil--Kisin module obtained from~$\gM$ in Lemma~\ref{v3-lem:BK-module-parabolic}. By Theorem~\ref{v3-thm:shape-uparrow}  we have
\begin{equation}
\label{v3-eq:19}
s(\gM^{\semis} )\uparrow s(\gMbar(\rho))\uparrow \mu(\rho).
\end{equation}
The $\LG$-Breuil--Kisin module~$\gM^{\semis}$ admits a reduction of structure group to~$\LM$. Applying Lemma~\ref{v3-lem:semisimple-etale-phi-structure} to~$\rhobar^{\semis}_{\Mhat}$, we may write $\Mrhobarss\sim tu^{\lambda}w$ as an $\LM$-\'etale $\varphi$-module, where~$w\in W_{\Mhat}$.
By Lemma~\ref{v3-lem:inertia-representation-from-etale-phi}, we have $\langle e,\alpha\rangle\notin\Z$ for each root~$\alpha\in\Phi_{\Mhat}$.
Following \cite{MR4402497}, we note that this means that the parahoric subgroup of
~$\Mhat(\Fpbar((u)))$ determined by~$e$ is an  Iwahori subgroup~$I_e$.
We write~$I_{e,1}$ for the unipotent radical of~$I_e$.

Write $b \coloneq tu^{\lambda}w$.
  After choosing a trivialization of the $\LM$-structure on~$\gM^{\semis}$, we can write $\gM^{\semis}\sim x \coloneq g^{-1}b\varphi^\gamma(g)$
  for some element~$g\in \Mhat(\Fpbar((u)))$.
 By \cite[Prop.\ 8, Appendix]{MR2435422}, we may write $g=yu^{a}k$ for some
$a\in X_*(\That)$, $y\in I_{e,1}$ and~$k\in L^+\Mhat$. After possibly making a change of trivialization, we may suppose that~$k=1$.

Choose~$d\ge 1$ with~$(w\gamma)^d=1$, and put~$m=p^d-1$; then~$me\in X_*(\That)$ by~\eqref{v3-eqn:defn-of-ei}.
We now consider the family of $\LM$-Breuil--Kisin modules~$\gM_s$ over $\Gm=\Spec\Fpbar[s,s^{-1}]$ for which $\gM_s\sim x_s$, with
\begin{equation}\label{eqn:family-over-Gm-irred}
x_s \coloneq s^{-m(e-a)}x(s^m u)s^{mp\gamma(e-a)};
\end{equation} here $x(s^m u)\in \Mhat(\Fpbar[s,s^{-1}]((u)))$ is the expression obtained from $x=x(u)$ by ``loop rotation''.
We have~$\gM_1=\gM^\semis$, and we claim that this extends to a family of $\LM$-Breuil--Kisin modules over~$\A^1$; more precisely~$x_s$ extends to an element of~$\Mhat(\Fpbar[s]((u)))$ with~$x_0=tu^{\lambda-a+pw\gamma a}w$.

To see this, 
we set $j\coloneq y^{-1}b\varphi^{\gamma}(y)b^{-1}$, so that \[x=g^{-1}b\varphi^\gamma(g)=u^{-a}jbu^{p\gamma a}.\] Noting that by~\eqref{v3-eqn:cyclic-equation-ej} we have \[s^{-me}b(s^m u)s^{mp\gamma e}=b,\]we see that
\begin{align*}
  x_s &= s^{-m(e-a)}x(s^m u)s^{mp\gamma(e-a)}\\ &= u^{-a}\left(s^{-me}j(s^m u)s^{me}\right)bu^{p\gamma a}.
\end{align*}By Lemma~\ref{lem:extension-Iwahori-family} below, \[j_s \coloneq s^{-me}\cdot j(s^m u)\cdot s^{me}\] extends to~$s=0$ with $j_0=1$, so~$x_s$  extends to~$s=0$  with~$x_0=u^{-a}bu^{p\gamma a}=tu^{\lambda-a+pw\gamma a}w$, as required.

Now, for each~$s\ne 0$, ~$x_s$ is obtained from~$x$ by loop rotation and by left and right multiplication by elements of $\Mhat(\Fpbar[s,s^{-1}])$.
These operations preserve both the shape of the Breuil--Kisin module and the property of being $\mu_p$-equivariant, so each~$\gM_s$ is $\mu_p$-equivariant of shape $s(\gM^{\semis})$.
Since $x_0=tu^{\lambda-a+pw\gamma a}w$,
 it follows from Proposition~\ref{prop:K-orbit-closure-mu-fixed} that ~$\gM_0$ has shape \[(\lambda-a+pw\gamma a)_{\dom}=s(\gM_0)\uparrow s(\gM^{\semis})\uparrow \mu(\rho)\](here we have used~\eqref{v3-eq:19}).
Since~$x_0=u^{-a}b\varphi^\gamma(u^a)$, replacing~$\lambda$ by~$\lambda-a+pw\gamma a$ gives~\eqref{eqn:t-u-lambda-main-thm}. The description~\eqref{eqn:rhobar-on-inertia-in-main-thm} of~$\rhobar^{\semis}|_{I_{\Qp}}$ follows from Lemma~\ref{v3-lem:inertia-representation-from-etale-phi}.
 Now let ~$v\in W_{\Mhat}$ be such that $v^{-1}\lambda$ is $\Mhat$-dominant.
By~\eqref{v3-eq:21} we have \[\Mrhobarss \sim t'u^{v^{-1}\lambda} (v^{-1}wv^{\gamma})\] for some~$t'$, so if we replace~$\lambda$ by~$v^{-1}\lambda$ and~$w$ by~$v^{-1}wv^{\gamma}$, we obtain~\eqref{item:lambda-is-M-dominant-and-w-in-WM}.
We can obtain~\eqref{item:lambda-is-dominant} in the same way, this time taking ~$v\in W$  such that $v^{-1}\lambda=\lambda_{\dom}$.
Finally if~$\LG=\Ghat$ then~$\gamma$ is the identity, so $v^{-1}wv^{\gamma}=v^{-1}wv\in W_{v^{-1}\Mhat v}$, and if we replace~$\Mhat$ by~$v^{-1}\Mhat v$, then we can arrange~\eqref{item:lambda-is-dominant} and~\eqref{item:lambda-is-M-dominant-and-w-in-WM} simultaneously, as required.
\end{proof}
The following lemma is a variant on~\cite[Lem.~2.2]{MR4402497}.
\begin{lemma}\label{lem:extension-Iwahori-family}Let~$\LM$ be a standard Levi subgroup of~$\LG$, let $\lambda\in X_*(\That)$,
 $t\in \That(\Fpbar)$, $w\in W_{\Mhat}$ be arbitrary,  and write $b \coloneq tu^{\lambda}w$.
Let ~$e\in X_*(\That)_{\Q}$ be the unique solution of the equation ~\eqref{v3-eqn:cyclic-equation-ej}, and let~$m\in\Z_{\ge 1}$ be such that $me\in X_*(\That)$.
  Assume that  $\langle e,\alpha\rangle\notin\Z$ for each root~$\alpha\in\Phi_{\Mhat}$,
  and let~$I_{e,1}$ be the unipotent radical of the Iwahori subgroup of ~$\Mhat(\Fpbar((u)))$ determined by~$e$.
  Then:
  \begin{enumerate}
  \item the map \[y\mapsto y^{-1}b\varphi^{\gamma}(y)b^{-1}\]takes~$I_{e,1}$ to~$I_{e,1}$.
  \item For any~$j\in I_{e,1}$, the element \[j_s \coloneq s^{-me}\cdot j(s^m u)\cdot s^{me}\in \Mhat(\Fpbar[s,s^{-1}]((u)))\] lies in $\Mhat(\Fpbar[s]((u)))$, and~$j_0=1$.
  \end{enumerate}
\end{lemma}
\begin{proof}For each~$\alpha\in\Phi_{\Mhat}$, choose a parametrization~$u_{\alpha}:\G_a\to U_{\alpha}$ of the corresponding root group.
By~\cite[Rem.\ 10, Appendix]{MR2435422} (more precisely, the formula
$
K_0=T(L)_1\,\mathfrak U_C(\mathcal O_L)
$ in loc.\
cit.), the group~$I_{e,1}$ is generated by the root subgroups $u_{\alpha}(u^{\lceil \langle e,\alpha\rangle\rceil}\Fpbar \llbracket u \rrbracket)$ for ~$\alpha\in \Phi_{\Mhat}$, together with  \[
\widehat T_1:=\ker\!\left(\widehat T(\Fpbar\llbracket u\rrbracket)\xrightarrow{\,u=0\,}\widehat T(\Fpbar)\right).
\]

For the first part, let~$y\in I_{e,1}$ be arbitrary; we must show that $b\varphi^{\gamma}(y)b^{-1}\in I_{e,1}$.
This is clear if~$y\in \widehat T_1$. For the root groups, writing~$\beta=w\gamma\alpha$, we have for~$r\in\Z$
\[
 b\varphi^\gamma\bigl(u_\alpha(u^r\Fpbar\llbracket u\rrbracket)\bigr)b^{-1}
 \subseteq u_\beta\bigl(u^{pr+\langle\lambda,\beta\rangle}\Fpbar\llbracket u\rrbracket\bigr).
\]
By~\eqref{v3-eqn:cyclic-equation-ej},
\[
 pr+\langle\lambda,\beta\rangle
 =p(r-\langle e,\alpha\rangle)+\langle e,\beta\rangle,
\]
and if~$r>\langle e,\alpha\rangle$, the exponent on the right-hand side is greater than~$\langle e,\beta\rangle$, as required.

For the second part, it again 
suffices to consider the cases that  $j\in \widehat T_1$ or $j\in u_{\alpha}(u^{\lceil \langle e,\alpha\rangle\rceil}\Fpbar \llbracket u \rrbracket)$ for some~$\alpha$.
The case $j\in \widehat T_1$ is clear. For the root subgroups, let~$f(u)=\sum_{r= \lceil\langle e,\alpha\rangle\rceil}^\infty c_r u^r$ with~$c_r\in\Fpbar$. Then
\[
 s^{-me}u_\alpha(f(s^m u))s^{me}
 =u_\alpha\left(\sum_{r= \lceil\langle e,\alpha\rangle\rceil}^\infty c_r s^{m(r-\langle e,\alpha\rangle)}u^r\right).
\]
Each exponent~$m(r-\langle e,\alpha\rangle)$ is a positive integer, so the series belongs to~$\Fpbar[s]((u))$ and specializes to~$0$ at~$s=0$, as required.\end{proof}
\begin{rem}
  \label{rem:cannot-always-be-dominant-and-in-M}In general if~$L\ne \Qp$, so that $\LG\ne \Ghat$, then we cannot simultaneously arrange conditions ~\eqref{item:lambda-is-dominant} and~\eqref{item:lambda-is-M-dominant-and-w-in-WM} in Theorem~\ref{v3-thm:existence-of-semisimple-etale-phi-module-with-invariants}.
  For example, take $\Ghat=\GL_2\times\GL_{2} $ and~$L=\Q_{p^{2}}$, with $(g_1 ,g_2 )^{\gamma}=(g_2 ,g_1 )$.
  (As explained in Section~\ref{subsec:K-unramified-representations-via-induction} below, this example is equivalent to studying crystalline representations~$\Gal_{\Q_{p^{2}}}\to \GL_2 (\Zpbar)$.)
Then if $\gMbar(\rho)\sim (u^{\lambda_1 },u^{-\lambda_2 })$ with~$\lambda_1 ,\lambda_2 $ strictly dominant, then~$\rhobar^{\semis}$ factors through~$\That$, but  $\gMbar(\rho)\not\sim (u^{\lambda_1 },u^{\lambda_2 })$.
  \end{rem}

\begin{rem}
  \label{rem:not-best-possible-BM}We do not know for which groups~$G$ and Hodge--Tate cocharacters~$\mu(\rho)$ the results of Theorems~\ref{v3-thm:shape-uparrow} and~\ref{v3-thm:existence-of-semisimple-etale-phi-module-with-invariants} are optimal.
  As we explain in Section~\ref{sec:weight-Serre-conjecture}, we find it plausible that Theorem~\ref{v3-thm:existence-of-semisimple-etale-phi-module-with-invariants} is optimal whenever $0\le \langle \alpha,\mu(\rho)\rangle\le p$ for all simple roots~$\alpha$.
  
  It is, however, possible to find examples of groups~$G$ and cocharacters~$\mu(\rho)$ for which these results are not optimal.
  One such examples is non-split groups, and in particular restrictions of scalars of split groups over unramified extensions of~$\Qp$ (for which see Section~\ref{subsec:K-unramified-representations-via-induction} below).
  Already for~$G=\Res_{\Q_{p^2}/\Qp}\GL_2$, it is easy to see that the results of Barlett's paper~\cite{MR4840254} provide additional constraints in some special cases.
  In a forthcoming sequel to this paper, we will use our methods to obtain~$\mu_{p^n}$-equivariant structures for~$n>1$, which we apply to give new constraints for general non-split groups, in particular explaining this class of examples.    However even for the split group~$\GL_3/\Qp$,  calculations using  the the Breuil--M\'ezard conjecture suggest that Theorem~\ref{v3-thm:existence-of-semisimple-etale-phi-module-with-invariants} fails to be optimal for at least some~$\mu(\rho)$.
\end{rem}

\begin{para}In the same way that Corollary~\ref{v3-cor:shape-p-small-weight} followed from Theorem~\ref{v3-thm:shape-uparrow}, if the Hodge--Tate weights are sufficiently small then we have the following result, which in the case~$G=\GL_n$ recovers~\cite[Thm.~1.0.1]{MR4055172}.

\begin{cor}
  \label{v3-cor:existence-of-semisimple-etale-phi-module-with-invariants=p-small}
 Suppose that~$\rho:\Gal_{\Qp}\to \LG(\Zpbar)$ is a crystalline $L$-parameter with Hodge--Tate weights~$\mu(\rho)$, and that  $\mu(\rho)\in \overline{\Delta}_p$  (i.e.\ for each positive root $\alpha$ of~$\Ghat$ we have $0\le \langle \alpha,\mu(\rho)\rangle\le p$). 

 Then the conclusions of Theorem~\ref{v3-thm:existence-of-semisimple-etale-phi-module-with-invariants} hold, with $\lambda_{\dom}=\mu(\rho)$.
\end{cor}
\begin{proof}
  Just as in the proof of Corollary~\ref{v3-cor:shape-p-small-weight}, if $\lambda_{\dom}\uparrow \mu(\rho)$ then~$\lambda_{\dom}=\mu(\rho)$, so the result follows immediately from Theorem~\ref{v3-thm:existence-of-semisimple-etale-phi-module-with-invariants}.
\end{proof}
\end{para}

\subsection{Representations of~\texorpdfstring{$\Gal_{K}$}{GalK}, \texorpdfstring{$K/\Qp$}{K/Qp} unramified}\label{subsec:K-unramified-representations-via-induction}
Let~$K/\Qp$ be an unramified finite extension (contained in our fixed algebraic closure~$\Qpbar$).
A pleasant consequence of working in the generality of $L$-groups of unramified reductive groups is that we can obtain a version of Theorem~\ref{v3-thm:existence-of-semisimple-etale-phi-module-with-invariants} for representations of~$\Gal_{K}$ with essentially no extra work.
To see this, we follow~\cite[\S 9.4]{zbMATH06991335}, to which we refer for more details.
(Strictly speaking, this reference makes some additional assumptions on the reductive groups, see~\cite[Hyp.~9.1.1]{zbMATH06991335}; but this hypothesis is not used for any of the results that we reference).

Let~$\LH$ be the $L$-group of an unramified connected reductive group~$H$ over~$K$; exactly as in Section~\ref{subsec:L-groups}, this is a semidirect product   \[\LH = \Hhat \rtimes \Gal (L/K),\] where~$\Hhat$ is a Chevalley group over~$\Zp$, with maximal torus~$\That_{\Hhat}$ and Weyl group~$W_{\Hhat}$.

As usual, if $X$ is a set (or a group, or a group scheme\dots) with a
left action of $\Gal(L/K)$, then we denote by $\Ind_{\Gal(L/K)}^{\Gal(L/\Qp)} X$ the induced
set (group, group scheme\dots) consisting of functions $\Gal(L/\Qp) \to X$
that are $\Gal(L/K)$-equivariant.
For $\gamma\in\Gal(L/\Qp)$ we let $\ev_\gamma : \Ind_{\Gal(L/K)}^{\Gal(L/\Qp)} X \to X$
denote the evaluation map at $\gamma$.
If $Y$ is a set of representatives of $\Gal(L/K) \backslash \Gal(L/\Qp)$, the $(\ev_y)_{y \in Y}$ provide
a non-canonical isomorphism $\Ind_{\Gal(L/K)}^{\Gal(L/\Qp)} X \isoto X^{[K:\Qp]}$ of sets  (groups, group schemes\dots).

In particular, we
 define
 \[
(\Ghat,\That) \coloneq  \Ind_{\Gal(L/K)}^{\Gal(L/\Qp)} (\Hhat,\That_{\Hhat}),
\]so that~$\Ghat$ is a Chevalley group over~$\Zp$, equipped  with an action of~$\Gal(L/\Qp)$, and we can form its $L$-group $\LG$.
For any $\Zp$-algebra $A$, the homomorphism $\ev_1 : \Ghat (A) \to \Hhat (A)$ is
 $\Gal(L/K)$-equivariant, so it extends to a homomorphism $\ev_1 : \Ghat (A)\rtimes \Gal (L/K) \to \LH(A)$, $(g,\sigma)\mapsto(\ev_1 (g),\sigma)$.
 By~\cite[Lem.~9.4.1]{zbMATH06991335},  we have a bijection between equivalence classes of $L$-parameters~$\rho:\Gal_{\Qp}\to \LG(A)$ and  equivalence classes of $L$-parameters~$\rho_{K}:\Gal_{K}\to \LH(A)$
sending $\rho$ to $\rho_K = \ev_1 (\rho |_{\Gal_{K}})$.

If $A=\Zpbar$ then $\rho$ is crystalline if and only if $\rho_K$ is crystalline, and in this case the isomorphism
\begin{equation}\label{eqn:res-scalars-cocharacters}
  X_{*}(\That) \cong \bigoplus_{\kappa:\, K\into \Qpbar} X_{*} (\That_{\Hhat} \times_{\kappa} \Qpbar) \cong
  \Ind_{\Gal(L/K)}^{\Gal(L/\Qp)} X_{*}(\That_{\Hhat})
\end{equation}identifies the Hodge--Tate weights~$\mu(\rho)$ of~$\rho$ with the tuple of labelled Hodge--Tate weights~$\{\mu_{\kappa}(\rho_{K})\}_{\kappa: K\into\Qpbar}$ of~$\rho_K$ (see~\cite[Lem.~9.3.2, (9.4.4)]{zbMATH06991335}).
Given a tuple of cocharacters~$\{\lambda_{\kappa}\}_{\kappa:K\into\Qpbar}$ and a tuple~$\{w_{\kappa}\}_{\kappa:K\into\Qpbar}$  of elements of~$W_{\Hhat}$, we have a corresponding pair~$\lambda\in X_{*}(\That)$, $w\in W_{\Ghat}$ via~\eqref{eqn:res-scalars-cocharacters} and the analogous relation of Weyl groups.
We can therefore define
\[
  \tau(\{\lambda_{\kappa}\},\{w_{\kappa}\}):I_K\to\That_{\Hhat}(\Fpbar)\subset\LH(\Fpbar)
\]via $\tau(\{\lambda_{\kappa}\},\{w_{\kappa}\}) \coloneq  \ev_1(\tau(\lambda,w))$ (noting that~$I_K=I_{\Qp}$); we leave to the reader the  exercise of formulating this more explicitly in similar terms to~\eqref{v3-eq:formula-for-inertial-character}.

The following theorem is then an easy consequence of Theorem~\ref{v3-thm:existence-of-semisimple-etale-phi-module-with-invariants}.
\begin{thm}\label{thm:existence-of-semisimple-etale-phi-module-with-invariants-L-parameter-version}
  Suppose that~$\rho:\Gal_K\to \LH(\Zpbar)$ is a crystalline $L$-parameter with Hodge--Tate weights~$(\mu_\kappa(\rho))_{\kappa}$. 
  Then there exist dominant cocharacters $\lambda_\kappa\uparrow \mu_\kappa(\rho)$ and elements~$w_\kappa\in W_{\Hhat}$  such that
  \[\rhobar^{\semis}|_{I_K}\cong \tau(\{\lambda_{\kappa}\},\{w_{\kappa}\}).\qedhere\]
  \end{thm}
\section{A Serre weight conjecture for semisimple Galois representations}\label{sec:weight-Serre-conjecture}
The weight part of Serre's conjecture is a somewhat loosely defined set of conjectures involving mod~$p$ automorphic forms, the conjectural categorical $p$-adic Langlands program, and the reductions modulo~$p$ of potentially semistable (and in particular crystalline) Galois representations.
Since most of this material lies well outside the scope of the rest of this paper, we recall only enough here to be able to state Theorem~\ref{thm:Serre-weight-upper-bound}; we refer the reader to~\cite{zbMATH06991335} and~\cite{MR4812709} for overviews and further references.

\subsection{Unramified groups}\label{subsec:unramified-groups}
We will work throughout this section with Galois representations valued in a class of not-necessarily-connected reductive groups, namely the $L$-groups of unramified reductive groups.
We begin by very briefly recalling some material from~\cite[\S 9]{zbMATH06991335}, to which we refer for further details.

In contrast to the rest of the paper, in this section we let~$G$ be a connected reductive group over $\mathbf{Z}_p$.
By~\cite[Cor.~6.2.4]{MR3362641}, $G$ is quasi-split, and splits over a finite \'etale extension of~$\Zp$.
In particular, $G_{\Qp}$ is unramified, i.e.\ quasi-split and split over an unramified extension of $\Qp$, and we let~$L/\Qp$ denote its splitting field.
 We let $B$ be a Borel subgroup of $G$ with Levi subgroup $T\subseteq B$, so $T$ is a maximal torus of $G$.
We have a canonical identification of character groups $X^*(T_{\Qpbar}) \cong X^*(T_{\Fpbar})$, which is compatible with the (left) Galois action of $\Gal_\Qp \onto \Gal_\Fp$.
We write $X^*(T)$ for this Galois module.
Let $W \coloneq  \big( N(T) / T\big) (\Qpbar) \cong \big( N(T) / T\big) (\Fpbar)$ denote the Weyl group.

Then we may define the dual group~$\Ghat$, a split connected reductive group over~$\Zp$ with maximal torus~$\widehat{T}$, and the $L$-group \[\LG \coloneq  \Ghat \rtimes \Gal (L/\Qp),\] a reductive group over $\Zp$, as in Section~\ref{subsec:L-groups}.
We canonically identify~$X^{*}(T)$ with~$X_{*}(\widehat{T})$, and identify~$W$ with the Weyl group of $\That$.
As in Section~\ref{subsec:L-groups}, we write~$\gamma$ for the geometric Frobenius element.
We will only consider representations $\Gal_\Qp \to \LG (\Fpbar)$ which are $L$-parameters in the sense of Definition~\ref{defn:L-parameter}, i.e.\ representations which are compatible with the projections to $\Gal (L/\Qp)$.

\begin{rem}
  \label{rem:restriction-scalars}As explained in~\cite[\S 9.4]{zbMATH06991335} (see also Section~\ref{subsec:K-unramified-representations-via-induction} above), by considering groups of the form $G=\Res_{\cO_{K}/\Zp}H$, this framework also applies to the weight part of Serre's conjecture for split (or even unramified) connected reductive groups over unramified extensions~$K/\Qp$.
\end{rem}

\subsection{Serre weights}\label{subsec:Serre weights}

We assume from now on the following hypothesis, which is~\cite[Hyp.~9.1.1]{zbMATH06991335}.

\begin{hypothesis}
  \label{hyp:twisting-element}The group~$G^{\der}$ is simply connected, the centre $Z(G)$ is connected, and~$G$ has a \emph{twisting element}, i.e.\ an $\eta\in X^{*}(T)^{\Gal_{\Qp}}$ such that $\langle\eta,\alpha^{\vee}\rangle=1$ for all simple coroots~$\alpha^{\vee}$.
  We fix such an~$\eta$ from now on, and frequently regard it as an element of~$X_{*}(\widehat{T})$ via the canonical identification of $X^{*}(T)$ with $X_{*}(\That)$.
\end{hypothesis}
\begin{rem}\label{rem:twisting-elements}If~$G$ is split and $G^{\der}$ is simply connected, then a twisting element always exists.
  For example, if~$G=\GL_n$ then we may take $\eta=(n-1,n-2,\dots,0)$.
Twisting elements were introduced in~\cite{MR3444225}, and they are a manifestation of the ``$\rho$-shift'' which occurs when passing between the weights of cohomological automorphic representations and their associated Galois representations.
Making different choices of twisting elements amounts to twisting the Langlands correspondence by a character, and is unimportant for us.
\end{rem}

A \emph{Serre weight} is an isomorphism class of irreducible $\Fpbar$-representations of~$G(\Fp)$.
These are parameterized as follows.
Write~$S$ for the set of simple coroots of~$G$, and let
\[
X^{*}(T)_1=\{\lambda\in X^{*}(T)\mid 1\le \langle \lambda,\alpha^{\vee}\rangle \le p \text{ for all }\alpha^{\vee}\in S\},
\]
\[
X^{*}(T)_0=\{\lambda\in X^{*}(T)\mid \langle \lambda,\alpha^{\vee}\rangle =0 \text{ for all }\alpha^{\vee}\in S\}.
\]
By~\cite[Lem.~9.2.4]{zbMATH06991335}, there is a bijection between $X^{*}(T)_1/(p\gamma-1)X^{*}(T)_0$ and the set of Serre
weights, sending~$\lambda\in X^{*}(T)_1$ to the representation~$F_{\lambda-\eta}$ given by the restriction to~$G(\Fp)$ of the irreducible algebraic representation of~$G$ of highest weight~$\lambda-\eta$.

\subsection{An optimistic conjecture}\label{subsec:optimistic-semisimple-Serre-weight-conjecture}
Bearing in mind the canonical isomorphism between $X^{*}(T)$ and $X_{*}(\That)$, we make the following definition.
\begin{defn}
  \label{defn:some-sets-of-Serre-weights}
  If~$\rhobar:\Gal_{\Qp}\to \LG(\Fpbar)$ is semisimple, then we let~$W^{\cris}(\rhobar)$ be the set of Serre weights~$F_{\lambda-\eta}$ for which~$\rhobar$ admits a crystalline lift $\rho:\Gal_{\Qp}\to\LG(\Zpbar)$ with Hodge--Tate cocharacter~$\lambda$.
  (Considering twists by crystalline characters, it is easy to see that the existence of such a lift only depends on~$\lambda$ modulo $(p\gamma-1)X^{*}(T)_0$.) We let~$W^{\explicit}(\rhobar)$ be the set of $F_{\lambda-\eta}$ having the property that there exists a dominant cocharacter~$\lambda'\uparrow \lambda$ and~$w\in W$ such that $\rhobar|_{I_{\Qp}}\cong \tau(\lambda',w)$.
\end{defn}

\begin{rem}
  \label{rem:need-not-be-p-restricted}Note that~$\lambda'$ in Definition~\ref{defn:some-sets-of-Serre-weights} need not be contained in $X^{*}(T)_{1}$.
\end{rem}
\begin{thm}
  \label{thm:Serre-weight-upper-bound}
  For any semisimple representation $\rhobar:\Gal_{\Qp}\to \LG(\Fpbar)$, we have $W^{\cris}(\rhobar)\subseteq W^{\explicit}(\rhobar)$.
\end{thm}
\begin{proof}
  This is immediate from Theorem~\ref{v3-thm:existence-of-semisimple-etale-phi-module-with-invariants}. 
\end{proof}
It is expected that there is a set of Serre weights~$W(\rhobar)$ associated to any~$\rhobar$ (semisimple or otherwise), and that this set satisfies a form of local-global compatibility, in the sense that the Serre weights associated to a global mod~$p$ Galois representation $\rbar:\Gal_{\Q}\to\LG(\Fpbar)$ are precisely~$W(\rbar|_{\Gal_{\Qp}})$. As before, we refer the reader to~\cite{zbMATH06991335} for more details, and here
we informally let~$W(\rhobar)$ denote this conjectural ``true'' set of Serre weights associated to~$\rhobar$.
One basic expected property of~$W(\rhobar)$ is that we have an inclusion $W(\rhobar)\subseteq W^{\cris}(\rhobar)$; this is known in the usual global settings in which the weight part of Serre's conjecture is considered, as a consequence of local-global compatibility for the Galois representations associated to automorphic forms.
In general this inclusion will be strict, but following~\cite[Conj.~7.5.3]{zbMATH06991335} (which is the case~$G=\GL_n$), we cautiously propose the following conjecture.
\begin{optimistic-conj}
  \label{optimistic-conj:crystalline-Serre-weight}
  If~$\rhobar:\Gal_{\Qp}\to\LG(\Fpbar)$ is semisimple, then~$W(\rhobar)=W^{\cris}(\rhobar)=W^{\explicit}(\rhobar)$.
\end{optimistic-conj}

\subsection{Comparison to earlier conjectures}\label{subsec:comparison-to-GHS}
Given the expectation that~$W(\rhobar)\subseteq W^{\cris}(\rhobar)$, Theorem~\ref{thm:Serre-weight-upper-bound} gives an upper bound for the set of Serre weights~$W(\rhobar)$; such a result is often referred to as a ``weight elimination result''.
We believe that Theorem~\ref{thm:Serre-weight-upper-bound} recovers or improves upon all such results in the literature. 
In the case~$G=\Res_{L/\Qp}\GL_{2}$, the weight part of Serre's conjecture is known, and we have~$W(\rhobar)=W^{\explicit}(\rhobar)$. For general unramified groups~$G$, a predicted set of Serre weights is defined in~\cite[Defn.~9.2.5]{zbMATH06991335}, generalizing the case~$G=\GL_n$ which was considered by Herzig in~\cite{MR2541127}.
We denote this set of weights by $\WHer(\rhobar)$ (it is called~$W^{?}(\rhobar)$ in \cite{zbMATH06991335}).
It is expected that $\WHer(\rhobar)$ should be the correct set of weights if~$\rhobar$ is sufficiently generic (but not otherwise).

The definition of $\WHer(\rhobar)$ is in terms of the reduction modulo~$p$ of Deligne--Lusztig representations;
the recent work~\cite{le2025genericdecompositionsdelignelusztigrepresentations} finds explicit formulas for these reductions under what are expected to be optimal genericity conditions.
It is then an exercise, which we carry out in Proposition~\ref{prop:generic-agreement-GHS} below, to compare the sets $W^{\explicit}(\rhobar)$ and $\WHer(\rhobar)$ under an appropriate genericity condition on~$\rhobar$. In the case~$G=\GL_n$ (without an explicit genericity condition) this goes back to~\cite[Prop.~6.28]{MR2541127}.

\begin{defn}
  \label{defn:h-eta}We let~$h_{\eta}=\max_{\alpha\in \Phi^+}\langle \alpha,\eta\rangle$, where~$\Phi^+$ is the set of positive roots of~$\Ghat$; equivalently, $h_{\eta}$ is the maximal height of the highest coroot of an irreducible sub-root system.
\end{defn}
In particular, if~$G=\GL_n$, we have $h_{\eta}=n-1$.
\begin{defn}\label{defn:rhobar-generic}
  We say that a semisimple representation $\rhobar:\Gal_{\Qp}\to\LG(\Fpbar)$ is \emph{generic} if we can write $\rhobar|_{I_{\Qp}}\cong \tau(\lambda_0 ,w)$ with $h_{\eta}+1<\langle \alpha,\lambda_0 \rangle <p-h_{\eta}-1$ for all positive roots~$\alpha$, so that in particular~$\lambda_0 \in \Delta_p$.
\end{defn}

\begin{lem}
  \label{lem:generic-implies-regular}
  If~$\rhobar:\Gal_{\Qp}\to\LG(\Fpbar)$ is generic, then $\Ghat^{\rhobar(I_{\Qp})}$ is a maximal torus in~$\Ghat$.
\end{lem}
\begin{proof}
 This follows easily from \cite[Lem.~10.1.10]{zbMATH06991335} and its proof, but for the convenience of the reader, we explain the argument in our language.
 By definition, we may assume that $\rhobar|_{I_{\Qp}}= \tau(\lambda_0 ,w)$ for $\lambda_0 \in \Delta_p$ as in Definition~\ref{defn:rhobar-generic}.
 As in the proof of \cite[Lem.~10.1.10]{zbMATH06991335}, our running hypothesis that~$Z(G)$ is connected guarantees that $\Ghat^{\rhobar(I_{\Qp})}$ is connected reductive, with Weyl group generated by the reflections~$s_{\alpha}$ fixing $\rhobar(I_{\Qp})$ pointwise.
 Fix~$d\ge 1$ with~$(w\gamma)^{d}=1$.
 Using again that~$Z(G)$ is connected (so that for each~$\alpha$ we have $\langle X_{*}(\That),\alpha\rangle=\Z$), it follows from~\eqref{v3-eqn:defn-of-ei} and~\eqref{v3-eq:18} that we need to show that for each (positive) root~$\alpha$, \[\sum_{n=0}^{d-1}p^{n}\langle \lambda_0 ,(w\gamma)^{n}\alpha\rangle\not\equiv
   0\pmod{p^d-1}.\]
 By Definition~\ref{defn:rhobar-generic} we have $|\langle\lambda_0 ,(w\gamma)^{n}\alpha\rangle|\le p-2$ for all~$\alpha,n$, so it suffices to show that $\sum_{n=0}^{d-1}p^{n}\langle \lambda_0 ,(w\gamma)^{n}\alpha\rangle\ne 0$.
 It suffices in turn to show that $\langle \lambda_0 ,\alpha\rangle\not\equiv 0\pmod{p}$; but this is immediate from Definition~\ref{defn:rhobar-generic}.
\end{proof}
\begin{lem}
  \label{lem:generic-iso-implies-weyl-conjugate}Suppose that~$\rhobar$ is generic, and write $\rhobar|_{I_{\Qp}}\cong \tau(\lambda_0 ,w)$ as in Definition~\ref{defn:rhobar-generic}.
  Then the set of~$\lambda'\in X_{*}(\That)$ such that there exists $w'\in W$ with $\tau(\lambda_0 ,w)\cong \tau(\lambda',w')$ is precisely the set of~$\lambda'$ which can be written as
  \begin{equation}\label{eqn:lambda0-w-orbit}
    \lambda'=\sigma^{-1}(\lambda_0 -\nu +p(w\gamma)\nu)
  \end{equation}
  for some $\sigma\in W$, $\nu\in X_{*}(\That)$.
\end{lem}
\begin{proof}
  By Lemma~\ref{lem:generic-implies-regular}, we have $\tau(\lambda_0 ,w)\cong \tau(\lambda',w')$ if and only if $\tau(\lambda_0 ,w)$ and $\tau(\lambda',w')$ are $N_{\Ghat}(\That)$-conjugate.
  By Lemma~\ref{v3-lem:inertia-representation-from-etale-phi}, this is equivalent to $u^{\lambda_0 }w$ and $u^{\lambda'}w'$ being twisted $\varphi$-conjugate in $N_{\Ghat}(\That)$, so the claim follows from~\eqref{v3-eq:21}.
\end{proof}

\begin{prop}
  \label{prop:generic-agreement-GHS}If~$\rhobar:\Gal_{\Qp}\to\LG(\Fpbar)$ is generic, then $W^{\explicit}(\rhobar)=\WHer(\rhobar)$. 
\end{prop}
\begin{proof}
By hypothesis, we can write $\rhobar|_{I_{\Qp}}\cong\tau(\lambda_0 ,w)$ for some~$\lambda_0$ as in Definition~\ref{defn:rhobar-generic}.
Suppose first that $F_{\lambda-\eta}\in W^{\explicit}(\rhobar)$. Bearing in mind Lemma~\ref{lem:generic-iso-implies-weyl-conjugate}, we see that there exists a dominant $\lambda'\uparrow \lambda$ such that~$\lambda'$ can be written in the form~\eqref{eqn:lambda0-w-orbit}.
By Proposition~\ref{prop:Bruhat-equivalent-uparrow}, there exist $\lambda_0' \in \overline{\Delta}_p$ and $\tu\le_{\Bru} \tw\in W_p$ such that $\lambda'=\tu(\lambda'_0 ), \lambda=\tw(\lambda'_0 )$, and $\tu(\Delta_p), \tw(\Delta_p)$ are dominant.

Then we have \begin{equation}\label{eqn:lambda-0-lambda-prime}\tu(\lambda'_0)=\lambda'=\sigma^{-1}(\lambda_0 -\nu +p(w\gamma)\nu).\end{equation}
We now argue as in~\cite[\S\S 5,6]{MR2541127} (see also the proof of~\cite[Lem.~2.4.9]{le2025weightserresconjecturecm}).
Let~$\sigma'\in W$ be such that $\sigma'\sigma^{-1}w\gamma\nu$ is dominant, let $h_{\nu}=\max_{\alpha\in\Phi}\langle \alpha,\nu \rangle$, and let~$\alpha_0$ be a root such that $\langle \alpha_0,\sigma'\sigma^{-1}w\gamma\nu \rangle=h_{\nu}$; evidently, we may suppose that~$\alpha_0$ is a highest root (i.e., $\alpha_0$ is the highest root of an irreducible sub-root system), and is in particular dominant. Since~$\alpha_0$ is dominant, and $\sigma'(\lambda')\le \lambda'\le\lambda$, we see that \begin{equation*}\label{h-nu-h-eta-inequality}\begin{split}
      ph_{\eta} & \ge  \langle \alpha_0,\lambda \rangle \\
      & \ge \langle \alpha_0,\lambda' \rangle \\
  & \ge \langle \alpha_0,\sigma'(\lambda') \rangle \\
  & = ph_{\nu} + \langle \alpha_0,\sigma'\sigma^{-1}\lambda_0  \rangle - \langle \alpha_0,\sigma'\sigma^{-1}\nu  \rangle \\
  & = ph_{\nu} + \langle (\sigma')^{-1} \sigma^{-1}\alpha_0,\lambda_0  \rangle - \langle (\sigma')^{-1} \sigma^{-1}\alpha_0,\nu  \rangle \\
  &> ph_{\nu} - (p-h_{\eta}-1) -h_{\nu},
\end{split}
\end{equation*}
where in the first inequality we used that~$\lambda\in X^{*}(T)_{1}$, and in the last inequality we used the assumption that~$\rhobar$ is generic.
This rearranges to give $h_{\nu}< h_{\eta} +1$, so $h_{\nu}\le h_{\eta}$.
Since~$\rhobar$ is generic, it follows that $\lambda_0 -\nu\in \Delta_p$.
It then follows from~\eqref{eqn:lambda-0-lambda-prime} that~$\lambda_0'\in\Delta_p$. Write $x\coloneq u^{p(\sigma^{-1}w\gamma)\nu}\sigma^{-1}$, and write~$\Omega$ for the stabilizer of~$\Delta_p$ in $pX^{*}(T)\rtimes W$.
Then the equality $\tu(\lambda_0')=x(\lambda_0 -\nu)$ implies that $\tu^{-1}x\in\Omega$.
Noting that we may write
\begin{equation}
  \label{eq:obscure-eqn-to-compare-to-LLH}
  u^{\lambda_0 }(w^{-1})^{\gamma^{-1}}=u^{\lambda_0-\nu}(w^{-1}\sigma)^{\gamma^{-1}} (u^{(\sigma^{-1}w\gamma)\nu}\sigma^{-1})^{\gamma^{-1}}
\end{equation}
it is now immediate from (the proof of) \cite[Prop.~2.5.2]{le2025weightserresconjecturecm}\footnote{At the beginning of \cite[\S 2.5]{le2025weightserresconjecturecm}, it is assumed that~$G$ is the restriction of scalars of a split group; however, this assumption is not used until after \cite[Prop.~2.5.2]{le2025weightserresconjecturecm}.} that~$F_{\lambda-\eta}\in \WHer(\rhobar)$.
 To see this, we explain the necessary translation between our notation and conventions and those of~\cite{le2025weightserresconjecturecm}.
First, our~$\lambda_0$ is their~$\tw(\taubar)(0)$ (and our hypothesis that $\rhobar$ is generic implies the hypothesis made on~$\tw(\taubar)(0)$ in the statement of~\cite[Prop.~2.5.2]{le2025weightserresconjecturecm}).
Their~$\omega$ is our~$\lambda_0-\nu$, and their~$s$ is our $(w^{-1})^{\gamma^{-1}}$ (to see this, compare their definition of~$\tau(w,s)$, which is found in~\cite[\S 2.4]{MR4549091}, to our~$\tau(\lambda,w)$).
Finally, their $\tu,\tw$ are our~$x^{\gamma^{-1}},(\tw\tu^{-1}x)^{\gamma^{-1}}$ respectively (up to rescaling by~$p$).
The condition in~\cite{le2025weightserresconjecturecm} that (in their notation) $t_{\mu}s\in t_{\omega}W\tu$ is provided by~\eqref{eq:obscure-eqn-to-compare-to-LLH}, so we see that $F_{\lambda-\eta}\in \WHer(\rhobar)$.

Thus $W^{\explicit}(\rhobar)\subseteq \WHer(\rhobar)$.
The converse implication is simpler, and follows easily by reversing the above logic; so $W^{\explicit}(\rhobar)=\WHer(\rhobar)$, as required.
\end{proof}

\begin{rem}\label{rem:this-bound-is-optimal}
  Definition~\ref{defn:rhobar-generic} seems to be the natural condition under which one can hope for the equality $W^{\explicit}(\rhobar)=\WHer(\rhobar)$. For example, for~$G=\GL_2$ and~$p\ge 5$, we have~$h_{\eta}=1$, and if we allowed $\langle \alpha,\lambda_0 \rangle=p-2$, then we could take (in the usual notation for cocharacters of~$\GL_2$) $\lambda_0=(p-1,1)$ and~$w=1$, giving $\tau(\lambda_0 ,w)=1\oplus\omega^{-1}$.
    Then~$W^{\operatorname{explicit}}(\rhobar)$ contains $3$ weights~$F_{\lambda-\eta}$, namely those with~$\lambda=(p,0)$, $(1,0)$, and~$(p-1,1)$, while by definition (see~\cite[\S 6]{MR2541127}) the set $\WHer(\rhobar)$ only consists of ``$p$-regular'' weights, which excludes~$\lambda=(p,0)$.
\end{rem}
\begin{rem}
  \label{rem:compare-to-LLLM-elimination}
  In the case~$G=\Res_{\cO_L/\Zp}\GL_n$, standard local-global compatibility results combined with Theorem~\ref{thm:Serre-weight-upper-bound} and Proposition~\ref{prop:generic-agreement-GHS} prove a ``weight elimination'' result for automorphic forms on unitary groups, showing that the set of weights of a generic~$\rhobar$ is contained in $\WHer(\rhobar)$.
  Here~$h_{\eta}=n-1$, so the generic~$\rhobar$ are those that can be written as $\tau(\lambda_0 ,w)$ with $n<\langle \alpha,\lambda_0 \rangle <p-n$ for all positive roots~$\alpha$.
  This improves on~\cite[Thm.~6.1]{le2025genericdecompositionsdelignelusztigrepresentations} (see also~\cite[Thm.~3.1.3]{le2025weightserresconjecturecm}), which proves the same result under the stronger hypothesis that $2n+1<\langle \alpha,\lambda_0 \rangle <p-(2n+1)$ for all positive roots~$\alpha$.
  In view of~\cite[Rem.~10]{MR3963975}, it seems unlikely that any further significant improvement is possible.
\end{rem}

\begin{rem}\label{rem:explicit-Serre-weights}
  For non-generic semisimple~$\rhobar$, ~\cite[Defn.~9.3.10]{zbMATH06991335} defines
  an ``explicit'' set of Serre weights~$W_{\operatorname{expl}}(\rhobar)$ by a rather involved (and somewhat ad hoc) recursive procedure; the idea was to write down an approximation to the set~$W^{\cris}(\rhobar)$, including all the Serre weights for which one could exhibit crystalline lifts, or for which considerations of functoriality or the Breuil--M\'ezard conjecture suggested that such lifts should exist.
  Since~$W_{\operatorname{expl}}(\rhobar)$ is expected to be a subset of~$W^{\cris}(\rhobar)$, while $W^{\cris}(\rhobar)\subseteq W^{\explicit}(\rhobar)$ by Theorem~\ref{thm:Serre-weight-upper-bound}, we expect that $W^{\explicit}(\rhobar)=W^{\cris}(\rhobar)$ whenever $W^{\operatorname{explicit}}(\rhobar)=W_{\operatorname{expl}}(\rhobar)$.

  We find it encouraging that at least for~$\GL_3$, the sets $W_{\operatorname{expl}}(\rhobar)$ and~$W^{\explicit}(\rhobar)$ coincide.
  For example, if~$\rhobar|_{I_{\Qp}}$ is trivial, then $W^{\explicit}(\rhobar)$ 
  consists of precisely~$4$ Serre weights.
    Indeed, since~$\GL_3$ is split, we can simultaneously arrange conditions~\eqref{item:lambda-is-dominant} and~\eqref{item:lambda-is-M-dominant-and-w-in-WM} in Theorem~\ref{v3-thm:existence-of-semisimple-etale-phi-module-with-invariants}, so we need to identify the $F_{\lambda-\eta}$ such that there exists a dominant cocharacter~$\lambda'\uparrow \lambda$ with~$\omega^{\lambda'}=1$.
    In the case~$p=2$ we have~ $\omega^{\lambda '}=1$ for all~$\lambda '$, so we may assume that~$p>2$ from now on.
  Write~$\lambda=(a,b,c)$, and take~$\eta=(2,1,0)$; since~$\lambda\in X^{*}(T)_1$ we have $1\le a-b,b-c\le p$.
  If~$\lambda'=\lambda$ then we must have $a-b=b-c=p-1$ and~$(p-1)|c$, which gives the Serre weight~$F_{(p-1,0,1-p)-\eta}=F_{(p-3,-1,1-p)}$.
  Otherwise we must have~$a-c>p$, and $\lambda'=(c+p,b,a-p)$, so that $\omega^{\lambda'}=1$ is equivalent to $(a,b,c)\equiv (1,0,-1)\pmod{p-1}$.
  This gives us 
  the possibilities~$F_{(p,0,-1)-\eta}=F_{(p-2,-1,-1)}$, $F_{(p,p-1,-1)-\eta}=F_{(p-2,p-2,-1)}$, and $F_{(p,0,-p)-\eta}=F_{(p-2,-1,-p)}$.
  These are precisely the weights in~$W_{\operatorname{expl}}(\rhobar)$, as explained in~\cite[Ex.~8.2.9]{zbMATH06991335}; note that already in this simple case, the definition of~$W_{\operatorname{expl}}(\rhobar)$ and its computation are rather more complicated than that of~$W^{\explicit}(\rhobar)$, and involve the consideration of ``obvious'', ``shadow'', and ``obscure'' weights.
\end{rem}

\begin{rem}
  \label{rem:non-ss}It would also be interesting to see what our techniques can say about non-semisimple representations~$\rhobar$.
As remarked in~\ref{v3-para:small-weight-shape}, our results recover the key ingredients in the resolution of the weight part of Serre's conjecture for~$\GL_2$ in~\cite{GLSII};
we do not expect that our results will have such strong consequences for general groups~$G$, but we anticipate that they can at least be used to prove non-trivial results for other groups of small rank.
\end{rem}

\appendix
\section{Tannakian formalism}\label{sec:appendix-Tannakian}
\subsection{Exact tensor functors}
Let $V$ be a discrete valuation ring, and let $G$ be a flat, affine, finite type
group scheme over $V.$
We denote by $\Rep_V G$ the category of representations of $G$ on finite free $V$-modules,
and write $\Ind\Rep_V G$ for the $\Ind$-category of $\Rep_V G,$ obtained by formally adjoining
filtered colimits to $\Rep_V G.$ There is a fully faithful functor from $\Ind\Rep_V G$ to the category of $V$-flat
$V[G]$-comodules, and this functor is an equivalence since any $V$-flat $V[G]$-comodule is a filtered colimit of $V$-finite
(and hence $V$-free) $V[G]$-comodules (see \cite[Cor.~to Prop.~2]{MR231831}).
It follows that we may think of objects of $\Ind\Rep_V G$ as $V[G]$-comodules. For example,
the Hopf algebra $\O_G$, thought of as a $G$-comodule via $G$ acting on itself by right translation, is in $\Ind\Rep_V G.$

Now let $\CC$ be a $V$-linear exact tensor category (see e.g.\ \cite[A.5]{imai2024tannakian}).
By definition, a \emph{$G$-object in $\CC$} (equivalently, an \emph{object of~$\CC$ with $G$-structure}) is a
$V$-linear exact tensor functor
$$ \omega: \Rep_V G \rightarrow \CC.$$
We denote the $\Ind$-extended functor by the same notation
$$ \omega: \Ind\Rep_V G \rightarrow\Ind\CC.$$
Below, we denote by $\omega_{\Triv}$ the forgetful tensor functor from
$\Rep_V G$ to the category of finite free~$V$-modules,
and use the same notation with a subscript for its base changes,
such as $\omega_{\Triv,R}$ or $\omega_{\Triv,\O}.$
We will refer to $G$-objects by prepending a $G$ to the name of the category, e.g.\ $G$-Breuil--Kisin modules.

\subsection{Torsors}\label{subsec:torsors} Let $R$ be a $V$-algebra.
Suppose that $\CC$ is the category of finite projective $R$-modules.
Then $\Ind\CC$ admits a fully faithful exact tensor functor to the category of
$R$-modules. In particular, for $N$ in $\Ind\CC,$ we can regard $\omega(N)$ as an
$R$-module. With this convention, $P\coloneq \Spec(\omega(\O_G))$ is a $G$-torsor over~$\Spec R$, 
and one can reconstruct $\omega$ from $P$ by setting $\omega(W) = (W\times P)/G$;
see \cite[Thm.~4.8]{Broshi}, \cite[Thm.~19.5.1]{ScholzeWeinsteinBerkeley}.

\begin{para}
The equivalence between these two points of view holds much more generally,
as we now explain, following \cite[Appendix A]{imai2024tannakian}.
Let $(\cX,\O)$ be a site ringed in $V$-algebras.
A vector bundle on $(\cX,\O)$ is a sheaf of $\O$-modules which is locally isomorphic to $\O^n,$
for some $n.$ We denote by $\Vect(\cX,\O)$ the category of vector bundles on $(\cX,\O).$

Let $G_{\O}$ denote the sheaf on $\cX$ induced by $G.$
That is, $G_{\O}(U) \coloneq G(\O(U)).$
A $G$-torsor on $(\cX,\O)$ is a sheaf on $\cX$ with a $G_{\O}$-action which is locally simply transitive.

Denote by $G$-$\Vect(\cX,\O)$ the category of
$V$-linear exact tensor functors $\omega: \Rep_V G \rightarrow \Vect(\cX,\O).$ We call $\omega$ locally
trivial if locally on $\cX,$ it is isomorphic to the functor $\omega_{\Triv,\O}$
given by $\omega_{\Triv,\O}(W)(U) = W\otimes_V \O(U).$
We denote by $G$-$\Vect(\cX,\O)^{lt}$ the category of locally trivial $V$-linear exact tensor functors.

If $P$ is a $G$-torsor on $(\cX,\cO)$ then we write $\omega_P\in G$-$\Vect(\cX,\O)^{lt}$ for the functor given by $\omega_P(W)=(W\times P)/G$.
\end{para}

\begin{prop}\label{prop:equivfunctors} The functor $P\mapsto \omega_P$ is 
  an equivalence of categories between the category of $G$-torsors on $(\cX,\O)$
and the category $G$-$\Vect(\cX,\O)^{lt}.$
\end{prop}

\begin{proof} By \cite[Prop.~A.17]{imai2024tannakian}, it suffices to show that
in the terminology of ${\it loc.~cit.},$
$G_{\O}$ is  reconstructable in $(\cX,\O)$; i.e., the natural map $G_{\cO}\to  \underline{\Aut}(\omega_{\Triv,\O})$ is an isomorphism.
We may check this locally on~$\cX$, so it suffices in turn to show that if  $R$ is a $V$-algebra, and  $\omega_{\Triv,R}$
denotes the functor sending $U$ in $\Rep_V G$ to the $R$-module $U\otimes_V R$, then the natural map $G(R) \rightarrow \Aut(\omega_{\Triv,R})$ is an isomorphism.

Recall that by \cite[Thm.~1.1]{Broshi} there is a faithful representation~$W$ of~$G$, together with a tensorial construction $t(W)$ and a locally split line bundle $L\subset t(W)$, such that $G=\Aut(W,L)$ as subfunctors of $GL(W)$.
Given $\alpha\in  \Aut(\omega_{\Triv,R})$, it follows that $\alpha_W\in \GL(W\otimes_VR)$ is an element of~$G(R)$, and it is easy to see that this gives an inverse to the natural map $G(R) \rightarrow \Aut(\omega_{\Triv,R})$, as required.
\end{proof}

\begin{cor}\label{cor:equivfunctors} 
Let $X$ be a stack on $V$-algebras, and let $(\cX,\O)$ be the topos associated to affine schemes over $X$, endowed with a topology between the syntomic and fpqc topologies.  Then any $\omega$ in $G$-$\Vect(\cX,\O)$ is locally trivial.
In particular the category of $G$-torsors on $(\cX,\O)$ is equivalent to $G$-$\Vect(\cX,\O).$
\end{cor}

\begin{proof} 
Forming the relative spectrum gives a map $f:P := \underline{\Spec}_X(\omega(\O_G)) \to X$. By descent from the case of affine schemes considered above, this map is a $G$-torsor. In particular, this map is also a syntomic cover as the same holds for the trivial torsor $G \times X \to X$. To finish, it is enough to observe that $f^* \circ \omega$ is canonically trivial, which is clear: for $U \in \Rep_V(G)$, we have a canonical isomorphism $U \otimes_V \mathcal{O}_G \simeq U_0 \otimes_V \mathcal{O}_G$ in $\Ind(\Rep_V(G))$, where $U_0$ denotes $U$ endowed with the trivial $G$-action, so application of $\omega(-)$ gives the claim.
\end{proof}

\begin{para}Write $\omega \mapsto P_{\omega}$ for the quasi-inverse to the functor of Proposition~\ref{prop:equivfunctors}. Suppose that
$$\omega, \omega': \Rep_V G \rightarrow \Vect(\cX,\O)$$
are locally trivial exact tensor functors, and that $f:\omega' \rightarrow \omega$
is a morphism of tensor functors (hence an isomorphism, by the argument of  \cite[Ch.~I, Prop.~5.2.3]{SaavedraTannakiennes}, since every object of $\Rep_V G$ is dualizable).
Then we obtain a morphism $P_f: P_{\omega'} \rightarrow P_{\omega}$
between the corresponding torsors.

The following lemma is used in the body of the paper to identify $G$-Breuil--Kisin modules with certain $G$-torsors
with Frobenius structures.
\end{para}

\begin{lem}\label{lem:Frobtorsor} Let $\varphi:(\cX,\cO)\to(\cX,\cO)$
be a morphism of locally ringed sites, with corresponding pullback functor $$ \varphi^*: \Vect(\cX,\O) \rightarrow \Vect(\cX,\O). $$
Let $\omega$ be in $G$-$\Vect(\cX,\O)^{lt}$. 
Then there is a canonical isomorphism
$$ P_{\varphi^*\circ\omega} \simeq \varphi^*P_{\omega}.$$
In particular, a morphism $\varphi^*\circ\omega \rightarrow \omega$ induces a morphism
$$ \varphi^*P_{\omega} \simeq P_{\varphi^*\circ\omega} \rightarrow P_{\omega}. $$
\end{lem}
\begin{proof} By the equivalence of Proposition \ref{prop:equivfunctors}, the $G$-torsor $\varphi^*P_{\omega}$
has the form $P_{\omega'}$ for some $\omega'$ in $G$-$\Vect(\cX,\O)^{lt},$ and it suffices to construct a
canonical isomorphism $ \omega'\simeq \varphi^*\circ\omega.$

For $W$ in $\Rep_V G,$ $R$ a $V$-algebra, and any $G$-torsor $P$ over $R,$ it is easy to see that the
formation of $(W\times P)/G$ commutes with arbitrary base change $R \rightarrow R';$ indeed this may be
checked locally in the fppf topology of $R,$ where one can reduce to the case $P \simeq G.$ Thus we have
\[ \omega'(W) = (W\times \varphi^*P_{\omega})/G \simeq \varphi^*((W\times P_{\omega})/G) \simeq \varphi^*(\omega(W))
= (\varphi^*\circ\omega)(W).\qedhere
\]
\end{proof}

\subsection{Restriction of scalars}\label{subsec:restriction-of-scalars}
We now study how $G$-objects behave under restriction of scalars.
Thus suppose that $V'/V$ is a finite flat extension of discrete valuation rings, and let $G'$ be a flat, affine, finite type group scheme over $V'.$ We set $G = \Res_{V'/V} G'.$
Let $(\cX,\O)$ be as above and let $\O' = \O\otimes_V V'.$

\begin{prop}\label{prop:Gobjectequiv} There is a canonical equivalence
$$ G'\text{-}\Vect(\cX,\O')^{lt} \simeq G\text{-}\Vect(\cX,\O)^{lt}.$$
\end{prop}
\begin{proof} By Proposition \ref{prop:equivfunctors}, it suffices to show that
$G'$-torsors on $(\cX,\O')$ are equivalent to $G$-torsors on $(\cX,\O).$
As
$$ G_{\O}(U) = G(\O(U)) \simeq G'(\O(U)\otimes_V V') = G'_{\O'}(U),$$
this is a tautology.
\end{proof}

\begin{para}\label{para:extension-scalars-categories} In certain situations, we can reformulate the above proposition purely in categorical terms.
For a $V$-linear category $\CC,$ we denote by $\CC\otimes_V V'$ the category consisting of pairs
$(F,\iota)$ where $F$ is an object of $\CC$ and $\iota$ is a map of $V$-algebras $V' \rightarrow \End F.$
Then we have the following.
\end{para}
\begin{cor}\label{cor:Gobjectequiv} Let $X$ be a stack over $V$ which has a flat  cover by a formally smooth, formally finite type formal $V$-scheme $Y$, and let $(\cX,\O)$ be the quasi-syntomic site of $X.$ Then there is a natural equivalence
$$ G'\text{-}(\Vect(\cX,\O)\otimes_V V') \simeq G\text{-}\Vect(\cX,\O).$$
\end{cor}
\begin{proof}We will deduce this from Proposition \ref{prop:Gobjectequiv}.
  Bearing in mind Corollary \ref{cor:equivfunctors}, we see that it suffices to show that the natural functor
  \begin{equation}\label{eqn:V-to-Vprime-Vect}
    G'\text{-}\Vect(\cX,\O') \to G'\text{-}(\Vect(\cX,\O)\otimes_V V')
  \end{equation}
  is an equivalence.

  An object of $\Vect(\cX,\O)\otimes_V V'$ is a vector bundle $M$ on $X$
equipped with an action of $V'.$ 
We first remark that, as an $\O_X\otimes_V V'$-module, $M$
is locally projective. Indeed, this can be checked on the cover $Y$ of $X.$
Since $M$ is locally free over $\O_Y$, it locally has depth equal to $\dim Y.$
As $\O_Y\otimes_V V'$ is a sheaf of regular rings which is finite flat over $\O_Y,$ this implies that $M$
is a projective $\O_Y\otimes_V V'$-module.

Next, note that the inclusion
$$ \Vect(\cX,\O') \subset \Vect(\cX,\O)\otimes_V V' $$
identifies $\Vect(\cX,\O')$ with the subcategory of locally projective $\O_X\otimes_V V'$-modules $M$
that are locally free. Local freeness is equivalent to asking that for any point $\gm$ of $Y,$
the rank of $M$ at the finitely many maximal ideals of the semi-local ring
$(\O_Y)_{\gm}\otimes_V V'$ is constant.

In order to see that~\eqref{eqn:V-to-Vprime-Vect} is an equivalence,
it is therefore sufficient to note that if $\omega$ is in $G'\text{-}(\Vect(\cX,\O)\otimes_V V'),$ then for any $W$ in $\Rep_{V'} G',$
$\omega(W)$ has constant rank equal to $\rk_{V'} W.$
To see this, note that after pulling back to $Y$ and then to the residue field $\kappa(\mathfrak n)$ of a maximal ideal
$\mathfrak n$ of $(\O_Y)_{\gm}\otimes_V V',$ we obtain an exact tensor functor
$\omega_{\mathfrak n}:\Rep_{V'}G'\to \Vect_{\kappa(\mathfrak n)}.$ We have $\omega_{\mathfrak n}(\wedge^{i}W)\cong \wedge^{i}\omega_{\mathfrak n}(W)$ for all~$i$ (cf.~\cite[Rem.~4.2]{Broshi}); considering the cases $i=\rk_{V'}W,$ $i=\rk_{V'}W+1$, we see that $\dim_{\kappa(\mathfrak n)}\omega_{\mathfrak n}(W)=\rk_{V'}W,$ as required.
\end{proof}

\begin{para}\label{para:scalarsnotn} 

Suppose $\CC$ is an exact $V$-linear tensor category. We will sometimes
slightly abuse notation and define a $G'$-object in $\CC$ to be a $G'$-object of $\CC\otimes_V V'.$
Then Corollary~\ref{cor:Gobjectequiv} can be reformulated as
$$ G'\text{-}\Vect(\cX,\O) \simeq G\text{-}\Vect(\cX,\O).$$

\end{para}
\printbibliography
\setlength{\leftskip}{\parindent}
\setlength{\parindent}{0pt}
\end{document}